\documentclass[12pt,reqno]{amsart} 
\usepackage{amssymb}
\usepackage{mathrsfs}
\usepackage{enumitem}
\usepackage[all]{xy}

\setlist[itemize]{leftmargin=5mm,labelindent=-20mm,itemsep=6pt}

\newlist{enuma}{enumerate}{3}
\setlist[enuma]{label=\rm(\alph*),ref=\alph*,itemsep=6pt,leftmargin=*}
\newlist{enumr}{enumerate}{3}
\setlist[enumr]{label=\rm(\roman*),ref=\roman*,align=left,itemsep=6pt,leftmargin=*}
\newlist{enumA}{enumerate}{3}
\setlist[enumA]{label=\rm(\Alph*),ref=\Alph*,itemsep=6pt,leftmargin=*}
\newlist{enumI}{enumerate}{3}
\setlist[enumI]{label=\rm(\Roman*),ref=\Roman*,align=left,itemsep=6pt,leftmargin=*}
\newlist{enum1}{enumerate}{3}
\setlist[enum1]{label=\rm(\arabic*),ref=\arabic*,itemsep=6pt,leftmargin=8mm}

\newcommand{\dsg}[2][]{[\![#2]\!]_{#1}}

\usepackage[usenames]{color}
\definecolor{darkgreen}{rgb}{0,0.7,0}
\definecolor{newcolor}{rgb}{0.9,0,0}
\definecolor{newercolor}{rgb}{0.2,0,1}
\definecolor{orange}{rgb}{1,0.7,0}

\newcommand{\newsect}[1]{\bigskip\section{#1}\setcounter{table}{0}}

\newcommand{\newsubb}[2]{\mbfx{\smallskip\subsection{#1}\label{#2}%
\leavevmode\noindent}}

\newcommand{\boldd}[1]{{\mathversion{bold}\textbf{#1}}}
\newcommand{\mbfx}[1]{{\boldmath #1\unboldmath}}

\newcommand{\bmid}{\mathrel{\big|}}

\newlength{\short}
\newcommand{\dbl}[2]{\renewcommand{\arraystretch}{1.0}%
\begin{array}{c}\rule{0pt}{12pt}#1\\#2\end{array}}

\newcommand{\halfup}[1]{\raisebox{2.2ex}[0pt]{$#1$}}

\newcommand{\0}[1]{\overset{\circ}{#1}}
\renewcommand{\2}[2]{\underset{#1}{#2}}
\newcommand{\3}[1]{\hat{#1}}
\newcommand{\4}[1]{\widebar{#1}}
\newcommand{\5}[1]{\widehat{#1}}
\newcommand{\7}{_c}
\let\8=\scriptstyle
\newcommand{\9}[1]{{}^{#1}\!}

\newcommand{\Par}[1]{\mathfrak{P}_{#1}}
\newcommand{\bb}{\mathfrak{b}}
\newcommand{\pp}{\mathfrak{p}}

\def\pair[#1,#2]{[\hskip-1.5pt[#1,#2]\hskip-1.5pt]}
\def\trp[#1,#2,#3]{[\hskip-1.5pt[#1,#2,#3]\hskip-1.5pt]}

\SelectTips{cm}{10} \UseTips   

\let\oldcirc=\circ
\renewcommand{\circ}{\mathchoice
    {\mathbin{\scriptstyle\oldcirc}}{\mathbin{\scriptstyle\oldcirc}}
    {\mathbin{\scriptscriptstyle\oldcirc}}
    {\mathbin{\scriptscriptstyle\oldcirc}}}

\def\beq#1\eeq{\begin{equation*}#1\end{equation*}}
\def\beqq#1\eeqq{\begin{equation}#1\end{equation}}

\numberwithin{equation}{section}
\def\theequation{\arabic{equation}}

\newtheorem{Thm}{Theorem}[section]
\newtheorem{Prop}[Thm]{Proposition}
\newtheorem{Cor}[Thm]{Corollary}
\newtheorem{Lem}[Thm]{Lemma}
\newtheorem{Defi}[Thm]{Definition}

\newtheorem{Hyp}[Thm]{Hypotheses}
\newtheorem{Not}[Thm]{Notation}
\newtheorem{Ex}[Thm]{Example}

\newtheorem{Th}{Theorem}

\newcommand{\widebar}[1]
      {\overset{{\mskip3mu\leaders\hrule height0.4pt\hfill\mskip3mu}}{#1}
      \vphantom{#1}}

\newcounter{let} 
\loop\stepcounter{let}
\expandafter\edef\csname cal\alph{let}\endcsname%
{\noexpand\mathcal{\Alph{let}}}
\ifnum\thelet<26\repeat

\newcommand{\tdef}[2][]{\expandafter\newcommand\csname#2\endcsname%
{#1\textup{#2}}}
\tdef{Iso}   \tdef{Aut}    \tdef{Out}    \tdef{Inn}    \tdef{Hom}
\tdef{End}   \tdef{Inj}    \tdef{map}    \tdef{Ker}    \tdef{Ob}
\tdef{Mor}   \tdef{Res}    \tdef{Id}     \tdef{Fr}     \tdef{Spin} 
\tdef{rk}    \tdef{conj}   \tdef{incl}   \tdef{proj}   \tdef{diag} 
\tdef{trf}   \tdef{Sol}    \tdef{He}     \tdef{Sz}     \tdef{cj}
\tdef{Rep}   \tdef{pr}    \tdef{Inndiag} \tdef{Outdiag}  \tdef{expt}
\tdef{supp}  \tdef{Isom}   \tdef{ord}    \tdef{Coker}
\tdef[_]{typ} \tdef[^]{op} \tdef[^]{ab}   

\newcommand{\htt}{\textup{ht}}

\newcommand{\scal}{_{\textup{sc}}}
\newcommand{\fdef}[1]{\expandafter\newcommand\csname#1\endcsname%
{\mathfrak{#1}}}
\fdef{X}  \fdef{red}  \fdef{foc}  \fdef{hyp}  \fdef{Lie} \fdef{Y}

\newcommand{\bbdef}[1]{\expandafter\newcommand%
\csname#1\endcsname{\mathbb{#1}}}
\bbdef{C} \bbdef{F} \bbdef{R} \bbdef{Z} \bbdef{N} \bbdef{Q} \bbdef{K}

\newcommand{\itdef}[1]{\expandafter\newcommand\csname#1\endcsname%
{\textit{#1}}}
\itdef{PSL}  \itdef{PSU}  \itdef{SL}  \itdef{SU}  \itdef{GL} \itdef{GU}
\itdef{Sp}   \itdef{PSp} \itdef{PSO} \itdef{SO}   \itdef{SD} \itdef{PGU} 
\itdef{PGL}

\newcommand{\GO}{O}

\newcommand{\fpbar}{\widebar{\F}_p}
\newcommand{\fqobar}{\widebar{\F}_{q_0}}

\newcommand{\gee}{\varepsilon}
\newcommand{\sminus}{\smallsetminus}
\newcommand{\lie}[3]{\def\test{#2}\def\tst{G}\ifx\test\tst{{}^{#1}#2_{#3}}
\else{{}^{#1}\!#2_{#3}}\fi}
\renewcommand{\*}{\,\lower6pt\hbox{\Large{\textup{*}}}\,}
\newcommand{\syl}[2]{\textup{Syl}_{#1}(#2)}
\newcommand{\sylp}[1]{\syl{p}{#1}}
\newcommand{\ordp}{\ord_p}
\renewcommand{\Im}{\textup{Im}}
\newcommand{\autf}{\Aut_{\calf}}

\newcommand{\outf}{\Out_{\calf}}

\newcommand{\homf}{\Hom_{\calf}}

\newcommand{\defeq}{\overset{\textup{def}}{=}}
\newcommand{\mxtwo}[4]{\bigl(\begin{smallmatrix}#1&#2\\#3&#4\end{smallmatrix}%
\bigr)}
\newcommand{\mxthree}[9]{\left(\begin{smallmatrix}#1&#2&#3\\#4&#5&#6\\
#7&#8&#9\end{smallmatrix}\right)}
\newcommand{\mxfoura}[8]{\left(\begin{smallmatrix}#1&#2&#3&#4\\#5&#6&#7&#8}
\newcommand{\mxfourb}[8]{\\#1&#2&#3&#4\\#5&#6&#7&#8\end{smallmatrix}\right)}

\let\emptyset=\varnothing
\renewcommand{\:}{\colon}
\newcommand{\pcom}{{}^\wedge_p}

\newcommand{\ploc}{{}_{(p)}}
\newcommand{\nsg}{\trianglelefteq}
\newcommand{\sd}[1]{\overset{{#1}}{\rtimes}}
\newcommand{\til}[1]{\widetilde{#1}}

\renewcommand{\gg}{\mathbb{G}}
\newcommand{\hh}{\mathbb{H}}
\newcommand{\gen}[1]{{\langle}#1{\rangle}}
\newcommand{\Gen}[1]{{\bigl\langle}#1{\bigr\rangle}}
\newcommand{\norm}[1]{{\Vert}#1{\Vert}}
\newcommand{\dg}{_{\textup{diag}}}

\newcommand{\longleft}[1]{\;{\leftarrow%
\count255=0 \loop \mathrel{\mkern-6mu}%
    \relbar\advance\count255 by1\ifnum\count255<#1\repeat}\;}
\newcommand{\longright}[1]{\;{\count255=0 \loop \relbar\mathrel{\mkern-6mu}%
    \advance\count255 by1\ifnum\count255<#1\repeat\rightarrow}\;}
\newcommand{\Right}[2]{\overset{#2}{\longright#1}}
\newcommand{\RIGHT}[3]{\mathrel{\mathop{\kern0pt\longright#1}
        \limits^{#2}_{#3}}}
\newcommand{\Left}[2]{{\buildrel #2 \over {\longleft#1}}}
\newcommand{\LEFT}[3]{\mathrel{\mathop{\kern0pt\longleft#1}\limits^{#2}_{#3}}
}
\newcommand{\dRIGHT}[3]{\mathrel{%
   \mathop{\vcenter{\baselineskip=0pt\hbox{$\kern0pt\longright#1$}%
   \hbox{$\kern0pt\longright#1$}}}\limits^{#2}_{#3}}}
\newcommand{\LRIGHT}[3]{\mathrel{%
   \mathop{\vcenter{\baselineskip=0pt\hbox{$\kern0pt\longleft#1$}%
   \hbox{$\kern0pt\longright#1$}}}\limits^{#2}_{#3}}}
\newcommand{\RLEFT}[3]{\mathrel{%
   \mathop{\vcenter{\baselineskip=0pt\hbox{$\kern0pt\longright#1$}%
   \hbox{$\kern0pt\longleft#1$}}}\limits^{#2}_{#3}}}
\newcommand{\onto}[1]{\;{\count255=0 \loop \relbar\mathrel{\mkern-6mu}%
    \advance\count255 by1
    \ifnum\count255<#1 \repeat \twoheadrightarrow}\;}
\newcommand{\Onto}[2]{\overset{#2}{\onto#1}}

\title{Automorphisms of fusion systems of finite simple groups of Lie type}

\author{Carles Broto}
\address{Departament de Matem\`atiques, Universitat Aut\`onoma de 
Barcelona, E--08193 Bellaterra, Spain}
\email{broto@mat.uab.es}
\thanks{C. Broto is partially supported by MICINN grant MTM2010-20692 and 
MINECO grant MTM2013-42293-P}
\author{Jesper M. M\o{}ller}
\address{Matematisk Institut, Universitetsparken 5, DK--2100 K\o{}benhavn, 
Denmark}
\email{moller@math.ku.dk}

\thanks{J. M\o{}ller is partially supported by the Danish National 
Research Foundation through the Centre for Symmetry and Deformation 
(DNRF92) and by Villum Fonden through the project Experimental Mathematics 
in Number Theory, Operator Algebras, and Topology.}
\author{Bob Oliver}
\address{Universit\'e Paris 13, Sorbonne Paris Cit\'e, LAGA, UMR 7539 du CNRS, 
99, Av. J.-B. Cl\'ement, 93430 Villetaneuse, France.}
\email{bobol@math.univ-paris13.fr}
\thanks{B. Oliver is partially supported by UMR 7539 of the CNRS, and by 
project ANR BLAN08-2\_338236, HGRT}

\thanks{All three authors wish to thank K\o benhavns Universitet, the 
Universitat Aut\`onoma de Barcelona, and especially the Centre for Symmetry 
and Deformation in Copenhagen, for their hospitality while much of this 
work was carried out.}


\begin{document}

\begin{abstract} For a finite group $G$ of Lie type and a prime $p$, we 
compare the automorphism groups of the fusion and linking systems of $G$ at 
$p$ with the automorphism group of $G$ itself. When $p$ is the defining 
characteristic of $G$, they are all isomorphic, with a very short list of 
exceptions. When $p$ is different from the defining characteristic, the 
situation is much more complex, but can always be reduced to a case where 
the natural map from $\Out(G)$ to outer automorphisms of the fusion or 
linking system is split surjective. This work is motivated in part by 
questions involving extending the local structure of a group by a group of 
automorphisms, and in part by wanting to describe self homotopy 
equivalences of $BG\pcom$ in terms of $\Out(G)$.
\end{abstract}

\maketitle

When $p$ is a prime, $G$ is a finite group, and $S\in\sylp{G}$, the 
\emph{fusion system} of $G$ at $S$ is the category $\calf_S(G)$ whose 
objects are the subgroups of $S$, and whose morphisms are those 
homomorphisms between subgroups induced by conjugation in $G$.  In this 
paper, we are interested in comparing automorphisms of $G$, when $G$ is a 
simple group of Lie type, with those of the fusion system of $G$ at a Sylow 
$p$-subgroup of $G$ (for different primes $p$).

Rather than work with automorphisms of $\calf_S(G)$ itself, it turns out to 
be more natural in many situations to study the group 
$\Out\typ(\call_S^c(G))$ of outer automorphisms of the \emph{centric 
linking system} of $G$. We refer to Section \ref{s:tame} for the 
definition of $\call_S^c(G)$, and to Definition \ref{d:aut(L)} for precise 
definitions of $\Out(S,\calf_S(G))$ and $\Out\typ(\call_S^c(G))$.  These 
are defined in such a way that there are natural homomorphisms
	\[ \Out(G) \Right4{\kappa_G} \Out\typ(\call_S^c(G)) 
	\Right4{\mu_G} \Out(S,\calf_S(G))
	\qquad\textup{and}\qquad \4\kappa_G=\mu_G\circ\kappa_G \,. \]
For example, if $S$ controls fusion in $G$ (i.e., if $S$ has a normal 
complement), then $\Out(S,\calf_S(G))=\Out(S)$, and $\4\kappa_G$ is induced 
by projection to $S$. The fusion system $\calf_S(G)$ is 
\emph{tamely realized by $G$} if $\kappa_G$ is split surjective, and is 
\emph{tame} if it is tamely realized by some finite group $G^*$ where 
$S\in\sylp{G^*}$ and $\calf_S(G)=\calf_S(G^*)$. Tameness plays an important 
role in Aschbacher's program for shortening parts of the proof of the 
classification of finite simple groups by classifying simple fusion systems 
over finite $2$-groups. We say more about this later in the introduction, 
just before the statement of Theorem \ref{ThT}.

By \cite[Theorem B]{BLO1}, $\Out\typ(\call_S^c(G))\cong\Out(BG\pcom)$:  the 
group of homotopy classes of self homotopy equivalences of the 
$p$-completed classifying space of $G$.  Thus one of the motivations for 
this paper is to compute $\Out(BG\pcom)$ when $G$ is a finite simple group 
of Lie type (in characteristic $p$ or in characteristic different from 
$p$), and compare it with $\Out(G)$.  

Following the notation used in \cite{GLS3}, for each prime $p$, we let 
$\Lie(p)$ denote the class of finite groups of Lie type in characteristic 
$p$, and let $\Lie$ denote the union of the classes $\Lie(p)$ for all 
primes $p$. (See Definition \ref{d:Lie} for the precise definition.)  We 
say that $G\in\Lie(p)$ is of \emph{adjoint type} if $Z(G)=1$, and is of 
\emph{universal type} if it has no nontrivial central extensions which are 
in $\Lie(p)$. For example, for $n\ge2$ and $q$ a power of $p$, $\PSL_n(q)$ 
is of adjoint type and $\SL_n(q)$ of universal type.

Our results can be most simply stated in the ``equi-characteristic case'':  
when working with $p$-fusion of $G\in\Lie(p)$.

\begin{Th} \label{ThE}
Let $p$ be a prime. Assume that $G\in\Lie(p)$ and is of 
universal or adjoint type, and also that 
$(G,p)\not\cong(\Sz(2),2)$.  Fix $S\in\sylp{G}$.  Then the composite 
homomorphism 
	\[ \4\kappa_G\: \Out(G) \Right5{\kappa_G} \Out\typ(\call_S^c(G)) 
	\Right5{\mu_G} \Out(S,\calf_S(G)) \]
is an isomorphism, and $\kappa_G$ and $\mu_G$ are isomorphisms except when 
$G\cong\PSL_3(2)$. 
\end{Th}

\begin{proof} Assume $G$ is of adjoint type.  When $G\not\cong\GL_3(2)$, 
$\mu_G$ is an isomorphism by \cite[Proposition 4.3]{limz-odd}\footnotemark\ 
or \cite[Theorems C \& 6.2]{limz}.  The injectivity of 
$\4\kappa_G=\mu_G\circ\kappa_G$ (in all cases) is shown in Lemma 
\ref{kappa_inj}. The surjectivity of $\kappa_G$ is shown in 
Proposition \ref{rk>2-tame} when $G$ has Lie rank at least three, and in 
Proposition \ref{rk1equi} when $G$ has Lie rank $1$ and $G\not\cong\Sz(2)$. 
When $G$ has Lie rank $2$, $\kappa_G$ is onto (when $G\not\cong\SL_3(2)$) 
by Proposition \ref{ThA:most}, \ref{ThA:n=4}, \ref{G2(2)-tame}, 
\ref{3D4(2)-tame}, or \ref{Tits-tame}. (See Notation 
\ref{G-setup-Q}\eqref{not8=} for the definition of Lie rank used here.)

If $G$ is of universal type, then 
by Proposition \ref{p:Gu->Ga}, $G/Z(G)\in\Lie(p)$ is of adjoint type 
where $Z(G)$ has order prime to $p$. 
Also, $\Out(G)\cong\Out(G/Z(G))$ by \cite[Theorem 2.5.14(d)]{GLS3}. 
Hence $\calf_S(G)\cong\calf_S(G/Z(G))$ and 
$\call_S^c(G)\cong\call_S^c(G/Z(G))$; and $\kappa_G$ and/or $\4\kappa_G$ 
is an isomorphism if $\kappa_{G/Z(G)}$ and/or $\4\kappa_{G/Z(G)}$, 
respectively, is an isomorphism. 
\end{proof}

\footnotetext{Steve Smith recently pointed out to the third author an error 
in the proof of this proposition. One can get around this problem either 
via a more direct case-by-case argument (see the remark in the middle of 
page 345 in \cite{limz-odd}), or by applying \cite[Theorem C]{O-Ch}. The 
proof of the latter result uses the classification of finite simple groups, 
but as described by Glauberman and Lynd \cite[\S\,3]{GLynd}, the proof in 
\cite{O-Ch} (for odd $p$) can be modified to use an earlier result of 
Glauberman \cite[Theorem A1.4]{Glauberman2}, and through that avoiding the 
classification.}

When $G=\PSL_3(2)$ and $p=2$, $\Out(G)\cong\Out(S,\calf_S(G))\cong C_2$, 
while $\Out\typ(\call_S^c(G))\cong C_2^2$.  When $G=\Sz(2)\cong 
C_5\sd{}C_4$ and $p=2$, $\Out(G)=1$, while 
$\Out\typ(\call_S^c(G))\cong\Aut(C_4)\cong C_2$.  Thus these groups are 
exceptions to Theorem \ref{ThE}.

To simplify the statement of the next theorem, for finite groups $G$ and 
$H$, we write $G\sim_pH$ to mean that there are Sylow subgroups 
$S\in\sylp{G}$ and $T\in\sylp{H}$, together with an isomorphism 
$\varphi\:S\Right2{\cong}T$ which induces an isomorphism of categories 
$\calf_S(G)\cong\calf_T(H)$ (i.e., $\varphi$ is fusion preserving in the 
sense of Definition \ref{d:aut(L)}). 

\begin{Th} \label{ThX}
Fix a pair of distinct primes $p$ and $q_0$, and a group $G\in\Lie(q_0)$ of 
universal or adjoint type. Assume that the Sylow $p$-subgroups of $G$ are 
nonabelian. Then there is a prime $q_0^*\ne{}p$, and a group 
$G^*\in\Lie(q_0^*)$ of universal or adjoint type, respectively, such that 
$G^*\sim_pG$ and $\kappa_{G^*}$ is split surjective.  
If, furthermore, $p$ is odd or $G^*$ has universal type, then $\mu_{G^*}$ is an 
isomorphism, and hence $\4\kappa_{G^*}$ is also split surjective.
\end{Th}

\begin{proof} 
\noindent\textbf{Case 1: } Assume $p$ is odd and $G$ is of universal type.  
Since $\mu_G$ is an isomorphism by \cite[Theorem C]{limz-odd}, $\kappa_G$ 
or $\kappa_{G^*}$ is (split) surjective if and only if $\4\kappa_G$ or 
$\4\kappa_{G^*}$ is.

By Proposition \ref{G-cases-odd}, we can choose a prime $q_0^*$ and a group 
$G^*\in\Lie(q_0^*)$ such that either 
\begin{enumerate}[label*=(1.\alph*) ,leftmargin=12mm]

\item $G^*\cong\gg(q^*)$ or ${}^2\gg(q^*)$, for some $\gg$ with Weyl group 
$W$ and $q^*$ a power of $q_0^*$, and has a $\sigma$-setup which satisfies the 
conditions in Hypotheses \ref{G-hypoth-X} and \ref{G-hypoth-X2}, and 
\medskip
\begin{enumerate}[label=(1.a.\arabic*) ,leftmargin=6mm,itemsep=4pt]
\item $-\Id\notin W$ and $G^*$ is a Chevalley group, or 
\item $-\Id\in W$ and $q^*$ has even order in $\F_p^\times$, or
\item $p\equiv3$ (mod $p$) and $p|(q^*-1)$; or
\end{enumerate}

\smallskip

\item $p=3$, $q_0^*=2$, $G\cong\lie3D4(q)$ or $\lie2F4(q)$ for $q$ some 
power of $q_0$, and $G^*\cong\lie3D4(q^*)$ or $\lie2F4(q^*)$ for $q^*$ some 
power of $2$.


\end{enumerate}
Also (by the same proposition), if $p=3$ and $G^*=F_4(q^*)$, then we can 
assume $q_0^*=2$. 

In case (1.b), $\4\kappa_{G^*}$ is split surjective by Proposition 
\ref{p:lowrank}. In case (1.a), it is surjective by Proposition 
\ref{kappa_onto}. In case (1.a.1), $\4\kappa_{G^*}$ is split by Proposition 
\ref{Ker(kappa)}(b,c). In case (1.a.3), $\4\kappa_{G^*}$ is split by Proposition 
\ref{Ker(kappa)}(b). In case (1.a.2), if $G^*$ is a Chevalley group, then 
$\4\kappa_{G^*}$ is split by Proposition \ref{Ker(kappa)}(c).

This leaves only case (1.a.2) when $G^*$ is a twisted group. The only 
irreducible root systems which have nontrivial graph automorphisms and for 
which $-\Id\in W$ are those of type $D_n$ for even $n$. Hence 
$G^*=\Spin_{2n}^-(q^*)$ for some even $n\ge4$. 
By the last statement in Proposition \ref{G-cases-odd}, $G^*$ is one of the 
groups listed in Proposition \ref{list-simp}, and so $q^n\equiv-1$ (mod $p$). 
Hence $\4\kappa_{G^*}$ is split surjective by Example \ref{ex:nonsplit2}(a), 
and we are done also in this case.

\smallskip

\noindent\textbf{Case 2: }  Now assume $p=2$ and $G$ is of universal type. 
By Proposition \ref{G-cases-2}, there is an odd prime $q_0^*$, a group 
$G^*\in\Lie(q_0^*)$, and $S^*\in\sylp{G^*}$, such that 
$\calf_S(G)\cong\calf_{S^*}(G^*)$ and $G^*$ has a $\sigma$-setup which 
satisfies Hypotheses \ref{G-hypoth-X} and \ref{G-hypoth-X2}. By the same 
proposition, if $G^*\cong G_2(q^*)$, then we can arrange that $q^*=5$ or 
$q_0^*=3$. If $G^*\cong G_2(5)$, then by Propositions \ref{p:G2(5)} and 
\ref{lim1-exceptional}, $G^*\sim_2 G_2(3)$, $\4\kappa_{G_2(3)}$ is split 
surjective, and $\mu_{G_2(3)}$ is injective.

In all remaining cases (i.e., $G^*\not\cong G_2(q^*)$ or $q_0^*=3$), 
$\4\kappa_{G^*}$ is split surjective by Proposition 
\ref{Ker(kappa)}(a). By Proposition \ref{lim1-classical} or 
\ref{lim1-exceptional}, $\mu_{G^*}$ is injective, and hence $\kappa_{G^*}$ 
is also split surjective.

\smallskip

\noindent\textbf{Case 3: } Now assume $G$ is of adjoint type. Then $G\cong 
G_u/Z$ for some $G_u\in\Lie(q_0)$ of universal type and $Z\le Z(G_u)$. By 
Proposition \ref{p:Gu->Ga}, $Z=Z(G_u)$ and has order prime to $q_0$. 

By Case 1 or 2, there is a prime $q_0^*\ne{}p$ and a group 
$G_u^*\in\Lie(q_0^*)$ of universal type such that 
$G_u^*\sim_pG_u$ and $\kappa_{G_u^*}$ is split surjective. Also, $G_u^*$ is 
$p$-perfect by definition of $\Lie(q_0^*)$ (and since $q_0^*\ne{}p$), and 
$H^2(G_u^*;\Z/p)=0$ by Proposition \ref{p:Gu->Ga}. Set 
$G^*=G_u^*/Z(G_u^*)$. By Proposition \ref{univ2adj}, with 
$G_u^*/O_{p'}(G_u^*)$ in the role of $G$, $\kappa_{G^*}$ is also split 
surjective.

It remains to check that $G\sim_pG^*$. Assume first that $G_u$ and $G_u^*$ 
have $\sigma$-setups which satisfy Hypotheses \ref{G-hypoth-X}. Fix 
$S\in\sylp{G_u}$ and $S^*\in\sylp{G_u^*}$, and a fusion preserving 
isomorphism $\varphi\:S\Right2{}S^*$ (Definition \ref{d:aut(L)}(a)). By 
Corollary \ref{c:Z*}, $Z(\calf_S(G_u))=O_p(Z(G_u))$ and 
$Z(\calf_{S^*}(G_u^*))=O_p(Z(G_u^*))$. Since $\varphi$ is fusion 
preserving, it sends $Z(\calf_S(G_u))$ onto $Z(\calf_{S^*}(G_u^*))$, and 
thus sends $O_p(Z(G_u))$ onto $O_p(Z(G_u^*))$. Hence $\varphi$ induces a 
fusion preserving isomorphism between Sylow subgroups of $G=G_u/Z(G_u)$ and 
$G^*=G_u^*/Z(G_u^*)$.

The only cases we considered where $G$ or $G^*$ does not satisfy Hypotheses 
\ref{G-hypoth-X} were those in case (1.b) above. In those cases, 
$G\cong\lie2F4(q)$ or $\lie3D4(q)$ and $G^*\cong\lie2F4(q^*)$ or 
$\lie3D4(q^*)$ for some $q$ and $q^*$, hence $G$ and $G^*$ are also of 
universal type ($d=1$ in the notation of \cite[Lemma 14.1.2(iii)]{Carter}), 
and so there is nothing more to prove. 
\end{proof}

The last statement in Theorem \ref{ThX} is \emph{not} true in general when 
$G^*$ is of adjoint type. For example, if $G^*\cong\PSL_2(9)$, $p=2$, and 
$S^*\in\syl2{G^*}$, then $\Out(G^*)\cong\Out\typ(\call^c_{S^*}(G^*))\cong 
C_2^2$, while $\Out(S^*,\calf_{S^*}(G^*))\cong C_2$. By comparison, if 
$\til{G}^*\cong\SL_2(9)$ is the universal group, then 
$\Out(\til{S}^*,\calf_{\til{S}^*}(\til{G}^*))\cong C_2^2$, and 
$\kappa_{\til{G}^*}$ and $\mu_{\til{G}^*}$ are isomorphisms.

As noted briefly above, a fusion system $\calf_S(G)$ is called \emph{tame} 
if there is a finite group $G^*$ such that $G^*\sim_pG$ and $\kappa_{G^*}$ 
is split surjective. In this situation, we say that $G^*$ \emph{tamely 
realizes} the fusion system $\calf_S(G)$. By \cite[Theorem B]{AOV1}, if 
$\calf_S(G)$ is not tame, then some extension of it is an ``exotic'' fusion 
system; i.e., an abstract fusion system not induced by any finite group.  
(See Section \ref{s:tame} for more details.) The original goal of this 
paper was to determine whether all fusion systems of simple groups of Lie 
type (at all primes) are tame, and this follows as an immediate consequence 
of Theorems \ref{ThE} and \ref{ThX}.  Hence this approach cannot be used to 
construct new, exotic fusion systems.

Determining which simple fusion systems over finite $2$-groups are tame, 
and tamely realizable by finite simple groups, plays an important role in 
Aschbacher's program for classifying simple fusion systems over $2$-groups 
(see \cite[Part II]{AKO} or \cite{A-gfit}). Given such a fusion system 
$\calf$ over a $2$-group $S$, and an involution $x\in S$, assume that the 
centralizer fusion system $C_\calf(x)$ contains a normal quasisimple 
subsystem $\cale\nsg C_\calf(x)$. If $\cale$ is tamely realized by a finite 
simple group $K$, then under certain additional assumptions, one can show 
that the entire centralizer $C_\calf(x)$ is the fusion system of some 
finite extension of $K$. This is part of our motivation for looking at this 
question, and is also part of the reason why we try to give as much 
information as possible as to which groups tamely realize which fusion 
systems.

\begin{Th} \label{ThT}
For any prime $p$ and any $G\in\Lie$ of universal or adjoint type, the 
$p$-fusion system of $G$ is tame. If the Sylow $p$-subgroups of $G$ are 
nonabelian, or if $p$ is the defining characteristic and 
$G\not\cong\Sz(2)$, then its fusion system is tamely realized by some 
other group in $\Lie$. 
\end{Th}

\begin{proof} If $S\in\sylp{G}$ is abelian, then the $p$-fusion in $G$ is 
controlled by $N_G(S)$, and $\calf_S(G)$ is tame by Proposition 
\ref{p:constrained}.  If $p=2$ and $G\cong\SL_3(2)$, then the fusion system 
of $G$ is tamely realized by $\PSL_2(9)$. In all other cases, the claims 
follow from Theorems \ref{ThE} and \ref{ThX}. 
\end{proof}

We have stated the above three theorems only for groups of Lie type, but in 
fact, we proved at the same time the corresponding results for the Tits 
group:

\begin{Th} \label{ThTits}
Set $G=\lie2F4(2)'$ (the Tits group). Then for each prime $p$, the 
$p$-fusion system of $G$ is tame. If $p=2$ or $p=3$, then $\kappa_G$ is an 
isomorphism.
\end{Th}

\begin{proof} The second statement is shown in Proposition \ref{Tits-tame} 
when $p=2$, and in Proposition \ref{p:lowrank} when $p=3$. When $p>3$, the 
Sylow $p$-subgroups of $G$ are abelian 
($|G|=2^{11}\cdot3^3\cdot5^2\cdot13$), so $G$ is tame by Proposition 
\ref{p:constrained}(b). 
\end{proof}

As one example, if $p=2$ and $G=\PSL_2(17)$, then $\kappa_G$ is not 
surjective, but $G^*=\PSL_2(81)$ (of adjoint type) has the same $2$-fusion 
system and $\kappa_{G^*}$ is an isomorphism \cite[Proposition 7.9]{BLO1}.  
Also, $\4\kappa_{G^*}$ is non-split surjective with kernel generated by the 
field automorphism of order two by \cite[Lemma 7.8]{BLO1}.  However, if we 
consider the universal group $\til{G}^*=\SL_2(81)$, then 
$\4\kappa_{\til{G}^*}$ and $\kappa_{\til{G}^*}$ are both isomorphisms by 
\cite[Proposition 5.5]{BL} (note that $\Out(S,\calf)=\Out(S)$ in this 
situation). 

As another, more complicated example, consider the case where $p=41$ and 
$G=\Spin_{4k}^-(9)$. By \cite[(3.2)--(3.6)]{Steinberg-aut}, 
$\Outdiag(G)\cong C_2$, and $\Out(G)\cong C_2\times C_4$ is generated by a 
diagonal element of order 2 and a field automorphism of order $4$ (whose 
square is a graph automorphism of order $2$). Also, $\mu_G$ is an 
isomorphism by Proposition \ref{lim1-classical}, so $\kappa_G$ is 
surjective, or split surjective, if and only if $\4\kappa_G$ is. We refer 
to the proof of Lemma \ref{classical(III.3)}, and to Table \ref{tb:III.2a} 
in that proof, for details of a $\sigma$-setup for $G$ in which the 
normalizer of a maximal torus contains a Sylow $p$-subgroup $S$. In 
particular, $S$ is nonabelian if $k\ge41$. By Proposition 
\ref{Ker(kappa)}(d) and Example \ref{ex:nonsplit2}(a,b), when $k\ge41$, 
$\4\kappa_G$ is surjective, $\4\kappa_G$ is split (with 
$\Ker(\4\kappa_G)=\Outdiag(G)$) when $k$ is odd, and $\4\kappa_G$ is not 
split ($\Ker(\4\kappa_G)\cong C_2\times C_2$) when $k$ is even. By 
Proposition \ref{OldPrA3}(c), when $k$ is even, $G\sim_{41}G^*$ for 
$G^*=\Spin_{4k-1}(9)$, and $\kappa_{G^*}$ is split surjective (with 
$\Ker(\kappa_{G^*})=\Outdiag(G^*)$) by Proposition \ref{Ker(kappa)}(c). 
Thus $\calf_S(G)$ is tame in all cases: tamely realized by $G$ itself when 
$k$ is odd and by $\Spin_{4k-1}(9)$ when $k$ is even. Note that when $k$ is 
odd, since the graph automorphism does not act trivially on any Sylow 
$p$-subgroup, the $p$-fusion system of $G$ (equivalently, of 
$\SO_{4k}^-(9)$) is not isomorphic to that of the full orthogonal group 
$\GO_{4k}^-(9)$, so by \cite[Proposition A.3(b)]{BMO1}, 
it is not isomorphic to that of $\Spin_{4k+1}(9)$ either 
(nor to that of $\Spin_{4k-1}(9)$ since its Sylow $p$-subgroups are 
smaller).

Other examples are given in Examples \ref{ex:nonsplit} and 
\ref{ex:nonsplit2}. For more details, in the situation of Theorem 
\ref{ThX}, about for which groups $G$ the homomorphism $\4\kappa_G$ is 
surjective or split surjective, see Propositions \ref{kappa_onto} 
and \ref{Ker(kappa)}. 

The following theorem was shown while proving Theorem \ref{ThX}, and 
could be of independent interest. The case where $p$ is odd was handled by 
Gorenstein and Lyons \cite[10-2(1,2)]{GL}. 

\begin{Th} \label{ThJ}
Assume $G\in\Lie(q_0)$ is of universal type for some odd prime $q_0$. Fix 
$S\in\syl2{G}$. Then $S$ contains a unique abelian subgroup of maximal 
order, \emph{except} when $G\cong\Sp_{2n}(q)$ for some $n\ge1$ and some 
$q\equiv\pm3$ (mod $8$). 
\end{Th}

\begin{proof} Assume $S$ is nonabelian; otherwise there is nothing to 
prove. Since $q_0$ is odd, and since the Sylow $2$-subgroups of 
$\lie2G2(3^{2k+1})$ are abelian for all $k\ge1$ \cite[Theorem 8.5]{Ree-G2}, 
$G$ must be a Chevalley or Steinberg group. If $G\cong\lie3D4(q)$, then 
(up to isomorphism) $S\in\syl2{G_2(q)}$ by \cite[Example 4.5]{BMO1}. 
So we can assume that $G\cong{}^r\gg(q)$ for some odd prime power $q$, 
some $\gg$, and $r=1$ or $2$. 

If $q\equiv3$ (mod $4$), then choose another prime power $q^*\equiv1$ (mod 
$4$) such that $v_2(q^*-1)=v_2(q+1)$ (where $v_2(m)=k$ if $2^k|n$ and 
$2^{k+1}\nmid n$). Then $\4{\gen{q^*}}=\4{\gen{-q}}$ and 
$\4{\gen{-q^*}}=\4{\gen{q}}$ as closed subgroups of $(\Z_2)^\times$. By 
\cite[Theorem A]{BMO1} (see also Theorem \ref{OldThA}), there is a group 
$G^*\cong{}^t\gg(q^*)$ (where $t\le2$) whose 2-fusion system is equivalent 
to that of $G$. We can thus assume that $q\equiv1$ (mod $4$). So by Lemma 
\ref{case(III.1)}, $G$ has a $\sigma$-setup which satisfies Hypotheses 
\ref{G-setup-X}. By Proposition \ref{p:AcharS}(a), $S$ contains a unique 
abelian subgroup of maximal order, unless $q\equiv5$ (mod $8$) and 
$G\cong\Sp_{2n}(q)$ for some $n\ge1$. 
\end{proof}

In fact, when $G\cong\Sp_{2n}(q)$ for $q\equiv\pm3$ (mod $8$), then 
$S\in\syl2{G}$ is isomorphic to $(Q_8)^n\rtimes P$ for 
$P\in\syl2{\Sigma_n}$, $S$ contains $3^n$ abelian subgroups of maximal 
order $2^{2n}$, and all of them are conjugate to each other in $N_G(S)$.


The main definitions and results about tame and reduced fusion systems are 
given in Section \ref{s:tame}.  We then set up our general notation for 
finite groups of Lie type in Sections \ref{s:not} and \ref{s:aut}, deal 
with the equicharacteristic case in Section \ref{s:=}, and with the cross 
characteristic case in Sections \ref{s:X1} and \ref{s:X2}. The kernel of 
$\mu_G$, and thus the relation between automorphism groups of the fusion 
and linking systems, is handled in an appendix. 

\smallskip

\textbf{Notation: } In general, when $\calc$ is a category and 
$x\in\Ob(\calc)$, we let $\Aut_\calc(x)$ denote the group of automorphisms 
of $x$ in $\calc$.  When $\calf$ is a fusion system and $P\in\Ob(\calf)$, 
we set $\outf(P)=\autf(P)/\Inn(P)$.  

For any group $G$ and $g\in{}G$, $c_g\in\Aut(G)$ denotes the automorphism 
$c_g(h)=ghg^{-1}$.  Thus for $H\le G$, $\9gH=c_g(H)$ and $H^g=c_g^{-1}(H)$.  
When $G,H,K$ are all subgroups of a group $\Gamma$, we define 
	\begin{align*} 
	T_G(H,K) &= \{g\in{}G\,|\,\9gH\le K\} \\
	\Hom_G(H,K) &= \{c_g\in\Hom(H,K) \,|\, g\in{}T_G(H,K)\}\,. 
	\end{align*}
We let $\Aut_G(H)$ be the group $\Aut_G(H)=\Hom_G(H,H)$.  When $H\le 
G$ (so $\Aut_G(H)\ge\Inn(H)$), we also write $\Out_G(H)=\Aut_G(H)/\Inn(H)$.  


\bigskip
\newpage

\newsect{Tame and reduced fusion systems}
\label{s:tame}

Throughout this section, $p$ always denotes a fixed prime. Before defining 
tameness of fusion systems more precisely, we first recall the definitions 
of fusion and linking systems of finite groups, and of automorphism groups 
of fusion and linking systems. 

\begin{Defi} \label{d:F&L}
Fix a finite group $G$ and a Sylow $p$-subgroup $S\le G$. 
\begin{enuma}

\item The \emph{fusion system} of $G$ is the category $\calf_S(G)$ whose 
objects are the subgroups of $S$, and where 
$\Mor_{\calf_S(G)}(P,Q)=\Hom_G(P,Q)$ for each $P,Q\le S$. 

\item A subgroup $P\le S$ is \emph{$p$-centric in $G$} if 
$Z(P)\in\sylp{C_G(P)}$; equivalently, if $C_G(P)=Z(P)\times C'_G(P)$ for a 
(unique) subgroup $C'_G(P)$ of order prime to $p$.

\item The \emph{centric linking system} of $G$ is the category 
$\call_S^c(G)$ whose objects are the $p$-centric subgroups of $G$, and 
where $\Mor_{\call_S^c(G)}(P,Q)=T_G(P,Q)/C'_G(P)$ for each pair of objects 
$P,Q$. Let $\pi\:\call_S^c(G)\Right2{}\calf_S(G)$ denote the natural 
functor: $\pi$ is the inclusion on objects, and sends the class of $g\in 
T_G(P,Q)$ to $c_g\in\Mor_{\calf_S(G)}(P,Q)$. 

\item For $P,Q\le S$ $p$-centric in $G$ and $g\in T_G(P,Q)$, we let 
$\dsg[P,Q]g\in\Mor_{\call_S^c(G)}(P,Q)$ denote the class of $g$, and set 
$\dsg[P]g=\dsg[P,P]g$ if $g\in N_G(P)$. For each subgroup $H\le N_G(P)$, 
$\dsg[P]H$ denotes the image of $H$ in $\Aut_\call(P)=N_G(P)/C'_G(P)$. 

\end{enuma}
\end{Defi}

The following definitions of automorphism groups are taken from 
\cite[Definition 1.13 \& Lemma 1.14]{AOV1}, where they are formulated more 
generally for abstract fusion and linking systems.

\begin{Defi} \label{d:aut(L)}
Let $G$ be a finite group with $S\in\sylp{G}$, and set $\calf=\calf_S(G)$ 
and $\call=\call_S^c(G)$.  
\begin{enuma} 
\item If $H$ is another finite group with $T\in\sylp{H}$, then an isomorphism 
$\varphi\:S\Right2{\cong}T$ is called \emph{fusion preserving} (with 
respect to $G$ and $H$) if for each $P,Q\le{}S$, 
	\[ \Hom_H(\varphi(P),\varphi(Q)) = 
	\varphi\circ\Hom_G(P,Q)\circ\varphi^{-1}\,. \]
(Composition is from right to left.)  Equivalently, $\varphi$ is 
fusion preserving if it induces an isomorphism of categories 
$\calf_S(G)\Right2{\cong}\calf_T(H)$.

\item Let $\Aut(S,\calf)\le\Aut(S)$ be the group of fusion preserving 
automorphisms of $S$.  Set $\Out(S,\calf)=\Aut(S,\calf)/\autf(S)$.

\item For each pair of objects $P\le Q$ in $\call$, set 
$\iota_{P,Q}=\dsg[P,Q]1\in\Mor_\call(P,Q)$, which we call the 
\emph{inclusion} in $\call$ of $P$ in $Q$. For each $P$, we call 
$\dsg{P}=\dsg[P]P\le\Aut_\call(P)$ the \emph{distinguished subgroup} 
of $\Aut_\call(P)$.

\item Let $\Aut\typ^I(\call)$ be the group of automorphisms $\alpha$ of the 
category $\call$ such that $\alpha$ sends inclusions to inclusions and 
distinguished subgroups to distinguished subgroups.
For $\gamma\in\Aut_\call(S)$, let $c_\gamma\in\Aut\typ^I(\call)$ be the 
automorphism which sends an object $P$ to $\pi(\gamma)(P)$, and sends 
$\psi\in\Mor_\call(P,Q)$ to $\gamma'\psi(\gamma'')^{-1}$ where $\gamma'$ 
and $\gamma''$ are appropriate restrictions of $\gamma$. Set 
	\[ \Out\typ(\call)
	=\Aut\typ^I(\call)\big/\{c_\gamma\,|\,\gamma\in\Aut_\call(S)\}\,. \]

\item Let $\kappa_G\:\Out(G)\Right3{}\Out\typ(\call)$ be the homomorphism 
which sends the class $[\alpha]$, for $\alpha\in\Aut(G)$ such that 
$\alpha(S)=S$, to the class of the induced automorphism of 
$\call=\call_S^c(G)$. 

\item Define $\mu_G\: \Out\typ(\call) \Right3{} \Out(S,\calf)$ by setting 
$\mu_G([\beta])=[\beta_S|_S]$ for $\beta\in\Aut\typ^I(\call_S^c(G))$, where 
$\beta_S$ is the induced automorphism of $\Aut_\call(S)$, and 
$\beta_S|_S\in\Aut(S)$ is its restriction to $S$ when we identify $S$ with 
its image in $\Aut_\call(S)=N_G(S)/C'_G(S)$.

\item Set $\4\kappa_G=\mu_G\circ\kappa_G\:\Out(G)\Right2{}\Out(S,\calf)$: 
the homomorphism which sends the class of $\alpha\in N_{\Aut(G)}(S)$ to 
the class of $\alpha|_S$.

\end{enuma}
\end{Defi}

By \cite[Lemma 1.14]{AOV1}, the above definition of $\Out\typ(\call)$ is 
equivalent to that in \cite{BLO2}, and by \cite[Lemma 8.2]{BLO2}, both are 
equivalent to that in \cite{BLO1}.  So by \cite[Theorem 4.5(a)]{BLO1}, 
$\Out\typ(\call_S^c(G))\cong\Out(BG\pcom)$: the group of homotopy classes 
of self homotopy equivalences of the space $BG\pcom$.

We refer to \cite[\S\,2.2]{AOV1} and \cite[\S\,1.3]{AOV1} for more details 
about the definitions of $\kappa_G$ and $\mu_G$ and the proofs that they 
are well defined. Note that $\mu$ is defined there for an arbitrary linking 
system, not necessarily one realized by a group.

We are now ready to define tameness. Again, we restrict attention to fusion 
systems of finite groups, and refer to \cite[\S\,2.2]{AOV1} for the 
definition in the more abstract setting.

\begin{Defi} \label{d:tame}
For a finite group $G$ and $S\in\sylp{G}$, the fusion system $\calf_S(G)$ 
is \emph{tame} if there is a finite group $G^*$ which satisfies:
\begin{itemize}  
\item there is a fusion preserving isomorphism $S\Right2{\cong}S^*$ for 
some $S^*\in\sylp{G^*}$; and 
\item the homomorphism 
$\kappa_{G^*}\:\Out(G^*)\longrightarrow\Out_{\textup{typ}}(\call_S^c(G^*))
\cong\Out(BG^*\pcom)$ is split surjective.
\end{itemize}
In this situation, we say that $G^*$ \emph{tamely realizes} the fusion system 
$\calf_S(G)$.
\end{Defi}


The above definition is complicated by the fact that two finite 
groups can have isomorphic fusion systems but different outer automorphism 
groups.  For example, set $G=PSL_2(9)\cong A_6$ and $H=PSL_2(7)\cong 
GL_3(2)$.  The Sylow subgroups of both groups are dihedral of order $8$, 
and it is not hard to see that any isomorphism between Sylow subgroups is 
fusion preserving.  But $\Out(G)\cong C_2^2$ while $\Out(H)\cong C_2$ (see 
Theorem \ref{St-aut} below).  Also, $\kappa_G$ is an isomorphism, while 
$\kappa_H$ fails to be onto (see \cite[Proposition 7.9]{BLO1}).  In 
conclusion, the $2$-fusion system of both groups is tame, even though 
$\kappa_H$ is not split surjective.

This definition of tameness was motivated in part in \cite{AOV1} by an 
attempt to construct new, ``exotic'' fusion systems (abstract fusion 
systems not realized by any finite group) as extensions of a known fusion 
system by an automorphism. Very roughly, if 
$\alpha\in\Aut\typ^I(\call_S^c(G))$ is not in the image of $\kappa_G$, and 
not in the image of $\kappa_{G^*}$ for any other finite group $G^*$ which has the 
same fusion and linking systems, then one can construct and extension of 
$\calf_S(G)$ by $\alpha$ which is not isomorphic to the fusion system of 
any finite group. This shows why we are interested in the surjectivity of 
$\kappa_G$; to see the importance of its being split, we refer to the proof 
of \cite[Theorem B]{AOV1}. 

It is usually simpler to work with automorphisms of a $p$-group 
which preserve fusion than with automorphisms of a linking system.  So in 
most cases, we prove tameness for the fusion system of a group $G$ 
by first showing that $\4\kappa_G=\mu_G\circ\kappa_G$ 
is split surjective, and then showing that $\mu_G$ is injective.  The 
following elementary lemma will be useful.

\begin{Lem} \label{kappa-lift}
Fix a finite group $G$ and $S\in\sylp{G}$, and set $\calf=\calf_S(G)$.  
Then 
\begin{enuma} 
\item $\4\kappa_G$ is surjective if and only if each 
$\varphi\in\Aut(S,\calf)$ extends to some $\4\varphi\in\Aut(G)$, and
\item $\Ker(\4\kappa_G)\cong C_{\Aut(G)}(S)/\Aut_{C_G(S)}(G)$.
\end{enuma}
\end{Lem}

\begin{proof} 
This follows from the following diagram 
	\beq 
	\xymatrix@C=30pt{
	0 \ar[r] & \Aut_{N_G(S)}(G) \ar[r] \ar@{->>}[d] & N_{\Aut(G)}(S) \ar[r] 
	\ar[d]^{\textup{restr}} & \Out(G) \ar[r] \ar[d]^{\4\kappa_G} & 0 \\
	0 \ar[r] & \Aut_{N_G(S)}(S) \ar[r] & \Aut(S,\calf) \ar[r] & 
	\Out(S,\calf) \ar[r] & 0 
	} \eeq
with exact rows. 
\end{proof}

The next lemma can be useful when $\kappa_G$ or $\4\kappa_G$ is surjective 
but not split.

\begin{Lem} \label{onto->split}
Fix a prime $p$, a finite group $G$, and $S\in\sylp{G}$. 
\begin{enuma} 
\item Assume $\5G\ge G$ is such that $G\nsg\5G$, $p\nmid|\5G/G|$, and 
$\Out_{\5G}(G)\le\Ker(\4\kappa_G)$. Then $\calf_S(\5G)=\calf_S(G)$ and 
$\call_S^c(\5G)\cong\call_S^c(G)$. 

\item If $\kappa_G$ is surjective and $\Ker(\kappa_G)$ has order prime to 
$p$, then there is $\5G\ge G$ as in (a) such that $\kappa_{\5G}$ is split 
surjective. In particular, $\calf_S(G)$ is tame, and is tamely realized by 
$\5G$.
\end{enuma}
\end{Lem}

\begin{proof} \textbf{(a) } Since $\Out_{\5G}(G)\le\Ker(\4\kappa_G)$, each 
coset of $G$ in $\5G$ contains an element which centralizes $S$. (Recall 
that $\4\kappa_G$ is induced by the restriction homomorphism from 
$N_{\Aut(G)}(S)$ to $\Aut(S,\calf)$.) Thus 
$\calf_S(\5G)=\calf_S(G)$ and $\call_S^c(\5G)=\call_S^c(G)$. 

\smallskip

\textbf{(b) } 
Since $G$ and $G/O_{p'}(Z(G))$ have isomorphic fusion systems 
at $p$, we can assume that $Z(G)$ is a $p$-group. Set 
$K=\Ker(\kappa_G)\le\Out(G)$. Since $H^i(K;Z(G))=0$ for $i=2,3$, by the 
obstruction theory for group extensions \cite[Theorems IV.8.7--8]{MacLane}, 
there is an extension $\5G$ of $G$ by $K$ such that $G\nsg\5G$, $\5G/G\cong 
K$, and $\Out_{\5G}(G)=K$. Since $K=\Ker(\kappa_G)\le\Ker(\4\kappa_G)$, 
$\calf_S(\5G)=\calf_S(G)$, and $\call_S^c(\5G)=\call_S^c(G)$ by (a). 

By \cite[Lemma 1.2]{OV2}, and since $K\nsg\Out(G)$ and $H^i(K;Z(G))=0$ for 
$i=1,2$, each automorphism of $G$ extends to an automorphism of $\5G$ which 
is unique modulo inner automorphisms. Thus $\Out(\5G)$ contains a subgroup 
isomorphic to $\Out(G)/K$, and $\kappa_{\5G}$ sends this subgroup 
isomorphically onto $\Out\typ(\call_S^c(\5G))$. So $\kappa_{\5G}$ is split 
surjective, and $\calf_S(G)$ is tame. 
\end{proof}

The next proposition is really a result about constrained fusion systems 
(cf. \cite[Definition I.4.8]{AKO}): it says that every constrained 
fusion system is tame. Since we are dealing here only with fusion systems 
of finite groups, we state it instead in terms of $p$-constrained groups.

\begin{Prop} \label{p:constrained}
Fix a finite group $G$ and a Sylow subgroup $S\in\sylp{G}$.
\begin{enuma} 
\item If $C_G(O_p(G))\le O_p(G)$, then $\kappa_G$ and $\mu_G$ are both 
isomorphisms:
	\[ \Out(G) \RIGHT4{\kappa_G}{\cong} \Out\typ(\call_S^c(G)) 
	\RIGHT4{\mu_G}{\cong} \Out(S,\calf_S(G))\,. \]
\item If $S$ is abelian, or more generally if $N_G(S)$ controls $p$-fusion in 
$G$, then $\calf_S(G)$ is tame, and is tamely realized by 
$N_G(S)/O_{p'}(C_G(S))$.
\end{enuma}
\end{Prop}

\begin{proof} \textbf{(a) } Set $Q=O_p(G)$, $\calf=\calf_S(G)$, and 
$\call=\call_S^c(G)$.  Then $\Aut_\call(Q)=G$, so ($\alpha\mapsto\alpha_Q$) 
defines a homomorphism $\Phi\:\Aut\typ^I(\call)\Right2{}\Aut(G,S)$.  For 
each $\alpha\in\Ker(\Phi)$, $\alpha_Q=\Id_G$ and hence $\alpha=\Id_\call$.  
(Here, it is important that $\alpha$ sends inclusions to inclusions.)  Thus 
$\Phi$ is an isomorphism.  Also, $\alpha=c_\gamma$ for some 
$\gamma\in\Aut_\call(S)$ if and only if $\alpha_Q=c_g$ for some 
$g\in{}N_G(S)$, so $\Phi$ factors through an isomorphism from 
$\Out\typ(\call)$ to $\Aut(G,S)/\Aut_G(S)\cong\Out(G)$, and this is an 
inverse to $\kappa_G$.  Thus $\kappa_G$ is an isomorphism.

In the terminology in \cite[\S\,I.4]{AKO}, $G$ is a model for 
$\calf=\calf_S(G)$.  By the uniqueness of models (cf. \cite[Theorem 
III.5.10(c)]{AKO}), each $\beta\in\Aut(S,\calf)$ extends to some 
$\chi\in\Aut(G)$, and $\chi$ is unique modulo $\Aut_{Z(S)}(G)$.  Hence 
$\4\kappa_G$ is an isomorphism, and so is $\mu_G$. 

\smallskip

\noindent\textbf{(b) } If $N_G(S)$ controls $p$-fusion in $G$, then 
$N_G(S)\sim_pG$. Also, $N_G(S)\sim_pG^*$ where $G^*=N_G(S)/O_{p'}(C_G(S))$, 
$G^*$ satisfies the hypotheses of (a), and hence tamely realizes 
$\calf_S(G)$. In particular, this holds whenever $S$ is abelian by 
Burnside's theorem. 
\end{proof}

When working with groups of Lie type when $p$ is not the defining 
characteristic, it is easier to work with the universal groups 
rather than those in adjoint form ($\mu_G$ is better behaved in such 
cases). The next proposition is needed to show that tameness for fusion 
systems of groups of universal type implies the corresponding result for 
groups of adjoint type.

\begin{Prop} \label{univ2adj}
Let $G$ be a finite $p$-perfect group such that $O_{p'}(G)=1$ and 
$H_2(G;\Z/p)=0$ (i.e., such that each central extension of $G$ by a 
finite $p$-group splits). Choose $S\in\sylp{G}$, 
and set $Z=Z(G)\le S$. If $\calf_S(G)$ is tamely realized by $G$, then 
$\calf_{S/Z}(G/Z)$ is tamely realized by $G/Z$. 
\end{Prop}

\begin{proof} Let $\calh$ be the set of all $P\le S$ such that $P\ge Z$ and 
$P/Z$ is $p$-centric in $G/Z$, and let 
$\call_S^\calh(G)\subseteq\call_S^c(G)$ be the full subcategory with object 
set $\calh$. By \cite[Lemma 2.17]{AOV1}, $\call_S^\calh(G)$ 
is a linking system associated to $\calf_S(G)$ in the sense of 
\cite[Definition 1.9]{AOV1}. Hence the homomorphism 
	\[ R\: \Out\typ(\call_S^c(G)) \Right5{\cong} \Out\typ(\call_S^\calh(G)) \]
induced by restriction is an isomorphism by \cite[Lemma 1.17]{AOV1}.

Set $\calf=\calf_S(G)$, $\call=\call_S^\calh(G)$, 
$\4G=G/Z$, $\4S=S/Z$, $\4\calf=\calf_{\4S}(\4G)$, and 
$\4\call=\call_{\4S}^c(\4G)$ for short. Consider the following square:
	\beqq 
	\vcenter{\xymatrix@C=35pt{
	\Out(G) \ar[r]^-{\kappa_G} \ar[d]^{\mu} & \Out\typ(\call) 
	\rlap{\,$\cong\Out\typ(\call_S^c(G))$} \\
	\Out(\4G) \ar[r]^-{\kappa_{\4G}} & 
	\Out\typ(\4\call) \ar[u]_{\nu}^{1-1} \rlap{\,.}
	}}\hphantom{XXXXXXXXXX} \label{e:Out-sq} \eeqq
Here, $\mu$ sends the class of an automorphism of $G$ to the class of the 
induced automorphism of $\4G=G/Z(G)$.

Assume that $\nu$ has been defined so that \eqref{e:Out-sq} commutes and 
$\nu$ is injective. If $\kappa_G$ is onto, then $\nu$ is onto and hence an 
isomorphism, so $\kappa_{\4G}$ is also onto. Similarly, if 
$\kappa_G$ is split surjective, then $\kappa_{\4G}$ is also split 
surjective. Thus $\4\calf$ is tamely realized by $\4G$ if $\calf$ is tamely 
realized by $G$, which is what we needed to show.

It thus remains to construct the monomorphism $\nu$, by sending the class 
of $\alpha\in\Aut\typ^I(\4\call)$ to the class of a lifting of $\alpha$ to 
$\call$. So in the rest of the proof, we show the existence and uniqueness 
of such a lifting.

Let $\pr\:\call\Right2{}\4\call$ denote the projection. Let 
$\End\typ^I(\call)$ be the monoid of functors from $\call$ to itself which 
send inclusions to inclusions and distinguished subgroups into distinguished 
subgroups. (Thus $\Aut\typ^I(\call)$ is the group of elements of 
$\End\typ^I(\call)$ which are invertible.) We will prove the following two 
statements:
\begin{enumerate}[label=\rm(\arabic*),ref=\arabic*,itemsep=6pt,leftmargin=8mm] 
\setcounter{enumi}{\theequation}

\item \label{is_lift} For each $\alpha\in\Aut\typ^I(\4\call)$, there is a 
functor $\til\alpha\in\End\typ^I(\call)$ such that 
$\pr\circ\til\alpha=\alpha\circ\pr$.

\item \label{uniq_lift} If $\beta\in\End\typ^I(\call)$ is such that 
$\pr\circ\beta=\pr$, then $\beta=\Id_\call$.

\end{enumerate} \setcounter{equation}{\theenumi}

Assume that \eqref{is_lift} and \eqref{uniq_lift} hold; we call 
$\til\alpha$ a ``lifting'' of $\alpha$ in the situation of \eqref{is_lift}. 
For each $\alpha\in\Aut\typ^I(\4\call)$, there are liftings 
$\til\alpha$ of $\alpha$ and $\til\alpha^*$ of $\alpha^{-1}$ in 
$\End\typ^I(\call)$, and these are inverses to each other by 
\eqref{uniq_lift}. Hence $\til\alpha\in\Aut\typ^I(\call)$, and is the 
unique such lifting of $\alpha$ by \eqref{uniq_lift} again. 

Define $\nu\:\Out\typ(\4\call)\Right2{}\Out\typ(\call)$ by setting 
$\nu([\alpha])=[\til\alpha]$ when $\til\alpha$ is the unique lifting of 
$\alpha$. This is well defined as a homomorphism on $\Aut\typ^I(\4\call)$ 
by the existence and uniqueness of the lifting; and it factors through 
$\Out\typ(\4\call)$ since conjugation by $\4\gamma\in\Aut_{\4\call}(\4S)$ 
lifts to conjugation by $\gamma\in\Aut_\call(S)$ for any 
$\gamma\in\pr_S^{-1}(\4\gamma)$. 

Thus $\nu$ is a well defined homomorphism, and is clearly injective. The 
square \eqref{e:Out-sq} commutes since for each $\beta\in\Aut(G)$ such that 
$\beta(S)=S$, $\kappa_G([\beta])$ and $\nu\kappa_{\4G}\mu([\beta])$ are 
the classes of liftings of the same automorphism of $\4\call$.

It remains to prove \eqref{is_lift} and \eqref{uniq_lift}.

\smallskip

\noindent\textbf{Proof of (\ref{is_lift}): } 
For each $\alpha\in\Aut\typ^I(\4\call)$, consider the pullback diagram
	\beqq 
	\vcenter{\xymatrix@C=40pt{
	\til\call \ar[rr]^-{\rho_1} \ar[d]^{\rho_2} && 
	\call \ar[d]^{\pr} \\
	\call \ar[r]^(0.65){\pr} \ar@{-->}[rru]^{\til\alpha} & \4\call 
	\ar[r]^-{\alpha}_-{\cong} & \4\call 
	\rlap{\,.} 
	}} \label{e:L-pb} \eeqq
Each functor in \eqref{e:L-pb} is bijective on objects, and the diagram 
restricts to a pullback square of morphism sets for each pair of 
objects in $\4\call$ (and their inverse images in $\call$ and 
$\til\call$). 

Since the natural projection 
$G\Right2{}\4G$ is a central extension with kernel $Z$, the projection 
functor $\pr\:\call\Right2{}\4\call$ is also a central extension of 
linking systems in the sense of \cite[Definition 6.9]{BCGLO2} with kernel $Z$. 
Since $\rho_2$ is the pullback of a central extension, 
it is also a central extension of linking systems by \cite[Proposition 
6.10]{BCGLO2}, applied with $\omega=\pr^*\alpha^*(\omega_0)\in Z^2(\call;Z)$, 
where $\omega_0$ is a 2-cocycle on $\4\call$ which determines the extension 
$\pr$. 
By \cite[Proposition 1.1]{BLO1}, $H^2(|\call|;\F_p)\cong H^2(G;\F_p)$, 
where the last group is zero by assumption. Hence $H^2(|\call|;Z)=0$, 
so $\omega$ is a coboundary, and $\rho_2$ is the product extension 
by \cite[Theorem 6.13]{BCGLO2}. In other words, 
$\til\call\cong\call_Z^c(Z)\times\call$, where $\call_Z^c(Z)$ has one 
object and automorphism group $Z$, and there is a subcategory 
$\call_0\subseteq\til\call$ (with the same objects) which is sent 
isomorphically to $\call$ by $\rho_2$. Set 
$\til\alpha=\rho_1\circ(\rho_2|_{\call_0})^{-1}$.

We first check that $\til\alpha$ sends distinguished subgroups to 
distinguished subgroups. Let $\pr_S\:S\Right2{}\4S=S/Z$ be the projection. 
Fix an object $P$ in $\call$, and set $Q=\til\alpha(P)$. Then 
$Q/Z=\alpha(P/Z)$, and $\alpha_{P/Z}(\dsg{P/Z})=\dsg{Q/Z}$, so 
$\til\alpha_P(\dsg{P})\le\pr_S^{-1}(\dsg{Q/Z})=\dsg{Q}$.

For each subgroup $P\in\Ob(\call)$, there is a unique element $z_P\in{}Z$ 
such that $\til\alpha(\iota_{P,S})= 
\iota_{\til\alpha(P),S}\circ\dsg[\til\alpha(P)]{z_P}$. Note that $z_S=1$. 
Define a new functor $\beta \colon \call\Right1{}\call$ by setting 
$\beta(P)=\til\alpha(P)$ on objects and for each 
$\varphi\in\Mor_\call(P,Q)$, $\beta(\varphi) = 
\dsg[\til\alpha(Q)]{z_Q}\circ \til\alpha(\varphi) 
\circ\dsg[\til\alpha(P)]{z_P}^{-1}$. Then $\beta$ is still a lifting of 
$\alpha$, and for each $P$:
	\[ \beta(\iota_{P,S}) = \dsg[S]{z_S}\circ \til\alpha(\iota_{P,S}) 
	\circ\dsg[\til\alpha(P)]{z_P}^{-1} 
	= \iota_{\til\alpha(P),S}\circ\dsg[\til\alpha(P)]{z_P} 
	\circ\dsg[\til\alpha(P)]{z_P}^{-1} 
	= \iota_{\til\alpha(P),S}\,. \]
For arbitrary $P\le Q$, since $\iota_{\til\alpha(P),\til\alpha(Q)}$ is the 
unique morphism whose composite with $\iota_{\til\alpha(Q),S}$ is 
$\iota_{\til\alpha(P),S}$ (see \cite[Lemma 1.10(a)]{BLO2}), $\beta$ sends 
$\iota_{P,Q}$ to $\iota_{\til\alpha(P),\til\alpha(Q)}$. 

Thus, upon replacing $\til\alpha$ by $\beta$, we can assume that 
$\til\alpha$ sends inclusions to inclusions. This finishes the proof of 
\eqref{is_lift}. 
                     
\smallskip

\noindent\textbf{Proof of (\ref{uniq_lift}): } Assume that 
$\beta\in\End\typ^I(\call)$ is a lift of the identity on $\4\call$. Let 
$\calb(Z)$ be the category with one object $\*$ and with morphism group 
$Z$. Define a functor $\chi\:\call\Right2{}\calb(Z)$ by sending all objects 
in $\call$ to $\*$, and by sending a morphism $\dsg{g}\in\Mor_\call(P,Q)$ 
to the unique element $z\in{}Z$ such that 
$\beta_{P,Q}(\dsg{g})=\dsg{gz}=\dsg{zg}$. (Recall that $Z\le Z(G)$.) 

Now, 
	\[ H^1(|\call|;\F_p)\cong H^1(|\call_S^c(G)|;\F_p)\cong 
	H^1(BG;\F_p)\cong H^1(G;\F_p)=0 \,, \]
where the first isomorphism holds by \cite[Theorem B]{BCGLO1} and the 
second by \cite[Proposition 1.1]{BLO1}. Hence 
$\Hom(\pi_1(|\call|),\F_p)\cong\Hom(H_1(|\call|),\F_p)\cong 
H^1(|\call|;\F_p)=0$, where the second isomorphism holds by the universal 
coefficient theorem (cf. \cite[Theorem III.4.1]{MacLane}), and so 
$\Hom(\pi_1(|\call|),Z)=0$. In particular, the homomorphism 
$\5\chi\:\pi_1(|\call|)\Right2{}\pi_1(|\calb(Z)|)\cong Z$ induced by $\chi$ 
is trivial. 

Thus for each $\psi\in\Mor_\call(P,Q)$, the loop in $|\call|$ formed by 
$\psi$ and the inclusions $\iota_{P,S}$ and $\iota_{Q,S}$ is sent to 
$1\in{}Z$. Since $\beta$ sends inclusions to inclusions, this proves that 
$\chi_{P,Q}(\psi)=1$, and hence that $\beta_{P,Q}(\psi)=\psi$. Thus 
$\beta=\Id_\call$. 
\end{proof}

By Proposition \ref{univ2adj}, when proving tameness for fusion systems of 
simple groups of Lie type, it suffices to look at the universal groups 
(such as $\SL_n(q)$, $\SU_n(q)$) rather than the simple groups 
($\PSL_n(q)$, $\PSU_n(q)$).  However, it is important to note that the 
proposition is false if we replace automorphisms of the linking systems by 
those of the fusion system.  For example, set $G=SL_2(3^4)$ and 
$\4G=PSL_2(3^4)$.  Then $S\cong Q_{32}$ and $\4S\cong D_{16}$, 
$\Out(S,\calf_S(G))=\Out(S)\cong\Out(G)\cong C_4\times C_2$ (and 
$\4\kappa_G$ is an isomorphism), while $\Out(\4G)\cong C_4\times C_2$ and 
$\Out(\4S,\calf_{\4S}(\4G))=\Out(\4S)\cong C_2\times C_2$.  

We already gave one example of two groups which have the same fusion system 
but different outer automorphism groups.  That is a special case of the 
main theorem in our earlier paper, where we construct many examples of 
different groups of Lie type with isomorphic fusion systems.  Since this 
plays a crucial role in Section \ref{s:X2}, where we handle the cross 
characteristic case, we restate the theorem here.  

As in the introduction, we write $G\sim_pH$ to mean that there is a 
fusion preserving isomorphism from a Sylow $p$-subgroup of $G$ to one of 
$H$.

\begin{Thm}[{\cite[Theorem A]{BMO1}}] \label{OldThA}
Fix a prime $p$, a connected reductive group scheme $\gg$ over $\Z$, and a 
pair of prime powers $q$ and $q'$ both prime to $p$.  Then the following 
hold.
\begin{enuma}
\item $\gg(q)\sim_p\gg(q')$ if 
$\widebar{\gen{q}}=\widebar{\gen{q'}}$ as subgroups of $\Z_p^\times$. 
\item If $\gg$ is of type $A_n$, $D_n$, or $E_6$, and $\tau$ is a graph 
automorphism of $\gg$, then 
$\9\tau\gg(q)\sim_p\9\tau\gg(q')$ if 
$\widebar{\gen{q}}=\widebar{\gen{q'}}$ as subgroups of $\Z_p^\times$. 
\item If the Weyl group of $\gg$ contains an element which acts on the 
maximal torus by inverting all elements, then 
$\gg(q)\sim_p\gg(q')$ (or 
$\9\tau\gg(q)\sim_p\9\tau\gg(q')$ for $\tau$ as in (b)) if 
$\widebar{\gen{-1,q}}=\widebar{\gen{-1,q'}}$ as subgroups of $\Z_p^\times$. 
\item If $\gg$ is of type $A_n$, $D_n$ for $n$ odd, or $E_6$, and $\tau$ is 
a graph automorphism of $\gg$ of order two, then 
$\9\tau\gg(q)\sim_p\gg(q')$ if 
$\widebar{\gen{-q}}=\widebar{\gen{q'}}$ as subgroups of $\Z_p^\times$. 
\end{enuma}
\end{Thm}

The next proposition is of similar type, but much more elementary.

\begin{Prop} \label{OldPrA3}
Fix an odd prime $p$, a prime power $q$ prime to $p$, 
$n\ge2$, and $\gee\in\{\pm1\}$. Then
\begin{enuma}


\item $\Sp_{2n}(q)\sim_p\SL_{2n}(q)$ if $\ord_p(q)$ is even;

\item $\Sp_{2n}(q)\sim_p\Spin_{2n+1}(q)$; and 

\item $\Spin_{2n}^\gee(q) \sim_p \Spin_{2n-1}(q)$ if $q$ is odd and 
$q^n\not\equiv\gee$ (mod $p$).

\end{enuma}
\end{Prop}

\begin{proof} If we replace $\Spin_m^\pm(q)$ by $\SO_m^\pm(q)$ in (b) and 
(c), then these three points are shown in \cite[Proposition A.3]{BMO1} as 
points (d), (a), and (c), respectively. When $q$ is a power of $2$, (b) 
holds because the groups are isomorphic (see \cite[Theorem 11.9]{Taylor}). 
So it remains to show that 
	\[ \Spin_m^\gee(q) \sim_p \Omega_m^\gee(q) \sim_p \SO_m^\gee(q) \]
for all $m\ge3$ (even or odd) and $q$ odd. The first equivalence holds 
since $p$ is odd and $\Omega_m^\gee(q)\cong\Spin_m^\gee(q)/K$ where $|K|=2$. 
The second holds by Lemma \ref{onto->split}(a), and since 
$\Out_{\SO_{m}^\gee(q)}(\Omega_{m}^\gee(q))$ is generated by the class of a diagonal 
automorphism of order $2$ (see, e.g., \cite[\S\,2.7]{GLS3}) and hence can 
be chosen to commute with a Sylow $p$-subgroup. This last statement is 
shown in Lemma \ref{Aut-diag} below, and holds since for appropriate 
choices of algebraic group $\4G$ containing the given group $G$, and of 
maximal torus $\4T\le\4G$, a Sylow $p$-subgroup of $G$ is contained in 
$N_{\4G}(\4T)$ (see \cite[Theorem 4.10.2]{GLS3}) and the diagonal 
automorphisms of $G$ are induced by conjugation by elements in $N_{\4T}(G)$ 
(see Proposition \ref{p:Autdiag(G)}(c)). 
\end{proof}

Theorem \ref{OldThA} and Proposition \ref{OldPrA3}, together with some 
other, similar relations in \cite{BMO1}, lead to the following 
proposition, which when $p$ is odd provides a relatively short list of 
``$p$-local equivalence class representatives'' for groups of Lie type in 
characteristic different from $p$.

\begin{Prop} \label{list-simp}
Fix an odd prime $p$, and assume $G\in\Lie(q_0)$ is of universal type for 
some prime $q_0\ne{}p$. Assume also that the Sylow $p$-subgroups of $G$ are 
nonabelian. Then there is a group $G^*\in\Lie(q_0')$ of universal type for 
some $q_0'\ne p$, such that $G^*\sim_pG$ and $G^*$ is one of the groups in 
the following list: 
\begin{enuma}  
\item $\SL_n(q')$ for some $n\ge p$; or
\item $\Spin_{2n}^\gee(q')$, where $n\ge p$, $\gee=\pm1$, 
$(q')^n\equiv\gee$ (mod $p$), and $\gee=+1$ if $n$ is odd; or
\item $\lie3D4(q')$ or $\lie2F4(q')$, where $p=3$ and $q'$ is a power of $2$; or 
\item $\gg(q')$, where $\gg=G_2$, $F_4$, $E_6$, $E_7$, or $E_8$, 
$p\,\big|\,|W(\gg)|$, and $q'\equiv1$ (mod $p$); or 
\item $E_8(q')$, where $p=5$ and $q'\equiv\pm2$ (mod $5$).
\end{enuma}
Furthermore, in all cases except (c), we can take $q_0'$ to 
be any given prime whose class generates $(\Z/p^2)^\times$, and choose $G^*$ 
so that $q'=(q_0')^b$ where $b|(p-1)p^k$ for some $k$.
\end{Prop}

\begin{proof} Let $q$ be such that $G\cong\9\tau\gg(q)$ for some $\tau$ and 
some $\gg$. Thus $q$ is a power of $q_0$.  Fix a prime $q_0'$ as specified 
above. By Lemma \ref{cond(i-iii)}(a), there are positive integers $b,c$, 
and powers $q'=(q_0')^b$ and $q^\vee=(q_0')^c$ such that 
$\4{\gen{q}}=\4{\gen{q'}}$, $\4{\gen{-q}}=\4{\gen{q^\vee}}$, and 
$b,c|(p-1)p^\ell$ for some $\ell\ge0$. 

\begin{enumr}

\item Assume $G\cong\Sz(q)$, $\lie2G2(q)$, $\lie2F4(q)$, or 
$G\cong\lie3D4(q)$.  Since $p\ne{}q_0$, and since $S\in\sylp{G}$ is 
nonabelian, $p$ divides the order of the Weyl group $W$ of $\gg$ by 
\cite[10-1]{GL}. The Weyl group of $B_2$ is a $2$-group, and $2$ and $3$ 
are the only primes which divide the orders of the Weyl groups of $G_2$, 
$F_4$, and $D_4$.  Hence $p=3$, $G\not\cong\lie2G2(q)$ since that is 
defined only in characteristic $3$, and so $G\cong\lie2F4(q)$ or 
$\lie3D4(q)$. Set $G^*=\lie2F4(q')$ or $\lie3D4(q')$, respectively, where 
$q_0'=2$. Then $G^*\sim_pG$, and we are in case (c). 

\item If $G=\SU_n(q)$ or $\lie2E6(q)$, then by Theorem \ref{OldThA}(d), 
$G\sim_pG^*$ where $G^*\cong\SL_n(q^\vee)$ or $E_6(q^\vee)$, respectively.  
So we can replace $G$ by a Chevalley group in these cases. 

\item Assume $G=\Sp_{2n}(q)$ for some $n$ and $q$. If $\ord_p(q)$ is even, 
then by Proposition \ref{OldPrA3}(a), $G\sim_p\SL_{2n}(q)$. If 
$\ord_p(q)$ is odd, then $\ord_p(q^\vee)$ is even (recall that 
$q^\vee\equiv-q$ (mod $p$)), and $G\sim_p\Sp_{2n}(q^\vee)$ by Theorem 
\ref{OldThA}(c). So $G$ is always $p$-locally equivalent to a linear group in 
this case.

\item Assume $G=\Spin_{2n+1}(q)$ for some $n$ and $q$. Then 
$G\sim_p\Sp_{2n}(q)$ by Proposition \ref{OldPrA3}(b). So $G$ is $p$-locally 
equivalent to a linear group by (iii).

\item If $G=\SL_n(q)$, set $G^*=\SL_n(q')$. Then $G^*\sim_pG$ by Theorem 
\ref{OldThA}(a), $n\ge p$ since the Sylow $p$-subgroups of $G$ are 
nonabelian, and we are in the situation of (a).

\item Assume $G=\Spin_{2n}^\gee(q)$ for some $n$ and $q$, and $\gee=\pm1$. 
If $q$ is a power of $2$, then by using point (a) or (b) of Theorem 
\ref{OldThA}, we can arrange that $q$ be odd. 
If $q^n\not\equiv\gee$ (mod $p$), then $G\sim_p\Spin_{2n-1}(q)$ by 
Proposition \ref{OldPrA3}(c), and this is $p$-equivalent to a linear 
group by (iv). So we are left with the case where $q^n\equiv\gee$ (mod $p$). 
If $n$ is odd and $\gee=-1$, set $G^*=\Spin_{2n}^+(q^\vee)\sim_pG$ (Theorem 
\ref{OldThA}(d)). 
Otherwise, set $G^*=\Spin_{2n}^\gee(q')\sim_pG$ (Theorem 
\ref{OldThA}(a,b)). In either case, we are in the situation of (b).

\end{enumr}

We are left with the cases where $G=\gg(q)$ for some exceptional Lie group 
$\gg$. By \cite[10-1(2)]{GL} and since the Sylow $p$-subgroups of $G$ are 
nonabelian, $p\bmid|W(\gg)|$. If $\ord_p(q)=1$, then $G^*=\gg(q')\sim_pG$ 
by Theorem \ref{OldThA}(a). If $\ord_p(q)=2$ and $\gg\ne E_6$, then 
$G^*=\gg(q^\vee)\sim_pG$ by Theorem \ref{OldThA}(c), where $q^\vee\equiv1$ 
(mod $p$). In either case, we are in the situation of (d). 

If $\ord_p(q)=2$ and $G=E_6(q)$, then $\4{\gen{q}}=\4{\gen{-q^2}}$ as 
closed subgroups of $\Z_p^\times$ (note that $v_p(q^2-1)=v_p((-q^2)^2-1)$). 
So by Theorem \ref{OldThA}(d) and Example 4.4 in \cite{BMO1}, 
$G=E_6(q)\sim_p\lie2E6(q^2)\sim_pF_4(q^2)$. So we can choose $G^*$ 
satisfying (d) as in the last paragraph.

Assume $\ord_p(q)>2$. By \cite[10-1(3)]{GL}, for $S\in\sylp{G}$ to be 
nonabelian, there must be some $n\ge1$ such that $p\cdot\ord_p(q)\bmid n$, and 
such that $q^n-1$ appears as a factor in the formula for $|\gg(q)|$ (see, 
e.g., \cite[Table 4-2]{GL} or \cite[Theorem 9.4.10 \& Proposition 
10.2.5]{Carter}). Since $\ord_p(q)|(p-1)$, this shows that the case 
$\ord_p(q)>2$ appears only for the group $E_8(q)$, and only when $p=5$ and 
$m=4$. In particular, $q,q'\equiv\pm2$ (mod $5$). Set $G^*=E_8(q')$; then 
$G^*\sim_pG$ by Theorem \ref{OldThA}(a), and we are in the situation of (e). 
\end{proof}

The following lemma was needed in the proof of Proposition \ref{list-simp} 
to reduce still further the prime powers under consideration.


\begin{Lem} \label{cond(i-iii)}
Fix a prime $p$, and an integer $q>1$ prime to $p$. 
\begin{enuma} 
\item If $p$ is odd, then for any prime $r_0$ whose class
generates $(\Z/p^2)^\times$, there is $b\ge1$ such 
that $\4{\gen{q}}=\4{\gen{(r_0)^b}}$, and $b|(p-1)p^\ell$ for some $\ell$.

\item If $p=2$, then either $\4{\gen{q}}=\4{\gen{3}}$, or 
$\4{\gen{q}}=\4{\gen{5}}$, or there are $\gee=\pm1$ and $k\ge1$ such 
that $\gee\equiv{}q$ (mod $8$) and $\4{\gen{q}}=\4{\gen{\gee\cdot3^{2^k}}}$.
\end{enuma}
\end{Lem}

\begin{proof} Since $q\in\Z$ and $q>1$, $\4{\gen{q}}$ is infinite. 

\smallskip

\textbf{(a) } If $p$ is odd, then for each $n\ge1$, 
$(\Z/p^n)^\times\cong(\Z/p)^\times\times(\Z/p^{n-1})$ is cyclic and 
generated by the class of $r_0$. Hence 
$\Z_p^\times\cong(\Z/p)^\times\times(\Z_p,+)$, and 
$\4{\gen{r_0}}=\Z_p^\times$.  Also, $\4{\gen{q}}\ge 1+p^\ell\Z_p$ for some 
$\ell\ge1$, since each infinite, closed subgroup of $(\Z_p,+)$ contains 
$p^k\Z_p$ for some $k$. 

Set $ b = [\Z_p^\times:\4{\gen{q}}] = 
[(\Z/p^\ell)^\times:\gen{q{+}p^\ell\Z}] \big| (p-1)p^{\ell-1}$. Then 
$\4{\gen{q}}=\4{\gen{(r_0)^b}}$.

\smallskip

\noindent\textbf{(b) } If $p=2$, then 
$\Z_2^\times=\{\pm1\}\times\4{\gen{3}}$, where $\4{\gen{3}}\cong(\Z_2,+)$. 
Hence the only infinite closed subgroups of $\4{\gen{3}}$ are those of the 
form $\4{\gen{3^{2^k}}}$ for some $k\ge0$. So 
$\4{\gen{q}}=\4{\gen{\gee\cdot3^{2^k}}}$ for some $k\ge0$ and some 
$\gee=\pm1$, and the result follows since $\4{\gen{5}}=\4{\gen{-3}}$. 
\end{proof}

We also note, for use in Section \ref{s:=}, the following more technical 
result.

\begin{Lem} \label{extend_phi}
Let $G$ be a finite group, fix $S\in\sylp{G}$, and set $\calf=\calf_S(G)$.  
Let $P\le S$ be such that $C_G(P)\le P$ and $N_S(P)\in\sylp{N_G(P)}$. Then 
for each $\varphi\in\Aut(S,\calf)$ such that $\varphi(P)=P$, 
$\varphi|_{N_S(P)}$ extends to an automorphism $\4\varphi$ of $N_G(P)$.
\end{Lem}

\begin{proof} Since $C_G(P)\le P$ and $N_S(P)\in\sylp{N_G(P)}$, $N_G(P)$ is 
a model for the fusion system $\cale=\calf_{N_S(P)}(N_G(P))$ in the sense 
of \cite[Definition I.4.8]{AKO}.  By the strong uniqueness property for 
models \cite[Theorem I.4.9(b)]{AKO}, and since $\varphi|_{N_S(P)}$ 
preserves fusion in $\cale$, $\varphi|_{N_S(P)}$ extends to an automorphism 
of the model. 
\end{proof}

The following elementary lemma will be useful in Sections \ref{s:X1} and 
\ref{s:X2}; for example, when computing orders of Sylow subgroups of groups 
of Lie type.

\begin{Lem} \label{v(q^i-1)}
Fix a prime $p$.  Assume $q\equiv1$ (mod $p)$, and $q\equiv1$ (mod $4$) if 
$p=2$.  Then for each $n\ge1$, $v_p(q^n-1)=v_p(q-1)+v_p(n)$.
\end{Lem}

\begin{proof} Set $r=v_p(q-1)$, and let $k$ be such that $q=1+p^rk$.  Then 
$q^n=1+np^rk+\xi$, where $v_p(np^rk)=v_p(n)+r$, and where each term in 
$\xi$ has strictly larger valuation. 
\end{proof}

\newpage

\newsect{Background on finite groups of Lie type} 
\label{s:not}

In this section and the next, we fix the notation to be used for finite 
groups of Lie type, and list some of the (mostly standard) results which 
will be needed later. We begin by recalling the following concepts used in 
\cite{GLS3}. We do not repeat the definitions of maximal tori and Borel 
subgroups in algebraic groups, but refer instead to 
\cite[\S\S\,1.4--1.6]{GLS3}.

\begin{Defi}[{\cite[Definitions 1.7.1, 1.15.1, 2.2.1]{GLS3}}] \label{d:Lie}
Fix a prime $q_0$. 
\begin{enuma} 

\item A connected algebraic group $\4G$ over $\fqobar$ is \emph{simple} if 
$[\4G,\4G]\ne1$, and all proper closed normal subgroups of $\4G$ are finite 
and central. If $\4G$ is simple, then it is of \emph{universal type} 
if it is simply connected, and of \emph{adjoint type} if $Z(\4G)=1$. 

\item A \emph{Steinberg endomorphism} of a connected simple algebraic 
group $\4G$ is a surjective algebraic endomorphism $\sigma\in\End(\4G)$ 
whose fixed subgroup is finite.

\item A \emph{$\sigma$-setup} for a finite group $G$ is a pair 
$(\4G,\sigma)$, where $\4G$ is a simple algebraic group over $\fqobar$, 
and where $\sigma$ is a Steinberg endomorphism of $\4G$ such that 
$G=O^{{q_0}'}(C_{\4G}(\sigma))$.

\item Let $\Lie(q_0)$ denote the class of finite groups with $\sigma$-setup 
$(\4G,\sigma)$ where $\4G$ is simple and is defined in characteristic 
$q_0$, and let $\Lie$ be the union of the classes $\Lie(q_0)$ for all 
primes $q_0$.  We say that $G$ is of universal (adjoint) type if $\4G$ is 
of universal (adjoint) type.

\end{enuma}
\end{Defi}

If $\4G$ is universal, then $C_{\4G}(\sigma)$ is generated by elements of 
$q_0$-power order (see \cite[Theorem 12.4]{Steinberg-end}), and hence 
$G=C_{\4G}(\sigma)$ in (c) above. In general, $C_{\4G}(\sigma) = G\cdot 
C_{\4T}(\sigma)$ (cf. \cite[Theorem 2.2.6]{GLS3}).

A \emph{root group} in a connected algebraic group $\4G$ over $\fqobar$ 
with a given maximal torus $\4T$ is a one-parameter closed subgroup (thus 
isomorphic to $\fqobar$) which is normalized by $\4T$.  The \emph{roots} of 
$\4G$ are the characters for the $\4T$-actions on the root groups, and lie 
in the $\Z$-lattice $X(\4T)=\Hom(\4T,\fqobar^\times)$ of characters of 
$\4T$. (Note that this is the group of \emph{algebraic} homomorphisms, and 
that $\Hom(\fqobar^\times,\fqobar^\times)\cong\Z$.) The roots are regarded 
as lying in the $\R$-vector space $V=\R\otimes_{\Z}\4T^*$. 
We refer to \cite[\S\,1.9]{GLS3} for details 
about roots and root subgroups of algebraic groups, and to \cite[Chapitre 
VI]{Bourb4-6} for a detailed survey of root systems.  

The following notation and hypotheses will be used throughout this 
paper, when working with a finite group of Lie type defined via a 
$\sigma$-setup.

\begin{Not} \label{G-setup}
Let $(\4G,\sigma)$ be a $\sigma$-setup for the finite group $G$, where 
$\4G$ is a connected, simple algebraic group over $\4\F_{q_0}$ for a prime 
$q_0$.  When convenient, we also write $\4G=\gg(\4\F_{q_0})$, 
where $\gg$ is a group scheme over $\Z$.
\begin{enumA}  

\item\label{not1} \boldd{The maximal torus and Weyl group of $\4G$.} 
Fix a maximal torus $\4T$ in $\4G$ such that $\sigma(\4T)=\4T$.  
Let $W=N_{\4G}(\4T)/\4T$ be the Weyl group of $\4G$ (and of $\gg$).

\item\label{not2} \boldd{The root system of $\4G$.} Let $\Sigma$ be the set 
of all roots of $\4G$ with respect to $\4T$, and let $\4X_\alpha<\4G$ 
denote the root group for the root $\alpha\in\Sigma$.  Thus 
$\4X_\alpha=\{x_\alpha(u)\,|\,u\in\fqobar\}$ with respect to some fixed 
Chevalley parametrization of $\4G$.  Set $V=\R\otimes_{\Z}\4T^*$: a real 
vector space with inner product $(-,-)$ upon which the Weyl group $W$ acts 
orthogonally.  Let $\Pi\subseteq\Sigma$ be a fundamental system of roots, 
and let $\Sigma_+\subseteq\Sigma$ be the set of positive roots with respect 
to $\Pi$.  For each $\alpha\in\Sigma_+$, let $\htt(\alpha)$ denote the 
\emph{height} of $\alpha$: the number of summands in the decomposition of 
$\alpha$ as a sum of fundamental roots. 

\smallskip

For each $\alpha\in\Sigma$, let $w_\alpha\in{}W$ be the reflection in the 
hyperplane $\alpha^\perp\subseteq V$.  

\smallskip

For $\alpha\in\Sigma$ and $\lambda\in\fqobar^\times$, let 
$n_\alpha(\lambda)\in\gen{\4X_\alpha,\4X_{-\alpha}}$ and 
$h_\alpha(\lambda)\in\4T\cap\gen{\4X_\alpha,\4X_{-\alpha}}$ be as defined 
in \cite[\S\,6.4]{Carter} or \cite[Theorem 1.12.1]{GLS3}:  the images of 
$\mxtwo0\lambda{-\lambda^{-1}}0$ and $\mxtwo{\lambda}00{{\lambda}^{-1}}$, 
respectively, under the homomorphism $\SL_2(\fqobar)\Right2{}\4G$ which 
sends $\mxtwo1{u}01$ to $x_\alpha(u)$ and $\mxtwo10{v}1$ to 
$x_{-\alpha}(v)$.  Equivalently, $n_\alpha({\lambda})= 
x_\alpha({\lambda})x_{-\alpha}(-{\lambda}^{-1}) x_\alpha({\lambda})$ and 
$h_\alpha({\lambda})=n_\alpha({\lambda})n_\alpha(1)^{-1}$. 

\item\label{not3} {}\boldd{The maximal torus, root system and Weyl group of $G$.} 
Set $T=\4T\cap{}G$.

\noindent Let $\tau\in\Aut(V)$ and $\rho\in\Aut(\Sigma)$ be the orthogonal 
automorphism and permutation, respectively, such that for each 
$\alpha\in\Sigma$, $\sigma(\4X_\alpha)=\4X_{\rho(\alpha)}$ and 
$\rho(\alpha)$ is a positive multiple of $\tau(\alpha)$. Set 
$W_0=C_W(\tau)$. 

\smallskip

\noindent If $\rho(\Pi)=\Pi$, then set $V_0=C_V(\tau)$, and let 
$\pr_{V_0}^\perp$ be the orthogonal projection of $V$ onto $V_0$.  
Let $\5\Sigma$ be the set of 
equivalence classes in $\Sigma$ determined by $\tau$, where 
$\alpha,\beta\in\Sigma$ are equivalent if $\pr_{V_0}^\perp(\alpha)$ is a 
positive scalar multiple of $\pr_{V_0}^\perp(\beta)$ (see \cite[Definition 
2.3.1]{GLS3} or \cite[\S\,13.2]{Carter}).  Let $\5\Pi\subseteq\5\Sigma_+$ 
denote the images in $\5\Sigma$ of $\Pi\subseteq\Sigma_+$.  \smallskip

\noindent For each $\5\alpha\in\5\Sigma$, set 
$\4X_{\5\alpha}=\gen{\4X_\alpha\,|\,\alpha\in\5\alpha}$ and 
$X_{\5\alpha}=C_{\4X_{\5\alpha}}(\sigma)$. When $\alpha\in\Sigma$ 
is of minimal height in its class $\5\alpha\in\5\Sigma$, and 
$q'=|X_{\5\alpha}\ab|$, then for $u\in\F_{q'}$, let 
$\5x_\alpha(u)\in{}X_{\5\alpha}$ be an element whose image under projection 
to $X_\alpha$ is $x_\alpha(u)$ (uniquely determined modulo 
$[X_{\5\alpha},X_{\5\alpha}]$).  
\smallskip

\noindent For $\alpha\in\Pi$ and 
$\lambda\in\4\F_{q_0}^\times$, let $\5h_{\alpha}(\lambda)\in T$ be 
an element in $G\cap\gen{h_\beta(\4\F_{q_0}^\times)\,|\,\beta\in\5\alpha}$ 
whose component in $h_\alpha(\fqobar^\times)$ is $h_\alpha(\lambda)$ (if 
there is such an element). 

\end{enumA}
\end{Not}

To see that $\tau$ and $\rho$ exist as defined in point \eqref{not3}, 
recall that the root groups $\4X_\alpha$ for $\alpha\in\Sigma$ are the 
unique closed subgroups of $\4G$ which are isomorphic to $(\fqobar,+)$ and 
normalized by $\4T$ (see, e.g., \cite[Theorem 1.9.5(a,b)]{GLS3}). Since 
$\sigma$ is algebraic (hence continuous) and bijective, $\sigma^{-1}$ sends 
root subgroups to root subgroups, and $\sigma$ permutes the root 
subgroups (hence the roots) since there are only finitely many of them. 
Using Chevalley's commutator formula, one sees that this permutation $\rho$ 
of $\Sigma$ preserves angles between roots, and hence (up to positive 
scalar multiple) extends to an orthogonal automorphism of $V$. 

These definitions of $\5x_\alpha(u)\in{}X_{\5\alpha}$ and 
$\5h_\alpha(\lambda)\in{}T$ are slightly different from the definitions in 
\cite[\S\,2.4]{GLS3} of elements $x_{\5\alpha}(u)$ and 
$h_{\5\alpha}(\lambda)$. We choose this notation to emphasize that these 
elements depend on the choice of $\alpha\in\Sigma$, not only on its class 
$\5\alpha\in\5\Sigma$. This will be important in some of the relations we 
need to use in Section \ref{s:=}.

\begin{Lem} \label{W0onT}
Under the assumptions of Notation \ref{G-setup}, 
the action of $W$ on $\4T$ restricts to an action of $W_0$ on $T$, and 
the natural isomorphism 
$N_{\4G}(\4T)/\4T\cong W$ restricts to an isomorphism
	\[ \bigl( N_G(T) \cap N_{\4G}(\4T) \bigr) \big/ T \cong 
	C_W(\tau)=W_0\,. \]
\end{Lem}

\begin{proof} For each $\alpha\in\Sigma$, 
$n_\alpha(1)=x_\alpha(1)x_{-\alpha}(-1)x_\alpha(1)$ represents the 
reflection $w_\alpha\in{}W$, and hence 
$\sigma(n_\alpha)\in\gen{X_{\rho(\alpha)},X_{-\rho(\alpha)}}\cap 
N_{\4G}(\4T)$ represents the reflection 
$w_{\rho(\alpha)}=\9\tau(w_\alpha)$. Since $W$ is generated by the 
$w_\alpha$ for $\alpha\in\Sigma$, we conclude that $\sigma$ and $\tau$ have 
the same action on $W$.

Thus the identification $N_{\4G}(\4T)/\4T\cong W$ restricts to the 
following inclusions: 
	\[ \bigl( N_G(T) \cap N_{\4G}(\4T) \bigr) \big/ T  \le 
	C_{N_{\4G}(\4T)}(\sigma)/C_{\4T}(\sigma) \le
	C_{N_{\4G}(\4T)/\4T}(\sigma) \cong C_W(\tau) = W_0\,. \]
If $w\in W_0$ represents the coset $x\4T\subseteq N_{\4G}(\4T)$, then 
$x^{-1}\sigma(x)\in\4T$. By the Lang-Steinberg theorem, each element of 
$\4T$ has the form $t^{-1}\sigma(t)$ for some $t\in\4T$, and hence we can 
choose $x$ such that $\sigma(x)=x$. Then $x\in C_{\4G}(\sigma)$, and hence 
$x$ normalizes $G=O^{q_0'}(C_{\4G}(\sigma))$ and $T=G\cap\4T$. 
Since $C_{\4G}(\sigma)=GC_{\4T}(\sigma)$ (see \cite[Theorem 2.2.6(g)]{GLS3} 
or \cite[Corollary 12.3(a)]{Steinberg-end}), some element of $x\4T$ 
lies in $N_G(T)$. So the above inclusions are equalities.
\end{proof}

The roots in $\4G$ are defined formally as characters of its maximal torus 
$\4T$. But it will be useful to distinguish the (abstract) root 
$\alpha\in\Sigma$ from the character 
$\theta_\alpha\in\Hom(\4T,\fqobar^\times)\subseteq V$.

For each root $\alpha\in\Sigma\subseteq V$, let $\alpha^\vee\in V^*$ be the 
corresponding co-root (dual root): the unique element such that 
$(\alpha^\vee,\alpha)=2$ and $w_\alpha$ is reflection in the hyperplane 
$\Ker(\alpha^\vee)$. Since we identify $V=V^*$ via a $W$-invariant inner 
product, $\alpha^\vee=2\alpha/(\alpha,\alpha)$. Point \eqref{^txb(u)} of 
the next lemma says that $\alpha^\vee=h_\alpha$, when we regard  
$h_\alpha\in\Hom(\fqobar^\times,\4T)$ as an element in $V^*$.

\begin{Lem} \label{theta-r}
Assume we are in the situation of \eqref{not1} and \eqref{not2} in 
Notation \ref{G-setup}. 
\begin{enuma} 

\item \label{CG(T)=T} We have $C_{\4G}(\4T)=\4T$. In particular, 
$Z(\4G)\le\4T$, and is finite of order prime to the defining characteristic 
$q_0$.

\item \label{T=prod} The maximal torus $\4T$ in $\4G$ is 
generated by the elements $h_\alpha(\lambda)$ for $\alpha\in\Pi$ and 
$\lambda\in\4\F_{q_0}^\times$. If $\4G$ is universal, 
and $\lambda_\alpha\in\4\F_{q_0}$ are such that 
$\prod_{\alpha\in\Pi}h_\alpha(\lambda_\alpha)=1$, then $\lambda_\alpha=1$ 
for each $\alpha\in\Pi$.  Thus 
	\[ \4T = \prod_{\alpha\in\Pi} h_\alpha(\4\F_{q_0}^\times), \] 
and $h_\alpha$ is injective for each $\alpha$. 

\item \label{^txb(u)} For each $\beta\in\Sigma$, let $\theta_\beta\in 
X(\4T)=\Hom(\4T,\4\F_{q_0}^\times)$ be the character such that
	\[ \9tx_\beta(u)=x_\beta(\theta_\beta(t){\cdot}u) \]
for $t\in\4T$ and $u\in\4\F_{q_0}$. Then 
	\[ \theta_\beta(h_\alpha(\lambda)) 
	= \lambda^{(\alpha^\vee,\beta)} \qquad
	\textup{for $\beta,\alpha\in\Sigma$, $\lambda\in\fqobar^\times$.} \]
The product homomorphism $\theta_{\Pi}=\prod\theta_\beta\: 
\4T\Right3{}\prod_{\beta\in\Pi}\fqobar^\times$ is surjective, and 
$\Ker(\theta_\Pi)=Z(\4G)$.

\item \label{hb.hc} If $\alpha,\beta_1,\ldots,\beta_k\in\Sigma$ and 
$n_1,\ldots,n_k\in\Z$ are such that $\alpha^\vee = n_1\beta_1^\vee+ \ldots 
+n_k\beta_k^\vee$, 
then for each $\lambda\in\4\F_{q_0}^\times$, 
$h_\alpha(\lambda) = h_{\beta_1}(\lambda^{n_1}) \cdots 
h_{\beta_k}(\lambda^{n_k})$. 

\item \label{wa(hb)} For each $w\in{}W$, $\alpha\in\Sigma$, and 
$\lambda\in\fqobar^\times$, and each $n\in N_{\4G}(\4T)$ 
such that $n\4T=w\in{}N_{\4G}(\4T)/\4T\cong W$,
$\9n(\4X_\alpha)=\4X_{w(\alpha)}$ and 
$\9n(h_\alpha(\lambda))=h_{w(\alpha)}(\lambda)$.  
For each $\alpha,\beta\in\Sigma$ and each $\lambda\in\fqobar^\times$,
	\[ w_\alpha(h_\beta(\lambda)) = h_{w_\alpha(\beta)}(\lambda)
	= h_\beta(\lambda) 
	h_\alpha(\lambda^{-(\beta^\vee,\alpha)}) \,. \]
Hence $w_\alpha(t)=t\cdot h_\alpha(\theta_\alpha(t))^{-1}$ for each 
$t\in\4T$.

\end{enuma}
\end{Lem}

\begin{proof} 
\noindent\textbf{(\ref{CG(T)=T})} By \cite[Proposition 24.1.A]{Humphreys}, 
the maximal torus $\4T$ is regular (i.e., contained in only finitely many 
Borel subgroups). So $C_{\4G}(\4T)=\4T$ by \cite[Corollary 
26.2.A]{Humphreys}. Hence $Z(\4G)\le\4T$, it is finite since $\4G$ 
is assumed simple, and so it has order prime to the defining characteristic 
$q_0$.

\smallskip


We claim that it suffices to prove the relations in 
\eqref{^txb(u)}--\eqref{wa(hb)} in the adjoint group $\4G/Z(\4G)$, and 
hence that we can use the results in \cite[\S\S\,7.1--2]{Carter}. For 
relations in $\4T$, this holds since $\4T$ is infinitely divisible and 
$Z(\4G)$ is finite (thus each homomorphism to $\4T/Z(\4G)$ has at most one 
lifting to $\4T$). For relations in a root group $\4X_\alpha$, this holds 
since each element of $\4X_\alpha{}Z(\4G)$ of order $q_0$ lies in 
$\4X_\alpha$, since $|Z(\4G)|$ is prime to $q_0$ by \eqref{CG(T)=T}.

\smallskip

\noindent\textbf{(\ref{T=prod}) } This is stated without proof in 
\cite[Theorem 1.12.5(b)]{GLS3}, and with a brief sketch of a proof in 
\cite[p. 122]{Steinberg-grr}. We show here how it follows from the 
classification of reductive algebraic groups in terms of root data 
(see, e.g., \cite[\S\,10]{Springer}). 

Consider the homomorphism
	\[ h_\Pi\:\til{T} \defeq \prod_{\alpha\in\Pi}\fqobar^\times 
	\Right6{}\4T \] 
which sends $(\lambda_\alpha)_{\alpha\in\Pi}$ to 
$\prod_{\alpha}h_\alpha(\lambda_\alpha)$. Then $h_\Pi$ is surjective with 
finite kernel (see \cite[\S\,7.1]{Carter}). It remains to show that it is 
an isomorphism when $G$ is of universal type.

We recall some of the notation used in \cite[\S\,7]{Springer}. To $\4G$ is 
associated the root datum 
$\bigl(X(\4T),\Sigma,X^\vee(\4T),\Sigma^\vee\bigr)$, where 
	\[ X(\4T) = \Hom(\4T,\fqobar^\times), \quad X^\vee(\4T) = 
	\Hom(\fqobar^\times,\4T),\quad 
	\Sigma^\vee=\{\alpha^\vee=h_\alpha\,|\,\alpha\in\Sigma\} 
	\subseteq X^\vee(\4T)\,. \]
As noted before, $X(\4T)$ and $X^\vee(\4T)$ are groups of algebraic 
homomorphisms, and are free abelian groups of finite rank dual to each 
other. Recall that $\Sigma\subseteq X(\4T)$, since we identify a root 
$\alpha$ with the character $\theta_\alpha$.


Set $Y^\vee=\Z\Sigma^\vee\subseteq X^\vee(\4T)$, and let $Y\supseteq 
X(\4T)$ be its dual. Then $(Y,\Sigma,Y^\vee,\Sigma^\vee)$ is still a root 
datum as defined in \cite[\S\,7.4]{Springer}. By \cite[Proposition 
10.1.3]{Springer} and its proof, it is realized by a connected algebraic 
group $\til{G}$ with maximal torus $\til{T}$, which lies in a central 
extension $f\:\til{G}\Right2{}\4G$ which extends $h_\Pi$. Since $\4G$ is 
of universal type, $f$ and hence $h_\Pi$ are isomorphisms.

\smallskip

\noindent\textbf{(\ref{^txb(u)}) } Let $\Z\Sigma\le{}V$ be the additive 
subgroup generated by $\Sigma$. In the notation of \cite[pp. 
97--98]{Carter}, for each $\alpha\in\Sigma$ and 
$\lambda\in\4\F_{q_0}^\times$, $h_\alpha(\lambda)=h(\chi_{\alpha,\lambda})$ 
where 
	\[ \chi_{\alpha,\lambda}\in\Hom(\Z\Sigma,\4\F_{q_0}^\times)
	\qquad\textup{is defined by}\qquad
	\chi_{\alpha,\lambda}(v)=\lambda^{2(\alpha,v)/(\alpha,\alpha)}
	= \lambda^{(\alpha^\vee,v)}. \]
Also, by \cite[p. 100]{Carter}, for each 
$\chi\in\Hom(\Z\Sigma,\4\F_{q_0}^\times)$, $\beta\in\Sigma$, and 
$u\in\4\F_{q_0}$, 
$\9{h(\chi)}x_\beta(u)=x_\beta(\chi(\beta){\cdot}u)$.  Thus there 
are homomorphisms $\theta_\beta\in\Hom(\4T,\4\F_{q_0}^\times)$, for each 
$\beta\in\Sigma$, such that 
$\9tx_\beta(u)=x_\beta(\theta_\beta(t){\cdot}u)$, and 
$\theta_\beta(h(\chi))=\chi(\beta)$ for each $\chi$.  
For each $\alpha\in\Sigma$ and $\lambda\in\4\F_{q_0}^\times$,
	\beqq \theta_\beta(h_\alpha(\lambda)) 
	= \theta_\beta(h(\chi_{\alpha,\lambda})) = \chi_{\alpha,\lambda}(\beta) 
	= \lambda^{(\alpha^\vee,\beta)}~. \label{e:thb(h)} \eeqq

Assume $t\in\Ker(\theta_\Pi)$. Thus $t\in\Ker(\theta_\alpha)$ for all 
$\alpha\in\Pi$, and hence for all $\alpha\in\Sigma\subseteq\Z\Pi$. So 
$[t,X_\alpha]=1$ for all $\alpha\in\Sigma$, these root subgroups generate 
$\4G$ (see \cite[Corollary 8.2.10]{Springer}), and this proves that 
$t\in{}Z(\4G)$. The converse is clear: $t\in Z(\4G)$ implies $t\in\4T$ by 
\eqref{CG(T)=T}, and hence $\theta_\beta(t)=1$ for all $\beta\in\Pi$ by 
definition of $\theta_\beta$.

It remains to show that $\theta_{\Pi}$ sends $\4T$ 
onto $\prod_{\beta\in\Pi}\fqobar$. Consider the homomorphisms
	\beqq \til{T} \defeq \prod_{\alpha\in\Pi} \fqobar^\times 
	\Right6{h_\Pi} \4T \Right6{\theta_\Pi} \prod_{\beta\in\Pi} 
	\fqobar^\times \,, \label{e:theta.h} \eeqq
where $h_\Pi$ was defined in the proof of \eqref{T=prod}. 
We just saw that $\theta_\Pi\circ h_\Pi$ has 
matrix $\bigl((\alpha^\vee,\beta)\bigr)_{\alpha,\beta\in\Pi}$, which has 
nonzero determinant since $\Pi\subseteq V$ and $\Pi^\vee\subseteq V^*$ are 
bases. Since $\fqobar^\times$ is divisible and its finite subgroups are 
cyclic, this implies that $\theta_\Pi\circ h_\Pi$ is onto, and hence 
$\theta_\Pi$ is onto.

\smallskip

\noindent\textbf{(\ref{hb.hc}) } This follows immediately from 
\eqref{^txb(u)}, where we showed, for $\alpha\in\Sigma$, that $\alpha^\vee$
can be identified with $h_\alpha$ in $\Hom(\fqobar^\times,\4T)\subseteq 
V^*$.

\smallskip

\noindent\textbf{(\ref{wa(hb)}) } The first statement 
($\9n(\4X_\alpha)=\4X_{w(\alpha)}$ and 
$\9n(h_\alpha(\lambda))=h_{w(\alpha)}(\lambda)$) is shown in \cite[Lemma 
7.2.1(ii) \& Theorem 7.2.2]{Carter}.  By the usual formula for an 
orthogonal reflection, $w_\alpha(\beta)= 
\beta-\frac{2(\alpha,\beta)}{(\alpha,\alpha)}\alpha = 
\beta-(\alpha^\vee,\beta)\alpha$. Here, we regard 
$w_\alpha$ as an automorphism of $V$ (not of $\4T$).  Since 
$w_\alpha(\beta)$ and $\beta$ have the same norm,
	\[ w_\alpha(\beta)^\vee = 
	\frac{2w_\alpha(\beta)}{(\beta,\beta)} = 
	\frac{2\beta}{(\beta,\beta)} - \frac{2(\alpha,\beta)}{(\beta,\beta)} 
	\cdot \frac{2\alpha}{(\alpha,\alpha)}
	= \beta^\vee - (\beta^\vee,\alpha) \cdot \alpha^\vee \,, \]
and by \eqref{hb.hc}, 
	\[ w_\alpha(h_\beta(\lambda)) = h_{w_\alpha(\beta)}(\lambda) 
	= h_\beta(\lambda) 
	h_\alpha(\lambda^{-(\beta^\vee,\alpha)}) 
	= h_\beta(\lambda) h_\alpha(\theta_\alpha(h_\beta(\lambda))^{-1}) \]
where the last equality follows from \eqref{^txb(u)}. Since $\4T$ is 
generated by the $h_\beta(\lambda)$ by \eqref{T=prod}, this implies that 
$w_\alpha(t)=t\cdot h_\alpha(\theta_\alpha(t))^{-1}$ for all $t\in\4T$. 
\end{proof}

For any algebraic group $H$, $H^0$ denotes its identity connected 
component.  The following proposition holds for any connected, reductive 
group, but we state it only in the context of Notation \ref{G-setup}.  
Recall the homomorphisms $\theta_\beta\in\Hom(\4T,\4\F_{q_0}^\times)$, 
defined for $\beta\in\Sigma$ in Lemma \ref{theta-r}\eqref{^txb(u)}.

\begin{Prop} \label{p:CG(T)}
Assume Notation \ref{G-setup}.  
For any subgroup $H\le\4T$, 
$C_{\4G}(H)$ is an algebraic group, $C_{\4G}(H)^0$ is reductive, and
	\beqq \begin{split} 
	C_{\4G}(H)^0 &= \gen{\4T,X_\alpha\,|\, \alpha\in\Sigma,~ 
	H\le\Ker(\theta_\alpha)} \\
	C_{\4G}(H) &= C_{\4G}(H)^0\cdot\{g\in{}N_{\4G}(\4T)\,|\,
	[g,H]=1\}~. \label{e:5.4a}
	\end{split} \eeqq
If, furthermore, $\4G$ is of universal type, then $Z(\4G)=C_{\4T}(W)$.  
\end{Prop}

\begin{proof}  The description of $C_{\4G}(H)^0$ is shown in 
\cite[Theorem 3.5.3]{Carter2} when $H$ is finite and cyclic, and the proof 
given there also applies in the more general case.  For each 
$g\in{}C_{\4G}(H)$, $c_g(\4T)$ is another maximal torus in $C_{\4G}(H)^0$, 
so $gh\in{}C_{N_{\4G}(\4T)}(H)$ for some $h\in{}C_{\4G}(T)^0$, and thus 
$C_{\4G}(H)=C_{\4G}(H)^0\cdot{}C_{N_{\4G}(\4T)}(H)$.  

Assume $\4G$ is of universal type. Since $Z(\4G)\le\4T$ by Lemma 
\ref{theta-r}\eqref{CG(T)=T}, we have $Z(\4G)\le{}C_{\4T}(W)$. Conversely, 
by Lemma \ref{theta-r}\eqref{T=prod}, for each $t\in \4T$ and each 
$\alpha\in\Sigma$, $\9t(x_\alpha(u))=x_\alpha(\theta_\alpha(t)u)$, and 
$\theta_{-\alpha}(t)=\theta_\alpha(t)^{-1}$.  Hence also 
$\9t(n_\alpha(1))=n_\alpha(\theta_\alpha(t))$ (see the formula for 
$n_\alpha(\lambda)$ in Notation \ref{G-setup}\eqref{not2}). If 
$t\in{}C_{\4T}(W)$, then $[t,n_\alpha(1)]=1$ for each $\alpha$, and since 
$\4G$ is of universal type, 
$\gen{\4X_{\alpha},\4X_{-\alpha}}\cong\SL_2(\fqobar)$.  Thus 
$\theta_\alpha(t)=1$ for all $\alpha\in\Sigma$, $t$ acts trivially on all 
root subgroups, and so $t\in Z(\4G)$. 
\end{proof}

We now look more closely at the lattice $\Z\Sigma^\vee$ generated by the 
dual roots. 

\begin{Lem} \label{l:CW(t)}
Assume Notation \ref{G-setup}(\ref{not1},\ref{not2}), and also that $\gg$ 
(and hence $\4G$) is of universal type. 
\begin{enuma} 

\item There is an isomorphism 
	\[ \Phi\:\Z\Sigma^\vee\otimes_{\Z}\fqobar^\times \Right5{}\4T \]
with the property that $\Phi(\alpha^\vee\otimes\lambda)=h_\alpha(\lambda)$ for each 
$\alpha\in\Sigma$ and each $\lambda\in\fqobar^\times$.
\end{enuma}
Fix some $\lambda\in\fqobar^\times$, and set $m=|\lambda|$. Set 
$\Phi_\lambda=\Phi(-,\lambda)\:\Z\Sigma^\vee\Right2{}\4T$.
\begin{enuma}[resume]
\item The map $\Phi_\lambda$ is $\Z[W]$-linear, 
$\Ker(\Phi_\lambda)=m\Z\Sigma^\vee$, and 
$\Im(\Phi_\lambda)=\{t\in\4T\,|\,t^{m}=1\}$.

\item Fix $t\in\4T$ and $x\in\Z\Sigma^\vee$ such that $\Phi_\lambda(x)=t$, 
and also such that 
	\[ \norm{x} < \tfrac12m\cdot\min\bigl\{\norm{\alpha^\vee} 
	\,\big|\, \alpha\in\Pi \bigr\}. \]
Then $C_W(t)=C_W(x)$.

\item If $m=|\lambda|\ge4$, then for each $\alpha\in\Sigma$, 
$C_W(h_\alpha(\lambda))=C_W(\alpha)$. 

\end{enuma}
\end{Lem}

\begin{proof} \textbf{(a,b) } Identify $\Z\Sigma^\vee$ as a subgroup of 
$\Hom(\fqobar^\times,\4T)$, and let 
	\[ \4\Phi\: \Z\Sigma^\vee \times \fqobar^\times \Right5{} \4T \]
be the evaluation pairing. This is bilinear, hence induces a homomorphism 
on the tensor product, and $\4\Phi(\alpha^\vee,\lambda)=h_\alpha(\lambda)$ 
by Lemma \ref{theta-r}\eqref{^txb(u)}. Since 
$\{\alpha^\vee\,|\,\alpha\in\Pi\}$ is a $\Z$-basis for $\Z\Sigma^\vee$ 
(since $\Sigma^\vee$ is a root system by \cite[\S\,VI.1, Proposition 
2]{Bourb4-6}), and since $\4G$ is of universal type, $\Phi$ is an isomorphism 
by Lemma \ref{theta-r}\eqref{T=prod}.

In particular, for fixed $\lambda\in\fqobar^\times$ of order $m$, 
$\Phi(-,\lambda)$ induces an isomorphism from the quotient group
$\Z\Sigma^\vee/m\Z\Sigma^\vee$ onto the $m$-torsion subgroup of $\4T$.

\smallskip

\noindent\textbf{(c) } Clearly, $C_W(x)\le 
C_W(t)$; it remains to prove the opposite inclusion. Fix $w\in C_W(t)$. By 
(a), $w(x)\equiv x$ (mod $m\Z\Sigma^\vee$).

Set $r=\min\bigl\{\norm{\alpha^\vee}\,\big|\,\alpha\in\Pi\bigr\}$.  For 
each $\alpha\in\Sigma$, $\norm{\alpha^\vee}=\sqrt{k}\cdot r$ for some
$k=1,2,3$, and hence $(\alpha^\vee,\alpha^\vee)\in r^2\Z$. For each 
$\alpha,\beta\in\Sigma$, 
$2(\alpha^\vee,\beta^\vee)\big/(\alpha^\vee,\alpha^\vee)\in\Z$ (cf. 
\cite[Definition 2.1.1]{Carter}), and hence 
$(\alpha^\vee,\beta^\vee)\in\frac12r^2\Z$. Thus $(x,x)\in r^2\Z$ 
for each $x\in\Z\Sigma^\vee$, and in particular,
$\min\bigl\{\norm{x}\,\big|\,0\ne{}x\in\Z\Sigma^\vee\bigr\}=r$.

By assumption, $\norm{w(x)}=\norm{x}<mr/2$, so $\norm{w(x)-x}<mr$. Since 
each nonzero element in $m\Z\Sigma^\vee$ has norm at least $mr$, this 
proves that $w(x)-x=0$, and hence that $w\in C_W(x)$.

\smallskip

\noindent\textbf{(d) } This is the special case of (b), where 
$x=\alpha^\vee$ and $t=h_\alpha(\lambda)$. 
\end{proof}

\begin{Lem} \label{scal-mod-m}
Assume Notation \ref{G-setup}, and assume also that $\gg$ is of universal 
type. Let $\Gamma<\Aut(V)$ be any finite group of isometries of 
$(V,\Sigma)$. Then there is an action of $\Gamma$ on $\4T$, where 
$g(h_\alpha(u))=h_{g(\alpha)}(u)$ for each $g\in\Gamma$, $\alpha\in\Sigma$, 
and $u\in\fqobar^\times$. 
Fix $m\ge3$ such that $q_0\nmid m$, and set $T_m=\{t\in\4T\,|\,t^m=1\}$. 
Then $\Gamma$ acts faithfully on $T_m$. If $1\ne{}g\in\Gamma$ and 
$\ell\in\Z$ are such that $g(t)=t^\ell$ for each $t\in{}T_m$, then 
$\ell\equiv-1$ (mod $m$). 
\end{Lem}

\begin{proof} The action of $\Gamma$ on $\4T$ is well defined by the 
relations in Lemma \ref{theta-r}(\ref{hb.hc},\ref{T=prod}). 

Now fix $m\ge3$ prime to $q_0$, and let $T_m<\4T$ be the $m$-torsion 
subgroup. It suffices to prove the rest of the lemma when $m=p$ is an odd 
prime, or when $m=4$ and $p=2$. Fix $\lambda\in\fqobar^\times$ of order 
$m$, and let $\Phi_\lambda\:\Z\Sigma^\vee\Right3{}\4T$ be the homomorphism 
of Lemma \ref{l:CW(t)}(a). By definition of $\Phi_\lambda$, it commutes 
with the actions of $\Gamma$ on $\Z\Sigma^\vee<V$ and on $T_m$. 

Assume $1\ne{}g\in\Gamma$ and $\ell\in\Z$ are such that $g(t)=t^\ell$ for 
each $t\in{}T_m$. Set $r=\dim(V)$, and let $B\in\GL_r(\Z)$ be the matrix for the 
action of $g$ on $\Z\Sigma^\vee$, with respect to some $\Z$-basis of 
$\Z\Sigma^\vee$. Then $|g|=|B|$, and $B\equiv\ell{}I$ (mod $mM_r(\Z)$). 
If $p=2$ ($m=4$), let $\mu\in\{\pm1\}$ be such that $\ell\equiv\mu$ (mod 
$4$). If $p$ is odd (so $m=p$), then let $\mu\in(\Z_p)^\times$ be such that 
$\mu\equiv\ell$ (mod $p$) and $\mu^{p-1}=1$. Set 
$B'=\mu^{-1}B\in\GL_r(\Z_p)$. Thus $B'$ also has finite order, and 
$B'\equiv I$ (mod $mM_r(\Z_p))$. 

The logarithm and exponential maps define inverse bijections
	\[ I + m M_r(\Z_p) \RLEFT5{\ln}{\exp} m 
	M_r(\Z_p)\,. \]
They are not homomorphisms, but they do have the property that 
$\ln(M^k)=k\ln(M)$ for each $M\in I+mM_r(\Z_p)$ and each $k\ge1$. In 
particular, the only element of finite order in 
$I+mM_r(\Z_p)$ is the identity. 
Thus $B'=I$, so $B=\mu{}I$. Since $\mu\in\Z$ and $B\ne{}I$, we have 
$\mu=-1$ and $B=-I$.
\end{proof}

The following lemma about the lattice $\Z\Sigma^\vee$ will also be useful 
when working with the Weyl group action on certain subgroups of $\4T$.

\begin{Lem} \label{l:Lambda}
Assume Notation \ref{G-setup}(\ref{not1},\ref{not2}). Set 
$\Lambda=\Z\Sigma^\vee$: the lattice in $V$ generated by the dual roots. 
Assume that there are $b\in W$ of order $2$, and a splitting 
$\Lambda=\Lambda_+\times\Lambda_-$, such that $\Lambda_+,\Lambda_-\ne0$ and 
$b$ acts on $\Lambda_\pm$ via $\pm\Id$. Then $\gg\cong C_n$ ($=\Sp_{2n}$) for some 
$n\ge2$.
\end{Lem}

\begin{proof} Fix $b\in{}w$ and a splitting 
$\Lambda=\Lambda_+\times\Lambda_-$ as above. When considering individual 
cases, we use the notation of Bourbaki \cite[Planches I--IX]{Bourb4-6} to 
describe the (dual) roots, lattice, and Weyl group. 
\begin{itemize} 
\item If $\gg=A_n$ ($n\ge2$), then 
$\Lambda=\bigl\{(a_0,\ldots,a_n)\in\Z^{n+1}\,\big|\,a_0+\ldots+a_n=0\bigr\}$, 
and $b$ exchanges certain coordinates pairwise. Choose $v\in\Lambda$ with 
coordinates $1$, $-1$, and otherwise $0$; where the two nonzero entries are 
in separate orbits of $b$ of which at least one is nonfixed. Then 
$v\notin\Lambda_+\times\Lambda_-$, a contradiction. 

\item If $\gg=G_2$, then as described in \cite[Planche IX]{Bourb4-6}, 
$\Lambda$ is generated by the dual fundamental roots $(1,-1,0)$ and 
$(\frac23,-\frac13,-\frac13)$, and does not have an orthogonal basis.

\item If $\gg=B_n$ ($n\ge3$), $D_n$ ($n\ge4$), or $F_4$, then $\Lambda<\Z^n$ is 
the sublattice of $n$-tuples the sum of whose coordinates is even. Also, 
$b$ acts by permuting the coordinates and changing sign (or we can assume 
it acts this way in the $F_4$ case). Choose $v$ with two $1$'s and the rest 
$0$, where the $1$'s are in separate $b$-orbits, of which either at least 
one is nonfixed, or both are fixed and exactly one is negated. Then 
$v\notin\Lambda_+\times\Lambda_-$, a contradiction. 

\item If $\gg=E_8$, then $\Lambda=\Lambda(E_8)<\R^8$ is generated by 
$\frac12(1,1,\ldots,1)$ and the $n$-tuples of integers whose sum is 
even. We can assume (up to conjugation) that $b$ acts as a signed 
permutation. Choose $v$ as in the last case.

\item If $\gg=E_7$, then $\Lambda<\R^8$ is the lattice of 
all $x=(x_1,\ldots,x_8)\in\Lambda(E_8)$ such that $x_7=-x_8$. 
Up to conjugation, $b$ can be again be assumed to act on 
$A$ via a signed permutation (permuting only the first six coordinates), 
and $v$ can be chosen as in the last case.

\item If $\gg=E_6$, then $\Lambda<\R^8$ is the lattice of 
all $x=(x_1,\ldots,x_8)\in\Lambda(E_8)$ such that $x_6=x_7=-x_8$. 
Also, $W$ contains a subgroup isomorphic to $2^4:S_5$ with 
odd index which acts on the remaining five coordinates via signed 
permutations. So $b$ and $v$ can be taken as in the last three cases. 
\qedhere
\end{itemize}
\end{proof}

We finish the section with a very elementary lemma.

It will be useful to know, in certain situations, that each coset of $\4T$ 
in $N_{\4G}(\4T)$ contains elements of $G$.

\begin{Lem} \label{l:gT}
Assume that we are in the situation of Notation 
\ref{G-setup}(\ref{not1},\ref{not2}). 
Assume also that $\sigma$ acts on $\4T$ via $(t\mapsto t^m)$ for some 
$1\ne m\in\Z$. Then for each $g\in{}N_{\4G}(\4T)$, $g\4T\cap 
C_{\4G}(\sigma)\ne\emptyset$. 
\end{Lem}

\begin{proof} Since $\sigma|_{\4T}\in Z(\Aut(\4T))$, we have 
$g^{-1}\sigma(g)\in C_{\4G}(\4T)=\4T$, the last equality by Lemma 
\ref{theta-r}\eqref{CG(T)=T}. So for each $t\in\4T$, $\sigma(gt)=gt$ if and 
only if $g^{-1}\sigma(g)=t^{1-m}$. Since $\4T\cong(\fqobar^\times)^r$ for 
some $r$, and $\fqobar$ is algebraically complete (and $1-m\ne0$), this 
always has solutions. 
\end{proof}

\newpage

\newsect{Automorphisms of groups of Lie type}
\label{s:aut}

Since automorphisms of $G$ play a central role in this paper, we need to 
fix our notation (mostly taken from \cite{GLS3}) for certain subgroups 
and elements of $\Aut(G)$. We begin with automorphisms of the algebraic 
group $\4G$.

\begin{Defi} \label{d:Aut(Gbar)}
Let $\4G$ and its root system $\Sigma$ be as in Notation 
\ref{G-setup}(\ref{not1},\ref{not2}).
\begin{enuma} 

\item When $q$ is any power of $q_0$ (the defining characteristic of 
$\4G$), let $\psi_q\in\End(\4G)$ be the field endomorphism defined by 
$\psi_q(x_\alpha(u))=x_\alpha(u^q)$ for each $\alpha\in\Sigma$ and each 
$u\in\fqobar$. Set $\Phi_{\4G}=\{\psi_{q_0^b}\,|\,b\ge1\}$: the monoid of 
all field endomorphisms of $\4G$. 

\item Let $\Gamma_{\4G}$ be the group or set of graph automorphisms of $\4G$ as 
defined in \cite[Definition 1.15.5(e)]{GLS3}. Thus when 
$(\gg,q_0)\ne(B_2,2)$, $(G_2,3)$, nor $(F_4,2)$, $\Gamma_{\4G}$ is the group of 
all $\gamma\in\Aut(\4G)$ of the form 
$\gamma(x_\alpha(u))=x_{\rho(\alpha)}(u)$ (all $\alpha\in\pm\Pi$ and 
$u\in\fqobar$) for some isometry $\rho$ of 
$\Sigma$ such that $\rho(\Pi)=\Pi$. If $(\gg,q_0)=(B_2,2)$, $(G_2,3)$, 
or $(F_4,2)$, then $\Gamma_{\4G}=\{1,\psi\}$, where for the 
angle-preserving permutation $\rho$ of $\Sigma$ which exchanges long and 
short roots and sends $\Pi$ to itself, 
$\psi(x_\alpha(u))=x_{\rho(\alpha)}(u)$ when $\alpha$ is a long root and 
$\psi(x_\alpha(u))=x_{\rho(\alpha)}(u^{q_0})$ when $\alpha$ is short.

\item A Steinberg endomorphism $\sigma$ of $\4G$ is ``standard'' if 
$\sigma=\psi_q\circ\gamma=\gamma\circ\psi_q$, where $q$ is a power of $q_0$ 
and $\gamma\in\Gamma_{\4G}$. A $\sigma$-setup $(\4G,\sigma)$ for a finite 
subgroup $G<\4G$ is standard if $\sigma$ is standard.
\end{enuma}
\end{Defi}

By \cite[Theorem 2.2.3]{GLS3}, for any $G$ with $\sigma$-setup 
$(\4G,\sigma)$ as in Notation \ref{G-setup}, $G$ is $\4G$-conjugate to a 
subgroup $G^*$ which has a standard $\sigma$-setup. This will be made more 
precise in Proposition \ref{nonstandard}(a).

Most of the time in this paper, we will be working with standard 
$\sigma$-setups. But there are a few cases where we will need to work with 
setups which are not standard, which is why this condition is not included 
in Notation \ref{G-setup}. 

Following the usual terminology, we call $G$ a ``Chevalley group'' if it 
has a standard $\sigma$-setup where $\gamma=\Id$ in the notation of 
Definition \ref{d:Aut(Gbar)}; i.e., if $G\cong\gg(q)$ where $q$ is some 
power of $q_0$.  In this case, the root groups $X_{\5\alpha}$ are all 
abelian and isomorphic to $\F_q$.  When $G$ has a standard $\sigma$-setup 
with $\gamma\ne\Id$, we refer to $G$ as a ``twisted group'', and the 
different possible structures of its root groups are described in 
\cite[Table 2.4]{GLS3}. We also refer to $G$ as a ``Steinberg group'' 
if $\gamma\ne\Id$ and is an algebraic automorphism of $\4G$; i.e., if $G$ 
is a twisted group and not a Suzuki or Ree group.

The following lemma will be useful in Sections \ref{s:X1} and \ref{s:X2}.

\begin{Lem} \label{l:tau}
Assume $\4G$ is as in Notation \ref{G-setup}(\ref{not1},\ref{not2}). Then 
for each algebraic automorphism $\gamma$ of $\4G$ which normalizes $\4T$, 
there is an orthogonal automorphism $\tau$ of $V$ such that 
$\tau(\Sigma)=\Sigma$, and 
	\[ \gamma(\4X_\alpha)=\4X_{\tau(\alpha)} \qquad\textup{and}\qquad
	\gamma(h_\alpha(\lambda))=h_{\tau(\alpha)}(\lambda) \]
for each $\alpha\in\Sigma$ and each $\lambda\in\fqobar^\times$. In 
particular, $\bigl|\gamma|_{\4T}\bigr|=|\tau|<\infty$. If, in addition, 
$\gamma$ normalizes each of the root groups $\4X_\alpha$ (i.e., 
$\tau=\Id$), then $\gamma\in\Aut_{\4T}(\4G)$.
\end{Lem}

\begin{proof} By \cite[Theorem 1.15.2(b)]{GLS3}, and since $\gamma$ is an 
algebraic automorphism of $\4G$, $\gamma=c_g\circ\gamma_0$ for some 
$g\in\4G$ and some $\gamma_0\in\Gamma_{\4G}$. Furthermore, $\gamma_0$ has 
the form: $\gamma_0(x_\alpha(u))=x_{\chi(\alpha)}(u)$ for all 
$\alpha\in\Sigma$ and $u\in\fqobar$, and some isometry $\chi\in\Aut(V)$ 
such that $\chi(\Pi)=\Pi$. 
Since $\gamma$ and $\gamma_0$ both normalize $\4T$, we have $g\in 
N_{\4G}(\4T)$. 


Thus by Lemma \ref{theta-r}\eqref{wa(hb)}, 
there is $\tau\in\Aut(V)$ such that $\tau(\Sigma)=\Sigma$, and 
$\gamma(\4X_\alpha)=\4X_{\tau(\alpha)}$ and 
$\gamma(h_\alpha(\lambda))=h_{\tau(\alpha)}(\lambda)$ for each 
$\alpha\in\Sigma$ and $\lambda\in\fqobar^\times$. In particular, 
$\bigl|\gamma|_{\4T}\bigr|=|\tau|$.

If $\tau=\Id$, then $\gamma_0=\Id$ and $g\in\4T$. Thus 
$\gamma\in\Aut_{\4T}(\4G)$. 
\end{proof}

We next fix notation for automorphisms of $G$.

\begin{Defi} \label{d:Aut(G)}
Let $\4G$ and $G$ be as in Notation 
\ref{G-setup}(\ref{not1},\ref{not2},\ref{not3}), where in 
addition, we assume the $\sigma$-setup is standard.
\begin{enuma}
\item Set 
	\[ \Inndiag(G) = \Aut_{\4T}(G)\Inn(G) \qquad\textup{and}\qquad
	\Outdiag(G)=\Inndiag(G)/\Inn(G)\,. \]

\item Set $\Phi_G=\bigl\{\psi_q|_G\,\big|\,q=q_0^b,~b\ge1\bigr\}$, the 
group of field automorphisms of $G$.  

\item If $G$ is a Chevalley group, set 
$\Gamma_G=\bigl\{\gamma|_G\,\big|\,\gamma\in\Gamma_{\4G}\bigr\}$, the group of graph 
automorphisms of $G$. Set $\Gamma_G=1$ if $G$ is a twisted group (a Steinberg, 
Suzuki, or Ree group).  \\

\end{enuma}
\end{Defi}

Note that in \cite[Definition 2.5.13]{GLS3}, when $G$ has a standard 
$\sigma$-setup $(\4G,\sigma)$, $\Inndiag(G)$ is defined to 
be the group of automorphisms induced by conjugation by elements of 
$C_{\4G/Z(\4G)}(\sigma)$ (lifted to $\4G$). By \cite[Lemma 2.5.8]{GLS3}, 
this is equal to $\Inndiag(G)$ as defined above when $\4G$ is of adjoint 
form, and hence also in the general case (since $Z(\4G)\le\4T$). 

Steinberg's theorem on automorphisms of groups of Lie type can now be 
stated.

\begin{Thm}[{\cite[\S\,3]{Steinberg-aut}}] \label{St-aut}
Let $G$ be a finite group of Lie type. Assume that $(\4G,\sigma)$ is a 
standard $\sigma$-setup for $G$, where $\4G$ is in adjoint or universal 
form. Then $\Aut(G)=\Inndiag(G)\Phi_G\Gamma_G$, where 
$\Inndiag(G)\nsg\Aut(G)$ and $\Inndiag(G)\cap(\Phi_G\Gamma_G)=1$.
\end{Thm}

\begin{proof} See, e.g., \cite[Theorem 2.5.12(a)]{GLS3} 
(together with \cite[Theorem 2.5.14(d)]{GLS3}).  Most of this follows 
from the main result in \cite{Steinberg-aut}, and from 
\cite[Theorems 30 \& 36]{Steinberg-lect}.  
\end{proof}

We also need the following characterizations of $\Inndiag(G)$ which are 
independent of the choice of $\sigma$-setup.

\begin{Prop} \label{p:Autdiag(G)}
Assume the hypotheses and notation in \ref{G-setup}. Then 
\begin{enuma} 
\item $C_{\4G}(G)=Z(\4G)$;

\item $N_{\4G}(G)=GN_{\4T}(G)$;

\item $\Inndiag(G)=\Aut_{\4T}(G)\Inn(G)=\Aut_{\4G}(G)$, and hence 
$\Outdiag(G)=\Out_{\4T}(G)$. 

\end{enuma}
In fact, (b) and (c) hold if we replace $\4T$ by \emph{any} 
$\sigma$-invariant maximal torus in $\4G$. 
\end{Prop}

\begin{proof} \textbf{(a) } Since the statement is independent of the 
choice of $\sigma$-setup, we can assume that $\sigma$ is standard. Set 
$\4U=\prod_{\alpha\in\Sigma_+}\4X_\alpha$ and 
$\4U^*=\prod_{\alpha\in\Sigma_+}\4X_{-\alpha}$. 

Fix $g\in C_{\4G}(G)$. Since $\4G$ has a $BN$-pair (see \cite[Proposition 
8.2.1]{Carter}), it has a Bruhat decomposition $\4G=\4B\4N\4B=\4U\4N\4U$ 
\cite[Proposition 8.2.2(i)]{Carter}, where $\4B=\4T\4U$ and 
$\4N=N_{\4G}(\4T)$. Write $g=unv$, where $u,v\in\4U$ and $n\in\4N$. For 
each $x\in\4U\cap G$, $\9gx=\9u(\9{nv}x)\in\4U$ implies that 
$\9{nv}x=\9n(\9vx)\in\4U$. 

Since $n\in{}N_{\4G}(\4T)$, conjugation by $n$ permutes the root groups of 
$\4G$, in a way determined by the class $w=n\4T\in W=N_{\4G}(\4T)/\4T$. 
Thus $w$ sends each (positive) root in the decomposition of $\9vx$ to a 
positive root. For each $\alpha\in\Sigma_+$, $\5x_\alpha(1)\in G$, 
$\9v(\5x_\alpha(1))$ has $\alpha$ in its decomposition, and hence 
$w(\alpha)\in\Sigma_+$.

Thus $w$ sends all positive roots to positive roots, so $w(\Pi)=\Pi$, and 
$w=1$ by \cite[Corollary 2.2.3]{Carter}. So $n\in\4T$, and 
$g=unv\in\4T\4U$.

By the same argument applied to the negative root groups, $g\in\4T\4U^*$.  
Hence $g\in\4T$.

For each $\alpha\in\Sigma$, $g\in\4T$ commutes with $\5x_\alpha(1)\in G$, 
and hence $g$ centralizes $\4X_\beta$ for each $\beta\in\5\alpha$ (Lemma 
\ref{theta-r}\eqref{^txb(u)}). Thus $g$ centralizes all root groups in 
$\4G$, so $g\in Z(\4G)$. 

\smallskip

\noindent\textbf{(b) } Let $\4T^*$ be any $\sigma$-invariant maximal 
torus in $\4G$. Fix $g\in N_{\4G}(G)$. Then $g^{-1}\cdot \sigma(g) \in 
C_{\4G}(G) = Z(\4G) \leq \4T^*$ by (a). By Lang's theorem \cite[Theorem 
2.1.1]{GLS3}, there is  $t\in \4T^*$ such that $g^{-1}\cdot \sigma(g) = 
t^{-1}\cdot \sigma(t)$. Hence $g t^{-1} \in C_{\4G}(\sigma ) =  G \cdot 
C_{\4T^*}(\sigma)$, where the last equality holds by \cite[Theorem 
2.2.6(g)]{GLS3}. So $g\in G\4T^*$, and $g\in{}GN_{\4T^*}(G)$ since $g$ 
normalizes $G$.

\smallskip

\noindent\textbf{(c) } By (b), $\Aut_{\4G}(G)=\Aut_{\4T^*}(G)\Inn(G)$ for 
each $\sigma$-invariant maximal torus $\4T^*$. By definition, 
$\Inndiag(G)=\Aut_{\4T^*}(G)\Inn(G)$ when $\4T^*$ is the maximal torus in a 
standard $\sigma$-setup for $G$. Hence 
$\Inndiag(G)=\Aut_{\4G}(G)=\Aut_{\4T^*}(G)\Inn(G)$ for all such $\4T^*$. 
\end{proof}

We refer to \cite[Definitions 1.15.5(a,e) \& 2.5.10]{GLS3} for more details 
about the definitions of $\Phi_G$ and $\Gamma_G$. The next proposition 
describes how to identify these subgroups when working in a nonstandard 
setup.

\begin{Prop} \label{nonstandard}
Assume $\4G$, $\4T$, and the root system of $\4G$, are as in Notation 
\ref{G-setup}(\ref{not1},\ref{not2}).  Let $\sigma$ be any Steinberg 
endomorphism of $\4G$, and set $G=O^{q_0'}(C_{\4G}(\sigma))$. 
\begin{enuma} 
\item There is a standard Steinberg endomorphism $\sigma^*$ of $\4G$ 
such that if we set $G^*=O^{q_0'}(C_{\4G}(\sigma^*))$, 
then there is $x\in\4G$ such that $G=\9x(G^*)$. 
\end{enuma}
Fix $G^*$, $\sigma^*$, and $x$ as in (a). Let $\Inndiag(G^*)$, 
$\Phi_{G^*}$, and $\Gamma_{G^*}$ be as in Definition \ref{d:Aut(G)} (with 
respect to the $\sigma$-setup $(\4G,\sigma^*)$). Set 
$\Inndiag(G)=c_x\Inndiag(G^*)c_x^{-1}$, $\Phi_G=c_x\Phi_{G^*}c_x^{-1}$, and 
$\Gamma_G=c_x\Gamma_{G^*}c_x^{-1}$, all as subgroups of $\Aut(G)$. Then the 
following hold.
\begin{enuma}[resume]
\item $\Inndiag(G)=\Aut_{\4G}(G)$. 

\item For each $\4\alpha\in\Phi_{\4G}\Gamma_{\4G}$ such that 
$\4\alpha|_{G^*}\in\Phi_{G^*}\Gamma_{G^*}$, and each 
$\beta\in\4\alpha\cdot\Inn(\4G)$ such that $\beta(G)=G$, $\beta|_G\equiv 
c_x(\4\alpha)c_x^{-1}$ (mod $\Inndiag(G)$). 

\item If $\psi_{q_0}$ normalizes $G$, then 
$\Inndiag(G)\Phi_G=\Inndiag(G)\gen{\psi_{q_0}|_G}$. 

\end{enuma}
Thus the subgroups $\Phi_G$ and $\Gamma_G$ are well defined modulo 
$\Inndiag(G)$, independently of the choice of standard $\sigma$-setup for 
$G$.
\end{Prop}

\begin{proof} \textbf{(a) } See, e.g., \cite[Theorem 2.2.3]{GLS3}: 
for any given choice of maximal torus, positive roots, and 
parametrizations of the root groups, each Steinberg automorphism of $\4G$ 
is conjugate, by an element of $\Inn(\4G)$, to a Steinberg automorphism of 
standard type.

\smallskip

\noindent\textbf{(b) } This follows immediately from Proposition 
\ref{p:Autdiag(G)}(c).

\smallskip

\noindent\textbf{(c) } By assumption, $\beta\equiv\4\alpha\equiv 
c_x\4\alpha{}c_x^{-1}$ (mod $\Inn(\4G)$). Since $\beta$ and 
$c_x\4\alpha{}c_x^{-1}$ both normalize $G$, $\beta|_G\equiv
c_x\alpha^*c_x^{-1}$ modulo $\Aut_{\4G}(G)=\Inndiag(G)$. 

\noindent\textbf{(d) } If $\psi_{q_0}$ normalizes $G$, then (c), applied with 
$\4\alpha=\beta=\psi_{q_0}$, implies that as elements of 
$\Aut(G)/\Inndiag(G)$, $[\psi_{q_0}|_G]=[c_x(\psi_{q_0}|_{G^*})c_x^{-1}]$ 
generates the image of $\Phi_G$. 
\end{proof}

\begin{Lem} \label{4T->4G}
Assume $\4G$, $\4T$, $\sigma$, $G=O^{q_0'}(C_{\4G}(\sigma))$, and the root 
system of $\4G$, are as in Notation \ref{G-setup}(\ref{not1},\ref{not2}). 
Assume that $\varphi\in\Aut(\4T)$ is the restriction of an algebraic 
automorphism of $\4G$ such that $[\varphi,\sigma|_{\4T}]=1$. Then there is 
an algebraic automorphism $\4\varphi\in\Aut(\4G)$ such that 
$\4\varphi|_{\4T}=\varphi$, $[\4\varphi,\sigma]=1$, and $\4\varphi(G)=G$.
\end{Lem}

\begin{proof} By assumption, there is 
$\4\varphi\in\Aut(\4G)$ such that $\4\varphi|_{\4T}=\varphi$. Also, 
$[\4\varphi,\sigma]$ is an algebraic automorphism of $\4G$ by \cite[Theorem 
1.15.7(a)]{GLS3}, it is the identity 
on $\4T$, and hence $[\4\varphi,\sigma]=c_t$ for some $t\in\4T$ by Lemma 
\ref{l:tau}. Using the Lang-Steinberg theorem, upon replacing 
$\4\varphi$ by $c_u\4\varphi$ for appropriate $u\in\4T$, we can arrange 
that $[\4\varphi,\sigma]=1$. In particular, $\4\varphi(G)=G$.
\end{proof}

The following proposition is well known, but it seems to be difficult to 
find references where it is proven.

\begin{Prop} \label{p:Gu->Ga}
Fix a prime $q_0$, and a group $G\in\Lie(q_0)$ of universal type. 
Then $Z(G)$ has order prime to $q_0$, $G/Z(G)\in\Lie(q_0)$ and 
is of adjoint type, and $Z(G/Z(G))=1$.
If $G/Z(G)$ is simple, then each central extension of $G$ by a group of 
order prime to $q_0$ splits (equivalently, $H^2(G;\Z/p)=0$ for all primes 
$p\ne{}q_0$).
\end{Prop}

\begin{proof} Let $(\4G,\sigma)$ be a $\sigma$-setup for $G$, and 
choose a maximal torus and positive roots in $\4G$. We can 
thus assume Notation \ref{G-setup}. 
By Lemma \ref{theta-r}\eqref{CG(T)=T}, $Z(\4G)$ is finite of order 
prime to $q_0$. Since $Z(G)\le C_{\4G}(G)=Z(\4G)$ by Proposition 
\ref{p:Autdiag(G)}(a), $Z(G)$ also has order prime to $q_0$. 


Set $\4G_a=\4G/Z(\4G)$ and let $G_a<\4G_a$ be the image of $G$ under 
projection. Thus $\4G_a$ is an algebraic group of adjoint type, and 
$G_a=O^{q_0{}'}(C_{\4G_a}(\sigma_a))\in\Lie(q_0)$ where 
$\sigma_a\in\End(\4G_a)$ is induced by $\sigma$. Also, $Z(G_a)\le 
Z(\4G_a)=1$ by Proposition \ref{p:Autdiag(G)}(a) again.

It remains to prove the statement about central extensions. When $G$ is a 
Chevalley group, this was shown in \cite[Th\'eor\`eme 4.5]{Steinberg-grr}. 
It was shown in \cite[Corollary 6.2]{Steinberg-grc2} when 
$G\cong\lie2An(q)$ for $n$ even, and in \cite{AG} when $G\cong\lie2G2(q)$ 
or $\Sz(q)$. The remaining cases follow by similar arguments (see \cite[9.4 
\& 12.4]{Steinberg-rag}). (See also \cite[\S\,1]{Curtis}.)
\end{proof}

The next proposition shows that in most cases, $C_{\4G}(T)=\4T$. In the 
next section, we will see some conditions which imply that 
$C_{\4G}(O_p(T))=\4T$ when $p$ is a prime different from the defining 
characteristic.

\begin{Prop} \label{p:CbarG(T)}
Let $(\4G,\sigma)$ be a $\sigma$-setup for $G$, where $\4G$ and $G$ are of 
universal type. Assume Notation \ref{G-setup}, and in particular, that we 
have fixed a maximal torus $\4T$ and a root system $\Sigma$ in $\4G$.
\begin{enuma} 

\item Assume that $C_{\4G}(T)^0\gneqq\4T$, where $(-)^0$ denotes the connected 
component of the identity. Then there is $\alpha\in\Sigma_+$ such that 
$\theta_\alpha(T)=1$. Also, there is $\beta\in\Hom(\4T,\fqobar^\times)$ 
such that $\theta_\alpha=\beta^{-1}\sigma^*(\beta)$; i.e., 
$\theta_\alpha(t)=\beta(t^{-1}\sigma(t))$ for each $t\in\4T$. 

\item If the $\sigma$-setup is standard, then $C_{\4G}(T)^0=\4T$ 
except possibly when $G\cong{}^r\gg(2)$ for some $\gg$ and some $r\le3$, or 
when $G\cong A_1(3)$, $C_n(3)$ for $n\ge2$, or $\lie2G2(3)$.

\item If $C_{\4G}(T)^0=\4T$, then $N_G(T)/T\cong W_0$. 

\end{enuma}
\end{Prop}

\begin{proof} 
\noindent\textbf{(a) } Since $G$ is of universal type, $G=C_{\4G}(\sigma)$ 
and $T=C_{\4T}(\sigma)$. Hence there is a short exact sequence
	\[ 1 \Right2{} T \Right4{} \4T \Right8{t\mapsto t^{-1}\sigma(t)} 
	\4T \,, \]
where the last map is onto by the Lang-Steinberg theorem. Upon dualizing, 
and regarding $\Hom(\4T,\fqobar^\times)$ additively, we get an exact sequence
	\[ 0 \Right2{} \Hom(\4T,\fqobar^\times) \Right5{\sigma^*-\Id} 
	\Hom(\4T,\fqobar^\times) \Right4{\textup{restr}} 
	\Hom(T,\fqobar^\times) \]
(see also \cite[Proposition 3.2.3]{Carter2}), where 
$\Hom(\4T,\fqobar^\times)$ is the 
group of \emph{algebraic} homomorphisms. Since $\theta_\alpha$ is in 
the kernel of the restriction map, by assumption, it has the form 
$\beta^{-1}\sigma^*(\beta)$ for some $\beta\in\Hom(\4T,\fqobar^\times)$.

\smallskip

\noindent\textbf{(b) } Let $P(\Sigma)$ and $Q(\Sigma)$ be as in 
\cite[\S\,VI.1.9]{Bourb4-6} (but with $\Sigma$ in place of $R$ to denote 
the root system). Thus $Q(\Sigma)=\Z\Sigma$, the integral lattice generated by 
$\Sigma$, and 
	\[ P(\Sigma) = \{ v\in V \,|\, (v,\alpha^\vee)\in\Z \textup{ for 
	all $\alpha\in\Sigma$} \} \ge Q(\Sigma) \,. \]

For each $v\in P(\Sigma)$, define $\theta_v\in 
X(\4T)=\Hom(\4T,\fqobar^\times)$ by setting 
$\theta_v(h_\alpha(\lambda))=\lambda^{(v,\alpha^\vee)}$ for $\alpha\in\Pi$ 
and $\lambda\in\fqobar$. Since $G$ is of universal type, this is a well 
defined homomorphism by Lemma \ref{theta-r}\eqref{T=prod}, and the same 
formula holds for all $\alpha\in\Sigma$ by Lemma 
\ref{theta-r}\eqref{hb.hc}. By Lemma \ref{theta-r}\eqref{^txb(u)}, this 
extends our definition of $\theta_\beta$ for $\beta\in\Sigma\subseteq 
P(\Sigma)$. 

Recall that $\Hom(\fqobar^\times,\fqobar^\times)\cong\Z$. For each 
$\theta\in X(\4T)$ and each $\alpha\in\Sigma$, let $n_{\theta,\alpha}\in\Z$ 
be such that $\theta(h_\alpha(\lambda))=\lambda^{n_{\theta,\alpha}}$ for 
all $\lambda\in\fqobar^\times$. For given $\theta$, there is 
$v\in{}P(\Sigma)$ such that $(v,\alpha^\vee)=n_{\theta,\alpha}$ for all 
$\alpha\in\Pi$, and hence (by Lemma \ref{theta-r}\eqref{hb.hc}) for all 
$\alpha\in\Sigma$. Then $\theta=\theta_v$ as defined above. In this way, we 
identify $P(\Sigma)$ with the lattice $X(\4T)$ of characters for $\4T$, 
while identifying $Q(\Sigma)$ with $\Z\Sigma$.

From the appendix to Chapter VI in \cite{Bourb4-6} (Planches I--IX), we 
obtain the following table:
	\[ \renewcommand{\arraystretch}{1.5}
	\begin{array}{|c||c|c|c|c|c|c|c|c|} \hline
	\textup{root system $\Sigma$} & A_n & C_n & B_n,~D_n 
	& G_2 & F_4 & 
	E_6 & E_7 & E_8 \\\hline
	\min\{\norm{v}\,|\,v\in P(\Sigma)\} & 
	\sqrt{n/(n{+}1)} & 1 & \min\{\sqrt{n/4},1\} & \sqrt2 
	& 1 & \sqrt{4/3} & \sqrt2 & \sqrt2 \\\hline
	\max\{\norm{\alpha}\,|\,\alpha\in \Sigma\} & 
	\sqrt2 & 2 & \sqrt2 & \sqrt6 & \sqrt2 
	& \sqrt2 & \sqrt2 & \sqrt2 \\\hline
	\end{array} \]
Here, the norms are given with respect to the descriptions of these 
lattices in \cite{Bourb4-6} as subgroups of Euclidean spaces.

Assume $C_{\4G}(T)^0\gneqq\4T$. By (a), there are $\alpha\in\Sigma_+$ and 
$\beta\in\Hom(\4T,\fqobar^\times)$ such that 
$\alpha=\beta^{-1}\sigma^*(\beta)$. If we regard $\alpha$ and $\beta$ as 
elements in the normed vector space $V$, then 
$\norm\alpha=\norm{\sigma^*(\beta)-\beta} 
\ge\norm{\sigma^*(\beta)}-\norm\beta$. If $G=\lie{r}\gg{}(q)$ (and $\sigma$ 
is a standard setup), then $\norm{\sigma^*(\beta)}=q\norm\beta$, except when 
$G$ is a Suzuki or Ree group in which case 
$\norm{\sigma^*(\beta)}=\sqrt{q}\norm\beta$. Thus 
	\[ \frac{\norm\alpha}{\norm\beta} + 1 \ge \begin{cases} 
	q & \textup{if $G$ is a Chevalley or Steinberg group} \\
	\sqrt{q} & \textup{if $G$ is a Suzuki or Ree group.}
	\end{cases} \]
By the above table, this is possible only if $q=2$, or if $G$ is isomorphic 
to one of the groups $A_1(3)$, $B_2(3)$, $C_n(3)$ ($n\ge3$), $\lie2G2(3)$, 
or $\lie2B2(8)$.

Assume $G\cong\lie2B2(8)\cong\Sz(8)$. It is most convenient to use the root 
system for $C_2$ constructed in \cite{Bourb4-6}: $P(\Sigma)=\Z^2$, and 
$\Sigma=\{(\pm2,0),(0,\pm2),(\pm1,\pm1)\}$. Then $\alpha$ and $\beta$ 
satisfy the above inequality only if $\norm\alpha=2$, $\norm\beta=1$, and 
$\norm{\alpha+\beta}=\sqrt8$. So $(\alpha,\beta)=\frac32$, which is 
impossible for $\alpha,\beta\in\Z^2$. Hence $C_{\4G}(T)^0=\4T$ in this 
case. 

\smallskip

\noindent\textbf{(c) } If $C_{\4G}(T)^0=\4T$, then $N_{\4G}(T)\le 
N_{\4G}(\4T)$, and so $N_G(T)/T\cong W_0$ by Lemma \ref{W0onT}.
\end{proof}

The following, more technical lemma will be needed in Section \ref{s:X2}.

\begin{Lem} \label{W0-action}
Assume the hypotheses and notation in \ref{G-setup}, and also that the 
$\sigma$-setup $(\4G,\sigma)$ is standard. Then under the action 
of $W_0$ on $\5\Sigma$, each orbit contains elements of $\5\Pi$.
\end{Lem}

\begin{proof} When $\rho=\Id$, this is \cite[Proposition 
2.1.8]{Carter}. When $\rho\ne\Id$, it follows from the descriptions of 
$W_0$ and $\5\Sigma$ in \cite[\S\S\,13.2--13.3]{Carter}. 
\end{proof}

\newpage

\newsect{The equicharacteristic case} 
\label{s:=}

\tdef{minSupp}   \tdef{Tr}
\newcommand{\hgt}{\textup{ht}}

The following notation will be used in this section. 

\begin{Not} \label{G-setup-Q}
Assume the notation in \ref{G-setup}, and also that 
$\rho(\Pi)=\Pi$, $q_0=p$, and $Z(\4G)=1$.  Thus $\4G=\gg(\fpbar)$ 
is a connected, simple group over $\fpbar$ in adjoint form, $\sigma$ is a 
Steinberg endomorphism of $\4G$ of standard form, and 
$G=O^{p'}(C_{\4G}(\sigma))$. 
\begin{enumA}[start=4]  

\item\label{not4=} Set $\4U=\Gen{\4X_\alpha\,\big|\,\alpha\in\Sigma_+}$ and 
$\4B\defeq N_{\4G}(\4U)=\4U\4T$ (the \emph{Borel subgroup} of $\4G$).  Set 
	\[ U=C_{\4U}(\sigma)=\gen{X_{\5\alpha}\,|\,\5\alpha\in\5\Sigma_+}\,, 
	\qquad B=N_G(U)\,,
	\qquad\textup{and}\qquad T=\4T\cap{}G. \]
Thus $U=\prod_{\5\alpha\in\5\Sigma_+}X_{\5\alpha}\in\sylp{G}$, and $B=UT$.   
(See, e.g., \cite[Theorems 2.3.4(d) \& 2.3.7]{GLS3}, or \cite[Theorems 
5.3.3(ii) \& 9.4.10]{Carter} in the case of Chevalley groups.)  When  
$\5J\subsetneqq\5\Pi$ is the image in $\5\Sigma_+$ of a $\tau$-invariant 
subset $J\subsetneqq\Pi$, let $U_{\5J}\le U$ be the subgroup generated by 
root groups for positive roots in $\Sigma_+{\sminus}\gen{J}$ (the 
\emph{unipotent radical subgroup} associated to $\5J$), and set 
$\Par{\5J}=N_G(U_{\5J})=B\Gen{X_{-\5\alpha}\,\big|\,\alpha\in\gen{J}}$ (the 
\emph{parabolic subgroup} associated to $\5J$).  Thus $U=U_\emptyset$ and 
$B=\Par\emptyset$.  We also write $U_{\5\alpha}=U_{\{\5\alpha\}}$ and 
$\Par{\5\alpha}=\Par{\{\5\alpha\}}$ for each $\5\alpha\in\5\Pi$. 


\item\label{not5=} The \emph{height} of a positive root 
$\alpha=\sum_{\gamma\in\Pi}n_\gamma\gamma\in\Sigma_+$ ($n_\gamma\ge0$) is 
defined by $\htt(\alpha)=\sum_{\gamma\in\Pi}n_\gamma$.   The height 
$\htt(\5\alpha)$ of a class of roots $\5\alpha\in\5\Sigma_+$ is the minimum 
of the heights of roots in the class $\5\alpha$.

\item\label{not6=} Set $\calf=\calf_U(G)$ and $\call=\call_U^c(G)$. 

\item\label{not7=} Set $U_0=\gen{X_{\5\alpha}\,|\,\5\alpha\in\5\Sigma_+,~ 
\5\alpha\cap{}\Pi=\emptyset}=\gen{X_{\5\alpha}\,|\,\htt(\5\alpha)\ge2}$.

\item\label{not8=} The \emph{Lie rank} of $G$ is equal to 
$|\5\Pi|$; equivalently, to the number of maximal parabolic subgroups 
containing $B$.

\end{enumA}

\end{Not}

For example, assume $\sigma=\psi_q\circ\gamma$, where $\gamma\in\Aut(\4G)$ 
is a graph automorphism which induces $\rho\in\Aut(\Sigma_+)$, and $\psi_q$ 
is the field automorphism induced by $t\mapsto{}t^q$. Then for 
$\5\alpha\in\5\Sigma$, $X_{\5\alpha}\cong\F_q$ when $\5\alpha=\{\alpha\}$ 
contains only one root, $X_{\5\alpha}\cong\F_{q^a}$ if 
$\5\alpha=\{\rho^i(\alpha)\}$ is the $\rho$-orbit of $\alpha$ with length 
$a$, and $X_{\5\alpha}$ is nonabelian if $\5\alpha$ contains a root 
$\alpha$ and sums of roots in its $\rho$-orbit.

We need the following, stronger version of Theorem \ref{St-aut}. 

\begin{Thm}[{\cite[\S\,3]{Steinberg-aut}}] \label{St-aut2}
Assume $G$ is as in Notation \ref{G-setup} and \ref{G-setup-Q}.  
If $\alpha\in\Aut(G)$ is such that $\alpha(U)=U$, then 
$\alpha=c_udfg$ for unique automorphisms $c_u\in\Aut_U(G)$, 
$d\in\Inndiag(G)=\Aut_{\4T}(G)$, $f\in\Phi_G$, and $g\in\Gamma_G$. 
\end{Thm}

\begin{proof} Let $N_{\Aut(G)}(U)\le\Aut(G)$ and 
$N_{\Inndiag(G)}(U)\le\Inndiag(G)$ be the subgroups of those automorphisms 
which send $U$ to itself.  Since $\Phi_G\Gamma_G\le N_{\Aut(G)}(U)$ by 
definition, Theorem \ref{St-aut} implies that 
$N_{\Aut(G)}(U)=N_{\Inndiag(G)}(U)\cdot(\Phi_G\Gamma_G)$, a semidirect 
product. Since $\Phi_G\cap\Gamma_G=1$, it remains to show that 
$N_{\Inndiag(G)}(U)=\Aut_U(G)\Aut_{\4T}(G)$ and 
$\Aut_U(G)\cap\Aut_{\4T}(G)=1$.  The first is immediate: since 
$\Aut_{\4T}(G)\le N_{\Aut(G)}(U)$ and $N_G(U)=TU$, 
	\begin{align*} 
	N_{\Inndiag(G)}(U) &= \bigl(\Inn(G)\Aut_{\4T}(G)\bigr)\cap 
			N_{\Aut(G)}(U) \\
	&= \Aut_{N_G(U)}(G)\Aut_{\4T}(G) = \Aut_U(G)\Aut_{\4T}(G)\,. 
	\end{align*}
Finally, if $c_u=c_t\in\Aut(G)$ where $u\in{}U$ and $t\in\4T$, then 
$c_u=\Id_G$, since $u$ has $p$-power order and $t$ has order prime to $p$.
\end{proof}

\begin{Lem} \label{kappa_inj}
Assume $G\in\Lie(p)$.  Then for $U\in\sylp{G}$, 
$\4\kappa_G$ sends $\Out(G)$ injectively into $\Out(U,\calf)$.
\end{Lem}

\begin{proof} Assume that $\4\kappa_{G/Z(G)}$ is injective. Since 
$Z(G)$ is a $p'$-group (since $Z(G)\le\4T$), and since $O^{p'}(G)=G$ by 
definition of $\Lie(p)$, $\Aut(G)$ injects into $\Aut(G/Z(G))$, and hence 
$\4\kappa_G$ is injective. It thus suffices to prove the lemma when $G$ is 
in adjoint form.

We can thus assume Notation \ref{G-setup-Q}.
By Lemma \ref{kappa-lift}, it will suffice to 
prove that $C_{\Aut(G)}(U)\le\Inn(G)$.  Fix $\beta\in\Aut(G)$ such that 
$\beta|_U=\Id_U$. By Theorem \ref{St-aut2}, there are unique automorphisms 
$c_u\in\Aut_U(G)$, $d\in\Aut_{\4T}(G)$, $f\in\Phi_G$, and $g\in\Gamma_G$ 
such that $\beta=c_u dfg$. 

If $g\ne\Id$, then it permutes the fundamental root groups nontrivially, 
while $c_u df|_U$ sends each such group to itself modulo higher root groups 
and commutators.  Hence $g=\Id$.  Similarly, $f=\Id$, since otherwise 
$\beta$ would act on the fundamental root groups (modulo higher root 
groups) via some automorphism other than a translation.  

Thus $\beta=c_ud$, where $d=c_t$ for some $t\in N_{\4T}(G)$.  Then $u$ has 
$p$-power order while $t$ has order prime to $p$, so $d|_U=c_t|_U=\Id$.  
By Lemma \ref{theta-r}\eqref{^txb(u)}, $c_t$ sends each root group in 
$\4U$ to itself via $x_\alpha(u)\mapsto 
x_\alpha(\theta_\alpha(t){\cdot}u)$ for some character 
$\theta_\alpha\in\Hom(\4T,\fpbar^\times)$ which is linear in 
$\alpha$. For each $\5\alpha\in \5\Sigma_+$, $c_t|_{X_{\5\alpha}}=\Id$ 
implies that $\theta_\alpha(t)=1$ for all $\alpha\in\5\alpha$.  Thus 
$\theta_\alpha(t)=1$ for all $\alpha\in\Sigma_+$, so $c_t=\Id_{\4G}$, and 
$\beta=c_u\in\Inn(G)$. 
\end{proof}

It now remains, when proving Theorem \ref{ThE}, to show the surjectivity 
of $\kappa_G$.  This will be done case-by-case.  We first handle groups of 
Lie rank at least three, then those of rank one, and finally those of rank 
two.

For simplicity, we state the next two propositions only for groups of 
adjoint type, but they also hold without this restriction. The first 
implies that each element of $\Aut(U,\calf)$ permutes the subgroups 
$U_{\5J}$ (as defined in Notation \ref{G-setup-Q}), and that each element 
of $\Aut\typ^I(\call_S^c(G))$ induces an automorphism of the amalgam of 
parabolics $\Par{\5J}$ for $\5J\subsetneqq\5\Pi$.

\begin{Prop} \label{Borel-Tits}
Assume Notation \ref{G-setup-Q}. For $1\ne{}P\le{}U$, the following are 
equivalent:
\begin{enumr} 
\item $P=U_{\5J}$ for some $\5J\subsetneqq\5\Pi$;

\item $P\nsg B$, $C_U(P)\le P$, and $O_p(\outf(P))=1$; and 

\item $P\nsg B$, $C_G(P)\le P$, and $O_p(N_G(P))=P$. 

\end{enumr}
Hence for each $\varphi\in\Aut(U,\calf)$, $\varphi$ permutes the subgroups 
$U_{\5J}$, and in particular permutes the subgroups $U_{\5\alpha}$ for 
$\5\alpha\in\5\Pi$.
\end{Prop}

\begin{proof} \noindent\boldd{(i)$\implies$(iii): } 
For each $\5J\subsetneqq\5\Pi$, $C_G(U_{\5J})=Z(U_{\5J})$ by 
\cite[Theorem 2.6.5(e)]{GLS3} (recall that $G$ is of adjoint type).  Also, 
$O_p(N_G(U_{\5J}))=O_p(\Par{\5J})=U_{\5J}$, and $U_{\5J}$ 
is normal in $B$ since $N_G(U_{\5J})=\Par{\5J}\ge B$. 

\noindent\boldd{(iii)$\implies$(ii): } This holds since $\outf(P)\cong 
N_G(P)/PC_G(P)$.

\noindent\boldd{(ii)$\implies$(i): } In this case, $P\nsg B$, so 
$N_G(P)\ge B$, and $N_G(P)=\Par{\5J}$ for some $\5J\subsetneqq\5\Pi$ (cf. 
\cite[Theorem 8.3.2]{Carter}). Then $P\le O_p(\Par{\5J})=U_{\5J}$. Also, 
$U_{\5J}C_G(P)\big/PC_G(P)\le O_p(N_G(P)/PC_G(P))=1$, so $U_{\5J}\le 
PC_G(P)$. Since $U_{\5J}\le U$, this implies that $U_{\5J}\le PC_U(P)=P$; 
i.e., that $P=U_{\5J}$. So (i) holds.

The last statement follows from the equivalence of (i) and (ii). 
\end{proof}

When $G$ has large Lie rank, Theorem \ref{ThE} now follows from 
properties of Tits buildings.

\begin{Prop} \label{rk>2-tame}
Assume $G\in\Lie(p)$ is of adjoint type and has Lie rank at least $3$. 
Fix $U\in\sylp{G}$.  Then $\kappa_G$ is split surjective.
\end{Prop}

\begin{proof} Set $\call=\call_U^c(G)$.  By Proposition \ref{Borel-Tits}, 
for each $\alpha\in\Aut\typ^I(\call)$, $\alpha$ permutes the subgroups 
$U_{\5J}$ for $\5J\subsetneqq\5\Pi$.  For each such $\5J$, 
$C_G(U_{\5J})=Z(U_{\5J})$, so $\Aut_\call(U_{\5J})=N_G(U_{\5J})=\Par{\5J}$.  
Thus $\alpha$ induces an automorphism of the amalgam of parabolic subgroups 
$\Par{\5J}$.  Since $G$ is the amalgamated sum of these subgroups by a 
theorem of Tits (see \cite[Theorem 13.5]{Tits} or \cite[p. 95, Corollary 
3]{Serre}), $\alpha$ extends to a unique automorphism $\4\alpha$ of $G$.  

Thus $\alpha\mapsto\4\alpha$ defines a homomorphism 
$\5s\:\Aut\typ^I(\call)\Right2{}\Aut(G)$.  If $\alpha=c_\gamma$ for 
$\gamma\in\Aut_\call(U)=N_G(U)$, then $\4\alpha$ is conjugation by 
$\gamma\in G$ and hence lies in $\Inn(G)$.  Hence $\5s$ factors through 
$s\:\Out\typ(\call)\Right2{}\Out(G)$, $\kappa_G\circ 
s=\Id_{\Out\typ(\call)}$, and thus $\kappa_G$ is split surjective.
\end{proof}


Before we can handle the rank $1$ case, two elementary lemmas are needed.

\begin{Lem} \label{Aut-lemma}
Let $G$ be a finite group with normal Sylow $p$-subgroup $S\nsg{}G$.  Fix 
subgroups $1=S_0<S_1<\cdots<S_k=S$ normal in $G$ such that the following hold:
\begin{enumr} 

\item $S_{k-1}\le\Fr(S)$;

\item $C_G(S)\le S$; and 

\item for each $1\le i\le k-1$, $S_i$ is characteristic in $G$, $[S,S_i]\le 
S_{i-1}$, $S_i/S_{i-1}$ has exponent $p$, and 
$\Hom_{\F_p[G/S]}(S/\Fr(S),S_i/S_{i-1})=0$ (i.e., no irreducible 
$\F_p[G/S]$-submodule of $S_i/S_{i-1}$ appears as a submodule of 
$S/\Fr(S)$).

\end{enumr}
Let $\alpha\in\Aut(G)$ be such that $[\alpha,S]\le S_{k-1}$.  Then 
$\alpha\in\Aut_S(G)$.
\end{Lem}

\begin{proof} For $1\ne g\in G$ of order prime to $p$, the conjugation 
action of $g$ on $S$ is nontrivial since $C_G(S)\le S$, and hence the 
conjugation action on $S/\Fr(S)$ is also nontrivial (see \cite[Theorem 
5.3.5]{Gorenstein}). Thus $G/S$ acts faithfully on $S/\Fr(S)$. Since 
$\alpha$ induces the identity on $S/\Fr(S)$, $\alpha$ also induces the 
identity on $G/S$.

Assume first that $\alpha|_S=\Id$. Since $S$ is a $p$-group 
and $G/S$ has order prime to $p$, $H^1(G/S;Z(S))=0$. So by \cite[Lemma 
1.2]{OV2}, $\alpha\in\Inn(G)$. If $g\in{}G$ is such that $\alpha=c_g$, then 
$[g,S]=1$ since $\alpha|_S=\Id$, and $g\in S$ since $G/S$ acts faithfully 
on $S/\Fr(S)$. Thus $\alpha\in\Aut_S(G)$ in this case.

In particular, this proves the lemma when $k=1$. So assume $k\ge2$. We can 
assume inductively that the lemma holds for $G/S_1$, and hence can 
arrange (after composing by an appropriate element of $\Aut_S(G)$) that 
$\alpha$ induces the identity on $G/S_1$.

Let $\varphi\in\Hom(S,S_1)$ be such that $\alpha(x)=x\varphi(x)$ for each 
$x\in S$ (a homomorphism since $S_1\le Z(S)$). Then $\varphi$ factors 
through $\4\varphi\in\Hom(S/\Fr(S),S_1)$ since $S_1$ is elementary abelian, 
and $\4\varphi$ is a homomorphism of $\F_p[G/S]$-modules since 
$\alpha(g)\equiv g$ (mod $S_1$) for each $g\in{}G$ (and $S_1\le Z(S)$). Thus 
$\varphi=1$ since $\Hom_{G/S}(S_k/S_{k-1},S_1)=0$ by (iii), so
$\alpha|_S=\Id$, and we already showed that this implies 
$\alpha\in\Aut_S(G)$. 
\end{proof}

The next lemma will be useful when checking the hypotheses of Lemma 
\ref{Aut-lemma}.

\begin{Lem} \label{l:Fq*onFq}
Fix a prime $p$ and $e\ge1$, and set $q=p^e$ and $\Gamma=\F_q^\times$. 
For each $a\in\Z$, set 
$V_a=\F_q$, regarded as an $\F_p\Gamma$-module with action 
$\lambda(x)=\lambda^ax$ for $\lambda\in\Gamma$ and $x\in\F_q$. 
\begin{enuma} 
\item For each $a$, $V_a$ is $\F_p\Gamma$-irreducible if and only if 
$a/\gcd(a,q-1)$ does not divide $p^t-1$ for any $t|e$, $t<e$. 

\item For each $a,b\in\Z$, $V_a\cong V_b$ as $\F_p\Gamma$-modules if 
and only if $a\equiv bp^i$ (mod $q-1$) for some $i\in\Z$.
\end{enuma}
\end{Lem}

\begin{proof} \noindent\textbf{(a) } Set $d=\gcd(a,q-1)$, and let $t$ be 
the order of $p$ in $(\Z/\frac{q-1}d)^\times$. Thus $t|e$ since 
$\frac{q-1}d\big|(p^e-1)$. If $t<e$, then $\lambda^a\in\F_{p^t}$ for each 
$\lambda\in\F_q$, so $0\ne\F_{p^t}\subsetneqq V_a$ is a proper 
$\F_p\Gamma$-submodule, and $V_a$ is reducible.

Conversely, if $V_a$ is reducible, then it contains a proper submodule 
$0\ne W\subsetneqq V_a$ of dimension $i$, some $0<i<e$. All $\Gamma$-orbits 
in $V_a{\sminus}0$, hence in $W{\sminus}0$, have length $\frac{q-1}d$, so 
$\frac{q-1}d\big|(p^i-1)$, and $t\le i<e$. 

\smallskip

\noindent\textbf{(b) } For each $a\in\Z$, let $\4V_a\cong\F_q$ be the 
$\F_q\Gamma$-module where $\Gamma$ acts via $\lambda(x)=\lambda^ax$. Then 
$\F_q\otimes_{\F_p}V_a\cong 
\4V_a\oplus\4V_{ap}\oplus\cdots\oplus\4V_{ap^{e-1}}$ as 
$\F_q\Gamma$-modules. Since $\4V_b\cong\4V_a$ if and only if $b\equiv a$ 
(mod $q-1$), $V_b\cong V_a$ if and only if $b\equiv ap^i$ (mod $q-1$) for 
some $i$. 
\end{proof}

In principle, we don't need to look at the fusion systems of the simple 
groups of Lie rank $1$ if we only want to prove tameness.  Their fusion is 
controlled by the Borel subgroup, so their fusion systems are tame by 
Proposition \ref{p:constrained}.  But the following proposition is 
needed when proving Theorem \ref{ThE} in its stronger form, and will also 
be used when working with groups of larger Lie rank. 

\begin{Prop} \label{rk1equi}
Fix a prime $p$, and a group $G\in\Lie(p)$ of Lie rank $1$.  
Assume $(G,p)\not\cong(\Sz(2),2)$.  
Then each $\varphi\in\Aut(U,\calf)$ extends to an automorphism of $G$.  
Also, if $[\varphi,U]\le[U,U]$, then $\varphi\in\Inn(U)$.
\end{Prop}

\begin{proof} If $G$ is of universal form, then $Z(G)$ is cyclic of order 
prime to $p$ by Proposition \ref{p:Gu->Ga}.  For each $Z\le Z(G)$, 
$\Out(G/Z)\cong\Out(G)$ by \cite[Theorem 2.5.14(d)]{GLS3}, and 
$\Out(U,\calf_U(G/Z))\cong\Out(U,\calf_U(G))$ since $G$ and $G/Z$ have the 
same $p$-fusion systems.  It thus suffices to prove the proposition when 
$G$ has adjoint form.

Assume first $G=\PSL_2(q)$.  Thus $U\cong\F_q$ (as an additive group), 
$T\cong C_{(q-1)/\gee}$ where $\gee=\gcd(q-1,2)$, and 
$\Gamma\defeq\Aut_T(U)$ is the subgroup of index $\gee$ in $\F_q^\times$.  
If $\varphi\in\Aut(U)$ is fusion preserving, then under these 
identifications, there is $\alpha\in\Aut(\Gamma)$ such that 
$\alpha(u)\varphi(v)=\varphi(uv)$ for each $u\in\Gamma\le\F_q^\times$ and 
$v\in\F_q$.  After composing with an appropriate diagonal automorphism 
(conjugation by a diagonal element of $PGL_2(q)$), we can assume that 
$\varphi(1)=1$.  Hence the above formula (with $v=1$) implies that 
$\alpha=\varphi|_{\Gamma}$, and thus that 
$\varphi(uv)=\varphi(u)\varphi(v)$ for each $u,v\in\F_q$ with $u\in\Gamma$.  
If $\gee=1$, then $\varphi$ acts as a field automorphism on $U$, hence is 
the restriction of a field automorphism of $G$, and we are done.  
Otherwise, there is $u\in\Gamma$ such that $\F_q=\F_p(u)$, $u$ and 
$\varphi(u)$ have the same minimal polynomial over $\F_p$, and there is 
$\psi\in\Aut(\F_q)$ (a field automorphism) such that $\psi(u)=\varphi(u)$.  
Thus $\psi(u^i)=\varphi(u^i)$ for each $i$, so $\psi=\varphi$ since both 
are additive homomorphisms, and hence $\varphi$ extends to a field 
automorphism of $G$.  (Note that this argument also holds when $q=3$ and 
$\Gamma=1$.)

Next assume $G=\PSU_3(q)$. Following the conventions in 
\cite[Satz II.10.12(b)]{Huppert}, we identify 
	\begin{align*} 
	U &= \bigl\{ \pair[a,b] \,\big|\, 
	a,b\in\F_{q^2},~ b+b^q=-a^{q+1} \bigr\} &
	&\textup{where} &
	\pair[a,b] &= \Bigl(\begin{smallmatrix} 1&a&b\\0&1&-a^q\\0&0&1
	\end{smallmatrix} \Bigr); \\
	T &= \bigl\{d(\lambda)\,\big|\, \lambda\in\F_{q^2}^\times\bigr\} &
	&\textup{where} &
	d(\lambda)&= \diag(\lambda^{-q},\lambda^{q-1},\lambda).
	\end{align*}
Here, whenever we write a matrix, we mean its class in $\PSU_3(q)$.  
Then $B=UT=N_G(U)\le G$ (see \cite[Satz II.10.12(b)]{Huppert}), and 
	\[ \pair[a,b]\cdot\pair[c,d] = \pair[a+c,b+d-ac^q] 
	\qquad\textup{and}\qquad
	\9{d(\lambda)}\pair[a,b] = \pair[\lambda^{1-2q}a,\lambda^{-1-q}b]\,. \]
Set $\gee=\gcd(2q-1,q^2-1)=\gcd(2q-1,q^2-2q) =\gcd(q+1,3)$.  Then 
$d(\lambda)=1$ exactly when $\lambda^\gee=1$, $C_T(U)=1$, and hence 
$|T|=|\Aut_{B}(U/Z(U))|=(q^2-1)/\gee$. If $q>2$, then $|T|$ does 
not divide $p^i-1$ 
for any power $1<p^i<q^2$, and by Lemma \ref{l:Fq*onFq}(a), $U/Z(U)$ and 
$Z(U)$ are both irreducible as $\F_p[T]$-modules. (Note, in particular, the 
cases $q=5$ and $q=8$, where $(U/Z(U),T)$ is isomorphic to $(\F_{25},C_8)$ 
and $(\F_{64},C_{21})$, respectively.)  

Fix $\varphi\in\Aut(U,\calf)$, and extend it to $\alpha\in\Aut(B)$ (Lemma 
\ref{extend_phi}).  Via the same argument as that used when $G=\PSL_2(q)$, 
we can arrange (without changing the class of $\varphi$ modulo 
$\Im(\4\kappa_G)$) that $\varphi\equiv\Id$ (mod $[U,U]$).  If $q>2$, then 
the hypotheses of Lemma \ref{Aut-lemma} hold (with $[U,U]<U<B$ in the role 
of $S_1<S_2=S<G$), so $\alpha\in\Aut_U(B)$ and $\varphi\in\Inn(U)$.  

If $G\cong\PSU_3(2)\cong C_3^2\sd{}Q_8$ (cf. \cite[p. 123--124]{Taylor}), 
then $U\cong Q_8$ and $T=1$, so 
$\Out(U,\calf)=\Out(U)\cong\Sigma_3$.  By Theorem \ref{St-aut} (or by 
direct computation), $\Out(G)=\Outdiag(G)\Phi_G$ has order six, since 
$|\Outdiag(G)|=\gcd(3,q+1)=3$ and $|\Phi_G|=2$.  Thus $\4\kappa_G$ is an 
isomorphism, since it is injective by Lemma \ref{kappa_inj}.

The proof when $G=Sz(q)$ is similar.  Set $\theta=\sqrt{2q}$.  We follow 
the notation in \cite[\S\,XI.3]{HB3}, and identify $U$ as the group of all 
$S(a,b)$ for $a,b\in\F_q$ and $T<B=N_G(U)$ as the group of all $d(\lambda)$ for 
$\lambda\in\F_q^\times$, with relations 
	\[ S(a,b)\cdot S(c,d) = S(a+c,b+d+a^\theta c) 
	\quad\textup{and}\quad
	\9{d(\lambda)}S(a,b) = S(\lambda a,\lambda^{1+\theta}b)\,. \]
As in the last case, we can arrange that $\varphi\in\Aut(U,\calf)$ is the 
identity modulo $[U,U]$.  Since $q\ge8$ ($q\ne2$ by hypothesis), $Z(U)$ and 
$U/Z(U)$ are nonisomorphic, irreducible $\F_2T$-modules by Lemma 
\ref{l:Fq*onFq}(a,b) (and since $Z(U)\cong V_{1+\theta}$ and $U/Z(U)\cong 
V_1$ in the notation of that lemma).  We can thus apply Lemma 
\ref{Aut-lemma} to show that $\varphi\in\Inn(U)$.

It remains to handle the Ree groups $\lie2G2(q)$, where $q=3^m$ for some 
odd $m\ge1$.  Set $\theta=\sqrt{3q}$.  We use the notation in \cite[Theorem 
XI.13.2]{HB3}, and identify $U=(\F_q)^3$ with multiplication given by 
	\[ (x_1,y_1,z_1){\cdot}(x_2,y_2,z_2) =
	(x_1+x_2,y_1+y_2+x_1{\cdot}x_2^\theta, z_1+z_2 -x_1{\cdot}y_2 
	+ y_1{\cdot}x_2 - x_1{\cdot}x_1^\theta{\cdot}x_2)\,. \]
Note that $x^{\theta^2}=x^3$.  Let $T\le B=N_G(U)$ be the set of all 
$d(\lambda)$ for $\lambda\in\F_q^\times$, acting on $U$ via 
	\[ \9{d(\lambda)}(x,y,z) = 
	(\lambda x, \lambda^{\theta+1} y, \lambda^{\theta+2}z). \]
Again, we first reduce to the case where $\varphi\in\Aut(U,\calf)$ is such 
that $[\varphi,U]\le[U,U]$, and extend $\varphi$ to $\alpha\in\Aut(B)$.  If 
$q>3$, then $U/[U,U]\cong V_1$, $[U,U]/Z(U)\cong V_{\theta+1}$, 
and $Z(U)\cong V_{\theta+2}$ are irreducible and pairwise nonisomorphic as 
$\F_3T$-modules by Lemma \ref{l:Fq*onFq} (for $V_a$ as defined in that 
lemma), since neither $\theta+1$ nor $\theta+2$ is a power of $3$. So 
$\varphi\in\Inn(U)$ by Lemma \ref{Aut-lemma}.


If $q=3$, then $U=\gen{a,b}$, where $|a|=9$, $|b|=3$, and $[a,b]=a^3$.  Set 
$Q_i=\gen{ab^i}\cong C_9$ ($i=0,1,2$): the three subgroups of $U$ 
isomorphic to $C_9$.  Let $\Aut^0(U)\le\Aut(U)$ be the group of those 
$\alpha\in\Aut(U)$ which send each $Q_i$ to itself.  For each such 
$\alpha$, the induced action on $U/Z(U)$ sends each subgroup of order three 
to itself, hence is the identity or $(g\mapsto g^{-1})$, and the latter is 
seen to be impossible using the relation $[a,b]=a^3$.  Thus each 
$\alpha\in\Aut^0(U)$ induces the identity on $U/Z(U)$ and on $Z(U)$, and has 
the form $\alpha(g)=g\varphi(g)$ for some $\varphi\in\Hom(U/Z(U),Z(U))$.  
So $\Aut^0(U)=\Inn(U)$ since they both have order $9$ (and clearly 
$\Inn(U)\le\Aut^0(U)$).  The action of $\Aut(U)$ on $\{Q_0,Q_1,Q_2\}$ thus 
defines an embedding of $\Out(U)$ into $\Sigma_3$, and the automorphisms 
$(a,b)\mapsto(ab,b)$ and $(a,b)\mapsto(a^{-1},b)$ show that 
$\Out(U)\cong\Sigma_3$.  Since $|\outf(U)|=2$ and 
$\autf(U)\nsg\Aut(U,\calf)$, it follows that $\Out(U,\calf)=1=\Out(G)$.  
(See also \cite[Theorem 2]{BC} for more discussion about $\Aut(U)$.)
\end{proof}


\newcommand{\bigG}{\mathfrak{G}}
\newcommand{\apt}{\mathscr{T}}

It remains to show that $\kappa_G$ (at the prime $p$) is surjective when 
$G\in\Lie(p)$ has Lie rank $2$, with the one exception when 
$G\cong\SL_3(2)$. Our proof is based on ideas taken from the article of 
Delgado and Stellmacher \cite{DS}, even though in the end, we do not 
actually need to refer to any of their results in our argument. The third 
author would like to thank Richard Weiss for explaining many of the details 
of how to apply the results in \cite{DS}, and also to Andy Chermak and 
Sergey Shpectorov for first pointing out the connection. 

Fix a prime $p$, and a finite group $G\in\Lie(p)$ of Lie rank two. We 
assume Notation \ref{G-setup} and \ref{G-setup-Q}. In particular, 
$(\4G,\sigma)$ is a $\sigma$-setup for $G$, $\4T\le\4G$ is a maximal torus, 
$U\in\sylp{G}$ is generated by the positive root subgroups, and $B=N_G(U)$ 
is a Borel subgroup. Set $\5\Pi=\{\5\alpha_1,\5\alpha_2\}$, and set 
$\Par1=\Par{\5\alpha_1}=\gen{B,X_{-\5\alpha_1}}$ and 
$\Par2=\Par{\5\alpha_2}=\gen{B,X_{-\5\alpha_2}}$: the two maximal parabolic 
subgroups of $G$ containing $B$. Our proofs are based on the following 
observation:

\begin{Lem} \label{reduce2star}
Assume, for $G\in\Lie(p)$ of rank $2$ and its amalgam of parabolics as 
above, that 
	\beqq \parbox{\short}{each automorphism of the amalgam 
	$(\Par1>B<\Par2)$ extends to an automorphism of $G$.} 
	\label{star} \tag{$*$} \eeqq
Then $\kappa_G$ is surjective.
\end{Lem}

Here, by an automorphism of the amalgam, we mean a pair $(\chi_1,\chi_2)$, 
where either $\chi_i\in\Aut(\Par{i})$ for $i=1,2$ or 
$\chi_i\in\Iso(\Par{i},\Par{3-i})$ for $i=1,2$, and also 
$\chi_1|_B=\chi_2|_B$. 

\begin{proof} Set $\call=\call_U^c(G)$ and $U_i=O_p(\Par{i})$. By 
Proposition \ref{Borel-Tits}, each $\chi\in\Aut\typ^I(\call)$ either sends 
$U_1$ and $U_2$ to themselves or exchanges them. For each $i=1,2$, 
$C_G(U_i)\le U_i$, so $\Aut_\call(U_i)=N_G(U_i)=\Par{i}$.  
Thus $\chi$ induces an automorphism of the amalgam $(\Par1>B<\Par2)$. By 
assumption, this extends to an automorphism $\4\chi$ of $G$, and 
$\kappa_G(\4\chi)=\xi$.
\end{proof}

Set $\bigG=\Par1\*_B\Par2$: the amalgamated free product over $B$. Let 
$\rho\:\bigG\Right2{}G$ be the natural surjective homomorphism. Since 
each automorphism of the amalgam induces an automorphism of $\bigG$, 
\eqref{star} holds if for each automorphism of $(\Par1>B<\Par2)$, the 
induced automorphism of $\bigG$ sends $\Ker(\rho)$ to itself.

Let $\Gamma$ be the tree corresponding to the amalgam $(\Par1>B<\Par2)$. 
Thus $\Gamma$ has a vertex $[g\Par{i}]$ for each coset $g\Par{i}$ (for all 
$g\in\bigG$ and $i=1,2$), and an edge $g(e_B)$ connecting $[g\Par1]$ to 
$[g\Par2]$ for each coset $gB$ in $\bigG$. Also, $\bigG$ acts on $\Gamma$ 
via its canonical action on the cosets, and in particular, it acts on 
$g(e_B)$ with stabilizer subgroup $\9gB$. 

Similarly, let $\Gamma_G$ be the graph of $G$ with respect to the same 
amalgam: the graph with vertex set $(G/\Par1)\cup(G/\Par2)$ and edge set 
$G/B$. Equivalently, since $\Par1$, $\Par2$, and $B$ are self-normalizing, 
$\Gamma_G$ is the graph whose vertices are the maximal parabolics in $G$ 
and whose edges are the Borel subgroups. Let $\5\rho\:\Gamma\Right2{}\Gamma_G$ 
be the canonical map which sends a vertex $[g\Par{i}]$ in $\Gamma$ to the 
vertex in $\Gamma_G$ corresponding to the image of $g\Par{i}$ in $G$.

Fix a subgroup $N\le G$ such that $(B,N)$ is a $BN$-pair for $G$, and such 
that $B\cap{}N=T$ and $N/T\cong W_0$ (where $T$ and $W_0$ are as defined in 
Notation \ref{G-setup}). We refer to \cite[\S\S\,8.2, 13.5]{Carter} for the 
definition of $BN$-pairs, and the proof that $G$ has a $BN$-pair $(B,N)$ 
which satisfies these conditions. In order to stay close to the notation in 
\cite{DS}, we also set $T$: their notation for the Cartan subgroup. For 
$i=1,2$, choose $t_i\in(N\cap\Par{i}){\sminus}B=(N\cap\Par{i}){\sminus}T$. 
Since $(N\cap\Par{i})/T\cong C_2$ and $N=\gen{N\cap\Par1,N\cap\Par2}$, we 
have $N=T\gen{t_1,t_2}$, consistent with the notation in \cite{DS}. Note 
that $T$ can be the trivial subgroup. We also regard the $t_i\in\Par{i}$ as 
elements of $\bigG$, and $T\le B$ as a subgroup of $\bigG$, when 
appropriate.

Let $\apt$ be the union of the edges in the $T\gen{t_1,t_2}$-orbit of $e_B$. 
Thus $\apt$ is a path of infinite length in $\Gamma$ of the following form: 
	\[ \cdots \xy 0;<2cm,0cm>:
	(0,0)*+{}; (7,0)*+{} **@{-};
	(1,0)*+{\bullet};
	(2,0)*+{\bullet};
	(3,0)*+{\bullet};
	(4,0)*+{\bullet};
	(5,0)*+{\bullet};
	(6,0)*+{\bullet};
	(1,-0.16)*+{\8[t_1t_2\Par1]};
	(2,-0.16)*+{\8[t_1\Par2]};
	(3,-0.16)*+{\8[\Par1]};
	(4,-0.16)*+{\8[\Par2]};
	(5,-0.16)*+{\8[t_2\Par1]};
	(6,-0.16)*+{\8[t_2t_1\Par2]};
	(0.5,0.12)*+{\8t_1t_2t_1(e_B)};
	(1.5,0.12)*+{\8t_1t_2(e_B)};
	(2.5,0.12)*+{\8t_1(e_B)};
	(3.5,0.12)*+{\8e_B};
	(4.5,0.12)*+{\8t_2(e_B)};
	(5.5,0.12)*+{\8t_2t_1(e_B)};
	(6.5,0.12)*+{\8t_2t_1t_2(e_B)};
	\endxy \cdots \]
Thus $\5\rho(\apt)$ is an apartment in the building $\Gamma_G$ under Tits's 
definition and construction of these structures in \cite[3.2.6]{Tits}. 

A path in $\Gamma$ is always understood not to double back on itself.

\begin{Lem} \label{r=n=s-1}
Let $G$ and $\Gamma$ be as above. Let $n\in\{3,4,6,8\}$ be such that 
$W_0\cong D_{2n}$. Then each path in $\Gamma$ of length 
at most $n+1$ is contained in $g(\apt)$ for some $g\in\bigG$.
\end{Lem}

\begin{proof} A path of length $1$ is an edge, and is in the $\bigG$-orbit 
of $e_B$ which has stabilizer group $B$. If $e_B$ is extended to a path of 
length $2$ with the edge $t_i(e_B)$ ($i=1$ or $2$), then this path has 
stabilizer group 
	\[ B\cap\9{t_i}B = 
	\prod_{\5\alpha\in\5\Sigma_+{\sminus}\{\5\alpha_i\}}X_{\5\alpha} 
	\cdot T \,. \]
(Recall that $\9{t_i}X_{\5\alpha_i}=X_{-\5\alpha_i}$, and 
$X_{-\5\alpha_i}\cap{}B=1$ by \cite[Lemma 7.1.2]{Carter}.) Thus the 
stabilizer subgroup has index $p^j$ in $B$, where 
$p^j=|X_{\5\alpha_i}|$. Furthermore, $|\Par{i}/B|=1+p^j$, since by 
\cite[Proposition 8.2.2(ii)]{Carter}, 
	\[ \Par{i}=B\cup(Bt_iB) \quad\textup{where}\quad
	|Bt_iB|=|B|\cdot|B/(B\cap\9{t_i}B)| = |B|\cdot p^j \,. \]
Hence there are exactly $p^j$ extensions of $e_B$ to a path of length $2$ 
containing the vertex $[\Par{i}]$ in the interior, and these are permuted 
transitively by $B$.

Upon continuing this argument, we see inductively that for all $2\le k\le 
n+1$, the paths of length $k$ starting at $e_B$ with endpoint $[\Par{3-i}]$ 
are permuted transitively by $B$, and of them, the one contained in $\apt$ 
has stabilizer subgroup the product of $T$ with $(n+1-k)$ root subgroups in 
$U$. (Recall that $B=TU$, and $U$ is the product of $n$ root subgroups.) 
Since $\bigG$ acts transitively on the set of edges in $\Gamma$, each path 
of length $k$ is in the $\bigG$-orbit of one which begins with 
$e_B$ (and with endpoint $[\Par1]$ or $[\Par2]$), and hence in the 
$\bigG$-orbit of a subpath of $\apt$.
\end{proof}

\begin{Prop} \label{DS-thm}
Let $G$, $\bigG$, and $(T,t_1,t_2)$ be as above, and let $n$ be such that 
$W_0\cong D_{2n}$. Assume that 
	\beqq \parbox{\short}{for each $(\chi_1,\chi_2)\in 
	\Aut\bigl(\Par1>B<\Par2\bigr)$, where $\chi_i\in\Aut(\Par{i})$ or 
	$\chi_i\in\Iso(\Par{i},\Par{3-i})$ for $i=1,2$, 
	we have $(\chi_1(t_1)\chi_2(t_2))^n\in\chi_1(T)$.}
	\label{dag} \tag{$\dagger$} \eeqq
Then \eqref{star} holds (each automorphism of $(\Par1>B<\Par2)$ extends to 
an automorphism of $G$), and hence $\kappa_G$ is onto.
\end{Prop}

\begin{proof} Let $\approx$ be the equivalence relation on the set of 
vertices in $\Gamma$ generated by setting $x\approx y$ if $x$ and $y$ are 
of distance $2n$ apart in some path in the $\bigG$-orbit of $\apt$. Since 
$T\gen{t_1,t_2}/T\cong D_{2n}$ as a subgroup of $N_G(T)/T$, the natural map 
$\5\rho\:\Gamma\Right2{}\Gamma_G$ sends $\apt$ to a loop of length $2n$, 
and hence sends all apartments in the $\bigG$-orbit of $\apt$ to loops of 
length $2n$. Hence $\Gamma\Right2{}\Gamma_G$ factors through 
$\Gamma/{\approx}$. 

We claim that 
	\begin{align} 
	&\textup{$\Gamma_G$ contains no loops of length strictly less than 
	$2n$; and} \label{claim2} \\
	&\textup{each pair of points in $\Gamma/{\approx}$ is connected by a 
	path of length at most $n$.} \label{claim1} 
	\end{align}
Assume \eqref{claim2} does not hold: let $L$ be a loop of minimal length 
$2k$ ($k<n$), and fix edges $\sigma_i=[x_i,y_i]$ in $L$ ($i=1,2$) such that 
the path from $x_i$ to $y_{3-i}$ in $L$ has length $k-1$. Since $\Gamma_G$ 
is a building whose apartments are loops of length $2n$ \cite[3.2.6]{Tits}, 
there is an apartment $\Sigma$ which contains $\sigma_1$ and $\sigma_2$. By 
\cite[Theorem 3.3]{Tits} or \cite[p. 86]{Brown}, there is a retraction of 
$\Gamma_G$ onto $\Sigma$. Hence the path from $x_i$ to $y_{3-i}$ in 
$\Sigma$ (for $i=1,2$) has length at most $k-1$, these two paths must be 
equal to the minimal paths in $L$ since there are no loops of length less 
than $2k$, and this contradicts the assumption that $L$ and $\Sigma$ are 
loops of different lengths. (See also \cite[\S\,IV.3, Exercise 1]{Brown}. 
Point \eqref{claim2} also follows since $\Gamma_G$ is a generalized $n$-gon 
in the sense of Tits \cite[p. 117]{Brown}, and hence any two vertices are 
joined by at most one path of length less than $n$.) 

Now assume \eqref{claim1} does not hold: let $x,y$ be vertices in 
$\Gamma$ such that the shortest path between their classes in 
$\Gamma/{\approx}$ has length $k\ge n+1$. Upon replacing $x$ and $y$ by 
other vertices in their equivalence classes, if needed, we can assume that 
the path $[x,y]$ in $\Gamma$ has length $k$. Let $x^*$ be the vertex in the 
path $[x,y]$ of distance $n+1$ from $x$. By Lemma \ref{r=n=s-1}, $[x,z]$ is 
contained in $g(\apt)$ for some $g\in\bigG$; let $x'$ be the vertex in 
$g(\apt)$ of distance $2n$ from $x$ and distance $n-1$ from $z$. Then 
$x'\approx x$, and $[x',y]$ has length at most $(n-1)+(k-n-1)=k-2$, a 
contradiction. This proves \eqref{claim1}. 

Assume the map $(\Gamma/{\approx})\Right2{}\Gamma_G$ induced by $\5\rho$ is 
not an isomorphism of graphs, and let $x$ and $y$ be distinct vertices in 
$\Gamma/{\approx}$ whose images are equal in $\Gamma_G$. By \eqref{claim1}, 
there is a path from $x$ to $y$ of length at most $n$, and of even length 
since the graph is bipartite. This path cannot have length $2$ since 
$\5\rho\:\Gamma\Right2{}\Gamma_G$ preserves valence, so its image in 
$\Gamma_G$ is a loop of length at most $n$, and this contradicts 
\eqref{claim2}. We conclude that $\Gamma_G\cong\Gamma/{\approx}$. 

Now let $(\chi_1,\chi_2)$ be an automorphism of the amalgam 
$(\Par1>B<\Par2)$. Let $\chi\in\Aut(\bigG)$ be the induced automorphism of 
the amalgamated free product, and let $\5\chi\in\Aut(\Gamma)$ be the 
automorphism which sends a vertex $[g\Par{i}]$ to $[\chi(g\Par{i})]$. Since 
$(\chi_1(t_1)\chi_2(t_2))^n=1$ in $G$ by assumption, $\5\rho(\5\chi(\apt))$ 
is a loop of length $2n$ in $\Gamma_G$. Hence $\5\rho\circ\5\psi$ factors 
through $(\Gamma/{\approx})\cong\Gamma_G$, and by an argument similar to 
that used to show that $\Gamma_G\cong\Gamma/{\approx}$, the induced map 
$\Gamma_G\Right2{}\Gamma_G$ is an automorphism of $\Gamma_G$. So $\chi$ 
sends $\Ker[\bigG\Right2{\rho}G]$ to itself, and thus induces an 
automorphism of $G$. The last statement ($\kappa_G$ is onto) now follows 
from Lemma \ref{reduce2star}.
\end{proof}

It remains to find conditions under which \eqref{dag} holds. 
The following proposition handles all but a small number of cases.

\begin{Prop} \label{ThA:most}
Assume $N=N_G(T)$ (and hence $N_G(T)/T$ is dihedral of order $2n$). Then 
\eqref{dag} holds, and hence each automorphism of the amalgam 
$(\Par1>B<\Par2)$ extends to an automorphism of $G$. In particular, 
\eqref{dag} and \eqref{star} hold, and hence $\kappa_G$ is onto, 
whenever $G=\lie{r}Xn(q)\in\Lie(p)$ has 
Lie rank $2$ for $q>2$ and $G\not\cong\Sp_4(3)$. 
\end{Prop}

\begin{proof} Assume that $N_G(T)=N=T\gen{t_1,t_2}$. Then the choices of 
the $t_i$ are unique modulo $T$. Also, any two choices of $T$ are 
$B$-conjugate, so each automorphism of the amalgam is $B$-conjugate to one 
which sends $\apt$ to itself. Thus \eqref{dag} holds, and so \eqref{star} 
follows from Proposition \ref{DS-thm}. 

The last statement now follows from Proposition \ref{p:CbarG(T)}. Note that 
if \eqref{dag} holds for $G$ of universal type, then it also holds for 
$G/Z(G)$ of adjoint type.
\end{proof}

What can go wrong, and what does go wrong when $G=\SL_3(2)$, is that an 
automorphism of the amalgam can send $t_1,t_2$ to another pair of elements 
whose product (modulo $T$) has order strictly greater than $2n$. This 
happens when $\apt$ is sent to another path not in the $\bigG$-orbit of 
$\apt$: one whose image in $\Gamma_G$ is a loop of a different length.

\begin{Ex} \label{ex:ThA:L3(2)}
Assume $G=\SL_3(2)$. In particular, $T=1$. Let $B$ be the group of upper 
triangular matrices, let $t_1$ and $t_2$ be the permutation matrices for 
$(1\,2)$ and $(2\,3)$, respectively, and set $\Par{i}=\gen{B,t_i}$. 

Consider the automorphism $\alpha$ of the amalgam which is the identity on 
$\Par1$ (hence on $B$), and which is conjugation by $e_{13}$ (the involution 
in $Z(B)$) on $\Par2$. Set $t'_i=\alpha(t_i)$. Thus 
	\[ t'_1=\mxthree010100001 \qquad \textup{and}\qquad
	t'_2=\mxthree111001010\,. \]
One checks that $t'_1t'_2$ has order $4$, so that $\gen{t'_1,t'_2}\cong 
D_8$ while $\gen{t_1,t_2}\cong D_6$. In other words, $\alpha$ sends the 
lifting (from $\Gamma_G$ to $\Gamma$) of a loop of length $6$ to the 
lifting of a loop of length $8$, hence is not compatible with the relation 
$\approx$, hence does not extend to an automorphism of $G$.
\end{Ex}

We are left with seven cases: four cases with $n=4$, two with $n=6$, and 
one with $n=8$. Those with $n=4$ are relatively easy to handle.

\begin{Prop} \label{ThA:n=4}
Assume $G$ is one of the groups $\Sp_4(2)$, $\PSp_4(3)$, $\PSU_4(2)$, or 
$\PSU_5(2)$. Then \eqref{dag} holds, and hence \eqref{star} also holds and 
$\kappa_G$ is onto. 
\end{Prop}

\begin{proof} In all cases, we work in the universal groups $\Sp_4(q)$ and 
$\SU_n(2)$, but the arguments are unchanged if we replace the subgroups 
described below by their images in the adjoint group. Recall that $p$ is 
always the defining characteristic, so the second and third cases are 
distinct, even though $\PSp_4(3)\cong\SU_4(2)$ (see 
\cite[\S\,3.12.4]{Wilson} or \cite[Corollary 10.19]{Taylor}). 

Let $(\chi_1,\chi_2)$ be an automorphism of $(\Par1>B<\Par2)$. Since all 
subgroups of $B$ isomorphic to $T$ are conjugate to $T$ by the 
Schur-Zassenhaus theorem, we can also assume that $\chi_i(T)=T$. Set 
$\chi_0=\chi_1|_B=\chi_2|_B$ and $t_i^*=\chi_i(t_i)$ for short; we must 
show that $|t_1^*t_2^*|=n=4$. Note that $t_1^*t_2^*$ has order at least 
$4$, since otherwise $\Gamma_G$ would contain a loop of length strictly 
less than $8=2n$, which is impossible by point \eqref{claim2} in the proof 
of Proposition \ref{DS-thm}. 

\smallskip

\noindent\boldd{$G=\Sp_4(2)\cong\Sigma_6$ : } Set $G'=[G,G]$: the subgroup 
of index $2$. The elements $x_\gamma(1)$ for $\gamma\in\Sigma$ are all 
$\Aut(G)$-conjugate: the long roots and the short roots are all 
$W$-conjugate and a graph automorphism exchanges them. Since these elements 
generate $G$, none of them are in $G'$. Hence for $i=1,2$, all involutions 
in 
	\[ \Gen{x_{\alpha_i}(1),x_{-\alpha_i}(1)} \cong \GL_2(2) \cong 
	\Sigma_3 \]
lie in $G{\sminus}G'$, and in particular, $t_i\in G{\sminus}G'$. 

Each automorphism of the amalgam sends the focal subgroup to itself (as a 
subgroup of $B$), and hence also sends the intersections $\Par{i}\cap{}G'$ 
to themselves. So $t_1^*,t_2^*\in G{\sminus}G'$, and $t_1^*t_2^*\in G'\cong 
A_6$. It follows that $|t_1^*t_2^*|\le5$, and $|t_1^*t_2^*|=4$ since every 
dihedral subgroup of order $10$ in $\Sigma_6$ is contained in $A_6$.

\smallskip 

\noindent\boldd{$G=\Sp_4(3)$ : } In this case, $T\cong C_2^2$, and 
$N_G(T)\cong\SL_2(3)\wr C_2$. Hence $N_G(T)/T\cong A_4\wr C_2$ contains 
elements of order $2$, $3$, $4$, and $6$, but no dihedral subgroups of 
order $12$. Since $t_1^*t_2^*$ has order at least $4$, $|t_1^*t_2^*|=4$, 
and condition \eqref{dag} holds.

\smallskip

\noindent\boldd{$G=\SU_n(2)$ for $n=4$ or $5$ : } We regard these as matrix 
groups via
	\[ \SU_n(2) = \bigl\{ M\in \SL_n(4) \,\big|\, \4M^t=M^{-1} \bigr\} 
	\qquad \textup{where}\qquad
	\4{\bigl(a_{ij}\bigr)}^t = \bigl(\4{a_{n+1-j,n+1-i}}\bigr)\,, \]
and where $\4x=x^2$ for $x\in\F_4$. We can then take $B$ to be the group of 
upper triangular matrices in $\SU_n(2)$, $U$ the group of strict upper 
triangular matrices, and $T$ the group of diagonal matrices. We thus have
	\begin{align*} 
	T &= \bigl\{ \diag(x,x^{-1},x^{-1},x) \,\big|\, x\in\F_4 \bigr\} 
	\cong C_3 && \textup{if $n=4$} \\
	T &= \bigl\{ \diag(x,y,xy,y,x) \,\big|\, x,y\in\F_4 \bigr\} 
	\cong C_3^2 && \textup{if $n=5$.} 
	\end{align*}
Since $N_G(T)$ must permute the eigenspaces of the action of $T$ on 
$\F_4^n$, we have $N_{\GU_n(2)}(T)\cong\GU_2(2)\wr C_2$ (if $n=4$) or 
$(\GU_2(2)\wr C_2)\times\F_4^\times$ (if $n=5$). So in both cases, 
	\[ N_G(T)/T \cong \PGU_2(2)\wr C_2 \cong \Sigma_3\wr C_2 \cong 
	C_3^2\rtimes D_8 \,. \]

Set $Q=N_G(T)/O_3(N_G(T))\cong D_8$, and let $\psi\:N_G(T)\Right2{}Q$ be 
the natural projection. Set $Q_0=\psi(C_G(T))$. Since 
$C_G(T)/T\cong\Sigma_3\times\Sigma_3$ (the subgroup of elements 
which send each eigenspace to itself), $Q_0\cong C_2^2$ and 
$C_G(T)=\psi^{-1}(Q_0)$. 

Choose the indexing of the parabolics such that $\Par1$ is the subgroup 
of elements which fix an isotropic point and $\Par2$ of those which fix an 
isotropic line. Thus 
	\[ \Par1=\left\{ \left. \mxthree{u}vx0Aw00u \right| \, 
	A\in\GU_{n-2}(2) \right\} 
	\quad\textup{and}\quad 
	\Par2= \begin{cases} 
	\left\{ \left. \mxtwo{A}X0{(\4A^t)^{-1}} \right| \, 
	A\in\SL_2(4) \right\} & \textup{if $n=4$} 
	\vphantom{\underset{|}{|}} \\
	\left\{ \left. \mxthree{A}vX0uw00{(\4A^t)^{-1}} \right| \, 
	A\in\GL_2(4) \right\} &  
	\textup{if $n=5$} 
	\end{cases} \]
Then $\psi(N_{\Par1}(T))\le Q_0$: no matrix in $\Par1$ can 
normalize $T$ and exchange its eigenspaces. Also, $N_B(T)$ contains 
$C_U(T)=\gen{e_{1,n}(1),e_{2,n-1}(1)}$, where $e_{i,j}(u)$ denotes the 
elementary matrix with unique off-diagonal entry $u$ in position $(i,j)$. 
Thus $Q_0\ge\psi(N_{\Par1}(T))\ge\psi(N_B(T))\cong C_2^2$, so these 
inclusions are all equalities. Also, $\Par2$ contains the permutation 
matrix for the permutation $(1\,2)(n{-}1\,n)$, this element 
exchanges the eigenspaces of rank $2$ for $T$, and so 
$\psi(N_{\Par2}(T))=Q$.

Since $T\gen{t_1,t_2}/T\cong D_8$, $\gen{\psi(t_1),\psi(t_2)}=Q$, and so 
$\psi(t_1)\in{}Q_0{\sminus}Z(Q)$ and 
$\psi(t_2)\in Q{\sminus}Q_0$. Since $(\chi_1,\chi_2)$ induces an 
automorphism of the amalgam $(Q>Q_0=Q_0)$, this implies that 
$\psi(t_1^*)\in{}Q_0{\sminus}Z(Q)$ and $\psi(t_2^*)\in{}Q{\sminus}Q_0$. But 
then $\gen{\psi(t_1^*),\psi(t_2^*)}=Q$ since these elements generate modulo 
$Z(Q)$, so $|t_1^*t_2^*|\in4\Z$, and $|t_1^*t_2^*|=4$ since 
$N_G(T)/T\cong\Sigma_3\wr C_2$ contains no elements of order $12$. 
\end{proof}

It remains to handle the groups $G_2(2)$, $\lie3D4(2)$, and $\lie2F4(2)$. 
In the first two cases, if $t_i^*$ is an arbitrary involution in 
$N_{\Par{i}}(T){\sminus}N_B(T)$ for $i=1,2$, then $t_1^*t_2^*$ can have 
order $6$ or $8$ (or order $7$ or $12$ when $G=G_2(2)$), and there does not 
seem to be any way to prove condition \eqref{dag} short of analyzing 
automorphisms of the amalgam sufficiently to prove \eqref{star} directly. 

Let $\{\alpha,\beta\}$ be a fundamental system in the root system of $G_2$ 
where $\alpha$ is the long root. Let $\alpha,\alpha',\alpha''$ be the three 
long positive roots, and $\beta,\beta',\beta''$ the three short positive 
roots, as described in \eqref{e:G2roots} below. 

Let $\gamma_0,\gamma_1,\gamma_2,\gamma_3$ denote the four fundamental roots 
in the $D_4$ root system, where $\gamma_0$ is in the center of the Dynkin 
diagram, and the other three are permuted cyclically by the triality 
automorphism.  Set $\gamma_{ij}=\gamma_i+\gamma_j$ (when it is a 
root), etc. We identify the six classes of positive roots in $\lie3D4$ with the 
roots in $G_2$ by identifying the following two diagrams:
	\beqq
	
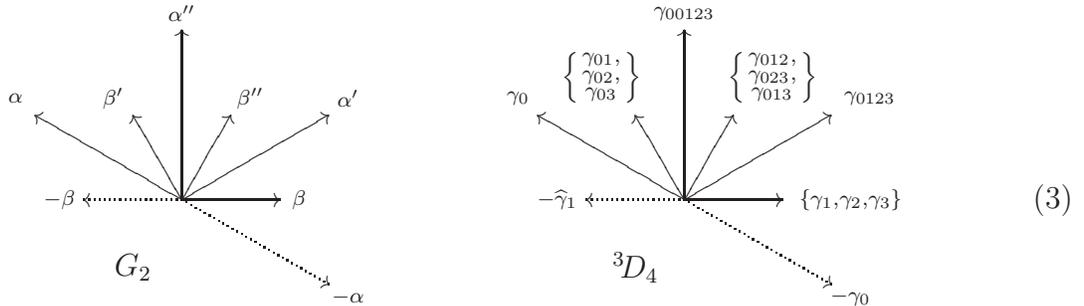
\xy 0;<1.3cm,0cm>:<0cm,1.126cm>::
	(1.2,0)*{\8\beta};
	(1.7,1.2)*+{\8\alpha'};
	(0.7,1.2)*+{\8\beta''};
	(-0.7,1.2)*+{\8\beta'};
	(-1.7,1.2)*+{\8\alpha};
	(0,2.2)*+{\8\alpha''};
	(-1.25,0)*+{\8-\beta};
	(1.7,-1.15)*+{\8-\alpha};
	(-0.5,-0.8)*+{G_2};
	\ar(1,0);(0,0);  
	\ar(-1.5,1);(0,0);
	\ar(-0.5,1);(0,0);
	\ar(0.5,1);(0,0);
	\ar(1.5,1);(0,0);
	\ar(0,2);(0,0);
	\ar@{.>}(-1,0);(0,0);
	\ar@{.>}(1.5,-1);(0,0);
	\endxy 
	\qquad\qquad
	\xy 0;<1.3cm,0cm>:<0cm,1.126cm>::
	(1.7,0)*{\8\{\gamma_1,\gamma_2,\gamma_3\}};
	(1.9,1.2)*+{\8\gamma_{0123}};
	(0.9,1.45)*+{\8\left\{\begin{smallmatrix}\gamma_{012},\\\gamma_{023},\\
	\gamma_{013} \end{smallmatrix}\right\}};
	(-0.85,1.45)*+{\8\left\{\begin{smallmatrix}\gamma_{01},\\\gamma_{02},\\
	\gamma_{03} \end{smallmatrix}\right\}};
	(-1.7,1.2)*+{\8\gamma_0};
	(0,2.2)*+{\8\gamma_{00123}};
	(-1.3,0)*+{\8-\5\gamma_1};
	(1.7,-1.15)*+{\8-\gamma_0};
	(-0.5,-0.8)*+{\lie3D4};
	\ar(1,0);(0,0);  
	\ar(-1.5,1);(0,0);
	\ar(-0.5,1);(0,0);
	\ar(0.5,1);(0,0);
	\ar(1.5,1);(0,0);
	\ar(0,2);(0,0);
	\ar@{.>}(-1,0);(0,0);
	\ar@{.>}(1.5,-1);(0,0);
	\endxy 
	\label{e:G2roots} \eeqq

The following list gives all nontrivial commutator relations among root 
subgroups of $G_2(q)$ or $\lie3D4(q)$ (see \cite[Theorems 1.12.1(b) \& 
2.4.5(b)]{GLS3}):  
	\begin{align} 
	[x_{\alpha}(u),x_{\beta}(v)] &\equiv 
	x_{\beta'}(\pm uv) x_{\beta''}(\pm uv^{1+q}) 
	&&\pmod{X_{\alpha'}X_{\alpha''}} \label{e:G1} \\
	[x_{\beta'}(u),x_{\beta}(v)] &\equiv x_{\beta''}(\pm(uv^{q}+u^{q}v)) 
	&&\pmod{X_{\alpha'}X_{\alpha''}} \label{e:G2} \\
	[x_{\alpha}(u),x_{\alpha'}(v)] &= x_{\alpha''}(\pm uv) \label{e:G3} \\
	[x_{\beta'}(u),x_{\beta''}(v)] &= x_{\alpha''}(\pm\Tr(uv^q)) 
	\label{e:G4} \\
	[x_{\beta''}(u),x_{\beta}(v)] &= x_{\alpha'}(\pm\Tr(u^qv))\,. 
	\label{e:G5} \end{align}
Again, $\Tr\:\F_{q^3}\Right2{}\F_q$ denotes the trace.  
Note that when $G=G_2(q)$, then $u,v\in\F_q$ in 
all cases, and hence $u^q=u^{q^2}=u$, $u^{q+q^2}=u^2$, and $\Tr(u)=3u$. 
When $G=\lie3D4(q)$, the notation $x_\beta(-)$, $x_{\beta'}(-)$, and 
$x_{\beta''}(-)$ is somewhat ambiguous (and formula \eqref{e:G2} depends on 
making the right choice), but this doesn't affect the arguments given below.

\begin{Prop} \label{G2(2)-tame}
Assume $p=2$ and $G=G_2(2)$. Then \eqref{star} holds: each automorphism 
of the amalgam $(\Par\alpha>B<\Par\beta)$ extends to an automorphism of $G$. 
(In fact, each automorphism of the amalgam is conjugation by some 
element of $B$.) In particular, $\kappa_G$ is onto.
\end{Prop}

\begin{proof} In this case, $T=1$, and 
	\[ \Par\alpha\cong (C_4\times C_4)\rtimes D_{12} 
	\qquad\textup{and}\qquad
	\Par\beta\cong (Q_8\times_{C_2}Q_8)\rtimes\Sigma_3 \,. \]
Also, $B=U$ has presentation $U=A\rtimes\gen{r,t}$, where 
	\[ A=\gen{a,b}\cong C_4\times C_4,~ \gen{r,t}\cong C_2^2,~ 
	\9ra=a^{-1},~ \9rb=b^{-1},~ \9ta=b,~ \9tb=a\,. \]
In terms of the generators $x_\gamma=x_\gamma(1)$ for $\gamma\in\Sigma_+$, 
we have $A=\Gen{x_{\beta'}x_\beta,x_{\beta''}x_\beta}$ and 
$\Omega_1(A)=\Gen{x_{\alpha'},x_{\alpha''}}$, and we can take 
$r=x_{\beta''}$, $t=x_\alpha$, and $a=x_{\beta}x_{\beta''}$ (and then 
$b=\9ta$). Note that \eqref{e:G2} takes the more precise form 
$[x_{\beta'},x_{\beta}]=x_{\alpha'}x_{\alpha''}$ in this case. Also, 
	\begin{align*} 
	U_{\alpha} &= A\gen{r} \cong (C_4\times C_4)\rtimes C_2 \\
	U_{\beta} &= \gen{ab^{-1},a^2t}\times_{\gen{a^2b^2}}\gen{ab,a^2rt} 
	\cong Q_8\times_{C_2}Q_8 \\
	U\cap G' &= A\gen{t} \cong C_4\wr C_2\,.
	\end{align*}
The last formula holds since $G'=[G,G]\cong\SU_3(3)$ has index two in $G$ 
(see \cite[\S\,4.4.4]{Wilson} or \cite[pp. 146--150]{Dickson}), since 
$x_\alpha,x_{\alpha'},x_{\alpha''}\in G'$ (note that 
$x_\alpha=[x_{-\beta},x_{\beta'}]$), and since $x_\beta$, $x_{\beta'}$, and 
$x_{\beta''}$ are all $G$-conjugate and hence none of them lies in $G'$.

Fix an automorphism $(\chi_\alpha,\chi_\beta)$ of the amalgam 
$(\Par\alpha>B<\Par\beta)$, and set $\chi_0=\chi_\alpha|_B=\chi_\beta|_B$. 
Then $\chi_0\in\Aut(U)$ sends each of the subgroups $U_{\alpha}$, 
$U_{\beta}$, and $U\cap G'$ to itself. Hence it sends each quaternion 
factor in $U_{\beta}$ to itself, and sends $U_{\alpha}\cap G'=\gen{a,b}$ to 
itself. After composing by an appropriate element of $\Aut_U(\Par\beta)$, 
we can arrange that $\chi_0(ab)=ab$ and $\chi_0(ab^{-1})=ab^{-1}$. In 
particular, $\chi_0$ induces the identity on $\Omega_1(A)$ and hence also 
on $A/\Omega_1(A)$. 

Let $g\in\Par\alpha$ be an element of order $3$, chosen so that 
$\9g(a^2)=b^2$ and $\9g(b^2)=a^2b^2$. The image of $\gen{g}$ in 
$\Par\alpha/A\cong D_{12}$ is normal, so $\chi_\alpha(g)\in{}Ag$. Let 
$x\in\Omega_1(A)$ be such that $\chi_\alpha(b)=\chi_0(b)=ax$. Then 
$\9gb\in\gen{ab,b^2}\le C_A(\chi_0)$, so $\9gb=\chi_\alpha(\9gb)=\9g(bx)$ 
implies that $\9gx=1$ and hence $x=1$. Thus $\chi_0|_A=\Id$. Also, 
$\chi_\alpha(\gen{g})\in\syl3{\Par\alpha}$ is conjugate to $\gen{g}$ by an 
element of $A$, so we can arrange that $\chi_\alpha(\gen{g})=\gen{g}$ and 
hence that $\chi_\alpha|_{A\gen{g}}=\Id$. But then $\chi_\alpha$ is the 
identity modulo $C_{\Par\alpha}(A\gen{g})=Z(A\gen{g})=1$, so 
$\chi_\alpha=\Id_{\Par\alpha}$.

Since $\chi_\beta|_{U_\beta}=\Id$, $\chi_\beta$ induces the identity modulo 
$C_{\Par\beta}(U_\beta)=Z(U_\beta)\cong C_2$. It thus has the form 
$\chi_\beta(x)=x\psi(x)$ for some $\psi\in\Hom(\Par\beta,Z(U_\beta))$. 
Hence $\chi_\beta=\Id$, since it is the identity on $U\in\syl2{\Par\beta}$. 
\end{proof}

\begin{Prop} \label{3D4(2)-tame}
Assume $p=2$ and $G=\lie3D4(2)$. Then \eqref{star} holds, and $\kappa_G$ is 
onto.
\end{Prop}

\begin{proof} In this case, $T\cong\F_8^\times\cong C_7$, 
$\Par\alpha/U_{\alpha}\cong C_7\times\Sigma_3$, and 
$\Par\beta/U_\beta\cong\SL_2(8)$. Also, by \eqref{e:G3} and \eqref{e:G4}, 
$U_\beta$ is extraspecial with center $X_{\alpha''}$. Fix an automorphism 
$(\chi_\alpha,\chi_\beta)$ of the amalgam $(\Par\alpha>B<\Par\beta)$, and 
set $\chi_0=\chi_\alpha|_B=\chi_\beta|_B$. We must show that $\chi_\alpha$ 
and $\chi_\beta$ are the restrictions of some automorphism of $G$.

By Theorem \ref{St-aut}, and since $\Outdiag(\SL_2(8))=1=\Gamma_{\SL_2(8)}$, 
$\Out(\Par\beta/U_\beta)\cong\Out(\SL_2(8))$ is generated by field 
automorphisms, and hence automorphisms which are restrictions of field 
automorphisms of $G$. So we can compose $\chi_{\beta}$ and 
$\chi_{\alpha}$ by restrictions of elements of 
$\Aut_{B}(G)\Phi_G=N_{\Aut_{\Par\beta}(G)}(U)\Phi_G$, to arrange that 
$\chi_{\beta}$ induces the identity on $\Par\beta/U_{\beta}$. Then, upon 
composing them by some element of $\Aut_U(G)$, we can also arrange that 
$\chi_0(T)=T$. Since $X_{\beta'}$ and $X_{\beta''}$ are dual to each other 
by \eqref{e:G4} and hence nonisomorphic as $\F_2[T]$-modules, $\chi_0$ 
sends each of them to itself. 

Since $\chi_0(T)=T$, $\chi_0$ sends 
$C_U(T)=X_{\alpha}X_{\alpha'}X_{\alpha''}\cong D_8$ to itself. It cannot 
exchange the two subgroups $X_{\alpha}X_{\alpha''}$ and 
$X_{\alpha'}X_{\alpha''}$ (the first is not contained in $U_\alpha$ and the 
second is), so $\chi_0|_{C_U(T)}\in\Inn(C_U(T))$. Hence
after composing by an element of $\Aut_{C_U(T)}(G)$, we can arrange that 
$\chi_0$ is the identity on this subgroup. Also, by applying \eqref{e:G1} 
with $u=1$, and since $\chi_0|_{X_\beta}\equiv\Id$ (mod $U_{\beta}$) and 
$[X_\alpha,U_{\beta}]\le X_{\alpha''}$, we see that $\chi_0$ is the 
identity on $X_{\beta'}X_{\beta''}$. We conclude that $\chi_0$ is the 
identity on $U_\beta$. 

Since $\chi_{\beta}$ induces the identity on $U_{\beta}$ and on 
$\Par\beta/U_\beta$, it has the form 
$\chi_\beta(x)=x\psi(x)$ (all $x\in\Par\beta$) for some 
	\[ \psi\in \Hom(\Par\beta/U_{\beta};Z(U_{\beta}))
	\cong\Hom(\SL_2(8),C_2)=1\,. \] 
So $\chi_{\beta}=\Id_{\Par\beta}$. 

Now, $C_{\Par\alpha}(T)\cong\Sigma_4\times C_7$, and $\Out(\Sigma_4)=1$. 
Hence $\chi_{\alpha}|_{C_{\Par\alpha}(T)}$ must be conjugation by some 
element $z\in{}Z(C_U(T))=X_{\alpha''}=Z(\Par\beta)$. After composing 
$\chi_\alpha$ and $\chi_\beta$ by restrictions of $c_z$, we can thus assume 
that $\chi_{\alpha}$ is the identity on $C_{\Par\alpha}(T)$ (and still 
$\chi_{\beta}=\Id_{\Par\beta}$). Since $\chi_{\alpha}|_U=\Id$ and 
$\Par\alpha=\gen{U,C_{\Par\alpha}(T)}$, we have 
$\chi_{\alpha}=\Id_{\Par\alpha}$. 
\end{proof}

It remains only to handle $\lie2F4(2)$ and the Tits group.

\begin{Prop} \label{Tits-tame}
Assume $G=\lie2F4(2)'$ or $\lie2F4(2)$.  Then $\kappa_G$ is an 
isomorphism.
\end{Prop}

\begin{proof} By the pullback square in \cite[Lemma 2.15]{AOV1} (and since 
$\Out\typ(\call)$ is independent of the choice of objects in $\call$ by 
\cite[Lemma 1.17]{AOV1}), $\kappa_G$ is an isomorphism when $G=\lie2F4(2)$ 
if it is an isomorphism when $G$ is the Tits group.  So from now on, we 
assume $G=\lie2F4(2)'$.

We adopt the notation for subgroups of $G$ used by Parrott \cite{Parrott}.  
Fix $T\in\syl2{G}$, and set $Z=Z(T)\cong C_2$, $H=C_G(Z)$, and $J=O_2(H)$.  Let 
$z\in{}Z$ be a generator. Then $H$ is the parabolic subgroup of order 
$2^{11}\cdot5$, $|J|=2^9$, and $H/J\cong C_5\sd{}C_4$.  Set $E=[J,J]$.  By 
\cite[Lemma 1]{Parrott}, $E=Z_2(J)=\Fr(J)\cong C_2^5$, and by the proof 
of that lemma, the Sylow $5$-subgroups of $H$ act irreducibly on $J/E\cong 
C_2^4$ and on $E/Z\cong C_2^4$. Since each element of $\Aut_{H/J}(J/E)$ 
sends $C_{J/E}(T/J)\cong C_2$ to itself,
	\beqq \Aut_{H/J}(J/E)=\{\Id_{J/E}\} \quad\textup{and}\quad
	|\Hom_{H/J}(J/E,E/Z)| \le |\Hom_{H/J}(J/E,J/E)|=2\,. 
	\label{e:|Hom|} \eeqq

Let $N>T$ be the other parabolic, and set $K=O_2(N)$.  Thus 
$N/K\cong\Sigma_3$, and $[T:K]=2$.

Fix $P\in\syl5{H}\subseteq\syl5{G}$ (so $P\cong C_5$). By \cite[p. 
674]{Parrott}, $H/E=(J/E)\cdot(N_G(P)/Z)$, where $N_G(P)/Z\cong H/J\cong 
C_5\sd{}C_4$. For each $\beta\in\Aut(H)$ such that $\beta(T)=T$, there is 
$\beta_1\equiv\beta$ (mod $\Aut_J(H)$) such that $\beta_1(P)=P$. Since each 
automorphism of $H/J$ which sends $T/J\cong C_4$ to itself is conjugation 
by an element of $T/J$, there is $\beta_2\equiv\beta_1$ (mod 
$\Aut_{N_T(P)}(H)$) such that $\beta_2$ induces the identity on $H/J$. 
By \eqref{e:|Hom|}, $\beta_2$ also induces the 
identity on $J/E$, and hence on $H/E=(J/E)\cdot(N_G(P)/Z)$. Thus
	\beqq N_{\Aut(H)}(T) = \Aut_T(H)\cdot \{\beta\in\Aut(H)\,|\, 
	\beta(P)=P,~ [\beta,H]\le E \}\,. \label{e:Aut(H)} \eeqq

Now set $\call=\call_T^c(G)$ for short, and identify $N=\Aut_\call(K)$ and 
$H=\Aut_\call(J)$. For each $\alpha\in\Aut\typ^I(\call)$, let 
$\alpha_H\in\Aut(H)$ and $\alpha_N\in\Aut(N)$ be the induced automorphisms, 
and set $\alpha_T=\alpha_H|_T=\alpha_N|_T$. Set 
	\[ \cala_0 = \bigl\{\alpha\in\Aut\typ^I(\call) \,\big|\, 
	[\alpha_H,H]\le E \textup{ and } \alpha_H|_P=\Id_P \bigr\}\,. \]
By \eqref{e:Aut(H)}, each class in $\Out\typ(\call)$ contains at least one 
automorphism in $\cala_0$.

Fix $\alpha\in\cala_0$. Since 
$[\alpha_H,H]$ must be normal in $H$, we have $[\alpha_H,H]\in\{E,Z,1\}$.  
If $[\alpha_H,H]=Z$, then $\alpha_H|_{JP}=\Id$, so 
$[\alpha_H,K]=[\alpha_N,K]=Z$, which is impossible since $Z$ is not normal 
in $N$ by \cite[Lemma 6]{Parrott} (or since $z\notin Z(G)$ and 
$G=\gen{H,N}$). Thus either $\alpha_H=\Id$, or $[\alpha_H,H]=E$.  

If $\alpha_H=\Id_H$, then $\alpha_N|_T=\Id$. In this case, $\alpha_N$ 
determines an element of $H^1(N/K;Z(K))$ whose restriction to 
$H^1(T/K;Z(K))$ is trivial, and since this restriction map for 
$H^1(-;Z(K))$ is injective (since $T/K\in\syl2{N/K}$), $\alpha_N\in\Inn(N)$ 
(see, e.g., \cite[Lemma 1.2]{OV2}). Hence $\alpha_N\in\Aut_Z(N)$ since 
$\alpha_N|_T=\Id$ (and $Z=Z(T)$). So $\alpha\in\Aut_Z(\call)$ in this case, 
and $[\alpha]=1\in\Out\typ(\call)$.

Set $\4H=H/Z$, and similarly for subgroups of $H$.  Let 
$\4\alpha_H\in\Aut(\4H)$ and $\4\alpha_T\in\Aut(\4T)$ be the automorphisms 
induced by $\alpha_H$ and $\alpha_T$, and set 
$\beta=\4\alpha_T|_{\overline{J}}$.  
Then $\4E=Z(\4J)$ since $E=Z_2(J)$, so $\beta(g)=g\varphi(\5g)$ for some 
$\varphi\in\Hom_{H/J}(J/E,\4E)$.  If 
$\varphi=1$, so that $[\alpha,J]\le Z$, then since $\alpha|_P=\Id$, we 
have $[\alpha_H,H]<E$ and so $\alpha_H=\Id$.

We have now constructed a homomorphism from $\cala_0$ to 
$\Hom_{H/J}(J/E,\4E)$ with kernel $\Aut_{Z}(\call)$. Thus 
	\[ |\Out\typ(\call)| \le |\cala_0/\Aut_Z(\call)| \le
	|\Hom_{H/J}(J/E,\4E)| \le 2 \,. \]
where the last inequality holds by \eqref{e:|Hom|}. 
Since $|\Out(G)|=2$ by \cite[Theorem 2]{GrL}, and since $\kappa_G$ is injective 
by Lemma \ref{kappa_inj}, this proves that $\kappa_G$ is an isomorphism.

Alternatively, this can be shown using results in \cite{Fan}.  Since 
$T/[T,T]\cong C_2\times C_4$ by the above description of $T/E$ (where 
$E\le[T,T]$), $\Aut(T)$ and hence $\Out\typ(\call)$ are $2$-groups.  
So each automorphism of the amalgam $H>T<N$ determines a larger 
amalgam.  Since the only extension of this amalgam is to that of 
$\lie2F4(2)$ by \cite[Theorem 1]{Fan}, $|\Out\typ(\call)|=2$. 
\end{proof}

\newpage

\newsect{The cross characteristic case: I}
\label{s:X1}

Throughout this section, we will work with groups $G=C_{\4G}(\sigma)$ which 
satisfy the conditions in Hypotheses \ref{G-hypoth-X} below.  In 
particular, \ref{G-hypoth-X}\eqref{not4x} implies that $G$ is not a Suzuki 
or Ree group. We will see in Section \ref{s:X2} (Proposition 
\ref{G-cases-odd}) that while these hypotheses are far from including all 
finite Chevalley and Steinberg groups, their fusion systems at the prime 
$p$ do include almost all of those we need to consider.

For any finite abelian group $B$, we denote its ``scalar automorphisms'' by
	\[ \psi_k^B\in\Aut(B),\qquad \psi_k^B(g)=g^k \qquad
	\textup{for all $k$ such that $(k,|B|)=1$} \]
and define the group of its scalar automorphisms
	\[ \Aut\scal(B) = \bigl\{ \psi_k^B \,\big|\, (k,|B|)=1 
	\bigr\} \le Z(\Aut(B))~. \]

\begin{Hyp} \label{G-hypoth-X}
Assume we are in the situation of Notation 
\ref{G-setup}(\ref{not1},\ref{not2},\ref{not3}). 
\begin{enumI}[widest=(I)]

\item\label{not4x} Let $p$ be a prime distinct from $q_0$ such that 
$p\big||W_0|$.  Assume also that 
$\sigma=\psi_q\circ\gamma=\gamma\circ\psi_q\in\End(\4G)$, where 
\medskip

\begin{itemize} 
\item $q$ is a power of the prime $q_0$;

\item $\psi_q\in\Phi_{\4G}$ is the field automorphism (see Definition 
\ref{d:Aut(Gbar)}(a)); and 

\item $\gamma\in\Aut(\4G)$ is an algebraic automorphism of finite order 
which sends $\4T$ to itself and commutes with $\psi_{q_0}$ (so that 
$\psi_{q_0}(G)=G$). 

\end{itemize}
\medskip

\noindent Also, there is a free $\gen{\tau}$-orbit of the form 
$\{\alpha_1,\alpha_2,\ldots,\alpha_s\}$ or 
$\{\pm\alpha_1,\pm\alpha_2,\ldots,\pm\alpha_s\}$ in $\Sigma$ such that the 
set $\{\alpha_1,\alpha_2,\ldots,\alpha_s\}$ is linearly independent in $V$.

\item  The algebraic group $\4G$ is of universal type, and $N_G(T)$ contains 
a Sylow $p$-subgroup of $G$. 

\item  Set $A=O_p(T)$. Assume 
one of the following holds: either 

\smallskip

\begin{enumerate}[label=\textup{(\Roman{enumIi}.\arabic*)},
ref=(\Roman{enumIi}.\arabic*),labelindent=3mm,leftmargin=!,itemsep=6pt,
topsep=6pt]

\item \label{easy_case}
$q\equiv1$ (mod $p$), $q\equiv1$ (mod $4$) if $p=2$, $|\gamma|\le2$, and 
$\gamma\in\Gamma_{\4G}$ (thus $\rho(\Pi)=\Pi$); or


\item \label{minus_case}
$p$ is odd, $q\equiv-1$ (mod $p$), $G$ is a Chevalley group (i.e., 
$\gamma\in\Inn(\4G)$), and $\gamma(t)=t^{-1}$ for each $t\in\4T$; or 


\item \label{messy_case}
$p$ is odd, $|\tau|=\ordp(q)\ge2$, $C_A(O_{p'}(W_0))=1$, 
$C_S(\Omega_1(A))=A$, $\Aut_G(A)=\Aut_{W_0}(A)$, 
	\[ N_{\Aut(A)}(\Aut_{W_0}(A)) \le
	\Aut\scal(A)\Aut_{\Aut(G)}(A) \]
where $\Aut_{\Aut(G)}(A)=\bigl\{\delta|_A\,\big|\,\delta\in\Aut(G),~ 
\delta(A)=A \bigr\}$, and 
	\[ \Aut_{W_0}(A)\cap\Aut\scal(A) \le
	\begin{cases} 
	\Gen{\gamma|_A} & \textup{if $2\big|\ordp(q)$ or $-\Id\notin W$} \\
	\Gen{\gamma|_A,\psi_{-1}^A} & \textup{otherwise,}
	\end{cases} \]
\end{enumerate}
\end{enumI} 
\end{Hyp}

Since $W_0$ acts on $T$ by Lemma \ref{W0onT}, it also acts on $A=O_p(T)$.

We will see in Lemma \ref{NG(T)} that the conditions $C_S(\Omega_1(A))=A$ 
(or $C_S(A)=A$ when $p=2$) and $\Aut_G(A)=\Aut_{W_0}(A)$, both assumed here 
in \ref{messy_case}, also hold in cases \ref{easy_case} and 
\ref{minus_case}. 

Recall, in the situation of \ref{messy_case}, that $|\tau|=|\gamma|_{\4T}|$ 
by Lemma \ref{l:tau}.

Note that the above hypotheses eliminate the possibility that $G$ be a 
Suzuki or Ree group. Since we always assume the Sylow $p$-subgroups are 
nonabelian, the only such case which needs to be considered here (when 
$q_0\ne{}p$) is that of $\lie2F4(q)$ when $p=3$, and this will be handled 
separately. 

By Lemma \ref{l:tau}, whenever 
$\sigma=\psi_q\circ\gamma$, and $\gamma$ is an algebraic automorphism of 
$\4G$ which normalizes $\4T$, there is $\tau\in\Aut(V)$ such that 
$\tau(\Sigma)=\Sigma$ and $\sigma(\4X_\alpha)=\4X_{\tau(\alpha)}$ for each 
$\alpha\in\Sigma$. So under Hypotheses \ref{G-hypoth-X}, the condition at 
the beginning of Notation \ref{G-setup}\eqref{not3} holds automatically, 
and with $\rho=\tau|_\Sigma$. To simplify the notation, throughout this 
section and the next, we write $\tau=\rho$ to denote this induced 
permutation of $\Sigma$. 

The following notation will be used throughout this section, in addition to 
that in Notation \ref{G-setup}.  
Note that $\5\Pi$ and $\5\Sigma$ are defined in Notation 
\ref{G-setup}\eqref{not3} only when $\rho(\Pi)=\Pi$, and hence only 
in case \ref{easy_case} of Hypotheses \ref{G-hypoth-X}. It will be 
convenient, in some of the proofs in this section, to extend this 
definition to case \ref{minus_case}.

Recall (Notation \ref{G-setup}) that for 
$\alpha\in\Sigma$, $w_\alpha\in{}W$ denotes the reflection in the 
hyperplane $\alpha^\perp\subseteq V$.  

\begin{Not} \label{G-setup-X}
Assume we are in the situation of Notation \ref{G-setup} and Hypotheses 
\ref{G-hypoth-X}.  
\begin{enumA}[start=4]

\item\label{not5x} If \ref{minus_case} holds, then set $\5\Sigma=\Sigma$, 
$\5\Pi=\Pi$, and $V_0=V$. Note that $W_0=W$ in this case.

\item\label{not6x} If \ref{easy_case} holds, then for each $\5\alpha\in\5\Sigma$, 
let $w_{\5\alpha}\in W_0$ be the element in 
$\gen{w_\alpha\,|\,\alpha\in\5\alpha}$ which acts on $V_0$ as the 
reflection across the hyperplane $\gen{\5\alpha}^\perp$, and which 
exchanges the positive and negative roots in the set 
$\gen{\5\alpha}\cap\Sigma$. (Such an element exists and lies in $W_0$ by 
\cite[Proposition 13.1.2]{Carter}.) 

\item\label{not7x} If \ref{easy_case} or \ref{minus_case} holds, then for each 
$\alpha\in\Sigma$ and each $\5\alpha\in\5\Sigma$, set 
	\begin{align*} 
	\4K_\alpha&=\gen{\4X_\alpha,\4X_{-\alpha}} &
	\4T_\alpha&=h_\alpha(\fqobar^\times) \\
	\4K_{\5\alpha} &= \gen{\4K_\alpha\,|\,\alpha\in\5\alpha} &
	\4T_{\5\alpha} &= \gen{\4T_\alpha\,|\,\alpha\in\5\alpha} \,.
	\end{align*}

\item\label{not8x} Set $N=N_G(T)/O_{p'}(T)$, and identify $A=O_p(T)$ with 
$T/O_{p'}(T)\nsg{}N$.  If \ref{easy_case} or \ref{minus_case} holds, 
then for $\5\alpha\in\5\Sigma$, set $A_{\5\alpha}=A\cap\4T_{\5\alpha}$.

\item\label{not9x} Fix $S\in\sylp{G}$ such that $A\le{}S\le{}N_G(T)$, and 
set $\calf=\calf_S(G)$.  (Recall that $N_G(T)$ contains a Sylow 
$p$-subgroup of $G$ by Hypotheses \ref{G-hypoth-X}(II).) Set 
	\[ \Aut(A,\calf) = \bigl\{\beta\in\Aut(A) \,\big|\,
	\beta=\4\beta|_A,~\textup{some}~
	\4\beta\in\Aut(S,\calf)\bigr\} \,. \]
Set $\Aut\dg(S,\calf)=C_{\Aut(S,\calf)}(A)=
\bigl\{\beta\in\Aut(S,\calf)\,\big|\,\beta|_A=\Id\bigr\}$, 
and let $\Out\dg(S,\calf)$ be the image of $\Aut\dg(S,\calf)$ in 
$\Out(S,\calf)$.  
\end{enumA}
\end{Not}

Note that when $(\4G,\sigma)$ is a standard setup (i.e., in case 
\ref{easy_case}), $W_0$ acts faithfully on $V_0$ (see \cite[Lemma 
13.1.1]{Carter}).

Recall that $N=N_G(T)/O_{p'}(T)$.  We identify $A=O_p(T)$ with 
$T/O_{p'}(T)\nsg N$.

\begin{Lem} \label{NG(T)}
Assume Hypotheses \ref{G-hypoth-X} and Notation \ref{G-setup-X}. 
\begin{enuma} 

\item If condition \ref{easy_case} or \ref{minus_case} holds, then 
$C_W(A)=1$, $C_{\4G}(A)=C_{\4G}(T)=\4T$, $C_G(A)=T$, and $C_S(A)=A$. If $p$ 
is odd, then $C_W(\Omega_1(A))=1$ and $C_S(\Omega_1(A))=A$. 

\item If $C_{\4G}(A)^0=\4T$ (in particular, if \ref{easy_case} or 
\ref{minus_case} holds), then 
$N_G(A)=N_G(T)\le N_{\4G}(\4T)$, and the inclusion of $N_G(T)$ in 
$N_{\4G}(\4T)$ induces isomorphisms $W_0\cong N_G(T)/T\cong N/A$. Thus 
$\Aut_G(A)=\Aut_{W_0}(A)$.

\end{enuma}
\end{Lem}

\begin{proof} \noindent\textbf{(a) } Assume condition \ref{easy_case} or 
\ref{minus_case} holds. We first prove that $C_W(A)=1$, and also that 
$C_W(\Omega_1(A))=1$ when $p$ is odd. 

If $p$ is odd, set $A_0=\Omega_1(A)$ and $\5p=p$. 
If $p=2$, set $A_0=\Omega_2(A)$ and $\5p=4$. Thus in all cases, $A_0$ is 
the $\5p$-torsion subgroup of $A$. Set $\gee=1$ if we are in case 
\ref{easy_case}, or $\gee=-1$ in case \ref{minus_case}. By assumption, 
$\5p|(q-\gee)$. Choose $\lambda\in\F_q^\times$ (or 
$\lambda\in\F_{q^2}^\times$ if $\gee=-1$) of order $\5p$. Set 
$\Pi=\{\alpha_1,\ldots,\alpha_r\}$.  Fix $w\in{}C_W(A_0)$.

Assume first $G=\gg(q)$, a Chevalley group. Then 
$T=\bigl\{t\in\4T\,\big|\,t^{q-\gee}=1\bigr\}$, and $A_0$ contains all 
elements of order $\5p$ in $\4T$. So $w=1$ by Lemma \ref{scal-mod-m}.

Now assume that $\Id\ne\gamma\in\Gamma_{\4G}$; i.e., $G$ is one of the 
Steinberg groups $\lie2An(q)$, $\lie2Dn(q)$, or $\lie2E6(q)$.  Then 
$C_{\4G}(\gamma)$ is a simple algebraic group of type $B_m$, $C_m$, or 
$F_4$ (cf. \cite[\S\,13.1--3]{Carter}) with root system $\5\Sigma\subseteq 
V_0=C_V(\tau)$, and $A_0$ contains all $\5p$-torsion in $C_{\4T}(\gamma)$. 
By Lemma \ref{scal-mod-m} again, $w|_{V_0}=\Id$. Since $w$ 
and $\tau$ are both orthogonal, $w$ also sends the $(-1)$-eigenspace for 
the action of $\tau$ to itself, and thus $w\in{}C_W(\tau)=W_0$.  But $W_0$ 
acts faithfully on $V_0$ (see, e.g., \cite[13.1.1]{Carter}), so $w=1$. 

Thus $C_W(A_0)=1$. Hence $C_{\4G}(A_0)=\4T$ by Proposition \ref{p:CG(T)}, and
the other statements follow immediately.

\smallskip

\noindent\textbf{(b) } If $C_{\4G}(A)^0=\4T$, then $N_{\4G}(T)\le 
N_{\4G}(A)\le N_{\4G}(\4T)$ (recall that $A$ is the $p$-power torsion in 
$T$).  If $g\in N_{\4G}(\4T)$ and $\sigma(g)=g$, then $g$ also normalizes 
$T=C_{\4T}(\sigma)$.  Thus $N_G(T)=N_G(A)\le N_{\4G}(\4T)$, and hence 
$N_G(T)/T\cong W_0$ by Lemma \ref{W0onT}. The identification $N/A\cong 
N_G(T)/T$ is immediate from the definition of $N$.
\end{proof}

Recall (Notation \ref{G-setup-X}\eqref{not7x}) that when case \ref{easy_case} of 
Hypotheses \ref{G-hypoth-X} holds (in particular, when $p=2$),  we set 
$\4K_{\5\alpha}=\gen{\4K_\alpha\,|\,\alpha\in\5\alpha}$ for 
$\5\alpha\in\5\Sigma$, where $\4K_\alpha=\gen{\4X_\alpha,\4X_{-\alpha}}$. 
The conditions in \ref{easy_case} imply that each class 
in $\5\Sigma$ is of the form $\{\alpha\}$, $\{\alpha,\tau(\alpha)\}$, or 
$\{\alpha,\tau(\alpha),\alpha+\tau(\alpha)\}$ for some $\alpha$.  This last 
case occurs only when $G\cong\SU_n(q)$ for some odd $n\ge3$ and some 
$q\equiv1$ (mod $p$ or mod $4$).  

\begin{Lem} \label{Kr,Hr}
Assume Hypotheses \ref{G-hypoth-X}, case \ref{easy_case}, and Notation 
\ref{G-setup-X}.  For each 
$\alpha\in\Sigma$, $\4K_\alpha\cong\SL_2(\fqobar)$.  For each 
$\5\alpha\in\5\Sigma$, $\4K_{\5\alpha}\cong{}\SL_2(\fqobar)$, 
$\SL_2(\fqobar)\times\SL_2(\fqobar)$, or $\SL_3(\fqobar)$ whenever the 
class $\5\alpha$ has order $1$, $2$, or $3$, respectively.  Also, 
$G\cap\4K_{\5\alpha}$ is isomorphic to $\SL_2(q)$, $\SL_2(q^2)$, or 
$\SU_3(q)$, respectively, in these three cases.  
\end{Lem}

\begin{proof} By Lemma \ref{W0-action}, each class in $\5\Sigma$ is 
in the $W_0$-orbit of a class in $\5\Pi$.  So it suffices to prove the 
statements about $\4K_\alpha$ and $\4K_{\5\alpha}$ when $\alpha\in\Pi$, and 
when $\5\alpha\in\5\Pi$ is its equivalence class.

By Lemma \ref{theta-r}\eqref{T=prod} (and since $\4G$ is universal), 
$\4K_\alpha\cong\SL_2(\fqobar)$ for each $\alpha\in\Pi$. So when 
$\alpha=\tau(\alpha)$ (when $|\5\alpha|=1$), 
$\4K_{\5\alpha}=\4K_\alpha\cong\SL_2(\fqobar)$.  

When $\alpha\ne\tau(\alpha)$ and they are not orthogonal, then 
$\4G\cong\SL_{2n+1}(\fqobar)$ for some $n$, and the inclusion of 
$\SL_3(\fqobar)$ is clear. When $\alpha\perp\tau(\alpha)$, then 
$[\4K_\alpha,\4K_{\tau(\alpha)}] = 1$, and $\4K_\alpha\cap 
\4K_{\tau(\alpha)}=1$ by Lemma \ref{theta-r}\eqref{T=prod} and since $G$ is universal, and 
since the intersection is contained in the centers of the two factors and 
hence in the maximal tori.  Hence $\bar K_{\widehat\alpha} = \gen{\bar 
X_{\pm \alpha}, \bar X_{\pm\tau(\alpha)}} \cong \4K_\alpha\times \4 
K_{\tau(\alpha)} \cong SL_2(\fqobar) \times SL_2(\fqobar)$.

In all cases, since $\4G$ is universal, $ G \cap \4K_{\5\alpha} =  
C_{\4G}(\sigma) \cap \4K_{\5\alpha} = C_{\4K_{\5\alpha}} (\sigma) $. 
If $\alpha=\tau(\alpha)$, then $\gamma$ acts trivially on $\4K_{\5\alpha}$, 
and $C_{\4K_{\5\alpha}} (\sigma) \cong SL_2(q)$. If 
$\alpha\perp\tau(\alpha)$ then $\gamma$ exchanges the two factors and 
$C_{\4K_{\5\alpha}} (\sigma) \cong SL_2(q^2)$. Finally, if 
$\alpha\ne\tau(\alpha)$ and they are not orthogonal, then $\gamma$ is the graph 
automorphism of $SL_3(\fqobar)$, so $C_{\4K_{\5\alpha}} (\sigma) \cong 
SU_3(q)$.
\end{proof}

We also recall here the definition of the focal subgroup of a saturated fusion 
system $\calf$ over a finite $p$-group $S$: 
	\[ \foc(\calf) = \Gen{ xy^{-1} \,\big|\, x,y\in S,~ 
	\textup{$x$ is $\calf$-conjugate to $y$} }. \]
By the focal subgroup theorem for groups (cf. \cite[Theorem 
7.3.4]{Gorenstein}), if $\calf=\calf_S(G)$ for some finite group $G$ with 
$S\in\sylp{G}$, then $\foc(\calf)=S\cap[G,G]$.

\begin{Lem} \label{Kr,Hr2}
Assume Hypotheses \ref{G-hypoth-X}, case \ref{easy_case} or 
\ref{minus_case}, and Notation \ref{G-setup-X}. Assume also that 
$|\5\Pi|\ge2$. Then the following hold.
\begin{enuma} 

\item If $p$ is odd, then $[w_{\5\alpha},A]=A_{\5\alpha}$ for each 
$\5\alpha\in\5\Sigma$. If $p=2$, then for each $\5\alpha\in\5\Sigma$, 
$[w_{\5\alpha},A]\le A_{\5\alpha}$ with index at most $2$, 
and $[w_{\5\alpha},A]=A_{\5\alpha}$ with the following exceptions:
\medskip
\begin{itemize} 
\item $\tau=\Id$, $\gg=C_n$ (or $B_2$), and $\5\alpha=\{\alpha\}$ 
where $\alpha$ is a long root; or

\item $|\tau|=2$, $\gg=D_n$ (or $A_3$), and 
$\5\alpha=\{\alpha,\tau(\alpha)\}$ where $\alpha\perp\tau(\alpha)$.

\item $|\tau|=2$, $\gg=A_{2n}$, and $|\5\alpha|=3$. 
\end{itemize}

\item For each $w\in W_0$ of order $2$, $w=w_{\5\alpha}$ for some 
$\5\alpha\in\5\Sigma$ if and only if $[w,A]$ is cyclic.

\item If $p=2$, then for each $\5\alpha\in\5\Sigma$, 
	\[ C_{\4G}(C_A(w_{\5\alpha}))=
	\begin{cases} 
	\4T\4K_{\5\alpha} & \textup{if $|\5\alpha|\le2$} \\
	\4T\4K_{\alpha+\tau(\alpha)} & \textup{if 
	$\5\alpha=\{\alpha,\tau(\alpha),\alpha+\tau(\alpha)\}$.}
	\end{cases} \]
If in addition, $|\5\alpha|\le2$, then 
	\[ A_{\5\alpha}=A\cap\bigl[C_G(C_A(w_{\5\alpha})),
	C_G(C_A(w_{\5\alpha}))\bigr] 
	=A\cap\foc(C_\calf(C_A(w_{\5\alpha}))). \]

\end{enuma}
\end{Lem}

\begin{proof} \noindent\textbf{(a) } If we are in case \ref{easy_case} of 
Hypotheses \ref{G-hypoth-X}, then by Lemma \ref{W0-action}, each orbit of 
$W_0$ under its action on $\5\Sigma$ contains an element of $\5\Pi$. If we 
are in case \ref{minus_case}, then $W_0=W$ and $\5\Sigma=\Sigma$, so the 
statement holds by the same lemma. So in either case, it suffices to prove 
this when $\5\alpha\in\5\Pi$.  

Fix $\alpha\in\Pi$, and let $\5\alpha\in\5\Pi$ be its class. Since 
$w_{\5\alpha}\in\gen{w_\alpha,w_{\tau(\alpha)}}$, $[w_{\5\alpha},A]\le 
A\cap\4T_{\5\alpha}=A_{\5\alpha}$ in all cases by Lemma 
\ref{theta-r}\eqref{wa(hb)}. By the same lemma, 
$w_{\5\alpha}(\5h_\alpha(\lambda))=\5h_\alpha(\lambda^{-1})$ for all 
$\lambda\in\fqobar^\times$ if $|\5\alpha|\le2$; and 
$w_{\5\alpha}(\5h_\alpha(\lambda))=\5h_\alpha(\lambda^{-q})$ for 
$\lambda\in\F_{q^2}^\times$ if $|\5\alpha|=3$. So 
$[w_{\5\alpha},A]=A_{\5\alpha}$ if $p$ is odd, and $[w_{\5\alpha},A]$ has 
index at most $2$ in $A_{\5\alpha}$ if $p=2$.

Assume now that $p=2$, and hence that $q\equiv1$ (mod $4$). 
If $\tau=\Id$ (and hence $\5\alpha=\{\alpha\}$), then for each $\beta\in\Pi$ and 
each $\lambda\in\F_q^\times$, Lemma \ref{theta-r}\eqref{wa(hb)} implies 
that
	\[ w_\alpha(h_\beta(\lambda)) = \begin{cases} 
	h_\beta(\lambda) & \textup{if $\beta\perp\alpha$} \\
	h_\beta(\lambda)h_\alpha(\lambda) 
	& \textup{if $\beta\not\perp\alpha$, $\norm\beta\ge\norm\alpha$} \\
	h_\beta(\lambda)h_\alpha(\lambda^k) 
	& \textup{if $\beta\not\perp\alpha$, 
	$\norm\alpha=\sqrt{k}\cdot\norm\beta$, $k=1,2,3$.}
	\end{cases} \]
(Note that $w_\alpha(\beta^\vee)=\beta^\vee$, $\beta^\vee+\alpha^\vee$, or 
$\beta^\vee+k\alpha^\vee$, respectively, in these three cases.) 
Since $T$ is generated by the $h_\beta(\lambda)$ for $\beta\in\Pi$ and 
$\lambda\in\F_q^\times$,, it follows that $[w_\alpha,A]$ has index $2$ in 
$A_\alpha$ exactly when there are roots with two lengths and ratio 
$\sqrt2$, $\alpha$ is a long root, and orthogonal to all other long roots 
in $\Pi$. This happens only when $\gg\cong C_n$ or $B_2$.

Now assume $|\tau|=2$. In particular, all roots in $\Sigma$ have the same 
length. By Lemma \ref{theta-r}\eqref{wa(hb)} again, for each 
$\beta\in\Pi{\sminus}\5\alpha$ such that $\beta\not\perp\alpha$ and with 
class $\5\beta\in\5\Pi$, we have 
	\[ w_{\5\alpha}(\5h_\beta(\lambda)) = \begin{cases} 
	\5h_\beta(\lambda)\5h_\alpha(\lambda) & \textup{if $|\5\beta|=1$ 
	and $\lambda\in\F_q^\times$} \\
	\5h_\beta(\lambda)\5h_\alpha(\lambda) & \textup{if $|\5\beta|\ge2$, 
	$|\5\alpha|=2$, and $\lambda\in\F_{q^2}^\times$} \\
	\5h_\beta(\lambda)\5h_\alpha(\lambda^{q+1}) & \textup{if 
	$|\5\beta|\ge2$, $|\5\alpha|=1$ or $3$, and $\lambda\in\F_{q^2}^\times$} 
	\end{cases} \]
Thus $[w_{\5\alpha},A]=A_{\5\alpha}$ exactly when $|\5\alpha|=1$, or 
$|\5\alpha|=2$ and there is some $\beta\in\Pi$ such that 
$\beta\not\perp\alpha$ and $\beta\ne\tau(\beta)$. The only cases where this 
does not happen are when $\gg=D_n$ or $A_3$ and $|\5\alpha|=2$, and when 
$\gg=A_{2n}$ and $|\5\alpha|\ge3$.

\smallskip

\noindent\textbf{(b) } For each $\5\alpha\in\5\Sigma$, $[w_{\5\alpha},A]\le 
A_{\5\alpha}$ by (a), and hence is cyclic. It remains to prove the 
converse.

When we are in case \ref{minus_case} (and hence the setup is not 
standard), it will be convenient to define $V_0=V$. (Recall that $V_0$ is 
defined in Notation \ref{G-setup}\eqref{not3} only when $(\4G,\sigma)$ is a 
standard setup.) Note that by assumption, $G$ is always a Chevalley group 
in this case. 

Let $w\in{}W_0$ be an element of order $2$ which is not equal to 
$w_{\5\alpha}$ for any $\5\alpha$. If $G$ is a Chevalley group (if $W_0=W$ 
and $V_0=V$), then $C_V(w)$ contains no points in the interior of any Weyl 
chamber, since $W$ permutes freely the Weyl chambers (see \cite[\S\,V.3.2, 
Th\'eor\`eme 1(iii)]{Bourb4-6}). Since $w$ is not the reflection in a root 
hyperplane, it follows that $\dim(V/C_V(w))\ge2$. If $G$ is a Steinberg 
group (thus in case \ref{easy_case} with a standard setup), then 
$W_0$ acts on $V_0$ as the Weyl group of a certain root system on $V_0$ 
(see \cite[\S\,13.3]{Carter}), so $\dim(V_0/C_{V_0}(w))\ge2$ by a similar 
argument. 

Set $\gee=+1$ if we are in case \ref{easy_case}, or $\gee=-1$ if we are 
in case \ref{minus_case}. Set $m=v_p(q-\gee)$, and choose 
$\lambda\in(\F_{q^2})^\times$ of order $p^m$. Set $\Lambda=\Z\Sigma^\vee$, 
regarded as the lattice in $V$ with $\Z$-basis 
$\Pi^\vee=\{\alpha^\vee\,|\,\alpha\in\Pi\}$. Let 
	\[ \Phi_\lambda\: \Lambda/p^m\Lambda \Right5{} \4T \]
be the $\Z[W]$-linear monomorphism of Lemma \ref{l:CW(t)}(a) with image the 
$p^m$-torsion in $\4T$. Thus $\Phi_\lambda(\alpha^\vee)=h_\alpha(\lambda)$ 
for each $\alpha\in\Sigma$. 
Also, $\sigma(h_\alpha(\lambda))=h_{\tau(\alpha)}(\lambda)$ for each 
$\alpha\in\Sigma$ ($\lambda^q=\lambda$ by assumption), and thus 
$\Phi_\lambda$ commutes with the actions of $\tau$ on $\Lambda<V$ and of 
$\sigma$ on $\4T$. 

Set $\Lambda_0=C_\Lambda(\tau)$ in case \ref{easy_case}, or 
$\Lambda_0=\Lambda$ in case \ref{minus_case}. Then 
$C_{\Lambda/p^m\Lambda}(\tau)=\Lambda_0/p^m\Lambda_0$ in case 
\ref{easy_case}, since $\tau$ permutes the basis $\Pi^\vee$ of $\Lambda$. 
We claim that $\Phi_\lambda$ restricts to a $\Z[W_0]$-linear 
isomorphism 
	\[ \Phi_0\: \Lambda_0/p^m\Lambda_0 \Right5{\cong} \Omega_m(A)\,, \]
where $\Omega_m(A)$ is the $p^m$-torsion subgroup of $A$ and hence of 
$T=C_{\4T}(\sigma)$. If $G$ is a Chevalley group (in either case 
\ref{easy_case} or \ref{minus_case}), then $\Lambda_0=\Lambda$, so 
$\Im(\Phi_0)$ is the $p^m$-torsion subgroup of $\4T$ and equal to 
$\Omega_m(A)$. If $G$ is a Steinberg group, then $\gee=+1$, each element of 
order dividing $p^m$ in $\4T$ is fixed by $\psi^q$, and hence lies in 
$\Omega_m(A)$ if and only if it is fixed by $\gamma$ (thus in 
$\Phi_\lambda(C_{\Lambda/p^m\Lambda}(\tau))$). 

Thus $[w,A]\ge[w,\Omega_m(A)]\cong[w,\Lambda_0/p^m\Lambda_0]$. Set 
$B=\Lambda_0/p^m\Lambda_0$ for short; we will show that $[w,B]$ is 
noncyclic. Set 
	\[ r=\rk(\Lambda_0)=\dim(V_0) \qquad\textup{and}\qquad
	s=\rk(C_{\Lambda_0}(w))=\dim_{\R}(C_{V_0}(w))\le r-2\,. \] 
For each $b\in{}C_B(w)$, and each $v\in\Lambda_0$ such that 
$b=v+p^m\Lambda_0$, $v+w(v)\in{}C_{\Lambda_0}(w)$ maps to $2b\in{}C_B(w)$. 
Thus $B\cong(\Z/p^m)^r$, while $\{2b\,|\,b\in{}C_B(w)\}$ is contained in 
$C_{\Lambda_0}(w)/p^mC_{\Lambda_0}(w)\cong(\Z/p^m)^s$. Since $p^m>2$ by 
assumption (and $r-s\ge2$), it follows that $B/C_B(w)\cong[w,B]$ is not cyclic.

\smallskip

\noindent\textbf{(c) } Fix $\5\alpha\in\5\Sigma$. We set up our notation as 
follows.
\begin{enumerate}[label=\bf Case (\arabic*): ,leftmargin=22mm] 

\item $|\5\alpha|=1$ or $3$. Set $\alpha^*=\alpha$ if $\5\alpha=\{\alpha\}$ 
(where $\tau(\alpha)=\alpha$), or $\alpha^*=\alpha+\tau(\alpha)$ if 
$\5\alpha=\{\alpha,\tau(\alpha),\alpha+\tau(\alpha)\}$. Set 
$w_{\5\alpha}=w_{\alpha^*}$, $W_{\5\alpha}=\gen{w_{\5\alpha}}$, and 
$\Delta=\{\pm\alpha^*\}\subseteq\Sigma$.  

\item $\5\alpha=\{\alpha,\tau(\alpha)\}$ where 
$\alpha\perp\tau(\alpha)$.  Set 
$w_{\5\alpha}=w_{\alpha}w_{\tau(\alpha)}$, 
$W_{\5\alpha}=\gen{w_{\alpha},w_{\tau(\alpha)}}$, and 
$\Delta=\{\pm\alpha,\pm\tau(\alpha)\}\subseteq\Sigma$.  

\end{enumerate}

In case (1), by Lemma \ref{theta-r}(\ref{^txb(u)},\ref{wa(hb)}), 
	\[ C_{\4T}(w_{\5\alpha})= C_{\4T}(w_{\alpha^*})= \Ker(\theta_{\alpha^*})
	= C_{\4T}(\4X_{\alpha^*}) = C_{\4T}(\4X_{-\alpha^*})\,. \]
Hence $C_{\4G}(C_A(w_{\5\alpha})) \ge C_{\4G}(C_{\4T}(w_{\5\alpha})) \ge 
\4T\gen{\4X_{\alpha^*},\4X_{-\alpha^*}} = \4T\4K_{\alpha^*}$. In case (2), 
by the same lemma, 
	\[ C_{\4T}(w_{\5\alpha})=C_{\4T}(\gen{w_\alpha,w_{\tau(\alpha)}}) 
	= C_{\4T}(\gen{\4X_{\alpha},\4X_{-\alpha},\4X_{\tau(\alpha)},
	\4X_{-\tau(\alpha)}})
	= C_{\4T}(\4K_\alpha\4K_{\tau(\alpha)}) \]
so that $C_{\4G}(C_A(w_{\5\alpha}))\ge \4T\4K_{\5\alpha}$.
This proves one of the inclusions in the first statement in (c).  By 
Proposition \ref{p:CG(T)}, the opposite inclusion will follow 
once we show that 
	\beqq C_W(C_A(w_{\5\alpha})) \le W_{\5\alpha} . 
	\label{e:w_in_cent} \eeqq

Fix $w\in{}C_W(C_A(w_{\5\alpha}))$.
\begin{itemize} 

\item Let $\beta\in\Sigma\cap\Delta^\perp$ be such that 
$\beta=\tau(\beta)$. Then $h_{\beta}(\lambda)\in{}C_A(w_{\5\alpha})$ for 
$\lambda\in\fqobar^\times$ of order $4$, so 
$w(h_{\beta}(\lambda))=h_{\beta}(\lambda)$, and $\beta\in{}C_V(w)$ by Lemma 
\ref{l:CW(t)}(c).

\item Let $\beta\in\Sigma\cap\Delta^\perp$ be such that $\beta\ne\tau(\beta)$, 
and set $\beta'=\tau(\beta)$ for short.  Let $r\ge2$ be such that $q\equiv1+2^r$ 
(mod $2^{r+1}$), and choose $\lambda\in\fqobar^\times$ of 
order $2^{r+1}$.  Set $a=1-2^r$, so $\lambda^a=\lambda^q$.  Then 
	\[ h_{\beta}(\lambda)h_{\beta'}(\lambda^a), ~
	h_{\beta}(\lambda^a)h_{\beta'}(\lambda)\in C_A(w_{\5\alpha})\le 
	C_{\4T}(w)\,. \]
Also, $\norm{\beta+a\beta'}=\norm{a\beta+\beta'}<(1-a)\norm{\beta}
=\frac12|\lambda|\norm{\beta}$ since $a<0$ and 
$\beta'\ne-\beta$ (since $\tau(\Sigma_+)=\Sigma_+$). Thus 
$\beta+a\beta',a\beta+\beta'\in{}C_V(w)$ by Lemma \ref{l:CW(t)}(b), so  
$\beta,\beta'\in C_V(w)$.

\item Let $\beta\in\Sigma$ be such that $\beta=\tau(\beta)$ and 
$\beta\notin\Delta^\perp$, and set $\eta=\beta+w_{\5\alpha}(\beta)$. Since 
$w_{\5\alpha}\tau=\tau w_{\5\alpha}$ in $\Aut(V)$, 
$\tau(\eta)=\eta$. Since $\beta\notin\Delta^\perp=C_V(w_{\5\alpha})$, we 
have $w_{\5\alpha}(\beta)\ne\beta$, and hence $\norm\eta<2\norm\beta$. 
For $\lambda\in\fqobar^\times$ of order $4$, 
$t=h_{\beta}(\lambda)h_{w_{\5\alpha}(\beta)}(\lambda)\in{}C_A(w_{\5\alpha})$, 
so $w(t)=t$, and $\eta=\beta+w_{\5\alpha}(\beta)\in C_V(w)$ by Lemma 
\ref{l:CW(t)}(b). 

\end{itemize}
Consider the set 
	\[ \Sigma^* = 
	\bigl(\Sigma\cap\Delta^\perp\bigr) \cup 
	\bigl\{\beta+w_{\5\alpha}(\beta)\,\big|\, \beta\in\Sigma,~ 
	\tau(\beta)=\beta,~ \beta\not\perp\Delta\bigr\} \subseteq V \,. \]
We have just shown that $w(\eta)=\eta$ for each $\eta\in\Sigma^*$, and 
hence that $w|_{\gen{\Sigma^*}}=\Id$.  From the description of the root 
systems in \cite[Planches I--IX]{Bourb4-6}, we see that 
$\Sigma\cap(\Sigma^*)^\perp=\Delta$, except when $\gg\cong A_2$ and 
$\tau\ne\Id$ (i.e., when $G\cong\SU_3(q)$).

Thus when $G\not\cong\SU_3(q)$, the only reflection hyperplanes which contain 
$\gen{\Sigma^*}$ are those in the set $\{\beta^\perp\,|\,\beta\in\Delta\}$. 
Fix a ``generic'' element $v\in\gen{\Sigma^*}$; i.e., one which is not 
contained in any of these hyperplanes. In case (1), $v$ is contained in 
only the one reflection hyperplane $\alpha^*{}^\perp$, and hence is in 
the closure of exactly two Weyl chambers for $(\Sigma,W)$: chambers which are 
exchanged by $w_{\5\alpha}$.  In case (2), $v$ is contained in the two 
reflection hyperplanes $\alpha^\perp$ and $\tau(\alpha)^\perp$, and hence 
in the closure of four Weyl chambers which are permuted freely and 
transitively by $W_{\5\alpha}=\gen{w_{\alpha},w_{\tau(\alpha)}}$.  Since 
$W$ permutes the Weyl chambers freely and transitively (see 
\cite[\S\,V.3.2, Th\'eor\`eme 1(iii)]{Bourb4-6}), and since 
$\gen{w,W_{\5\alpha}}$ permutes the chambers whose closures contain $v$, we 
have $w\in W_{\5\alpha}$. 

This proves \eqref{e:w_in_cent} when $G\not\cong\SU_3(q)$. If 
$G\cong\SU_3(q)$, then $h_{\alpha^*}(-1)\in C_A(w_{\5\alpha})$. But no 
element of order $2$ in $\4T<\SL_3(\fqobar)$ centralizes the full Weyl 
group $W\cong\Sigma_3$, so \eqref{e:w_in_cent} also holds in this case.

If $|\5\alpha|\le2$, then 
	\[ C_{G}(C_A(w_{\5\alpha}))= G\cap C_{\4G}(C_A(w_{\5\alpha})) =
	T(G\cap\4K_{\5\alpha}) \]
where by Lemma \ref{Kr,Hr}, $G\cap\4K_{\5\alpha}\cong\SL_2(q)$ or 
$\SL_2(q^2)$. Hence $C_G(C_A(w_{\5\alpha}))$ has commutator subgroup 
$G\cap\4K_{\5\alpha}$, and focal subgroup $A_{\5\alpha}$. Since 
$C_\calf(C_A(w_{\5\alpha}))$ is the fusion system of 
$C_G(C_A(w_{\5\alpha}))$ (cf. \cite[Proposition I.5.4]{AKO}), this proves 
the last statement. 
\end{proof}

\begin{Lem} \label{CG(T)-2}
Assume Hypotheses \ref{G-hypoth-X}, case \ref{easy_case}, and Notation 
\ref{G-setup-X}.  
\begin{enuma} 

\item Assume that all classes in $\5\Sigma$ have order $1$ or $2$.  
(Equivalently, $\tau(\alpha)=\alpha$ or $\tau(\alpha)\perp\alpha$ for each 
$\alpha\in\Sigma$.)  Then $C_{\4T}(W_0)=C_{\4T}(W)=Z(\4G)$, and 
$Z(G)=C_T(W_0)$.

\item Assume that $\5\Sigma$ contains classes of order $3$.  Then 
$\4G\cong\SL_{2n-1}(\fqobar)$ and $G\cong\SU_{2n-1}(q)$ for some $n\ge2$.  
Also, $C_{\4T}(W_0)\cong\fqobar^\times$, and $\sigma(t)=t^{-q}$ for all 
$t\in C_{\4T}(W_0)$.

\end{enuma}
\end{Lem}

\begin{proof} \textbf{(a) }  Assume that $\tau(\alpha)=\alpha$ or 
$\tau(\alpha)\perp\alpha$ for each $\alpha\in\Sigma$. We first show, for 
each $\5\alpha=\{\alpha,\tau(\alpha)\}\in\5\Pi$, that 
$C_{\4T}(w_{\5\alpha})=C_{\4T}(w_\alpha,w_{\tau(\alpha)})$.  This is clear 
if $\alpha=\tau(\alpha)$.  If $\alpha\perp\tau(\alpha)$, then 
$w_{\5\alpha}=w_\alpha{}w_{\tau(\alpha)}$, so if $t\in 
C_{\4T}(w_{\5\alpha})$, then $w_\alpha(t)=w_{\tau(\alpha)}(t)$ and 
$t^{-1}w_\alpha(t)=t^{-1}w_{\tau(\alpha)}(t)$. Also, 
$t^{-1}w_\alpha(t)\in\4T_\alpha$ and $t^{-1}w_{\tau(\alpha)}(t)\in 
\4T_{\tau(\alpha)}$ by Lemma \ref{theta-r}\eqref{wa(hb)}.  Since 
$\4T_\alpha\cap \4T_{\tau(\alpha)}=1$ by Lemma \ref{theta-r}\eqref{T=prod}, 
$t^{-1}w_\alpha(t)=1$, and hence $t\in 
C_{\4T}(w_\alpha,w_{\tau(\alpha)})$.  

Since $W=\gen{w_{\alpha}\,|\,\alpha\in\Pi}$, this proves that 
$C_{\4T}(W_0)=C_{\4T}(W)$. Since $\4G$ is universal, $C_{\4T}(W)=Z(\4G)$ by 
Proposition \ref{p:CG(T)}. In
particular, $C_T(W_0)\le G\cap{}Z(\4G)\le Z(G)$; while $Z(G)\le C_T(W_0)$ 
since $C_G(T)=T$ by Lemma \ref{NG(T)}(a).

\smallskip

\noindent\textbf{(b) } Assume $\5\Sigma$ contains a class of order $3$.  
Then by \cite[(2.3.2)]{GLS3}, $\gamma\ne\Id$, $\gg=\SL_{2n-1}$, and 
$G\cong\SU_{2n-1}(q)$ (some $n\ge2$).  Also, if we identify 
	\[ \4T = \bigl\{\diag(\lambda_1,\ldots,\lambda_{2n-1})\,\big|\, 
	\lambda_i\in\fqobar^\times,~ 
	\lambda_1\lambda_2\cdots\lambda_{2n-1}=1 \bigr\} \,, \]
and identify $W=\Sigma_{2n-1}$ with its action on $\4T$ permuting the 
coordinates, then 
	\[ \gamma\bigl(\diag(\lambda_1,\ldots,\lambda_{2n-1})\bigr)=
	\diag(\lambda_{2n-1}^{-1},\ldots,\lambda_1^{-1}), \] 
and $W_0\cong C_2\wr\Sigma_{n-1}$ is generated by the permutations 
$(i\,2n{-}i)$ and $(i\,j)(2n{-}i\,2n{-}j)$ for $i,j<n$.  So $C_{\4T}(W_0)$ 
is the group of all matrices $\diag(\lambda_1,\ldots,\lambda_{2n-1})$ such 
that $\lambda_i=\lambda_1$ for all $i\ne{}n$ and 
$\lambda_n=\lambda_1^{-(2n-2)}$, and $C_{\4T}(W_0)\cong\fqobar^\times$. 
Also, $\gamma$ inverts $C_{\4T}(W_0)$, so $\sigma(t)=t^{-q}$ for 
$t\in{}C_{\4T}(W_0)$. 
\end{proof}

Recall (Notation \ref{G-setup-X}\eqref{not9x}) that $\Aut(A,\calf)$ is the 
group of automorphisms of $A$ which extend to elements of $\Aut(S,\calf)$. 
The next result describes the structure of $\Aut(A,\calf)$ for a group $G$ 
in the situation of case~\ref{easy_case} or \ref{minus_case} of 
Hypotheses~\ref{G-hypoth-X}. Recall that $W_0$ acts faithfully on $A$ by 
Lemma~\ref{NG(T)}(a), and hence that $W_0\cong\Aut_N(A)=\Aut_{N_G(T)}(A)$ 
by Lemma~\ref{NG(T)}(b). It will be convenient to identify $W_0$ with this 
subgroup of $\Aut(A)$. Since each element of $\Aut(A,\calf)$ is fusion 
preserving, this group normalizes and hence acts on $W_0$, and 
$W_0\Aut(A,\calf)$ is a subgroup of $\Aut(A)$. 

For convenience, we set $\Aut_{\Aut(G)}(A)=\bigl\{\delta|_A\,\big|\,
\delta\in\Aut(G),~ \delta(A)=A \bigr\}$. 

\begin{Lem} \label{Hyp1-N(W0)}
Assume that $G$ and $(\4G,\sigma)$ satisfy Hypotheses 
\ref{G-hypoth-X}, case \ref{easy_case} or \ref{minus_case}. Assume also 
Notation \ref{G-setup-X}. 
\begin{enuma} 

\item $C_{W_0\Aut(A,\calf)}(W_0) \le W_0\Aut\scal(A)$.

\item $\Aut(A,\calf) \le \Aut\scal(A)\Aut_{\Aut(G)}(A)$,
with the exceptions \smallskip
\begin{itemize} 
\item $(G,p)\cong(\lie2E6(q),3)$, or
\item $(G,p)\cong(G_2(q),2)$ and $q_0\ne3$, or 
\item $(G,p)\cong(F_4(q),3)$ and $q_0\ne2$. 
\end{itemize}

\item In all cases, $\Aut(A,\calf)\cap\Aut\scal(A)\Aut_{\Aut(G)}(A)$ has 
index at most $2$ in $\Aut(A,\calf)$. 
\end{enuma}
\end{Lem}

\begin{proof} Recall that in Notation \ref{G-setup}\eqref{not3}, $V_0$, 
$\5\Sigma$, and $\5\Pi$ are defined when $\rho(\Pi)=\Pi$, and hence in case 
\ref{easy_case} of Hypotheses \ref{G-hypoth-X}. In case \ref{minus_case}, 
we defined $V_0=V$, $\5\Sigma=\Sigma$, and $\5\Pi=\Pi$ in Notation 
\ref{G-setup-X}\eqref{not5x}. So under the hypotheses of the lemma (and 
since $G$ is always a Chevalley group in case \ref{minus_case}), we have 
$V_0=V$ and $\5\Pi=\Pi$ if and only if $G$ is a Chevalley group. 


If $\alpha\in\Pi$ and $\alpha+\tau(\alpha)\in\Sigma$, then 
$\4T_{\alpha+\tau(\alpha)}\le\4T_\alpha\4T_{\tau(\alpha)}$ by Lemma 
\ref{theta-r}\eqref{hb.hc}.  Hence 
$\4T_{\5\alpha}=\4T_\alpha\4T_{\tau(\alpha)}$ (the maximal torus in 
$\4K_{\5\alpha}$) for each $\alpha\in\Pi$.  So by Lemma 
\ref{theta-r}\eqref{T=prod}, in all cases,
	\beqq T = C_{\4T}(\psi_q\gamma) 
	= \prod_{\5\alpha\in\5\Pi}C_{\4T_{\5\alpha}}(\psi_q\gamma) 
	\qquad\textup{and hence}\qquad
	A = \prod_{\5\alpha\in\5\Pi}A_{\5\alpha}~. \label{e:Tsplit} \eeqq

Set
	\[ \gee= \begin{cases} 
	+1 & \textup{if $q\equiv1$ (mod $p$)\quad (case \ref{easy_case})}\\
	-1 & \textup{if $q\equiv-1$ (mod $p$) and $p$ is odd\quad (case \ref{minus_case})}
	\end{cases}
	\qquad \textup{and}\qquad m=v_p(q-\gee)\,. \]
By assumption, $\gee=1$ if $G$ is a Steinberg group or if $p=2$, and $m>0$ 
in all cases. 

If $G$ is a Chevalley group, then $\tau=\gee\cdot\Id_V$, so $W_0=W$. Also, 
$\sigma=\gamma\psi_q$ acts on $\4T$ via $\sigma(t)=t^{\gee{}q}$, so 
$T=\{t\in\4T\,|\,t^{q-\gee}=1\}$, and $A=\{t\in\4T\,|\,t^{p^m}=1\}$. Thus 
for each $\alpha\in\Sigma$, $T_\alpha\cong C_{q-\gee}$ and $A_\alpha\cong 
C_{p^k}$. 

Now assume $G$ is a Steinberg group ($\tau\ne\Id$). 
For each $\5\alpha\in\5\Pi$, either 
\begin{itemize} 

\item $\5\alpha\cap\Pi=\{\alpha,\tau(\alpha)\}$ for some $\alpha\in\Pi$ 
such that $\alpha\ne\tau(\alpha)$, in which case $\sigma=\psi_q\gamma$ acts 
on $\4T_{\5\alpha}=\4T_\alpha\times{}\4T_{\tau(\alpha)}$ by sending $(a,b)$ 
to $(b^q,a^q)$, and so $C_{\4T_{\5\alpha}}(\psi_q\gamma)\cong{}C_{q^2-1}$; 
\quad or 

\item $\5\alpha=\{\alpha\}$ for some $\alpha\in\Pi$ such that 
$\alpha=\tau(\alpha)$, in which case $\psi_q\gamma$ acts on 
$\4T_{\5\alpha}=\4T_\alpha$ via $(a\mapsto{}a^q)$, and  
$C_{\4T_{\5\alpha}}(\psi_q\gamma)\cong{}C_{q-1}$.

\end{itemize}
Since $v_p(q^2-1)=m$ ($p$ odd) or $m+1$ ($p=2$), we have now shown that in 
all cases,
	\beqq A_{\5\alpha}\cong C_{p^m}\textup{ if $p$ is odd; } \qquad
	A_{\5\alpha}\cong \begin{cases} 
	C_{2^m} & \textup{if $p=2$ and $|\5\alpha|=1$}\\
	C_{2^{m+1}} & \textup{if $p=2$ and $|\5\alpha|\ge2$.} 
	\end{cases} \label{e:4.11a} \eeqq

\smallskip

\noindent\textbf{Step 1: } We first prove that 
	\beqq \varphi\in C_{W_0\Aut(A,\calf)}(W_0) \quad\implies\quad 
	\varphi(A_{\5\alpha})=A_{\5\alpha} \textup{ for all 
	$\5\alpha\in\5\Sigma$.} \label{e:4.11o} \eeqq
If $p$ is odd, then $A_{\5\alpha}=[w_{\5\alpha},A]$
by Lemma \ref{Kr,Hr2}(a), so \eqref{e:4.11o} is immediate.

Next assume that $p=2$, and also that $|\5\alpha|\le2$. Write 
$\varphi=w\circ\varphi_0$, where $w\in{}W_0$ and 
$\varphi_0\in\Aut(A,\calf)$. Then 
$\varphi_0(C_A(w_{\5\alpha}))=w^{-1}(C_A(w_{\5\alpha})) =C_A(w_{\5\beta})$, 
where $\5\beta=w^{-1}(\5\alpha)$.  By definition of $\Aut(A,\calf)$ 
(Notation \ref{G-setup-X}),  $\varphi_0=\4\varphi_0|_A$ for some 
$\4\varphi_0\in\Aut(S,\calf)$. Since $\4\varphi_0$ is fusion preserving, it 
sends $\foc(C_\calf(C_A(w_{\5\alpha})))$ onto
$\foc(C_\calf(C_A(w_{\5\beta})))$. Since these focal subgroups are 
$A_{\5\alpha}$ and $A_{\5\beta}$, respectively, by Lemma \ref{Kr,Hr2}(c), 
$\varphi(A_{\5\alpha})=w(A_{\5\beta})=A_{w(\5\beta)}=A_{\5\alpha}$ also in 
this case (the second equality by Lemma \ref{theta-r}\eqref{wa(hb)}).

It remains to consider the case where $p=2$ and $|\5\alpha|=3$, and thus 
where $G\cong\SU_{2n+1}(q)$ for some $n\ge1$. There is a subgroup 
$(H_1\times\cdots\times H_n)\rtimes\Sigma_n<G$ of odd index, where 
$H_i\cong\GU_2(q)$. Fix $S_i\in\syl2{H_i}$; then $S_i\cong\SD_{2^k}$ where 
$k=v_2(q^2-1)+1\ge4$. Let $A_i,Q_i<S_i$ denote the 
cyclic and quaternion subgroups of index $2$ in $S_i$. Then we can take 
$A=A_1\times\cdots\times A_n\cong (C_{2^{k-1}})^n$,
$N=(S_1\times\cdots\times S_n)\rtimes\Sigma_n$, and $S\in\syl2{N}$. 

There are exactly $n$ classes $\5\alpha_1,\ldots,\5\alpha_n\in\5\Sigma_+$ 
of order $3$, which we label so that $[w_{\5\alpha_i},A]\le A_i$ 
($[w_{\5\alpha_i},A]=A\cap Q_i$). Equivalently, these are chosen so that 
$w_{\5\alpha_i}$ acts on $A$ via conjugation by an element of 
$S_i{\sminus}A_i$. Let $\alpha^*_i\in\Sigma_+$ be the root in the class 
$\5\alpha_i$ which is the sum of the other two. 

Write $\varphi=w\circ\varphi_0$, where $w\in{}W_0$ and 
$\varphi_0\in\Aut(A,\calf)$, and let $\4\varphi_0\in\Aut(S,\calf)$ be such 
that $\varphi_0=\4\varphi_0|_A$. For each $1\le i\le n$, 
$\varphi_0(C_A(w_{\5\alpha_i}))=w^{-1}(C_A(w_{\5\alpha_i})) 
=C_A(w_{\5\alpha_{f(i)}})$, where $f\in\Sigma_n$ is such that 
$\5\alpha_{f(i)}=w^{-1}(\5\alpha_i)$.  Since $\4\varphi_0$ is fusion 
preserving, it sends $\foc(C_\calf(C_A(w_{\5\alpha_i})))$ onto 
$\foc(C_\calf(C_A(w_{\5\alpha_{f(i)}})))$. By Lemma \ref{Kr,Hr2}(c), 
$C_G(C_A(w_{\5\alpha_i}))=G\cap(\4T\4K_{\alpha^*_i})$, its commutator 
subgroup is $G\cap\4K_{\alpha^*_i}\cong\SL_2(q)$, and hence 
$\foc(C_\calf(C_A(w_{\5\alpha_i})))=Q_i$. Thus $\4\varphi_0(Q_i)=Q_{f(i)}$.

For each $i$, set $Q^*_i=\gen{Q_j\,|\,j\ne{}i}$. Then $C_G(Q^*_i)$ is the 
product of $G\cap\4K_{\5\alpha_i}\cong\SL_3(q)$ (Lemma \ref{Kr,Hr}) with 
$Z(Q^*_i)$. Thus $\4\varphi_0$ sends $\foc(C_\calf(Q^*_i))=S_i$ to 
$\foc(C_\calf(Q^*_{f(i)}))=S_{f(i)}$, and hence $\varphi_0(A_i)=A_{f(i)}$. 
So $\varphi(A_i)=w(A_{f(i)})=A_i$ for each $i$ where $A_i=A_{\5\alpha_i}$, 
and this finishes the proof of \eqref{e:4.11o}.

\smallskip

\noindent\textbf{Step 2: } We next prove point (a): that 
$C_{W_0\Aut(A,\calf)}(W_0) \le W_0\Aut\scal(A)$. 
Let $\varphi\in W_0\Aut(A,\calf)$ be an element 
which centralizes $\Aut_N(A)\cong N/A\cong{}W_0$. By \eqref{e:4.11o}, 
$\varphi(A_{\5\alpha})=A_{\5\alpha}$ for each $\5\alpha\in\5\Sigma$. 
Since $A_{\5\alpha}$ is cyclic for each $\5\alpha\in\5\Sigma_+$ by 
\eqref{e:4.11a}, $\varphi|_{A_{\5\alpha}}$ is multiplication by some unique 
$u_{\5\alpha}\in(\Z/q_{\5\alpha})^\times$, where 
$q_{\5\alpha}=|A_{\5\alpha}|$.  We must show that $u_{\5\alpha}$ is 
independent of $\5\alpha$.

Assume first that $\tau=\Id$. By \eqref{e:4.11a}, 
$|A_\alpha|=p^m$ for each $\alpha\in\Pi$.  Fix $\alpha_1,\alpha_2\in\Pi$ 
and $\beta\in\Sigma_+$ such that $\frac1k\beta=\frac1k\alpha_1+\alpha_2$, 
where either
\begin{itemize} 
\item $k=1$ and all three roots have the same length; or

\item $k\in\{2,3\}$ and 
$\norm\beta=\norm{\alpha_1}=\sqrt{k}\cdot\norm{\alpha_2}$.
\end{itemize}
The relation between the three roots is chosen so that $h_\beta(\lambda)= 
h_{\alpha_1}(\lambda)h_{\alpha_2}(\lambda)$ for all 
$\lambda\in\fqobar^\times$ by Lemma \ref{theta-r}\eqref{hb.hc}.  Hence 
$u_{\alpha_1}\equiv u_\beta\equiv u_{\alpha_2}$ (mod $p^m$) by 
\eqref{e:Tsplit}.  By the connectivity of the Dynkin diagram, the 
$u_\alpha$ for $\alpha\in\Pi$ are all equal, and $\varphi\in\Aut\scal(A)$. 

Now assume $|\tau|=2$; the argument is similar but slightly more 
complicated.  By assumption, $\gg$ is of type $A_n$, $D_n$, or $E_n$; i.e., 
all roots have the same length.  Set $m'=v_p(q^2-1)$; then 
$m'=m$ if $p$ is odd, and $m'=m+1$ if $p=2$. Fix 
$\alpha_1,\alpha_2\in\Pi$ such that $\alpha_1\ne\tau(\alpha_2)$ and 
$\beta\defeq\alpha_1+\alpha_2\in\Sigma_+$.  Choose 
$\lambda\in\4\F_{q_0}^\times$ of order $p^{m'}$.

If $\alpha_1\ne\tau(\alpha_1)$ and $\alpha_2\ne\tau(\alpha_2)$, then 
$|A_{\5\alpha_1}|=|A_{\5\alpha_2}|=p^{m'}$ by \eqref{e:4.11a}, and 
	\[ \5h_{\alpha_1}(\lambda)\5h_{\alpha_2}(\lambda) =
	h_{\alpha_1}(\lambda)h_{\tau(\alpha_1)}(\lambda^q)
	h_{\alpha_2}(\lambda)h_{\tau(\alpha_2)}(\lambda^q)
	= h_{\beta}(\lambda)h_{\tau(\beta)}(\lambda^q)
	= \5h_\beta(\lambda) \in A_{\5\beta} \,. \]
Hence 
	\[ \bigl(\5h_{\alpha_1}(\lambda)\5h_{\alpha_2}(\lambda)\bigr) 
	^{u_{\5\beta}} 
	= \varphi\bigl(\5h_{\alpha_1}(\lambda)\5h_{\alpha_2}(\lambda)\bigr)
	= \5h_{\alpha_1}(\lambda)^{u_{\5\alpha_1}} \cdot
	\5h_{\alpha_2}(\lambda)^{u_{\5\alpha_2}} \,, \]
and together with \eqref{e:Tsplit}, this proves that 
$u_{\5\alpha_1}\equiv u_{\5\beta}\equiv u_{\5\alpha_2}$ (mod $p^{m'}$).  

If $\tau(\alpha_i)=\alpha_i$ for $i=1,2$, then a similar argument shows 
that $u_{\5\alpha_1}\equiv u_{\5\beta}\equiv u_{\5\alpha_2}$ (mod 
$p^{m}$).  It remains to handle the case where 
$\alpha_1\ne\tau(\alpha_1)$ and $\alpha_2=\tau(\alpha_2)$.  In this case,  
$|A_{\5\alpha_1}|=p^{m'}$ and $|A_{\5\alpha_2}|=p^{m}$ by \eqref{e:4.11a}, 
and these groups are generated by $\5h_{\alpha_1}(\lambda)=
h_{\alpha_1}(\lambda)h_{\tau(\alpha_1)}(\lambda^q)$ and 
$h_{\alpha_2}(\lambda^{q+1})$, respectively.  Then 
	\[ \5h_{\alpha_1}(\lambda)\5h_{\alpha_2}(\lambda^{q+1}) =
	h_{\alpha_1}(\lambda)h_{\tau(\alpha_1)}(\lambda^q)
	h_{\alpha_2}(\lambda^{q+1})
	= h_{\beta}(\lambda)h_{\tau(\beta)}(\lambda^q) = \5h_\beta(\lambda)
	\in A_{\5\beta} \,, \]
so 
	\[ \bigl(\5h_{\alpha_1}(\lambda)\5h_{\alpha_2}(\lambda^{q+1})\bigr) 
	^{u_{\5\beta}} = 
	\varphi\bigl(\5h_{\alpha_1}(\lambda)\5h_{\alpha_2}(\lambda^{q+1})\bigr) 
	= \5h_{\alpha_1}(\lambda)^{u_{\5\alpha_1}} \cdot
	h_{\alpha_2}(\lambda^{q+1}) ^{u_{\5\alpha_2}} \,, \]
and $u_{\5\alpha_1}\equiv u_{\5\beta}\equiv u_{\5\alpha_2}$ (mod $p^{m}$) 
by \eqref{e:Tsplit} again.

Since the Dynkin diagram is connected, and since the subdiagram of nodes in 
free orbits in the quotient diagram is also connected, this shows that the 
$u_{\5\alpha}$ are all congruent for $\5\alpha\in\5\Pi$ (modulo $p^m$ or 
$p^{m'}$, depending on where they are defined), and hence that 
$\varphi\in\Aut\scal(A)$. 

\smallskip

\noindent\textbf{Step 3: } Consider the subset 
$W_{\5\Pi}=\{w_{\5\alpha}\,|\,\5\alpha\in\5\Pi\}$. We need to study the 
subgroup $N_{W_0\Aut(A,\calf)}(W_{\5\Pi})$: the group of elements of 
$W_0\Aut(A,\calf)$ which permute the set $W_{\5\Pi}$.  Note that 
$W_0=\gen{W_{\5\Pi}}$ (see, e.g., \cite[Proposition 13.1.2]{Carter}, and 
recall that $W_0=W$ and $\5\Pi=\Pi$ in case \ref{minus_case}). 
We first show that 
	\beqq \Aut(A,\calf) \le W_0N_{W_0\Aut(A,\calf)}(W_{\5\Pi})\,. 
	\label{e:4.11b} \eeqq

Write $\5\Pi=\{\5\alpha_1,\dots,\5\alpha_k\}$, ordered so that for each 
$2\le{}i\le{}k$, $\5\alpha_i$ is orthogonal to all but one of the 
$\5\alpha_j$ for $j<i$.  Here, $\5\alpha_i\perp\5\alpha_j$ means orthogonal 
as vectors in $V_0$. Thus $w_{\5\alpha_i}$ commutes with all but one of the 
$w_{\5\alpha_j}$ for $j<i$.  By inspection of the Dynkin diagram of $\gg$ 
(or the quotient of that diagram by $\tau$), this is always possible.

Fix $\varphi\in\Aut(A,\calf)$. In particular, $\varphi$ normalizes $W_0$ 
(recall that we identify $W_0=\Aut_{W_0}(A)$) since $\varphi$ is fusion 
preserving. (Recall that $\Aut_G(A)=\Aut_{W_0}(A)$ by Lemma 
\ref{NG(T)}(b).) We must show that some element of $\varphi{}W_0$ 
normalizes the set $W_{\5\Pi}$. 

By definition of $\Aut(A,\calf)$ (Notation \ref{G-setup-X}), 
$\varphi=\4\varphi|_A$ for some $\4\varphi \in \Aut(S,\calf)$. Since 
$\4\varphi$ is fusion preserving, $\varphi$ normalizes 
$\autf(A)=\Aut_G(A)$, where $\Aut_G(A)\cong N/A\cong W_0$ since $C_N(A)=A$ 
by Lemma \ref{NG(T)}(a). Thus there is a unique automorphism 
$\5\varphi\in\Aut(W_0)$ such that $\5\varphi(w)=\varphi\circ w\circ 
\varphi^{-1}$ for each $w\in{}W_0$. 

For each $i$, since $|\5\varphi(w_{\5\alpha_i})|=2$ and 
$[\5\varphi(w_{\5\alpha_i}),A]\cong[w_{\5\alpha_i},A]$ is 
cyclic, $\5\varphi(w_{\5\alpha_i})=w_{\5\alpha'_i}$ for some 
$\5\alpha'_i\in\5\Sigma$ by Lemma \ref{Kr,Hr2}(b), where $\5\alpha'_i$ is 
uniquely determined only up to sign. For $i\ne{}j$, 
	\[ \5\alpha_i\perp\5\alpha_j ~\iff~ 
	[w_{\5\alpha_i},w_{\5\alpha_j}]=1 ~\iff~
	[\5\varphi(w_{\5\alpha_i}),\5\varphi(w_{\5\alpha_j})]=1 ~\iff~
	\5\alpha'_i\perp\5\alpha'_j\,. \]
So using the assumption about 
orthogonality, we can choose successively 
$\5\alpha'_1,\5\alpha'_2,\ldots,\5\alpha'_k$ so that 
$\5\varphi(w_{\5\alpha_i})=w_{\5\alpha'_i}$ for each $i$, and 
$\gen{\5\alpha'_i,\5\alpha'_j}\le0$ for $i\ne{}j$.  

For each $i\ne{}j$, since 
$|w_{\5\alpha_i}w_{\5\alpha_j}|=|w_{\5\alpha'_i}w_{\5\alpha'_j}|$, the 
angle (in $V_0$) between $\5\alpha_i$ and $\5\alpha_j$ is equal to that 
between $\5\alpha'_i$ and $\5\alpha'_j$ (by assumption, both angles are 
between $\pi/2$ and $\pi$). The roots $\5\alpha'_i$ for $1\le i\le k$ 
thus generate $\5\Sigma$ as a root system on $V_0$ with Weyl group $W_0$, 
and hence are the fundamental roots for another Weyl chamber for 
$\5\Sigma$. (Recall that $\5\Sigma=\Sigma$, $V_0=V$, and $W_0=W$ in case 
\ref{minus_case}.) Since $W_0$ permutes the Weyl chambers transitively 
\cite[\S\,VI.1.5, Theorem 2(i)]{Bourb4-6}, there is $w\in{}W_0$ which sends 
the set $\{w_{\5\alpha_i}\}$ onto $\{\5\varphi(w_{\5\alpha_i})\}$. Thus 
$c_w^{-1}\circ\varphi\in N_{W_0\Aut(A,\calf)}(W_{\5\Pi})$, so $\varphi\in 
W_0N_{W_0\Aut(A,\calf)}(W_{\5\Pi})$, and this proves \eqref{e:4.11b}. 

\smallskip

\noindent\textbf{Step 4: } Set $\Aut_{W_0\Aut(A,\calf)}(W_{\5\Pi})= 
N_{W_0\Aut(A,\calf)}(W_{\5\Pi})\big/ C_{W_0\Aut(A,\calf)}(W_{\5\Pi})$: the 
group of permutations of the set $W_{\5\Pi}$ which are induced by elements 
of $W_0\Aut(A,\calf)$. By (a) (Step 2) and \eqref{e:4.11b}, and since 
$W_0=\gen{W_{\5\Pi}}$, there is a surjection 
	\beqq 
	\Aut_{W_0\Aut(A,\calf)}(W_{\5\Pi}) \Onto5{\textup{onto}} 
	\frac{W_0N_{W_0\Aut(A,\calf)}(W_{\5\Pi})}
	{W_0C_{W_0\Aut(A,\calf)}(W_{\5\Pi})} 
	= \frac{W_0\Aut(A,\calf)}{W_0\Aut\scal(A)} \,.
	\label{e:4.11d} \eeqq
To finish the proof of the lemma, we must show that each element of 
$\Aut_{W_0\Aut(A,\calf)}(W_{\5\Pi})$ is represented by an element of 
$\Aut_{\Aut(G)}(A)$ (i.e., the restriction of an automorphism of $G$), with 
the exceptions listed in point (b). 

In the proof of Step 3, we saw that each element of 
$\Aut_{W_0\Aut(A,\calf)}(W_{\5\Pi})$ preserves angles between the 
corresponding elements of $\5\Pi$, and hence induces an automorphism of the 
Coxeter diagram for $(V_0,\5\Sigma)$ (i.e., the Dynkin diagram without 
orientation on the edges).

\smallskip

\noindent\textbf{Case 1: } Assume $G=\gg(q)$ is a Chevalley group. The 
automorphisms of the Coxeter diagram of $\gg$ are well known, and we have 
	\beqq 
	\bigl|\Aut_{W_0\Aut(A,\calf)}(W_{\5\Pi})\bigr| 
	\le \begin{cases} 
	6 & \textup{if $\gg=D_4$} \\
	2 & \textup{if $\gg=A_n$ ($n\ge2$), $D_n$ ($n\ge5$), 
	$E_6$, $B_2$, $G_2$, or $F_4$} \\
	1 & \textup{otherwise.}
	\end{cases} \label{e:4.11c} \eeqq
In case \ref{easy_case} (i.e., when the setup is standard), all of these 
automorphisms are realized by restrictions of graph automorphisms in 
$\Gamma_G$ (see \cite[\S\S\,12.2--4]{Carter}), except possibly when $G\cong 
B_2(q)$, $G_2(q)$, or $F_4(q)$. In case \ref{minus_case}, with the 
same three exceptions, each such automorphism is realized by some graph 
automorphism $\varphi\in\Gamma_{\4G}$, and $\varphi|_{\4T}$ commutes with 
$\sigma|_{\4T}\in Z(\Aut(\4T))$. Hence by 
Lemma \ref{4T->4G}, $\varphi|_T$ extends to an automorphism of $G$ whose 
restriction to $A$ induces the given symmetry of the Coxeter diagram. 
Together with \eqref{e:4.11d}, this proves the lemma 
for Chevalley groups, with the above exceptions. 

If $G\cong B_2(q)$ or $F_4(q)$ and $p\ne2$, then 
$\bigl|\Aut_{W_0\Aut(A,\calf)}(W_{\5\Pi})\bigr|=2$, and the nontrivial 
element is represented by an element of $\Aut_{\Gamma_G}(A)$ exactly when 
$q_0=2$. This proves the lemma in these cases, and a similar argument holds 
when $G\cong G_2(q)$ and $p\ne3$. 

It remains to check the cases where $(G,p)\cong(B_2(q),2)$, $(G_2(q),3)$, 
or $(F_4(q),2)$. We claim that $\Aut_{W_0\Aut(A,\calf)}(W_{\5\Pi})=1$ in 
these three cases; then the three groups in \eqref{e:4.11d} are trivial, 
and so  $\Aut(A,\calf)\le W_0\Aut\scal(A)$. If $(\gg,p)=(B_2,2)$ or 
$(G_2,3)$, then with the help of Lemma 
\ref{theta-r}(\ref{hb.hc},\ref{T=prod}), one shows that the subgroups 
$\Omega_1(A_\alpha)$ are all equal for $\alpha$ a short root, and are all 
distinct for the distinct (positive) long roots. More precisely, of the 
$p+1$ subgroups of order $p$ in $\Omega_1(A)\cong C_p^2$, one is equal to 
$A_\alpha$ when $\alpha$ is any of the short roots in $\Sigma_+$, while 
each of the other $p$ is equal to $A_\alpha$ for one distinct long root 
$\alpha$. Since $\Omega_1(A_\alpha)=\Omega_1([w_\alpha,A])$ for each 
$\alpha$, no element of $N_{W_0\Aut(A,\calf)}(W_{\5\Pi})$ can exchange the 
long and short roots, so $\Aut_{W_0\Aut(A,\calf)}(W_{\5\Pi})=1$. 

Now assume $(\gg,p)=(F_4,2)$. Let $\alpha,\beta\in\Pi$ be such that 
$\alpha$ is long, $\beta$ is short, and $\alpha\not\perp\beta$. Then 
$\alpha$ and $\beta$ generate a root system of type $B_2$, and by the 
argument in the last paragraph, no element of 
$N_{W_0\Aut(A,\calf)}(W_{\5\Pi})$ can exchange them. Thus no element in 
$N_{W_0\Aut(A,\calf)}(W_{\5\Pi})$ can exchange the long and short roots in 
$\gg$, so again $\Aut_{W_0\Aut(A,\calf)}(W_{\5\Pi})=1$.

\smallskip

\noindent\textbf{Case 2: } Assume $G$ is a Steinberg group. In particular, 
we are in case \ref{easy_case}. The Coxeter diagram for the root system 
$(V_0,\5\Sigma)$ has type $B_n$, $C_n$, or $F_4$ (recall that we excluded 
the triality groups $\lie3D4(q)$ in Hypotheses \ref{G-hypoth-X}), and hence 
has a nontrivial automorphism only when it has type $B_2$ or $F_4$. It thus 
suffices to consider the groups $G=\lie2A3(q)$, $\lie2A4(q)$, and 
$\lie2E6(q)$. 

For these groups, the elements $\5h_\alpha(\lambda)$ for 
$\lambda\in\F_q^\times$, and hence the $(q-1)$-torsion in the subgroups 
$T_{\5\alpha}$ for $\5\alpha\in\5\Sigma_+$, have relations similar to those 
among the corresponding subgroups of $T$ when $G=B_2(q)$ or $F_4(q)$. This 
follows from Lemma \ref{l:CW(t)}(a): if $\lambda\in\F_q^\times$ is a 
generator, then $\Phi_\lambda$ restricts to an isomorphism from 
$C_{\Z\Sigma^\vee}(\tau)/(q-1)$ to the $(q-1)$-torsion in $T$, and the 
elements in $\5\Pi$ can be identified in a natural way with a basis for 
$C_{\Z\Sigma^\vee}(\tau)$. Hence when $p=2$, certain subgroups 
$\Omega_1(A_{\5\alpha})$ are equal for distinct $\5\alpha\in\5\Sigma_+$, 
proving that no element in $N_{W_0\Aut(A,\calf)}(W_{\5\Pi})$ can exchange 
the two classes of roots. Thus the same argument as that used in Case 1 
when $(G,p)=(B_2(q),2)$ or $(F_4(q),2)$ applies to prove that 
$N_{W_0\Aut(A,\calf)}(W_{\5\Pi})=\Aut\scal(A)$ in these cases. 

Since $p\big||W_0|$ by Hypotheses \ref{G-hypoth-X}(I), we 
are left only with the case where $p=3$ and $G=\lie2E6(q)$ for some 
$q\equiv1$ (mod $3$). Then $(V_0,\5\Sigma)$ is the root system of $F_4$, 
so $\Aut(A,\calf)\cap W_0\Aut\scal(A)$ has index at most $2$ in 
$\Aut(A,\calf)$ by \eqref{e:4.11d} and \eqref{e:4.11c}. Thus (c) holds in 
this case. (In fact, the fusion system of $G$ is isomorphic to that 
of $F_4(q)$ by \cite[Example 4.4]{BMO1}, and does have an 
``exotic'' graph automorphism.) 
\end{proof}

We now look at groups which satisfy any of the cases 
\ref{easy_case}, \ref{minus_case}, or \ref{messy_case} in Hypotheses 
\ref{G-hypoth-X}. Recall that $\4\kappa_G=\mu_G\circ\kappa_G\: \Out(G) 
\Right2{}\Out(S,\calf)$.

\begin{Lem} \label{Aut-diag}
Assume Hypotheses \ref{G-hypoth-X} and Notation 
\ref{G-setup-X}. Then each $\varphi\in\Aut\dg(S,\calf)$ 
is the restriction of a diagonal automorphism of $G$.  More precisely, 
$\4\kappa_G$ restricts to an epimorphism from $\Outdiag(G)$ onto 
$\Out\dg(S,\calf)$ whose kernel is the $p'$-torsion subgroup. 
Also, $C_A(W_0)=O_p(Z(G))$.
\end{Lem}

\begin{proof} 
In general, whenever $H$ is a group and $B\nsg{}H$ is a normal abelian 
subgroup, we let $\Aut\dg(H,B)$ be the group of all $\varphi\in\Aut(H)$ 
such that $\varphi|_B=\Id_B$ and $[\varphi,H]\le{}B$, and let 
$\Out\dg(H,B)$ be the image of $\Aut\dg(H,B)$ in $\Out(H)$.  There is a 
natural isomorphism 
$\Aut\dg(H,B)/\Aut_B(H)\RIGHT2{\eta_{H,B}}{\cong}H^1(H/B;B)$ (cf. 
\cite[2.8.7]{Sz1}), and hence $H^1(H/B;B)$ surjects onto $\Out\dg(H,B)$. If 
$B$ is centric in $H$ (if $C_H(B)=B$), then $\Out\dg(H,B)\cong 
H^1(H/B;B)$ since $\Aut_B(H)=\Inn(H)\cap\Aut\dg(H,B)$. 

In particular, $\Out\dg(S,A)$ is a $p$-group since $H^1(S/A;A)$ is a 
$p$-group. Also, $C_S(A)=A$ by Lemma \ref{NG(T)}(a) (or by assumption 
in case \ref{messy_case}), and hence we have 
$\Out\dg(S,A)\cong\Aut\dg(S,A)/\Aut_A(S)$. So $\Aut\dg(S,A)$ is a 
$p$-group, and its subgroup $\Aut\dg(S,\calf)$ is a $p$-group. It follows 
that
	\[ \Aut\dg(S,\calf)\cap\Aut_G(S) = \Aut\dg(S,\calf)\cap\Inn(S) = 
	\Aut_A(S)\,, \]
and thus $\Out\dg(S,\calf)\cong\Aut\dg(S,\calf)/\Aut_A(S)$. 

Since $\Outdiag(G)=\Out_{\4T}(G)$ by Proposition \ref{p:Autdiag(G)}(c), 
$\4\kappa_G(\Outdiag(G))\le\Out\dg(S,\calf)$, and in particular, 
$\4\kappa_G$ sends all torsion prime to $p$ in $\Outdiag(G)$ to the 
identity.  It remains to show that it sends $O_p(\Outdiag(G))$ 
isomorphically to $\Out\dg(S,\calf)$.  

Consider the following commutative diagram of automorphism groups and 
cohomology groups:
	\beqq \vcenter{\xymatrix@C=40pt{
	\Out\dg(S,\calf) \cong
	\Aut\dg(S,\calf)/\Aut_A(S) \ \ar[r]^-{\chi} \ar@<14mm>[d]^{\incl} 
	& H^1(\Aut_{G}(A);A) \ar[d]^{\rho_2} \\
	\Out\dg(S,A) \cong
	\Aut\dg(S,A)/\Aut_A(S) \ar[r]^-{\eta_{S,A}}_-{\cong} & 
	H^1(\Aut_S(A);A) \rlap{\,.}
	}} \label{e:4.12a0} \eeqq
Here, $\rho_2$ is induced by restriction, and is injective by 
\cite[Theorem XII.10.1]{CE} and since $\Aut_S(A)\in\sylp{\Aut_G(A)}$ (since 
$A\nsg S\in\sylp{G}$). 
For each $\omega\in\Aut\dg(S,\calf)$, since $\omega$ is fusion preserving, 
$\eta_{S,A}([\omega])\in{}H^1(\Aut_S(A);A)$ is stable with respect to 
$\Aut_G(A)$-fusion, 
and hence by \cite[Theorem XII.10.1]{CE} is the restriction of a unique 
element $\chi([\omega])\in{}H^1(\Aut_{G}(A);A)$.

The rest of the proof splits into two parts, depending on which of cases 
\ref{easy_case}, \ref{minus_case}, or \ref{messy_case} in Hypotheses 
\ref{G-hypoth-X} holds. Recall that $\autf(A)=\Aut_G(A)=\Aut_{W_0}(A)$: the 
second equality by Lemma \ref{NG(T)}(b) in cases \ref{easy_case} or 
\ref{minus_case}, or by assumption in case \ref{messy_case}. 

\smallskip

\noindent\textbf{Cases \ref{minus_case} and \ref{messy_case}: } We show 
that in these cases, $\Outdiag(G)$, $\Out\dg(S,\calf)$, $Z(G)$, and 
$C_A(W_0)$ all have order prime to $p$. Recall that 
$p$ is odd in both cases. By hypothesis in case \ref{messy_case}, and since 
$\gamma|_{\4T}\in{}O_{p'}(W_0)$ inverts $\4T$ in case \ref{minus_case}, 
$C_A(O_{p'}(W_0))=1$.  In particular, $C_A(W_0)=1$. Since $Z(G)\le Z(\4G)$ 
by Proposition \ref{p:Autdiag(G)}(a), and $Z(\4G)\le\4T$ by 
Lemma \ref{theta-r}\eqref{CG(T)=T}, $Z(G)\le G\cap{}C_{\4T}(W)\le C_T(W_0)$, so 
$O_p(Z(G))\le C_A(W_0)=1$. This proves the last statement.

Now, $O_p(\Outdiag(G))=1$ since $\Outdiag(G)\cong Z(G)$ (see \cite[Theorem 
2.5.12(c)]{GLS3}) and $O_p(Z(G))=1$. 
Also, 
	\[ H^1(\Aut_{G}(A);A) = H^1(\Aut_{W_0}(A);A)\cong 
	H^1(\Aut_{W_0}(A)/\Aut_{O_{p'}(W_0)}(A);C_A(O_{p'}(W_0)))=0 \]
since $A$ is a $p$-group and $C_A(O_{p'}(W_0))=1$. Hence 
$\Out\dg(S,\calf)=1$ by diagram \eqref{e:4.12a0}.

\smallskip

\noindent\textbf{Case \ref{easy_case}: } Since $C_W(A)=1$ by Lemma 
\ref{NG(T)}(a) (and since $\Aut_G(A)=\Aut_{W_0}(A)$), we can identify 
$H^1(\Aut_G(A);A)=H^1(W_0;A)$. 
Consider the following commutative diagram of automorphism groups and 
cohomology groups
	\beqq \vcenter{\xymatrix@C=40pt{
	O_p(\Outdiag(G)) \ar[r]^-{R} \ar[dd]_{\4\kappa_G} & 
	O_p(\Out\dg(N_G(T),T)) \ar[r]^-{\eta_{N(T),T}}_-{\cong} 
	\ar[d]_{\cong}^-{\sigma_1} & 
	H^1(W_0;T)\ploc \ar[d]_{\cong}^-{\sigma_2} \\
	& \Out\dg(N,A) \ar[r]^-{\eta_{N,A}}_-{\cong} \ar[d]^{\rho_1} 
	& H^1(W_0;A) \ar[d]^{\rho_2} \\
	\Out\dg(S,\calf) \ar[r]^-{\incl} \ar[ur]^{\chi_0} 
	\ar@/^3pc/[urr]^(0.3){\chi}
	& \Out\dg(S,A) \ar[r]^-{\eta_{S,A}}_-{\cong} & 
	H^1(S/A;A) 
	}} \label{e:4.12a} \eeqq
where $R$ is induced by restriction to $N_G(T)$.  By Lemma \ref{NG(T)}(a), 
$T$ is centric in $N_G(T)$ and $A$ is centric in $N$, so the three $\eta$'s 
are well defined and isomorphisms (i.e., 
$\Out\dg(N,A)=\Aut\dg(N,A)/\Aut_A(N)$, etc.). The maps $\sigma_i$ are 
induced by dividing out by $O_{p'}(T)$, and are isomorphisms since 
$A=O_p(T)$.  The maps $\rho_i$ are induced by restriction, and are 
injective since $S/A\in\sylp{W_0}$ (see \cite[Theorem XII.10.1]{CE}).  


Consider the short exact sequence
	\[ 1 \Right2{} T \Right4{} \4T \Right4{\Psi} \4T \Right2{} 1, \]
where $\Psi(t)=t^{-1}\cdot\gamma\psi_q(t)=t^{-1}\gamma(t^q)$ for $t\in\4T$.  
Let 
	\beqq 1 \Right2{} C_T(W_0) \Right3{} C_{\4T}(W_0) 
	\Right3{\Psi_*} C_{\4T}(W_0) \Right3{\delta} H^1(W_0;T) 
	\Right3{\theta} H^1(W_0;\4T) \label{e:4.12b} \eeqq
be the induced cohomology exact sequence for the $W_0$-action, and recall 
that $H^1(W_0;A)\cong H^1(W_0;T)_{(p)}$ by \eqref{e:4.12a}. We claim that 
\begin{enumerate}[label*=(\arabic*) ,leftmargin=10mm,itemsep=6pt]
\setcounter{enumi}{\theequation}
\item\label{e:4.12i} $|O_p(\Outdiag(G))|=|\Im(\delta)_{(p)}|=|O_p(Z(G))|
=|C_A(W_0)|$; 
\item\label{e:4.12ii} $R$ is injective; and 
\item\label{e:4.12iii} $\chi(\Out\dg(S,\calf))\le\Ker(\theta)$. 
\end{enumerate}\addtocounter{equation}{3}
These three points will be shown below.  It then follows from the 
commutativity of diagram \eqref{e:4.12a} (and since 
$\Im(\delta)=\Ker(\theta)$) that $\4\kappa_G$ sends 
$O_p(\Outdiag(G))$ isomorphically onto $\Out\dg(S,\calf)$.

\smallskip

\noindent\textbf{Proof of \ref{e:4.12i} and \ref{e:4.12ii}: } Assume first 
that $\gamma\ne\Id$ and $\gg=\SL_{2n-1}$ (some $n\ge1$).  Thus 
$G\cong{}\SU_{2n-1}(q)$.  By \cite[3.4]{Steinberg-aut}, $\Outdiag(G)$ 
and $Z(G)$ are cyclic of order $(q+1,2n-1)$, and hence have no 
$p$-torsion (recall $p|(q-1)$).  By Lemma \ref{CG(T)-2}(b), 
$C_{\4T}(W_0)\cong\fqobar^\times$, and $\sigma(u)=u^{-q}$ for 
$u\in{}C_{\4T}(W_0)$. Thus $\Psi_*(u)=u^{-1}\sigma(u)=u^{-1-q}$ for 
$u\in{}C_{\4T}(W_0)$, so $\Psi_*$ is onto, and $\Im(\delta)=1\cong 
O_p(\Outdiag(G))$ in this case. Also, $C_T(W_0)=\Ker(\Psi_*)$ has 
order $q+1$, so $C_A(W_0)=O_p(C_T(W_0))=1$. 

Now assume $\gamma=\Id$ or $\gg\ne{}\SL_{2n-1}$.  
By Lemma \ref{CG(T)-2}, in all such cases,
	\beqq C_{\4T}(W_0) = C_{\4T}(W) = Z(\4G)
	\qquad\textup{and}\qquad
	C_T(W_0)=Z(G)~. \label{e:4.12c} \eeqq
In particular, these groups are all finite, and hence 
$|\Im(\delta)|=|Z(G)|$ by the exactness of \eqref{e:4.12b}.  By 
\cite[Theorem 2.5.12(c)]{GLS3}, $\Outdiag(G)\cong{}Z(G)$ in all cases, and 
hence $|\Outdiag(G)|=|\Im(\delta)|$.  

If $[\varphi]\in\Ker(R)$, then we can assume that it is the class of 
$\varphi\in\Aut_{\4T}(G)$. Thus $\varphi=c_x$ for some $x\in N_{\4T}(G)$, 
and $\varphi|_{N_G(T)}=c_y$ for some $y\in N_G(T)$ which centralizes $A$. 
Then $y\in{}C_G(A)=T$ by Lemma \ref{NG(T)}(a), and upon replacing $\varphi$ 
by $c_y^{-1}\circ\varphi$ and $x$ by $y^{-1}x$ (without changing the class 
$[\varphi]$), we can arrange that $\varphi|_{N_G(T)}=\Id$. Then 
$x\in{}C_{\4T}(W_0)$ since it centralizes $N_G(T)$ (and since 
$N_G(T)/T\cong{}W_0$ by Lemma \ref{NG(T)}(b)), so $x\in{}Z(\4G)$ by 
\eqref{e:4.12c}, and hence $\varphi=\Id_G$.  Thus $R$ is injective.

\smallskip

\newcommand{\muu}{\underbar{c}}

\noindent\textbf{Proof of \ref{e:4.12iii}: }  
Fix $\varphi\in\Aut\dg(S,\calf)$. Choose $\4\varphi\in\Aut\dg(N,A)$ 
such that $\4\varphi|_S=\varphi$ (i.e., 
such that $[\4\varphi]=\chi_0([\varphi])$ in diagram \eqref{e:4.12a}).
Recall that $W_0\cong N/A$ by Lemma \ref{NG(T)}(b). Let 
$\muu\:W_0\cong{}N/A\Right2{}A$ be such that $\4\varphi(g)=\muu(gA){\cdot}g$ 
for each $g\in{}N$; thus $\eta_{N,A}([\varphi])=[\muu]$.  We must show that 
$\theta([\muu])=1$:  that this is a consequence of $\varphi$ being fusion 
preserving.

For each $\5\alpha\in\5\Pi$, set $u_{\5\alpha}=\muu(w_{\5\alpha})$.  
Thus for $g\in{}N$, $\4\varphi(g)=u_{\5\alpha}g$ if $g\in{}w_{\5\alpha}$ 
(as a coset of $A$ 
in $N$).  Since $w_{\5\alpha}^2=1$, $g^2=\4\varphi(g^2)=(u_{\5\alpha}g)^2$, 
and hence  $w_{\5\alpha}(u_{\5\alpha})=u_{\5\alpha}^{-1}$. We claim that 
$u_{\5\alpha}\in{}A_{\5\alpha}=A\cap{}\4K_{\5\alpha}$ for each 
$\5\alpha\in\5\Pi$.  
\begin{itemize} 

\item If $p$ is odd, then $u_{\5\alpha}\in A_{\5\alpha}$, since 
$A_{\5\alpha}=\{a\in{}A\,|\,w_{\5\alpha}(a)=a^{-1}\}$ by Lemma 
\ref{theta-r}\eqref{wa(hb)}.  

\item If $p=2$, $w_{\5\alpha}\in{}S/A$, and $|\5\alpha|\le2$, choose 
$g_{\5\alpha}\in{}S\cap\4K_{\5\alpha}$ such that 
$w_{\5\alpha}=g_{\5\alpha}A$. (For example, if we set 
$g=\prod_{\alpha\in\5\alpha}n_\alpha(1)$ (see Notation 
\ref{G-setup}\eqref{not2}), then $g\in N_G(T)$ represents the class 
$w_{\5\alpha}\in W_0$, and is $T$-conjugate to an element of 
$S\cap\4K_{\5\alpha}$.) By Lemma \ref{Kr,Hr2}(c), 
$C_G(C_A(w_{\5\alpha}))=G\cap\4T\4K_{\5\alpha}$, where 
$G\cap\4K_{\5\alpha}\cong\SL_2(q)$ or $\SL_2(q^2)$ by Lemma \ref{Kr,Hr}. 
Hence 
	\[ \qquad \foc(C_\calf(C_A(w_{\5\alpha}))) 
	= \foc(C_G(C_A(w_{\5\alpha}))) 
	= S\cap[G\cap\4T\4K_{\5\alpha},G\cap\4T\4K_{\5\alpha}]
	= S\cap\4K_{\5\alpha} \]
(see the remarks before Lemma \ref{Kr,Hr2}), 
and $g_{\5\alpha}$ lies in this subgroup. Since $\varphi$ is fusion 
preserving, 
$\varphi(g_{\5\alpha})\in\foc(C_\calf(C_A(w_{\5\alpha})))$.  
By Lemma \ref{Kr,Hr2}(c) again,
	\[ u_{\5\alpha} = \varphi(g_{\5\alpha}) \cdot g_{\5\alpha}^{-1} \in
	A\cap\foc(C_\calf(C_A(w_{\5\alpha}))) = A_{\5\alpha}\,. \]

\item If $p=2$, $w_{\5\alpha}\in{}S/A$, and 
$\5\alpha=\{\alpha,\tau(\alpha),\alpha^*\}$ where 
$\alpha^*=\alpha+\tau(\alpha)$, then $w_{\5\alpha}=w_{\alpha^*}$. Choose 
$g_{\5\alpha}\in{}S\cap{}\4K_{\alpha^*}$ such that 
$g_{\5\alpha}A=w_{\5\alpha}\in N/A$. (For example, there is such a $g_{\5\alpha}$ 
which is $T$-conjugate to $n_{\alpha^*}(1)$.)  By Lemma \ref{Kr,Hr2}(c), 
$C_G(C_A(w_{\5\alpha}))=G\cap\4T\4K_{\alpha^*}$, 
$G\cap\4K_{\alpha^*}\cong\SL_2(q)$, and hence 
$g_{\5\alpha}\in\foc(C_\calf(C_A(w_{\5\alpha})))$. So 
$\varphi(g_{\5\alpha})\in\foc(C_\calf(C_A(w_{\5\alpha})))$ since 
$\varphi|_S$ is fusion preserving.  By Lemma \ref{Kr,Hr2}(c),
	\[ u_{\5\alpha} = \varphi(g_{\5\alpha}) \cdot g_{\5\alpha}^{-1} \in
	A\cap \foc(C_\calf(C_A(w_{\5\alpha}))) 
	= A \cap \4K_{\alpha^*} \le A_{\5\alpha}\,. \]

\item If $p=2$ and $w_{\5\alpha}\notin{}S/A\in\syl2{W_0}$, then it is 
$W_0$-conjugate to some other reflection $w_{\5\beta}\in{}S/A$ (for 
$\5\beta\in\5\Sigma_+$), $\muu(w_{\5\beta})\in{}A_{\5\beta}$ by the 
above argument, and hence 
$u_{\5\alpha}=\muu(w_{\5\alpha})\in{}A_{\5\alpha}$.
\end{itemize}


Consider the homomorphism
	\[ \Phi = (\Phi_\alpha)_{\alpha\in\Pi} \: \4T \Right5{} 
	\prod_{\alpha\in\Pi} \4T_\alpha 
	\qquad\textup{where}\qquad
	\Phi_\alpha(t)=t^{-1}w_\alpha(t) \quad
	\textup{$\forall~ t\in\4T,~ \alpha\in\Pi$.} \]
Since $W=\gen{w_\alpha\,|\,\alpha\in\Pi}$, we have 
$\Ker(\Phi)=C_{\4T}(W)=Z(\4G)$ is finite (Proposition \ref{p:CG(T)}). Thus 
$\Phi$ is (isomorphic to) a homomorphism from $(\fqobar^\times)^r$ to 
itself with finite kernel (where $r=|\Pi|$), and any such homomorphism is 
surjective since $\fqobar^\times$ has no subgroups of finite index. 

Choose elements $v_\alpha\in\4T_\alpha$ for $\alpha\in\Pi$ as follows.
\begin{itemize} 
\item If $\5\alpha=\{\alpha\}$ where $\tau(\alpha)=\alpha$, we set
$v_\alpha=u_{\5\alpha}$. 

\item If $\5\alpha=\{\alpha,\tau(\alpha)\}$, where 
$\alpha\perp\tau(\alpha)$, then 
$\4T_{\5\alpha}=\4T_\alpha\times\4T_{\tau(\alpha)}$, and we let 
$v_\alpha,v_{\tau(\alpha)}$ be such that $v_\alpha 
v_{\tau(\alpha)}=u_{\5\alpha}$. 

\item If $\5\alpha=\{\alpha,\tau(\alpha),\alpha^*\}$ where 
$\alpha^*=\alpha+\tau(\alpha)$, then 
$u_{\5\alpha}=h_\alpha(\lambda)h_{\tau(\alpha)}(\lambda')$ for some 
$\lambda,\lambda'\in\fqobar^\times$,
	\[ w_{\5\alpha}(h_\alpha(\lambda)h_{\tau(\alpha)}(\lambda')) = 
	h_\alpha(\lambda'{}^{-1})h_{\tau(\alpha)}(\lambda^{-1}) \]
by Lemma \ref{theta-r}\eqref{wa(hb)}, and $\lambda=\lambda'$ since 
$w_{\5\alpha}(u_{\5\alpha})=u_{\5\alpha}^{-1}$. Set 
$v_\alpha=h_\alpha(\lambda)$  and $v_{\tau(\alpha)}=1$. (This depends on 
the choice of $\alpha\in\5\alpha\cap\Pi$.)

\end{itemize}
Let $t\in\4T$ be such that $\Phi(t)=(v_\alpha)_{\alpha\in\Pi}$. We claim 
that $t^{-1}w_{\5\alpha}(t)=u_{\5\alpha}$ for each $\5\alpha\in\5\Pi$. This 
is clear when $|\5\alpha|\le2$. If 
$\5\alpha=\{\alpha,\tau(\alpha),\alpha^*\}$ and $\lambda$ are as above, then 
	\begin{align*} 
	w_{\5\alpha}(t)= w_{\alpha^*}(t) &= w_{\tau(\alpha)} w_{\alpha} 
	w_{\tau(\alpha)}(t) = w_{\tau(\alpha)} (w_{\alpha}(t))
	= w_{\tau(\alpha)} \bigl( t \cdot h_{\alpha}(\lambda) \bigr) \\
	&= t \cdot w_{\tau(\alpha)} (h_{\alpha}(\lambda)) 
	= t \cdot h_{\alpha^*}(\lambda) = t \cdot u_{\5\alpha}.
	\end{align*}

Thus $\muu(w_{\5\alpha})=dt(w_{\5\alpha})$ for each $\5\alpha\in\5\Pi$. 
Since $W_0=\gen{w_{\5\alpha}\,|\,\5\alpha\in\5\Pi}$ (and since $\muu$ and 
$dt$ are both cocycles), this implies that $\muu=dt$, and hence that 
$[\muu]=0$ in $H^1(W_0;\4T)$. 
\end{proof}

As one consequence of Lemma \ref{Aut-diag}, the $Z^*$-theorem holds for 
these groups. This is known to hold for all finite groups (see 
\cite[\S\,7.8]{GLS3}), but its proof for odd $p$ depends on the 
classification of finite simple groups, which we prefer not to assume here.

\begin{Cor} \label{c:Z*}
Assume that $G\in\Lie(q_0)$, $p\ne{}q_0$, and $S\in\sylp{G}$ satisfy Hypotheses 
\ref{G-hypoth-X}. Then $Z(\calf_S(G))=O_p(Z(G))$. 
\end{Cor}

\begin{proof} By Lemma \ref{Aut-diag}, $O_p(Z(G))=C_A(W_0)$, where 
$C_S(A)=A$ and $\Aut_G(A)=\Aut_{W_0}(A)$ by Lemma \ref{NG(T)}(a,b) or by 
hypothesis (in Case \ref{G-hypoth-X}\ref{messy_case}). Hence 
$Z(\calf_S(G))\le O_p(Z(G))$, while the other inclusion is clear. 
\end{proof}

We now need the following additional hypotheses, in order to be able to 
compare $\Aut\scal(A)$ with the group of field automorphisms of $G$. 

\begin{Hyp} \label{G-hypoth-X2}
Fix a prime $p$ and a prime power $q$. Assume that $q=q_0^b$ where $q_0$ 
is prime, $b\ge1$, $q_0\ne p$, and 
\begin{enumr}
\item $q_0\equiv\pm3$ (mod $8$) if $p=2$;
\item the class of $q_0$ generates $(\Z/p^2)^\times$ if $p$ is odd; and 
\item $b|(p-1)p^\ell$ for some $\ell\ge0$. 
\end{enumr} 
We will also say that ``$G$ satisfies Hypotheses \ref{G-hypoth-X2}'' (for a 
given prime $p$) if $G\cong{}^t\gg(q)$ for some $t$ and $\gg$, and some $q$ 
which satisfies the above conditions.
\end{Hyp}

By Hypothesis \ref{G-hypoth-X}\eqref{not4x}, $\psi_{q_0}(G)=G$, and thus 
all field endomorphisms of $\4G$ normalize $G$. When $G$ has a 
\emph{standard} $\sigma$-setup, $\Phi_G$ was defined to be the group of 
restrictions of such endomorphisms $\psi_{q_0^a}\in\Phi_{\4G}$ for $a\ge0$. Under 
our Hypotheses \ref{G-hypoth-X}, this applies only when we are in case 
\ref{easy_case} (although Proposition \ref{nonstandard} describes the 
relation between $\Phi_G$ and $\psi_{q_0}$ in the other cases). In what 
follows, it will be useful to set 
	\[ \5\Phi_G=\gen{\psi_{q_0}|_G}\le\Aut(G). \]
By Proposition \ref{nonstandard}(d), 
$\Inndiag(G)\5\Phi_G=\Inndiag(G)\Phi_G$, although $\5\Phi_G$ can be 
strictly larger than $\Phi_G$ ($\5\Phi_G\cap\Inndiag(G)$ need not be 
trivial). Note that since each element of this group acts on $\4T$ via 
$(t\mapsto t^r)$ for some $r$, $\5\Phi_G$ normalizes $T$ and each of its 
subgroups. 

Recall that $\tau\in\Aut(V)$ is the automorphism induced by $\sigma$, and 
also denotes the induced permutation of $\Sigma$.

\begin{Lem} \label{Phi->scal}
Assume Hypotheses \ref{G-hypoth-X} and \ref{G-hypoth-X2} and Notation 
\ref{G-setup-X}. Let 
	\[ \chi_0\:\5\Phi_G\Right6{}\Aut(A,\calf) \]
be the homomorphism induced by restriction from $G$ to $A$. Set 
$m=|\tau|=\bigl|\gamma|_{\4T}\bigr|$. Then the following hold.
\begin{enuma} 

\item Either $T$ has exponent $q^m-1$; or $p$ is odd, $m=\ordp(q)$, $m$ is 
even, and $(q^{m/2}+1)\big|\expt(T)\big|(q^m-1)$. 

\item If $p$ is odd, then $\chi_0(\5\Phi_G)=\Aut\scal(A)$. If $p=2$, then 
$\chi_0(\5\Phi_G)$ has index $2$ in $\Aut\scal(A)$, and 
$\Aut\scal(A)=\Im(\chi_0)\gen{\psi_{-1}^A}$. 

\item If $p=2$, then $\chi_0$ is injective.  If $p$ is odd, then 
	\[ \Ker(\chi_0) = \begin{cases} 
	\gen{\psi_q|_G} = \gen{\gamma|_G} & \textup{in case \ref{easy_case}} 
	\\
	\gen{(\psi_q|_G)^m} = \gen{\gamma^m|_G} = 
	\5\Phi_G\cap\Aut_{\4T}(G) & \textup{in cases \ref{minus_case} and 
	\ref{messy_case}.}
	\end{cases} \]


\end{enuma}
\end{Lem}

\begin{proof} We first recall some of the assumptions in 
cases (III.1--3) of Hypotheses \ref{G-hypoth-X}:
	\beqq \renewcommand{\arraystretch}{1.5}
	\renewcommand{\arraycolsep}{3mm}
	\begin{array}{|l||l|l|}
	\hline
	\textup{case \ref{easy_case}} & \ordp(q)=1,~ m=|\gamma|, 
	\textup{ and } m\le2 
	&  \\\hline
	\textup{case \ref{minus_case}} & \ordp(q)=m=2 
	& \textup{$p$ is odd} \\\hline
	\textup{case \ref{messy_case}} & \ordp(q)=m 
	& \textup{$p$ is odd} \\\hline
	\end{array} \label{e:list_cases} \eeqq
(Recall that $\gamma$ is a graph automorphsm in case \ref{easy_case}, so 
$|\gamma|=|\tau|=m$.) In all of these cases, $p|(q^m-1)$ since $\ordp(q)|m$. 

\smallskip

\noindent\textbf{(a) } For each $t\in{}T=C_{\4T}(\psi_q\circ\gamma)$, 
$t^q=\psi_q(t)=\gamma^{-1}(t)$. Hence 
$t=\gamma^{-m}(t)=(\psi_q)^m(t)=t^{q^m}$, and $t^{q^m-1}=1$. Thus 
$\expt(T)\big|(q^m-1)$. 

By Hypotheses \ref{G-hypoth-X}\eqref{not4x}, there is a linearly 
independent subset $\Omega=\{\alpha_1,\ldots,\alpha_s\}\subseteq\Sigma$ 
such that either $\Omega$ or $\pm\Omega=\{\pm\alpha_1,\ldots,\pm\alpha_s\}$ 
is a free $\gen{\tau}$-orbit in $\Sigma$. Assume $\Omega$ is a free orbit 
(this always happens in case \ref{easy_case}). 
In particular, $m=|\tau|=s$. For each 
$1\ne\lambda\in\fqobar^{\times}$ such that $\lambda^{q^s-1}=1$, the element 
	\[ t(\lambda)= \prod_{i=0}^{m-1} h_{\tau^i(\alpha_1)}(\lambda^{q^i}) 
	\]
is fixed by $\sigma=\psi_q\circ\gamma$ (recall 
$\sigma(h_\beta(\lambda))=h_{\tau(\beta)}(\lambda^q)$ for each 
$\beta\in\Sigma$ by Lemma \ref{l:tau}). Hence $t(\lambda)\in{}T$, 
and $t(\lambda)\ne1$ when $\lambda\ne1$ by Lemma 
\ref{theta-r}(\ref{hb.hc},\ref{T=prod}). Thus $T$ contains the subgroup 
$\{t(\lambda)\,|\,\lambda^{q^m-1}=1\}$ of order $q^m-1$, this subgroup is 
cyclic (isomorphic to a subgroup of $\fqobar^{\times}$), and hence 
$\expt(T)=q^m-1$.

Assume now that $\pm\Omega$ is a free $\gen{\tau}$-orbit (thus 
$m=|\tau|=2s$). In particular, we are not in case \ref{easy_case}, so $p$ 
is odd and $m=\ordp(q)$. 
Then $\tau^i(\alpha_1)=-\alpha_1$ for some $0<i<2s$, and 
$i=s$ since $\tau^{2i}(\alpha_1)=\alpha_1$. For 
each $1\ne\lambda\in\fqobar^{\times}$ such that $\lambda^{q^s+1}=1$, 
	\[ t(\lambda)= \prod_{i=0}^{s-1} h_{\tau^i(\alpha_1)}(\lambda^{q^i}) 
	\]
is fixed by $\sigma=\psi_q\circ\gamma$ by Lemma \ref{l:tau} and since 
$h_{\tau^s(\alpha_1)}(\lambda^{q^s})=h_{-\alpha_1}(\lambda^{-1}) 
=h_{\alpha_1}(\lambda)$. Hence 
$t(\lambda)\in{}T$, and $t(\lambda)\ne1$ when $\lambda\ne1$ by Lemma 
\ref{theta-r} again. Thus  $\{t(\lambda)\,|\,\lambda^{q^e+1}=1\}\le T$ is 
cyclic of order $q^s+1$, and so $(q^s+1)\big|\expt(T)$.

\smallskip

\noindent\textbf{(b) } By definition, $\Im(\chi_0)=\chi_0(\5\Phi_G)$ is 
generated by $\chi_0(\psi_{q_0})=\psi_{q_0}|_A$, which acts on $A$ via 
$(a\mapsto a^{q_0})$. If $p$ is odd, then by Hypotheses 
\ref{G-hypoth-X2}(ii), the class of $q_0$ generates $(\Z/p^2)^{\!\times}$, 
and hence generates $(\Z/p^k)^{\!\times}$ for each $k>0$. So 
$\Im(\chi_0)=\Aut\scal(A)$ in this case.

If $p=2$, then $q_0\equiv\pm3$ (mod $8$) by Hypotheses 
\ref{G-hypoth-X2}(i). So for each $k\ge2$, $\gen{q_0}$ has index $2$ in 
$(\Z/2^k)^{\!\times}=\gen{q_0,-1}$. Hence $\Im(\chi_0)=\gen{\psi_{q_0}|_A}$ has 
index $2$ in $\Aut\scal(A)=\gen{\psi_{q_0}|_A,\psi_{-1}^A}$.

\smallskip

\noindent\textbf{(c) } Set $\phi_0=\psi_{q_0}|_G$, a generator of 
$\5\Phi_G$. Then $(\phi_0)^b=\psi_q|_G=(\gamma|_G)^{-1}$ since 
$G=C_{\4G}(\psi_q\circ\gamma)$, and so $\bigl|\phi_0|_T\bigr|$ divides 
$b|\gamma|_{\4T}|=bm$. Also, 
$(\phi_0)^{bm}=(\gamma|_G)^{-m}\in\Aut_{\4T}(G)$ by Lemma \ref{l:tau}.


By (a), either $\expt(T)=q^m-1$; or $m$ is even, $p$ is odd, $\ordp(q)=m$, 
and $(q^{m/2}+1)\big|\expt(T)\big|(q^{m}-1)$. In the latter case, 
$v_p(q^{m/2}+1)=v_p(q^m-1)>0$ since $p\nmid(q^{m/2}-1)$. Thus 
	\beqq \expt(A)=p^e \qquad\textup{where}\qquad
	e=v_p(q^m-1) = v_p(q_0{}^{bm}-1) >0\,. \label{e:e=v_p} \eeqq

If $p=2$, then we are in case \ref{easy_case}. In particular, 
$q=q_0^b\equiv1$ (mod $4$), and $m\le2$. Also, $b$ (and hence $bm$) is a 
power of $2$ by Hypotheses \ref{G-hypoth-X2}(iii). If $bm=1$, then 
$q=q_0\equiv5$ (mod $8$), so $e=v_2(q-1)=2$. If $bm$ is even, then 
$e=v_2(q_0^{bm}-1)=v_2(q_0^2-1)+v_2(bm/2)=3+v_2(bm/2)$ by Lemma 
\ref{v(q^i-1)}. Thus in all cases, $e=2+v_2(bm)$. So 
$\Im(\chi_0)\le\Aut\scal(A)\cong(\Z/2^e)^{\!\times}$ has order $2^{e-2}=bm$. 
Since $(\psi_{q_0}|_G)^{bm}=(\psi_q|_G)^m=(\gamma^{-1}|_G)^m=\Id_G$ (recall 
$m=|\gamma|$ in case \ref{easy_case}), $\chi_0$ is injective.

Now assume $p$ is odd, and set $m_0=\ordp(q)$. Then $b|(p-1)p^\ell$ for 
some $\ell\ge0$ by Hypotheses \ref{G-hypoth-X2}(iii), and $q=q_0^b$ where 
the class of $q_0$ generates $(\Z/p^k)^{\!\times}$ for each $k\ge1$. 
For $r\in\Z$, $q^r=q_0^{br}\equiv1$ (mod $p$) if and only if $(p-1)|br$. 
Hence $bm_0=b\cdot\ordp(q)=(p-1)p^\ell$ for some $\ell\ge0$. Since 
$v_p(q_0^{p-1}-1)=1$, and since $m=m_0$ or $2m_0$, Lemma \ref{v(q^i-1)} 
implies that
	\[ e=v_p(q^m-1)=v_p(q_0^{bm}-1)=v_p(q_0^{bm_0}-1)
	=1+v_p(p^\ell)=1+\ell\,. \]

Thus $\ell=e-1$, where $p^e=\expt(A)$ by \eqref{e:e=v_p}, so 
$|\Aut\scal(A)|=(p-1)p^{e-1}=bm_0$. Since $\chi_0$ sends the generator 
$\phi_0$ of $\5\Phi_G$ to the generator $\chi_0(\phi_0)$ of $\Aut\scal(A)$, 
this proves that 
$\Ker(\chi_0)=\gen{\psi_q^{m_0}|_G}=\gen{\gamma^{m_0}|_G}$. The 
descriptions in the different cases now follow immediately. Note that in 
cases \ref{minus_case} and \ref{messy_case} (where $m=m_0$), 
$\phi_0^{bm}=\gamma^{-m}|_G\in\Aut_{\4T}(G)$ by Lemma \ref{l:tau}. 
The converse is immediate: $\5\Phi_G\cap\Aut_{\4T}(G)\le\Ker(\chi_0)$. 
\end{proof}

Before applying these results to describe $\Out(S,\calf)$ and the 
homomorphism $\4\kappa_G$, we need to know in which cases the subgroup $A$ 
is characteristic in $S$. 

\begin{Prop} \label{p:AcharS}
Assume Hypotheses \ref{G-hypoth-X} and Notation \ref{G-setup-X}. 
\begin{enuma} 

\item If $p=2$, then $A$ is characteristic in $S$, and is the unique 
abelian subgroup of $S$ of order $|A|$, except when $q\equiv5$ (mod $8$) 
and $G\cong\Sp_{2n}(q)$ for some $n\ge1$.

\item If $p$ is odd, then $A$ is characteristic in $S$, and $\Omega_1(A)$ 
is the unique elementary abelian subgroup of $S$ of maximal rank, except 
when $p=3$, $q\equiv1$ (mod $3$), $v_3(q-1)=1$, and $G\cong\SU_3(q)$ or 
$G_2(q)$.
\end{enuma}
In all cases, each normal subgroup of $S$ isomorphic to $A$ is 
$N_G(S)$-conjugate to $A$.
\end{Prop}

\begin{proof} If $p$ is odd, then by \cite[10-2(1,2)]{GL}, there is a 
unique elementary $p$-subgroup $E\le S$ of rank equal to that of $A$ 
(denoted $r_{m_0}$ in \cite{GL}), except when $p=3$ and $G$ is isomorphic 
to one of the groups $\SL_3(q)$ ($q\equiv1$ (mod $3$)), $\SU_3(q)$ 
($q\equiv-1$ (mod $3$)), or $G_2(q)$, $\lie3D4(q)$, or $\lie2F4(q)$ 
($q\equiv\pm1$ (mod $3$)). When there is a unique such subgroup $E$, then 
$A=C_S(E)$ by Lemma \ref{NG(T)}(a) (or by assumption in case 
\ref{messy_case}), and hence $A$ is characteristic in $S$. 

Among the exceptions, $\SL_3(q)$ and $G_2(q)$ are the only ones which 
satisfy Hypotheses \ref{G-hypoth-X}. In both cases, $S$ is an extension of 
$A\cong(C_{3^\ell})^2$ by $C_3$, where $\ell=v_3(q-1)$, and where 
$Z(S)=C_A(S)$ has order $3$. If $\ell>1$, then 
$A$ is the unique abelian subgroup of index $p$ in $S$. If 
$\ell=1$, then $S$ is extraspecial of order $3^3$ and exponent $3$. By 
Theorem \ref{OldThA}(a), we can assume $q=4$ without changing the 
isomorphism type of the fusion system, so $G$ contains $\SU_3(2)$. This is 
a semidirect product $S\sd{}Q_8$ (cf. \cite[p. 123--124]{Taylor}), and 
hence the four subgroups of $S$ of order $9$ are $N_G(S)$-conjugate. 

It remains to prove the proposition when $p=2$. We use 
\cite[\S\,2]{O-Ch} as a reference for information about best offenders, 
since this contains what we need in a brief presentation. Assume $A$ is not 
the unique abelian subgroup of $S$ of order $|A|$. Then there is an abelian 
subgroup $1\ne{}B\le W_0$ such that $|B|\cdot|C_A(B)|\ge|A|$. In other 
words, the action of the Weyl group $W_0$ on $A$ has a nontrivial best 
offender \cite[Definition 2.1(b)]{O-Ch}. Hence by Timmesfeld's replacement 
theorem \cite[Theorem 2.5]{O-Ch}, there is a quadratic best offender $1\ne 
B\le W_0$: an offender such that $[B,[B,A]]=1$. 

We consider three different cases.

\smallskip

\noindent\boldd{Case 1: $G\cong\gg(q)$ is a Chevalley group, where either 
$q\equiv1$ (mod $8$), or $G\not\cong\Sp_{2n}(q)$ 
for any $n\ge1$. } Set $n=\rk(A)=\rk(T)$: the Lie rank of $G$ (or of 
$\gg$). Set $\ell=v_2(q-1)\ge2$. Then $A\cong(C_{2^\ell})^n$ is the group of 
all $2^\ell$-torsion elements in $T$ (or in $\4T$). 
Since the result is clear when $n=1$ ($G\cong\SL_2(q)\cong\Sp_2(q)$, $A\cong 
C_{2^\ell}$, and $S\cong Q_{2^{\ell+1}}$), we assume $n\ge2$.

Let $\Lambda=\Z\Sigma^\vee$ be the lattice in $V$ generated by the dual 
roots. By Lemma \ref{l:CW(t)}(a), there are $\Z[W]$-linear isomorphisms 
$A\cong\Lambda/2^\ell\Lambda$ and $\Omega_1(A)\cong\Lambda/2\Lambda$.

Assume first that $B$ acts faithfully on $\Omega_1(A)$. Since $B$ has 
quadratic action, it is elementary abelian \cite[Lemma 2.4]{O-Ch}. Set 
$k=\rk(B)$; thus $B\cong C_2^k$ and $|A/C_A(B)|\le2^k$. 

Since the $B$-action on $V$ is faithful, the characters 
$\chi\in\Hom(B,\{\pm1\})$ which have nontrivial eigenspace on $V$ generate 
the dual group $B^*$. So we can choose a basis $\chi_1,\ldots,\chi_k$ for 
$B^*$ such that each $\chi_i$ has nontrivial eigenspace. Let $b\in B$ be 
the unique element such that $\chi_i(b)=-1$ for each $i=1,\ldots,k$. Let 
$V_+,V_-$ be the $\pm1$-eigenspaces for the $b$-action on $V$, and set 
$\Lambda_{\pm}=\Lambda\cap V_{\pm}$. By construction, $\dim(V_-)\ge k$.

Let $v\in\Lambda$ be an element whose class modulo $2^\ell\Lambda$ is 
fixed by $b$, and write $v=v_++v_-$ where $v_\pm\in V_\pm$. Then 
$2v_-=v-b(v)\in2^\ell\Lambda\cap V_-=2^\ell\Lambda_-$, so 
$v_-\in2^{\ell-1}\Lambda_-$ and $v_+=v-v_-\in\Lambda\cap V_+=\Lambda_+$. 
Thus $C_{\Lambda/2^\ell\Lambda}(b) 
=(\Lambda_+\times2^{\ell-1}\Lambda_-)/2^\ell\Lambda$.
Set $r=\rk(\Lambda_-)=\dim(V_-)\ge k$; then 
	\begin{align*} 
	2^k\ge |A/C_A(B)|\ge |A/C_A(b)| 
	= |\Lambda/(\Lambda_+\times2^{\ell-1}\Lambda_-)|
	&= 2^{r(\ell-1)}\cdot|\Lambda/(\Lambda_+\times\Lambda_-)| \\
	&\ge 2^{k(\ell-1)}\cdot|\Lambda/(\Lambda_+\times\Lambda_-)| . 
	\end{align*}
In particular, $\Lambda=\Lambda_+\times\Lambda_-$. But then $b$ acts 
trivially on $\Lambda/2\Lambda$, hence on $\Omega_1(A)$, which contradicts 
our assumption.


Thus $B$ does not act faithfully on $\Omega_1(A)$. Set 
$B_0=C_B(\Omega_1(A))\cong C_B(\Lambda/2\Lambda)\ne1$. If $-\Id_V\in{}B_0$, 
then it inverts $A$, $[B,\Omega_1(A)]\le[B,[B_0,A]]=1$ since $B$ acts 
quadratically, so $B=B_0$, and 
$|B_0|\ge|A/C_A(B)|\ge|A/\Omega_1(A)|=2^{(\ell-1)n}$. If $b\in{}B_0$ is 
such that $b^2=-\Id_V$, then $b$ defines a $\C$-vector space structure on 
$V$, and hence does not induce the identity on $\Lambda/2\Lambda$, a 
contradiction. 

Thus there is $b\in B_0$ which does not act on $V$ via $\pm\Id$. Let 
$V_{\pm}\ne0$ be the $\pm1$-eigenspaces for the $b$-action on $V$, and set 
$\Lambda_\pm=\Lambda\cap V_\pm$. For each $v\in\Lambda$, 
$v-b(v)\in2\Lambda$ since $b$ acts trivially on 
$\Omega_1(A)\cong\Lambda/2\Lambda$.  Set $v=v_++v_-$, where $v_\pm\in 
V_\pm$. Then $2v_-=v-b(v)\in2\Lambda\cap V_-=2\Lambda_-$ implies that 
$v_-\in\Lambda_-$, and hence $v_+\in\Lambda_+$. Thus 
$V\in\Lambda_+\times\Lambda_-$, so by Lemma \ref{l:Lambda}, $\gg=C_n$. By 
assumption, $q\equiv1$ (mod $8$), so $\ell\ge3$, and 
$[b,[b,\Lambda/2^\ell\Lambda]]\ge4\Lambda_-/2^\ell\Lambda_-\ne1$, 
contradicting the assumption that $B$ acts quadratically on $A$. 

\smallskip

\noindent\boldd{Case 2: $G\cong\Sp_{2n}(q)$ for some $n\ge1$ and some 
$q\equiv5$ (mod $8$). } Fix subgroups $H_i\le G$ ($1\le i\le n$) and 
$K<G$ such that $H_i\cong\Sp_2(q)$ for each $i$, $K\cong\Sigma_n$ 
is the group of permutation matrices (in $2\times2$ blocks), and $K$ 
normalizes $H=H_1\times\cdots\times H_n$ and permutes the factors in the 
obvious way. We can also fix isomorphisms 
$\chi_i\:H_i\Right2{\cong}\Sp_2(q)$ such that the action of $K$ on the 
$H_i$ commutes with the $\chi_i$. 

Fix subgroups $\5A<\5Q<\Sp_2(q)$, where $\5Q\cong Q_8$ (a Sylow 
2-subgroup), and $\5A\cong C_4$ is contained in the maximal torus. Set 
$Q_i=\chi_i^{-1}(\5Q)$ and $A_i=\chi_i^{-1}(\5A)$, and set 
$Q=Q_1Q_2\cdots{}Q_n$ and $A=A_1A_2\cdots{}A_n$. Thus $A=O_2(T)$ is as in 
Hypotheses \ref{G-hypoth-X}(III): the 2-power torsion in the maximal torus 
of $G$. By \cite[\S\,I]{CF}, $S=QR$ for some 
$R\in\syl2{K}$. Also, $W\cong QK/A\cong C_2\wr\Sigma_n$ acts on 
$A$ via signed permutations of the coordinates. 

Let $B$ be any nontrivial best offender in $W$ on $A$. Consider the action 
of $B$ on the set $\{1,2,\ldots,n\}$, let $X_1,\ldots,X_k$ be the set of 
orbits, and set $d_i=|X_i|$. For $1\le i\le k$, let $A_i\le A$ be the 
subgroup of elements whose coordinates vanish except for those in positions 
in $X_i$; thus $A_i\cong(C_4)^{d_i}$ and $A=A_1\times\cdots\times A_k$. Set 
$B_i=B/C_B(A_i)$; then $|B|\le\prod_{i=1}^k|B_i|$. Since $B$ is abelian, 
either $|B_i|=d_i$ and $B_i$ permutes the coordinates freely, or 
$|B_i|=2d_i$ and there is a unique involution in $B_i$ which inverts all 
coordinates in $A_i$. In the first case, $|C_{A_i}(B_i)|=4$, and so 
$|B_i|\cdot|C_{A_i}(B_i)|=d_i\cdot4\le4^{d_i}=|A_i|$ with equality only if 
$d_i=1$. In the second case, $|C_{A_i}(B_i)|=2$, and again 
$|B_i|\cdot|C_{A_i}(B_i)|=2d_i\cdot2\le4^{d_i}=|A_i|$ with equality only if 
$d_i=1$. Since 
	\[  \prod_{i=1}^k|A_i| = |A| \le |B|\cdot|C_A(B)| 
	= |B|\cdot\prod_{i=1}^k|C_{A_i}(B_i)| \le 
	\prod_{i=1}^k \bigl(|B_i| \cdot |C_{A_i}(B_i)| \bigl)
	 \,, \]
we conclude that $d_i=1$ for all 
$i$, and hence that $B$ acts only by changing signs in certain coordinates. 

For each $1\le i\le n$, let $\pr_i\:Q\Right2{}Q_i$ be the projection onto 
the $i$-th factor. If $A^*\le S$ is abelian of order $4^n$, then 
$A^*A/A$ is a best offender in $W$ on $A$, and hence $A^*\le Q$ by 
the last paragraph. Also, $\pr_i(A^*)$ is cyclic of order at 
most $4$ for each $i$, and since $|A^*|=4^n$, $\pr_i(A^*)\cong C_4$ for 
each $i$ and $A^*=\prod_{i=1}^n\pr_i(A^*)$. Thus there are exactly $3^n$ 
such subgroups. 

Now assume $A^*\nsg S$, and set $A^*_i=\pr_i(A^*)\le Q_i$ for short. Since 
$A^*$ is normal, the subgroups $\chi_i(A^*_i)\le\5Q<\Sp_2(q)$ are equal for 
all $i$ lying in any $R$-orbit of the set $\{1,2,\ldots,n\}$. Hence we can 
choose elements $x_1,x_2,\ldots,x_n$, where $x_i\in 
N_{H_i}(Q_i)\cong\SL_2(3)$ and $\9{x_i}(A_i)=A^*_i$ for each $i$, and such 
that $\chi_i(x_i)\in\Sp_2(q)$ is constant on each $R$-orbit. Set 
$x=x_1x_2\cdots{}x_n$; then $\9xA=A^*$, and $x\in N_G(S)$. 

\smallskip

\noindent\boldd{Case 3: $G$ is a Steinberg group. } Assume 
$\gamma\in\Gamma_{\4G}$ is a graph automorphism of order $2$, and that 
$G=C_{\4G}(\sigma)$ where $\sigma=\gamma\psi_q$. Set 
$G_0=C_{\4G}(\gamma,\psi_q)$; thus $G_0\le G$. Set $\ell=v_2(q-1)\ge2$. We 
must again show that the action of $W_0$ on $A$ has no nontrivial best 
offenders.

If $G\cong\lie2E6(q)$ or $\Spin_{2n}^-(q)$ ($n\ge4$), then $G_0\cong 
F_4(q)$ or $\Spin_{2n-1}(q)$, respectively, and $W_0$ is the Weyl group of 
$G_0$. If $1\ne B\le W_0$ is a best offender in $W_0$ on $A$, then it is 
also a best offender on $\Omega_\ell(A)\le G_0$, which is impossible by 
Case 1.

If $G\cong\SU_{2n+1}(q)\cong\lie2A{2n}(q)$, then 
$S\cong(\SD_{2^{\ell+2}})^n\sd{}R$ for some $R\in\syl2{\Sigma_n}$ \cite[pp. 
143--144]{CF}. Thus $A\cong(C_{2^{\ell+1}})^n$, $W_0\cong C_2\wr\Sigma_n$, 
$\Sigma_n<W_0$ acts on $A$ by permuting the coordinates, and the subgroup 
$W_1\cong(C_2)^n$ in $W_0$ has a basis each element of which acts on one 
coordinate by $(a\mapsto a^{2^{\ell}-1})$. If $B\le W_0$ is a nontrivial 
quadratic best offender on $A$, then it is also a best offender on 
$\Omega_\ell(A)$ \cite[Lemma 2.2(a)]{O-Ch}, hence is contained in $W_1$ by 
the argument in Case 2, which is impossible since no nontrivial element in 
this subgroup acts quadratically. Thus $A$ is characteristic in this case.

It remains to consider the case where 
$G\cong\SU_{2n}(q)\cong\lie2A{2n-1}(q)$. Since the case 
$\SU_2(q)\cong\Sp_2(q)$ has already been handled, we can assume $n\ge2$. 
Set $\5G=\GU_{2n}(q)>G$, set 
$G_0=\GU_2(q)\times\cdots\times\GU_2(q)\le\5G$, and set 
$G_1=N_{\5G}(G_0)\cong\GU_2(q)\wr\Sigma_n$. Then $G_1$ has odd index in $\5G$  
\cite[pp. 143--144]{CF}, so we can assume $S\le G_1\cap{}G$. Fix 
$H_0\in\syl2{G_0}$; thus $H_0\cong(\SD_{2^{\ell+2}})^n$. Since 
$v_2(q+1)=1$, and since the Sylow 2-subgroups of $\SU_2(q)$ are quaternion,
	\[ G\cap H_0 = \Ker\bigl[ H_0 \cong (\SD_{2^{\ell+2}})^n 
	\Right3{\chi^n} C_2^n \Right3{\textup{sum}} C_2 \bigr] \,, \]
where $\chi\:\SD_{2^{\ell+2}}\Right2{}C_2$ is the surjection with 
quaternion kernel. As in the last case, $W_0\cong C_2\wr\Sigma_n$ with 
normal subgroup $W_1\cong C_2^n$. If $B\le W_0$ is a nontrivial quadratic 
best offender on $A$, then it is also a best offender on $\Omega_\ell(A)$ 
\cite[Lemma 2.2(a)]{O-Ch}, so $B\le W_1$ by the argument used in Case 2. Since 
no nontrivial element in $W_1$ acts quadratically on $A$, we 
conclude that $A$ is characteristic in this case.
\end{proof}

The next lemma is needed to deal with the fact that not all fusion 
preserving automorphisms of $A$ lie in $\Aut(A,\calf)$ (since they need not 
extend to automorphisms of $S$). 

\begin{Lem} \label{Aut(A,F)-Aut(S,F)}
Let $G$ be any finite group, fix $S\in\sylp{G}$, and let $S_0\nsg S$ be a 
normal subgroup. Let $\varphi\in\Aut(G)$ be such that $\varphi(S_0)=S_0$ 
and $\varphi|_{S_0}\in N_{\Aut(S_0)}(\Aut_S(S_0))$.  Then there is 
$\varphi'\in\Aut(G)$ such that $\varphi'|_{S_0}=\varphi|_{S_0}$, 
$\varphi'(S)=S$, and $\varphi'\equiv\varphi$ (mod $\Inn(G)$). 
\end{Lem}

\begin{proof} Since $\varphi|_{S_0}$ normalizes $\Aut_S(S_0)$, and 
$c_{\varphi(g)}=\varphi{}c_g\varphi^{-1}$ for each $g\in{}G$, we have 
$\Aut_{\varphi(S)}(S_0)=\9\varphi\Aut_S(S_0)=\Aut_S(S_0)$. Hence 
$\varphi(S)\le C_G(S_0)S$. Since $S$ normalizes $C_G(S_0)$ and 
$S\in\sylp{C_G(S_0)S}$, we have $\varphi(S)=\9xS$ for some 
$x\in{}C_G(S_0)$. Set $\varphi'=c_x^{-1}\circ\varphi\in\Aut(G)$; then 
$\varphi'(S)=S$ and $\varphi'|_{S_0}=\varphi|_{S_0}$. 
\end{proof}

In the next two propositions, we will be referring to the short exact 
sequence
	\beqq 1 \Right2{} \Aut\dg(S,\calf) \Right5{} N_{\Aut(S,\calf)}(A) 
	\Right5{R} \Aut(A,\calf) \Right2{} 1 \,. \label{e:Aut_dg(A)} \eeqq
Here, $R$ is induced by restriction, and $\Aut(A,\calf)=\Im(R)$ and 
$\Aut\dg(S,\calf)=\Ker(R)$ by definition of these two groups (Notation 
\ref{G-setup-X}\eqref{not9x}). By Proposition \ref{p:AcharS}, in all cases, 
each class in $\Out(S,\calf)$ is representated by elements of 
$N_{\Aut(S,\calf)}(A)$.

\begin{Prop} \label{kappa_onto}
Assume Hypotheses \ref{G-hypoth-X} and \ref{G-hypoth-X2} and Notation 
\ref{G-setup-X}. Then $\4\kappa_G$ is surjective, except in the following 
cases:
\begin{itemize} 
\item $(G,p)\cong(\lie2E6(q),3)$, or
\item $(G,p)\cong(G_2(q),2)$ and $q_0\ne3$, or 
\item $(G,p)\cong(F_4(q),3)$ and $q_0\ne2$. 
\end{itemize}
In the exceptional cases, $|\Coker(\4\kappa_G)|\le2$.
\end{Prop}


\begin{proof} We first claim that for $\varphi\in \Aut(S,\calf)$,
	\beqq \varphi(A)=A \quad\textup{and}\quad 
	\varphi|_A\in\Aut\scal(A)\Aut_{\Aut(G)}(A) 
	\qquad\implies\qquad
	[\varphi]\in\Im(\4\kappa_G)\,. \label{e:onto} \eeqq
To see this, fix such a $\varphi$. 
By Lemma \ref{Phi->scal}(b), each element of $\Aut\scal(A)$, or of 
$\Aut\scal(A)/\gen{\psi_{-1}^A}$ if $p=2$, is the restriction of an 
element of $\5\Phi_G$. If $p=2$, then we are in case \ref{easy_case}, the 
$\sigma$-setup is standard, and hence the inversion automorphism 
$\psi_{-1}^A$ is the restriction of an inner automorphism of $G$ (if 
$-\Id_V\in W$) or an element of $\Inn(G)\Gamma_G$. Thus 
$\varphi|_A$ extends to an automorphism of $G$.

Now, $\varphi|_A$ normalizes $\Aut_S(A)$ since $\varphi(S)=S$. So by Lemma 
\ref{Aut(A,F)-Aut(S,F)}, $\varphi|_A$ is the restriction of an automorphism 
of $G$ which normalizes $S$, and hence is the restriction of an element 
$\psi\in\Aut(S,\calf)$ such that $[\psi]\in\Im(\4\kappa_G)$. Then 
$\varphi\psi^{-1}\in\Ker(R)=\Aut\dg(S,\calf)$ by the 
exactness of \eqref{e:Aut_dg(A)}, and $[\varphi\psi^{-1}]\in\Im(\4\kappa_G)$ 
by Lemma \ref{Aut-diag}. So $[\varphi]\in\Im(\4\kappa_G)$, which proves 
\eqref{e:onto}. 

By Proposition \ref{p:AcharS}, each class in $\Out(S,\calf)$ is represented 
by an element of $N_{\Aut(S,\calf)}(A)$. Hence by \eqref{e:onto}, 
$|\Coker(\4\kappa_G)|$ is at most the index of 
$\Aut(A,\calf)\cap\Aut\scal(A)\Aut_{\Aut(G)}(A)$ in $\Aut(A,\calf)$. So by Lemma 
\ref{Hyp1-N(W0)}, $|\Coker(\4\kappa_G)|\le2$, and 
$\4\kappa_G$ is surjective with the exceptions listed above.
\end{proof}

We now want to refine Proposition \ref{kappa_onto}, and finish the proof of 
Theorem \ref{ThX}, by determining $\Ker(\4\kappa_G)$ in each case where 
\ref{G-hypoth-X} and \ref{G-hypoth-X2} hold and checking whether it is 
split. In particular, we still want to show that each of these fusion 
systems is tamely realized by some finite group of Lie type (and not just 
an extension of such a group by outer automorphisms).

Since $O_{p'}(\Outdiag(G))\le\Ker(\4\kappa_G)$ in all cases by Lemma 
\ref{Aut-diag}, $\4\kappa_G$ induces a quotient homomorphism 
	\[ \0\kappa_G\: \Out(G)/O_{p'}(\Outdiag(G)) \Right5{} \Out(S,\calf) 
	\,, \]
and it is simpler to describe $\Ker(\0\kappa_G)$ than $\Ker(\4\kappa_G)$. 
The projection of $\Out(G)$ onto the quotient $\Out(G)/O_{p'}(\Outdiag(G))$ 
is split: by Steinberg's theorem (Theorem \ref{St-aut}), it splits back to 
$O_p(\Outdiag(G))\Phi_G\Gamma_G$. Hence $\0\kappa_G$ is split surjective if 
and only if $\4\kappa_G$ is split surjective.

\begin{Prop} \label{Ker(kappa)}
Assume Hypotheses \ref{G-hypoth-X} and \ref{G-hypoth-X2} and Notation 
\ref{G-setup-X}. Assume also that none of the following hold: neither
\begin{itemize} 
\item $(G,p)\cong(\lie2E6(q),3)$, nor
\item $(G,p)\cong(G_2(q),2)$ and $q_0\ne3$, nor 
\item $(G,p)\cong(F_4(q),3)$ and $q_0\ne2$. 
\end{itemize}
\begin{enuma} 
\item If $p=2$, then $\0\kappa_G$ is an isomorphism, and $\4\kappa_G$ is 
split surjective.

\item Assume that $p$ is odd, and that we are in the situation of case 
\ref{easy_case} of Hypotheses \ref{G-hypoth-X}. Then $\sqrt{q}\in\N$, and 
	\[ \Ker(\0\kappa_G) = \begin{cases} 
	\Gen{[\psi_{\sqrt{q}}]}\cong C_2 & \textup{if $\gamma=\Id$ and 
	$-\Id\in W$} \\
	\Gen{[\gamma_0\psi_{\sqrt{q}}]}\cong C_2 & \textup{if $\gamma=\Id$ and 
	$-\Id\notin W$} \\
	\Gen{[\psi_{\sqrt{q}}]}\cong C_4 & \textup{if $\gamma\ne\Id$ ($G$ is a 
	Steinberg group)}
	\end{cases} \]
where in the second case, $\gamma_0\in\Gamma_G$ is a graph automorphism of 
order $2$. Hence $\4\kappa_G$ and $\0\kappa_G$ are split 
surjective if and only if either $\gamma=\Id$ and $-\Id\notin W$, or 
$p\equiv3$ (mod $4$). 

\item Assume that $p$ is odd, and that we are in the situation of case 
\ref{minus_case} or
\ref{messy_case} of Hypotheses \ref{G-hypoth-X}. Assume also that $G$ is a Chevalley 
group ($\gamma\in\Inn(\4G)$), and that $\ordp(q)$ is even or $-\Id\notin 
W_0$. Let $\Phi_G,\Gamma_G\le\Aut(G)$ be as in Proposition 
\ref{nonstandard}. Then $\Phi_G\cap\Ker(\0\kappa_G)=1$, so 
$|\Ker(\0\kappa_G)|\le|\Gamma_G|$, and $\4\kappa_G$ and $\0\kappa_G$ are 
split surjective.


\item Assume that $p$ is odd, and that we are in the situation of case 
\ref{messy_case} of Hypotheses \ref{G-hypoth-X}. 
Assume also that $G$ is a Steinberg group ($\gamma\notin\Inn(\4G)$), and 
that $\ordp(q)$ is even. Then 
	\[ \Ker(\0\kappa_G) = \begin{cases} 
	\Gen{[\gamma|_G]}\cong C_2 & \textup{if 
	$\gamma|_A\in\Aut_{W_0}(A)$} \\
	1 & \textup{otherwise.}
	\end{cases} \]
Hence $\4\kappa_G$ and $\0\kappa_G$ are split surjective if and only if 
$q$ is an odd power of $q_0$ or 
$\Ker(\4\kappa_G)=O_{p'}(\Outdiag(G))$. If $\4\kappa_G$ is not split 
surjective, then its kernel contains a graph automorphism of order $2$ in 
$\Out(G)/\Outdiag(G)$.

\end{enuma}
\end{Prop}

\begin{proof} In all cases, $\kappa_G$ is surjective by Proposition 
\ref{kappa_onto} (with the three exceptions listed above).

By definition and Proposition \ref{p:AcharS}, 
	\[ \Out(S,\calf) = \Aut(S,\calf)/\autf(S) 
	\cong N_{\Aut(S,\calf)}(A) \big/ N_{\autf(S)}(A)\,. \]
Also, $\Out\dg(S,\calf)$ is the image in $\Out(S,\calf)$ of 
$\Aut\dg(S,\calf)$. Since $N_{\autf(S)}(A)$ is the group of automorphisms 
of $S$ induced by conjugation by elements in $N_G(S)\cap{}N_G(A)$, 
the short exact sequence \eqref{e:Aut_dg(A)} 
induces a quotient exact sequence 
	\beqq 1 \Right2{} \Out\dg(S,\calf) \Right5{} \Out(S,\calf) 
	\Right5{\4R} \Aut(A,\calf)\big/\Aut_{N_G(S)}(A) \Right2{} 1 \,. 
	\label{e:4.15a} \eeqq

We claim that 
	\beqq \Aut_{N_G(S)}(A) = \Aut(A,\calf) \cap \Aut_G(A) \,. 
	\label{e:4.15b} \eeqq
That $\Aut_{N_G(S)}(A)$ is contained in the two other groups is 
clear. Conversely, assume $\alpha\in\Aut(A,\calf)\cap\Aut_G(A)$. Then 
$\alpha=c_g|_A$ for some $g\in{}N_G(A)$, and 
$\alpha\in{}N_{\Aut(A)}(\Aut_S(A))$ since it is the restriction of an 
element of $\Aut(S,\calf)$. Hence $g$ normalizes $SC_G(A)$, and since 
$S\in\sylp{SC_G(A)}$, there is $h\in{}C_G(A)$ such that $hg\in{}N_G(S)$. Thus 
$\alpha=c_g|_A=c_{hg}|_A\in\Aut_{N_G(S)}(A)$, and this finishes the proof 
of \eqref{e:4.15b}.

By Lemma \ref{Aut-diag}, $\4\kappa_G$ sends $\Outdiag(G)$ onto 
$\Out\dg(S,\calf)$ with kernel $O_{p'}(\Outdiag(G))$. Hence by the 
exactness of \eqref{e:4.15a}, restriction to $A$ induces an isomorphism
	\begin{multline} 
	\Ker(\0\kappa_G) \RIGHT3{\4R_0}{\cong} 
	\Ker\bigl[ \Out(G)/\Outdiag(G) \Right3{} 
	\Aut(A,\calf)/\Aut_{N_G(S)}(A) \bigr] \\
	= \Ker\bigl[ \Out(G)/\Outdiag(G) \Right3{} 
	N_{\Aut(A)}(\Aut_G(A))/\Aut_G(A) \bigr] \,, 
	\label{e:4.15c} 
	\end{multline}
where the equality holds by \eqref{e:4.15b}. 


Recall that for each $\ell$ prime to $p$, $\psi_\ell^A\in\Aut\scal(A)$ 
denotes the automorphism $(a\mapsto a^\ell)$.


\smallskip

\noindent\textbf{(a,b) } Under either assumption (a) or (b), we are in case 
\ref{easy_case} of Hypotheses \ref{G-hypoth-X}. In particular, 
$(\4G,\sigma)$ is a standard $\sigma$-setup for $G$. Set $k=v_p(q-1)$; then 
$k\ge1$, and $k\ge2$ if $p=2$. 

If $p$ is odd, then by Hypotheses \ref{G-hypoth-X2}(b), the class of $q_0$ 
generates $(\Z/p)^\times$. Since $q=q_0^b\equiv1$ (mod $p$), this implies 
that $(p-1)|b$. In particular, $b$ is even, and $\sqrt{q}=q_0^{b/2}\in\N$. 

Since $\Out(G)/\Outdiag(G)\cong\Phi_G\Gamma_G$ by Theorem \ref{St-aut}, 
where $\Phi_G\Gamma_G$ normalizes $T$ and hence $A$, \eqref{e:4.15c} takes 
the form
	\beqq \Ker(\0\kappa_G) \cong 
	\bigl\{ \varphi\in\Phi_G\Gamma_G \,\big|\, \varphi|_A\in\Aut_{W_0}(A) 
	\bigr\} \,. 
	\label{e:4.15d} \eeqq
In fact, when $\Ker(\0\kappa_G)$ has order prime to $p$ (which is the case 
for all examples considered here), the isomorphism in \eqref{e:4.15d} is an 
equality since $\Outdiag(G)/O_{p'}(\Outdiag(G))$ is a $p$-group. 

Assume first that $G=\gg(q)$ is a Chevalley group. Thus $\sigma=\psi_q$ 
where $q\equiv1$ (mod $p$), and 
$A=\bigl\{t\in\4T\,\big|\,t^{p^k}=1\bigr\}$. By Lemma \ref{scal-mod-m} 
(applied with $m=p^k\ge3$), the group $\Aut_W(\4T)\Aut_{\Gamma_{\4G}}(\4T)$
acts faithfully on $A$, and its action intersects $\Aut\scal(A)$ only in 
$\gen{\psi_{-1}^A}$. By Lemma \ref{Phi->scal}(b,c), restriction to $A$ sends 
$\5\Phi_G$ isomorphically onto $\Aut\scal(A)$ if $p$ is odd, and with index 
$2$ if $p=2$. So when $\gg$ is not one of the groups $B_2$, $F_4$, or 
$G_2$, then $\Phi_G\Gamma_G$ acts faithfully on $A$, and 
	\[ \bigl\{ \varphi\in\Phi_G\Gamma_G \,\big|\, \varphi|_A\in\Aut_{W_0}(A) 
	\bigr\} = \begin{cases} 
	1 & \textup{if $p=2$} \\ 
	\gen{\psi_{\sqrt{q}}} & \textup{if $p$ is odd and $-\Id\in W$} \\
	\gen{\gamma_0\psi_{\sqrt{q}}} & \textup{if $p$ is odd and $-\Id\notin W$}
	\end{cases} \]
where in the last case, $\gamma_0\in\Gamma_G$ is a graph automorphism such 
that the coset $\gamma_0W$ contains $-\Id$. (Note that 
$b=(p-1)p^\ell$ for some $\ell\ge0$ by Hypotheses \ref{G-hypoth-X2}(b,c) 
and since $p|(q-1)$. Hence $\sqrt{q}\equiv-1$ modulo $p^k=\expt(A)$, and 
$\psi_{\sqrt{q}}|_A=\psi_{-1}^A$.)

Thus by \eqref{e:4.15d}, $\0\kappa_G$ is injective if $p=2$, and 
$|\Ker(\0\kappa_G)|=2$ if $p$ is odd. When $p$ is odd, since 
$\Ker(\0\kappa_G)$ is normal of order prime to $p$ in $\Out(G)$ (hence of 
order prime to $|O_{p}(\Outdiag(G))|$), $\Ker(\0\kappa_G)$ is generated by 
$[\psi_{\sqrt{q}}]$ if $-\Id\in{}W$ (i.e., if there is an inner 
automorphism which inverts $\4T$ and hence $A$), or by 
$[\gamma_0\psi_{\sqrt{q}}]$ otherwise for $\gamma_0$ as above. In the 
latter case, $\0\kappa_G$ is split since it sends 
$O_p(\Outdiag(G))\Phi_GO_3(\Gamma_G)$ isomorphically onto $\Out(S,\calf)$ 
(recall $\Gamma_G\cong C_2$ or $\Sigma_3$). When 
$\Ker(\0\kappa_G)=\Gen{[\psi_{\sqrt{q}}]}$, the map is split if and only if 
$4\nmid|\Phi_G|=b$, and since $b=(p-1)p^m$ for some $m$, this holds exactly 
when $p\equiv3$ (mod $4$).

If $(G,p)\cong(B_2(q),2)$, $(F_4(q),2)$, or $(G_2(q),3)$, then since 
$q_0\ne{}p$, $\Gamma_G=1$. So similar arguments show that 
$\Ker(\0\kappa_G)=1$, $1$, or $\Gen{[\psi_{\sqrt{q}}]}\cong C_2$, 
respectively, and that $\0\kappa_G$ is split in all cases.

Next assume $G=G_2(q)$, where $p=2$, $q=3^b$, and $b$ is a power of $2$. 
Then $b\ge2$ since $q\equiv1$ (mod $4$). The above argument shows that 
$\Phi_G$ injects into $\Out(S,\calf)$. Since $\Out(G)$ is cyclic of order 
$2b$, generated by a graph automorphism whose square generates $\Phi_G$ 
(and since $2|b$), $\Out(G)$ injects into $\Out(S,\calf)$.

If $G=F_4(q)$, where $p=3$, $q=2^b$, and $b=2\cdot3^\ell$ for some 
$\ell\ge0$, then the same argument shows that $\Phi_G$ injects into 
$\Out(S,\calf)$. Since $\Out(G)$ is cyclic of order $2b=4\cdot3^\ell$, 
generated by a graph automorphism whose square generates $\Phi_G$, 
$\Out(G)$ injects into $\Out(S,\calf)$. 

It remains to handle the Steinberg groups. Let $\hh$ be such that 
$C_{\4G}(\gamma)=\hh(\fqobar)$: a simple algebraic group by \cite[Theorem 
1.15.2(d)]{GLS3}. In particular, $G\ge H=\hh(q)$. Also, $W_0$ is the Weyl 
group of $\hh$ by \cite[Theorem 1.15.2(d)]{GLS3} (or by the proof of 
\cite[Theorem 8.2]{Steinberg-end}). By Lemma \ref{scal-mod-m} applied to 
$\hh(\fqobar)$, $W_0$ acts faithfully on $A\cap{}H=\Omega_k(A)$, and 
intersects $\Aut\scal(A)$ at most in $\gen{\psi_{-1}^A}$. 

If $p=2$, then by Lemma \ref{Phi->scal}(b), $\psi_{-1}^A$ is not the 
restriction of an element in $\Phi_G$. Also, $\Phi_G\cong C_{2b}$ is sent 
injectively into $\Aut\scal(A)$ by Lemma \ref{Phi->scal}(c), so 
$\0\kappa_G$ is injective by \eqref{e:4.15d}. 

If $p$ is odd, then $\psi_{q_0}|_A$ has order $b$ in $\Aut\scal(A)$ by 
Lemma \ref{Phi->scal}(c). Since $(\psi_{q_0})^{b/2}=\psi_{\sqrt{q}}$ where 
$\sqrt{q}\equiv-1$ (mod $p$) (recall $b|(p-1)p^\ell$ for some $\ell$ by 
Hypotheses \ref{G-hypoth-X2}(iii)), $\psi_{q_0}|_A$ has order $b/2$ modulo 
$\Aut_{W_0}(A)$. So by \eqref{e:4.15d} and the remark afterwards, and since 
$\Phi_G$ is cyclic of order $2b$, 
$\Ker(\0\kappa_G)=\Gen{[\psi_{\sqrt{q}}]}\cong C_4$. In particular, 
$\0\kappa_G$ is split only if $b/2$ is odd; equivalently, $p\equiv3$ (mod 
$4$). 

\smallskip

\noindent\textbf{(c,d) } In both of these cases, $p$ is odd, $\ordp(q)$ is 
even or $-\Id\notin W$, and we are in the situation of case 
\ref{minus_case} or \ref{messy_case} in Hypothesis \ref{G-hypoth-X}. 
Then $\gamma|_G=(\psi_q|_G)^{-1}$ since $G\le C_{\4G}(\gamma\psi_q)$. Also, 
$\psi_{q_0}(G)=G$ by \ref{G-hypoth-X}\eqref{not4x}, and hence 
$\gamma(G)=G$. Since $\psi_{q_0}$ and $\gamma$ both normalize $\4T$ by 
assumption or by construction, they also normalize $T=G\cap\4T$ and 
$A=O_p(T)$. By Proposition \ref{nonstandard}(d), $[\psi_{q_0}]$ generates 
the image of $\Phi_G$ in $\Out(G)/\Outdiag(G)$. 

We claim that in all cases,
	\beqq \Aut_G(A)=\Aut_{W_0}(A) \qquad\textup{and}\qquad
	\Aut_G(A)\cap\Aut\scal(A) \le \gen{\gamma|_A}\,. \label{e:4.15e} \eeqq
This holds by assumption in case \ref{messy_case}, and since $\ordp(q)$ 
is even or $-\Id\notin W_0$. In case \ref{minus_case}, the first 
statement holds by Lemma \ref{NG(T)}(b), and the second by Lemma 
\ref{scal-mod-m} (and since $W_0=W$ and 
$A$ contains all $p^k$-torsion in $\4T$).

\smallskip

\noindent\textbf{(c) } Assume in addition that $G$ is a Chevalley group. 
Thus $\gamma\in\Inn(\4G)$, so 
$\gamma|_G\in\Inndiag(G)=\Inn(G)\Aut_{\4T}(G)$ by Proposition 
\ref{nonstandard}(b), and hence $\gamma|_A\in\Aut_G(A)$. Also, 
$\gamma|_A=(\psi_q|_A)^{-1}=(\psi_{q_0}|_A)^{-b}$ since 
$\sigma=\gamma\psi_q$ centralizes $G\ge A$. Since $\psi_{q_0}|_A$ has order 
$b\cdot\ordp(q)$ in $\Aut\scal(A)$ by Lemma \ref{Phi->scal}(c), its class 
in $N_{\Aut(A)}(\Aut_G(A))/\Aut_G(A)$ has order $b$ by \eqref{e:4.15e}. 

Thus by \eqref{e:4.15c}, $\0\kappa_G$ sends $O_p(\Outdiag(G))\Phi_G$ 
injectively into $\Out(S,\calf)$. Since $\Gamma_G$ is isomorphic to $1$, 
$C_2$, or $\Sigma_3$ (and since $\0\kappa_G$ is onto by Proposition 
\ref{kappa_onto}), $\0\kappa_G$ and $\4\kappa_G$ are split. 

\smallskip

\noindent\textbf{(d) } Assume $G$ is a Steinberg group and 
$\ordp(q)$ is even. In this case, $\gamma\notin\Inn(\4G)$, and 
$\Out(G)/\Outdiag(G)\cong\Phi_G$ is cyclic of order $2b$, generated by the 
class of $\psi_{q_0}|_G$. Hence by \eqref{e:4.15c}, $\Ker(\0\kappa_G)$ is 
isomorphic to the subgroup of those $\psi\in\Phi_G$ such that 
$\psi|_A\in\Aut_G(A)$. By \eqref{e:4.15e} and since 
$\psi_q|_A=\gamma^{-1}|_A$, $\Aut_G(A)\cap\Aut\scal(A)\le\gen{\psi_q^A}$. 
Thus $|\Ker(\0\kappa_G)|\le2$, and 
	\beq |\Ker(\0\kappa_G)|=2 \quad\iff\quad 
	\gamma|_A\in\Aut_G(A) = \Aut_{W_0}(A)\,. \eeq
When $\Ker(\0\kappa_G)\ne1$, $\0\kappa_G$ is split if and only if 
$4\nmid|\Phi_G|=2b$; i.e., when $b$ is odd.
\end{proof}

In the situation of Proposition \ref{Ker(kappa)}(c), if $-\Id\notin W$, 
then $\Ker(\0\kappa_G)=\Gen{[\gamma_0\psi_{\sqrt{q}}]}$ where $\gamma_0$ is 
a nontrivial graph automorphism. If $-\Id\in{}W$ (hence $\ordp(q)$ is 
even), then $\0\kappa_G$ is always injective: either because $\Gamma_G=1$, 
or by the explicit descriptions in the next section of the setups when 
$\ordp(q)=2$ (Lemma \ref{invert(III.2)}), or when $\ordp(q)>2$ and 
$\gg=D_{2n}$ (Lemma \ref{classical(III.3)}).

The following examples help to illustrate some of the complications in the 
statement of Proposition \ref{Ker(kappa)}.

\begin{Ex} \label{ex:nonsplit}
Set $p=5$.  If $G=\Spin_{4k}^-(3^4)$, $\Sp_{2k}(3^4)$, or 
$\SU_k(3^4)$ ($k\ge5$), then by Proposition 
\ref{Ker(kappa)}(b), $\4\kappa_G$ is surjective 
but not split. (These groups satisfy case \ref{easy_case} of Hypotheses 
\ref{G-hypoth-X} by Lemma \ref{case(III.1)}.) The fusion systems of the 
last two are tamely realized by $\Sp_{2\ell}(3^2)$ and $\SL_n(3^2)$, 
respectively (these groups satisfy case \ref{minus_case} by Lemma 
\ref{invert(III.2)}, hence Proposition \ref{Ker(kappa)}(c) applies). 
The fusion system of $\Spin_{4k}^-(3^4)$ is also 
realized by $\Spin_{4k}^-(3^2)$, but not tamely (Example 
\ref{ex:nonsplit2}(b)). It is tamely realized by $\Spin_{4k-1}(3^2)$ 
(see Propositions \ref{OldPrA3}(c) and \ref{Ker(kappa)}(c)).\\
\end{Ex}


\newpage

\newsect{The cross characteristic case: II}
\label{s:X2}

In Section \ref{s:X1}, we established certain conditions on a finite group 
$G$ of Lie type in characteristic $q_0$, on a $\sigma$-setup for $G$, and 
on a prime $p\ne{}q_0$, and then proved that the $p$-fusion system of $G$ 
is tame whenever those conditions hold. It remains to prove that for each 
$G$ of Lie type and each $p$ different from the characteristic, there is 
another group $G^*$ whose $p$-fusion system is tame by the results of 
Section \ref{s:X1}, and is isomorphic to that of $G$. 

We first list the groups which satisfy case \ref{easy_case} of Hypotheses 
\ref{G-hypoth-X}.

\begin{Lem} \label{case(III.1)}
Fix a prime $p$ and a prime power $q\equiv1$ (mod $p$), where $q\equiv1$ 
(mod $4$) if $p=2$.  Assume $G\cong\gg(q)$ for some simple group scheme 
$\gg$ over $\Z$ of universal type, or $G\cong{}^2\gg(q)$ for $\gg\cong 
A_n$, $D_n$, or $E_6$ of universal type.  Then $G$ has a 
$\sigma$-setup $(\4G,\sigma)$ such that Hypotheses \ref{G-hypoth-X}, 
case \ref{easy_case} holds.
\end{Lem}

\begin{proof} Set $\4G=\gg(\4\F_q)$, and let $\psi_q\in\Phi_{\4G}$ be the field 
automorphism.  Set $\sigma=\gamma\psi_q\in\End(\4G)$, where $\gamma=\Id$ if 
$G\cong\gg(q)$, and $\gamma\in\Gamma_{\4G}$ has order $2$ if 
$G\cong{}^2\gg(q)$.  

\smallskip

\noindent\boldd{$N_G(T)$ contains a Sylow $p$-subgroup of $G$. } If 
$\gamma=\Id$, then by \cite[Theorem 9.4.10]{Carter} (and since $G$ is in 
universal form), $|G|=q^N\prod_{i=1}^r(q^{d_i}-1)$ for some integers 
$N,d_1,\ldots,d_r$ ($r=\rk(\gg)$), where $d_1d_2\cdots{}d_r=|W|$ by 
\cite[Theorem 9.3.4]{Carter}. Also, $|T|=(q-1)^r$, $N_G(T)/T\cong W$, and 
so 
	\[ v_p(|G|) = \sum_{i=1}^r v_p(q^{d_i}-1) 
	= \sum_{i=1}^r\bigl( v_p(q-1) + v_p(d_i) \bigr) 
	= v_p(|T|)+v_p(|W|) = v_p(N_G(T)) \,, \]
where the second equality holds by Lemma \ref{v(q^i-1)}.


If $|\gamma|=2$, then by \cite[\S\S\,14.2--3]{Carter}, for $N$ and $d_i$ as 
above, there are $\gee_i,\eta_i\in\{\pm1\}$ for $1\le i\le r$ such that 
$|G|=q^N\prod_{i=1}^r(q^{d_i}-\gee_i)$ and $|T|=\prod_{i=1}^r(q-\eta_i)$. 
(More precisely, the $\eta_i$ are the eigenvalues of the $\gamma$-action on 
$V$, and polynomial generators $I_1,\ldots,I_r\in\R[x_1,\ldots,x_r]^W$ can 
be chosen such that $\deg(I_i)=d_i$ and $\tau(I_i)=\gee_iI_i$.)  
By \cite[Proposition 14.2.1]{Carter}, 
	\[ |W_0| = \lim_{t\to1} \,
	\prod_{i=1}^r\Bigl(\frac{1-\gee_it^{d_i}}{1-\eta_it}\Bigr)
	\quad\implies\quad \renewcommand{\arraystretch}{1.5}
	\begin{array}{l}
	\bigl| \{1\le i\le r\,|\,\gee_i=1\}\bigr|
	=\bigl|\{1\le i\le r\,|\,\eta_i=1\}\bigr|\\
	\textup{and}\quad |W_0|=\prod\{d_i\,|\,\gee_i=+1\}.
	\end{array} \]
Also, $v_p(q^d+1)=v_p(q+1)$ for all $d\ge1$: they are both $0$ if $p$ 
is odd, and both $1$ if $p=2$. Hence
	\[ v_p(|G|) - v_p(|T|) = \sum_{\substack{i=1\\\gee_i=+1}}^r 
	v_p\Bigl(\frac{q^{d_i}-1}{q-1}\Bigr) = 
	\sum_{\substack{i=1\\\gee_i=+1}}^r v_p(d_i) = v_p(|W_0|)
	= v_p(|N_G(T)|) - v_p(|T|) \]
by Lemma \ref{v(q^i-1)} again, and so $N_G(T)$ contains a Sylow 
$p$-subgroup of $G$.

\smallskip

\noindent\boldd{The free $\gen{\gamma}$-orbit} $\{\alpha\}$ (if 
$\gamma=\Id$) or $\{\alpha,\tau(\alpha)\}$ (if $|\gamma|=2$ and 
$\alpha\ne\tau(\alpha)$), for any $\alpha\in\Sigma$, satisfies the 
hypotheses of this condition.

\smallskip

\noindent\boldd{$[\gamma,\psi_{q_0}]=\Id$ } since $\gamma\in\Gamma_{\4G}$.
\end{proof}

We are now ready to describe the reduction, when $p=2$, to groups with 
$\sigma$-setups satisfying Hypotheses \ref{G-hypoth-X}.

\begin{Prop} \label{G-cases-2}
Assume $G\in\Lie(q_0)$ is of universal type for some odd prime $q_0$. 
Fix $S\in\syl2{G}$, and assume $S$ is nonabelian.  Then there is an odd 
prime $q_0^*$, a group $G^*\in\Lie(q_0^*)$ of universal type, and $S^*\in\syl2{G^*}$, such 
that $\calf_S(G)\cong\calf_{S^*}(G^*)$, and $G^*$ has a $\sigma$-setup which 
satisfies case \ref{easy_case} of Hypotheses \ref{G-hypoth-X} and also Hypotheses 
\ref{G-hypoth-X2}. Moreover, if $G^*\cong G_2(q^*)$ where $q^*$ is a power 
of $q_0^*$, then we can arrange that either $q^*=5$ or $q_0^*=3$. 
\end{Prop}

\begin{proof} Since $q_0$ is odd, and since the Sylow $2$-subgroups of 
$\lie2G2(3^{2k+1})$ are abelian for all $k\ge1$ \cite[Theorem 8.5]{Ree-G2}, 
$G$ must be a Chevalley or Steinberg group. If $G\cong\lie3D4(q)$, then 
$\calf$ is also the fusion system of $G_2(q)$ by \cite[Example 4.5]{BMO1}. 
So we can assume that $G\cong{}^\tau\gg(q)$ for some odd prime power $q$, 
some $\gg$, and some graph automorphism $\tau$ of order $1$ or $2$. 

Let $\gee\in\{\pm1\}$ be such that $q\equiv\gee$ (mod $4$). By Lemma 
\ref{cond(i-iii)}, there is a prime $q_0^*$ and $k\ge0$ such that 
$\4{\gen{q}}=\4{\gen{\gee\cdot (q_0^*)^{2^k}}}$, where either $q_0^*=5$ and 
$k=0$, or $q_0^*=3$ and $k\ge1$. 

If $\gee=1$, then set $G^*={}^\tau\gg((q_0^*)^{2^k})$, and fix 
$S^*\in\syl2{G^*}$. Then $\calf_{S^*}(G^*)\cong\calf_S(G)$ by Theorem 
\ref{OldThA}(a), $G^*$ satisfies case \ref{easy_case} of Hypotheses 
\ref{G-hypoth-X} by Lemma \ref{case(III.1)} (and since $(q_0^*)^{2^k}\equiv1$ 
(mod $4$)), and $G^*$ also satisfies Hypotheses \ref{G-hypoth-X2}.

Now assume $\gee=-1$. If $-\Id$ is in the Weyl group of $G$, then set 
$G^*={}^\tau\gg((q_0^*)^{2^k})$. If $-\Id$ is not in the Weyl group, then 
$\gg=A_n$, $D_n$ for $n$ odd, or $E_6$, and we set $G^*=\gg((q_0^*)^{2^k})$ if 
$\tau\ne\Id$, and $G^*={}^2\gg((q_0^*)^{2^k})$ if $G=\gg(q)$. In all cases, for 
$S^*\in\sylp{G^*}$, $\calf_{S^*}(G^*)\cong\calf_S(G)$ by Theorem 
\ref{OldThA}(c,d), $G^*$ satisfies case \ref{easy_case} of Hypotheses 
\ref{G-hypoth-X} by Lemma \ref{case(III.1)} again, and also satisfies 
Hypotheses \ref{G-hypoth-X2}.

By construction, if $\gg=G_2$, then either $q_0^*=3$ or 
$(q_0^*)^{2^k}=5$. 
\end{proof}

When $G\cong G_2(5)$ and $p=2$, $G$ satisfies Hypotheses \ref{G-hypoth-X} 
and \ref{G-hypoth-X2}, but $\4\kappa_G$ is not shown to be surjective in 
Proposition \ref{kappa_onto} (and in fact, it is not surjective). Hence 
this case must be handled separately.

\begin{Prop} \label{p:G2(5)}
Set $G=G_2(5)$ and $G^*=G_2(3)$, and fix $S\in\syl2{G}$ and 
$S^*\in\syl2{G^*}$. 
Then $\calf_{S^*}(G^*)\cong\calf_{S}(G)$ as fusion systems, and 
$\4\kappa_{G^*}=\mu_{G^*}\circ\kappa_{G^*}$ is an isomorphism from 
$\Out(G^*)\cong C_2$ onto $\Out(S^*,\calf_{S^*}(G^*))$. 
\end{Prop}

\begin{proof} The first statement follows from Theorem \ref{OldThA}(c). 
Also, $|\Out(G)|=2$ and $|\Out(G^*)|=1$ by Theorem \ref{St-aut}, 
and since $G$ and $G^*$ have no field automorphisms and all diagonal 
automorphisms are inner (cf. \cite[3.4]{Steinberg-aut}), and $G=G_2(3)$ has a 
nontrivial graph automorphism while $G^*=G_2(5)$ does not 
\cite[3.6]{Steinberg-aut}.
Since $G$ satisfies Hypotheses \ref{G-hypoth-X} and \ref{G-hypoth-X2}, 
$|\Coker(\4\kappa_{G})|\le2$ by Proposition \ref{kappa_onto}, so 
$|\Out(S,\calf_{S}(G))|\le2$.

By \cite[Proposition 4.2]{O-rk4}, $S^*$ contains a unique subgroup $Q\cong 
Q_8\times_{C_2}Q_8$ of index $2$. Let $x\in{}Z(Q)=Z(S^*)$ be the central 
involution. Set $\4G=G_2(\4\F_3)>G^*$. Then $C_{\4G}(x)$ is connected since 
$\4G$ is of universal type \cite[Theorem 8.1]{Steinberg-end}, so 
$C_{\4G}(x)\cong\SL_2(\4\F_3)\times_{C_2}\SL_2(\4\F_3)$ by Proposition 
\ref{p:CG(T)}. Furthermore, any outer (graph) automorphism which 
centralizes $x$ exchanges the two central factors $\SL_2(\4\F_3)$. Hence 
for each $\alpha\in\Aut(G^*){\sminus}\Inn(G^*)$ which normalizes $S^*$, 
$\alpha$ exchanges the two factors $Q_8$, and in particular, does not 
centralize $S^*$. Thus $\4\kappa_{G^*}$ is injective, and hence an 
isomorphism since $|\Out(G^*)|=2$ and 
$|\Out(S^*,\calf_{S^*}(G^*))|=|\Out(S,\calf_{S}(G))|\le2$. 
\end{proof}

We now turn to case \ref{minus_case} of Hypotheses \ref{G-hypoth-X}. 

\begin{Lem} \label{invert(III.2)}
Fix an odd prime $p$, and an odd prime power $q$ prime to $p$ such that 
$q\equiv-1$ (mod $p$).  Let $G$ be one of the groups $\Sp_{2n}(q)$, 
$\Spin_{2n+1}(q)$, $\Spin_{4n}^+(q)$ ($n\ge2$), $G_2(q)$, $F_4(q)$, $E_7(q)$, 
or $E_8(q)$ (i.e., $G=\gg(q)$ for some $\gg$ whose Weyl group contains 
$-\Id$), and assume that the Sylow $p$-subgroups of $G$ are nonabelian. 
Then $G$ has a $\sigma$-setup $(\4G,\sigma)$ such that Hypotheses 
\ref{G-hypoth-X}, case \textup{\ref{minus_case}}, hold.
\end{Lem}

\begin{proof} Assume $q=q_0^b$ where $q_0$ is prime and $b\ge1$. Set 
$\4G=\gg(\fqobar)$, and let $\4T<\4G$ be a maximal torus. 
Set $r=\rk(\4T)$ and $k=v_p(q+1)$. 

In all of these cases, $-\Id\in{}W$, so there is a coset $w_0\in 
N_{\4G}(\4T)/\4T$ which inverts $\4T$. Fix $g_0\in{}N_{\4G}(\4T)$ such that 
$g_0\4T=w_0$ and $\psi_{q_0}(g_0)=g_0$ (Lemma \ref{l:gT}). Set 
$\gamma=c_{g_0}$ and $\sigma=\gamma\circ\psi_q$. We identify 
$G=O^{q_0'}(C_{\4G}(\sigma))$, $T=G\cap\4T$, and $A=O_p(T)$. Since 
$\sigma(t)=t^{-q}$ for each $t\in\4T$, $T\cong(C_{q+1})^r$ is the 
$(q+1)$-torsion subgroup of $\4T$, and $A\cong(C_{p^k})^r$. 

\smallskip

\noindent\boldd{$N_G(T)$ contains a Sylow $p$-subgroup of $G$. } In all 
cases, by \cite[Theorem 9.4.10]{Carter} (and since $G$ is in universal 
form), $|G|=q^N\prod_{i=1}^r(q^{d_i}-1)$, where $d_1d_2\cdots{}d_r=|W|$ by 
\cite[Theorem 9.3.4]{Carter}. Also, the $d_i$ are all even in the cases 
considered here (see \cite[Theorem 25]{Steinberg-lect} or 
\cite[Corollary 10.2.4 \& Proposition 10.2.5]{Carter}). Hence by Lemma 
\ref{v(q^i-1)} and since $p$ is odd,
	\begin{align*} 
	v_p(|G|) 
	&= \sum_{i=1}^r v_p(q^{d_i}-1) 
	= \sum_{i=1}^r v_p\bigl((q^2)^{d_i/2}-1\bigr) 
	= \sum_{i=1}^r \bigl(v_p(q^2-1) +  v_p(d_i/2)\bigr) \\
	&= r\cdot v_p(q+1) + \sum_{i=1}^r v_p(d_i) 
	= v_p(|T|) + v_p(|W|) = v_p(|N_G(T)|)\,. 
	\end{align*}


\smallskip

\noindent\boldd{$[\gamma,\psi_{q_0}]=\Id$ } since $\gamma=c_{g_0}$ and 
$\psi_{q_0}(g_0)=g_0$. 

\smallskip

\noindent\boldd{A free $\gen{\gamma}$-orbit in $\Sigma$. } For each 
$\alpha\in\Sigma$, $\{\pm\alpha\}$ is a free $\gen{\gamma}$-orbit.
\end{proof}

We now consider case \ref{messy_case} of Hypotheses \ref{G-hypoth-X}. By 
\cite[10-1,2]{GL}, when $p$ is odd, each finite group of Lie type has a 
$\sigma$-setup for which $N_G(T)$ contains a Sylow $p$-subgroup of $G$. 
Here, we need to construct such setups explicitly enough to be able to 
check that the other conditions in Hypotheses \ref{G-hypoth-X} hold.

When $p$ is a prime, $A$ is a finite abelian $p$-group, and 
$\Id\ne\xi\in\Aut(A)$ has order prime to $p$, we say that $\xi$ is a 
\emph{reflection} in $A$ if $[A,\xi]$ is cyclic.  In this case, there is a 
direct product decomposition $A=[A,\xi]\times C_A(\xi)$, and we call 
$[A,\xi]$ the \emph{reflection subgroup} of $\xi$. This terminology will be 
used in the proofs of the next two lemmas.

\newcommand{\XX}{\boldsymbol{\tau}}
\newcommand{\YY}{\boldsymbol{\xi}}

\begin{Lem} \label{classical(III.3)}
Fix an odd prime $p$, and an odd prime power $q$ prime to $p$ such that 
$q\not\equiv1$ (mod $p$).  Let $G$ be one of the classical groups 
$\SL_n(q)$, $\Sp_{2n}(q)$, $\Spin_{2n+1}(q)$, or $\Spin_{2n}^{\pm}(q)$, and 
assume that the Sylow $p$-subgroups of $G$ are nonabelian. Then $G$ has a 
$\sigma$-setup $(\4G,\sigma)$ such that case \ref{messy_case} of 
Hypotheses \ref{G-hypoth-X} holds.
\end{Lem}

\begin{proof} Set $m=\ord_p(q)$; $m>1$ by assumption. 
We follow Notation \ref{G-setup}, except that we have yet to fix the 
$\sigma$-setup for $G$. Thus, for example, $q_0$ is the prime of which 
$q$ is a power. 

When defining and working with the $\sigma$-setups for the spinor groups, 
it is sometimes easier to work with orthogonal groups than with their 
2-fold covers. For this reason, throughout the proof, we set $\gg\7=\SO_\ell$ 
when $\gg=\Spin_\ell$, set $\4G\7=\SO_\ell(\fqobar)$ when 
$\4G=\Spin_\ell(\fqobar)$, and let $\chi\:\4G\Right2{}\4G\7$ be the 
natural surjection. We then set $G\7=C_{\4G\7}(\sigma)\cong\SO_\ell^\pm(q)$, 
once $\sigma$ has been chosen so that 
$G=C_{\4G}(\sigma)\cong\Spin_\ell^\pm(q)$, and set $\4T\7=\chi(\4T)$ and 
$T\7=C_{\4T\7}(\sigma)$. Also, in order to prove the lemma without 
constantly considering these groups as a separate case, we set 
$\4G\7=\4G$, $G\7=G$, $\chi=\Id$, etc. when $G$ is linear or 
symplectic. Thus $G_c$ and $\4G_c$ are classical groups in all cases.

Regard $\4G\7$ as a subgroup of $\Aut(\4V,\bb)$, where $\4V$ is an 
$\fqobar$-vector space of dimension $n$, $2n$, or $2n+1$, and $\bb$ is a 
bilinear form. Explicitly, $\bb=0$ if $\gg=\SL_n$, and $\bb$ has 
matrix $\mxtwo01{-1}0^{\oplus n}$ if $\gg=\Sp_{2n}$, $\mxtwo0110^{\oplus 
n}$ if $\gg=\Spin_{2n}$, or $\mxtwo0110^{\oplus n}\oplus(1)$ if 
$\gg=\Spin_{2n+1}$.  Let $\4T\7$ be the group of diagonal matrices in $\4G\7$, 
and set 
	\[ [\lambda_1,\ldots,\lambda_n] = \begin{cases} 
	\diag(\lambda_1,\ldots,\lambda_n) & \textup{if $\gg=\SL_n$} \\
	\diag(\lambda_1,\lambda_1^{-1},\ldots,\lambda_n,\lambda_n^{-1}) & 
	\textup{if $\gg=\Sp_{2n}$ or $\Spin_{2n}$} \\
	\diag(\lambda_1,\lambda_1^{-1},\ldots,\lambda_n,\lambda_n^{-1},1) & 
	\textup{if $\gg=\Spin_{2n+1}$.} \\
	\end{cases} \]
In this way, we identify the maximal torus $\4T\7<\4G\7$ with 
$(\fqobar^\times)^n$ in the symplectic and orthogonal cases, and with 
a subgroup of $(\fqobar^\times)^n$ in the linear case.

Set $W^*=W$ (the Weyl group of $\gg$ and of $\gg\7$), except when 
$\gg=\Spin_{2n}$, in which case we let $W^*<\Aut(\4T\7)$ be the group of all 
automorphisms which permute and invert the coordinates. Thus in this last 
case, $W^*\cong\{\pm1\}\wr\Sigma_n$, while $W$ is the group of signed 
permutations which invert an even number of coordinates (so $[W^*:W]=2$). 
Since $W^*$ induces a group of isometries of the root system for $\Spin_{2n}$ 
and contains $W$ with index $2$, it is generated by $W$ and the restriction 
to $\4T\7$ of a graph automorphism of order $2$ (see, e.g., \cite[\S\,VI.1.5, 
Proposition 16]{Bourb4-6}). 

We next introduce some notation in order to identify certain elements in $W^*$.
For each $r,s$ such that $rs\le n$, let $\XX_r^s\in\Aut(\4T\7)$ be the Weyl 
group element induced by the permutation $(1\cdots 
r)({r{+}1}\cdots2r)\cdots ((s{-}1)r+1\cdots sr)$; i.e., 
	\[ \XX_r^s([\lambda_1,\ldots,\lambda_n]) = 
	[\lambda_r,\lambda_1,\ldots,\lambda_{r-1},
	\lambda_{2r},\lambda_{r+1},\ldots,
	\lambda_{sr},\lambda_{(s-1)r+1},\ldots,\lambda_{sr-1},
	\lambda_{sr+1},\ldots]. \]
For $1\le i\le n$, let $\YY_i\in\Aut(\4T)$ be the automorphism which 
inverts the $i$-th coordinate.  Set $\XX_{r,+1}^s=\XX_r^s$ and 
$\XX_{r,-1}^s=\XX_r^s\YY_r\YY_{2r}\cdots\YY_{sr}$. Thus for 
$\theta=\pm1$,
	\[ \XX_{r,\theta}^s([\lambda_1,\ldots,\lambda_n]) = 
	[\lambda_r^{\theta},\lambda_1,\ldots,\lambda_{r-1},
	\lambda_{2r}^{\theta},\lambda_{r+1},\ldots
	\lambda_{sr}^{\theta},\lambda_{(s-1)r+1},\ldots,\lambda_{sr-1},
	\lambda_{sr+1},\ldots]. \]

Recall that $m=\ordp(q)$. Define parameters $\mu$, $\theta$, $k$, and 
$\kappa$ as follows:
	\[ \renewcommand{\arraystretch}{1.5}
	\renewcommand{\arraycolsep}{5mm}
	\begin{array}{rllll}
	\textup{if $m$ is odd~:} & \mu=m &  \theta=1  & &
	\kappa=[n/\mu]=[n/m] \\
	\textup{if $m$ is even~:} & \mu=m/2 & \theta=-1 &  
	\raisebox{2.1ex}[0pt]{$k=[n/m]$} & \kappa=[n/\mu]=[2n/m] \,.
	\end{array} \]

We can now define our $\sigma$-setups for $G$ and $G\7$. 
Recall that we assume $m>1$. 
In Table \ref{tb:III.2a}, we define an element $w_0\in{}W^*$, and then 
describe $T\7=C_{\4T\7}(w_0\circ\psi_q)$ and $W_0^*=C_{W^*}(w_0)$ (where 
$W_0=C_W(w_0)$ has index at most $2$ in $W_0^*$). 
	\begin{table}[ht]
	\[ \renewcommand{\arraystretch}{1.5}
	\begin{array}{|c|c|c|c|c|} \hline
	G\7 & \textup{conditions} & w_0=\gamma|_{\4T\7} & T\7 & W_0^*  \\
	\hline\hline
	\SL_n(q) &  & \XX_m^k & 
	(C_{q^m-1})^k\times C_{q-1}^{n-mk-1} & (C_m\wr\Sigma_k)\times H
	\\\hline
	\Sp_{2n}(q) & & & & \\
	\cline{1-1}
	\SO_{2n+1}(q) &  & \XX_{\mu,\theta}^{\kappa} & 
	(C_{q^{\mu}-\theta})^\kappa\times C_{q-1}^{n-\kappa\mu} & 
	(C_{2\mu}\wr\Sigma_\kappa)\times H \\ \cline{1-2}
	 & \gee=\theta^\kappa & & & \\\cline{2-5} 
	 & {\gee\ne\theta^\kappa,~\mu{\nmid}n} & 
	\XX_{\mu,\theta}^{\kappa}\,\YY_n 
	& (C_{q^{\mu}-\theta})^{\kappa}\times C_{q-1}^{n-\kappa\mu-1} 
	\times C_{q+1}
	& (C_{2\mu}\wr\Sigma_{\kappa})\times H \\\cline{2-5}
	\halfup{\SO_{2n}^{\gee}(q)} & 
	\dbl{\gee\ne\theta^\kappa,~ \mu|n}{\theta=-1} & 
	\XX_{\mu,\theta}^{\kappa-1}
	& (C_{q^{\mu}-\theta})^{\kappa-1}\times C_{q-1}^{\mu} 
	&  \\\cline{2-4}
	 & \dbl{\gee\ne\theta^\kappa,~ \mu|n}{\theta=+1} & 
	\XX_{\mu,\theta}^{\kappa-1}\,\YY_n 
	& (C_{q^{\mu}-\theta})^{\kappa-1}\times C_{q-1}^{\mu-1} \times 
	C_{q+1} 
	& \raisebox{3.0ex}[0pt]{$(C_{2\mu}\wr\Sigma_{\kappa-1})\times H$} 
	\\\hline
	\multicolumn{5}{|c|}{\textup{\small{In all cases, 
	$T\xrightarrow{\chi}T\7$ has kernel and cokernel of order $\le2$, 
	and so $A=O_p(T)\cong O_p(T\7)$.}}} \\\hline
	\end{array}
	\]
	\caption{} \label{tb:III.2a}
	\end{table}
In all cases, we choose $\gamma\in\Aut(\4G\7)$ as follows. Write 
$w_0=w_0'\circ\gamma_0|_{\4T\7}$ for some $w'_0\in{}W$ and 
$\gamma_0\in\Gamma_{\4G\7}$ (possibly $\gamma_0=\Id$). Choose 
$g_0\in{}N_{\4G\7}(\4T\7)$ such that $g_0\4T\7=w'_0$ and $\psi_{q_0}(g_0)=g_0$ 
(Lemma \ref{l:gT}), and set $\gamma=c_{g_0}\circ\gamma_0$. Then 
$[\gamma,\psi_{q_0}]=\Id_{\4G\7}$, since $c_{g_0}$ and $\gamma_0$ both commute 
with $\psi_{q_0}$, and we set $\sigma=\gamma\circ\psi_q=\psi_q\circ\gamma$. 
When $\gg=\Spin_{2n}$ or $\Spin_{2n+1}$, since $\4G$ is a perfect group 
and $\Ker(\chi)\le Z(\4G)$, $\gamma$ and $\sigma$ lift to unique 
endomorphisms of $\4G$ which we also denote $\gamma$ and $\sigma$ (and 
still $[\gamma,\psi_{q_0}]=1$ in $\Aut(\4G)$). 

Thus $G\cong C_{\4G}(\sigma)$ and $G\7\cong C_{\4G\7}(\sigma)$ in all 
cases, and we identify these groups. Set $T=C_{\4T}(\sigma)$, 
$T\7=C_{\4T\7}(\sigma)$, $W_0^*=C_{W^*}(\gamma)$, and $W_0=C_{W}(\gamma)$. 
If $\gg=\Spin_{2n+1}$ or $\Spin_{2n}$, then $\chi(T)$ is the kernel of 
the homomorphism $T\7\to\Ker(\chi)$ which sends $\chi(t)$ to 
$t^{-1}\sigma(t)$, and thus has index at most $2$ in $T\7$. Since $p$ is 
odd, this proves the statement in the last line of Table \ref{tb:III.2a}.

In the description of $W_0^*$ in Table \ref{tb:III.2a}, $H$ always denotes 
a direct factor of order prime to $p$. The first factor in the description 
of $W_0^*$ acts on the first factor in that of $T$, and $H$ acts on the 
other factors. 

When $G\7=\SL_n(q)$ and $m|n$, the second factor $C_{q-1}^{-1}$ in the 
description of $T$ doesn't make sense. It should be interpreted to mean 
that $T$ is ``a subgroup of index $q-1$ in the first factor 
$(C_{q^m-1})^k$''. 

Recall that $T\7=C_{\4T\7}(\gamma\circ\psi_q)$.  When 
$U=(\fqobar^\times)^{\mu}$, then 
	\begin{align*} 
	C_U(\XX_{\mu,\theta}^1\circ\psi_q) 
	&= \bigl\{(\lambda,\lambda^q,\lambda^{q^2},\ldots, 
	\lambda^{q^{\mu-1}}) \,\big|\, 
	(\lambda^{q^{\mu-1}})^{q\theta}=\lambda \bigr\} \\
	&= \bigl\{(\lambda,\lambda^q,\lambda^{q^2},\ldots, 
	\lambda^{q^{\mu-1}} ) \,\big|\, \lambda^{q^{\mu}-\theta}=1 \bigr\} 
	\cong C_{q^{\mu}-\theta}\,. 
	\end{align*}
This explains the description of $T\7$ in the 
symplectic and orthogonal cases: it is always the direct product of 
$(C_{q^{\mu}-\theta})^{\kappa}$ or $(C_{q^{\mu}-\theta})^{\kappa-1}$ with a 
group of order prime to $p$. (Note that $p|(q{+}1)$ only when $m=2$; i.e., 
when $\theta=-1$ and $1=\mu|n$.)

Since the cyclic permutation $(1\,2\,\cdots\,\mu)$ generates its own 
centralizer in $\Sigma_{\mu}$, the centralizer of $\XX_{\mu,\theta}^1$ in 
$\{\pm1\}\wr\Sigma_{\mu}<\Aut((\fqobar^\times)^{\mu})$ is generated by 
$\XX_{\mu,\theta}^1$ and $\psi_{-1}^{\4T}$. If $\theta=-1$, then 
$(\XX_{\mu,\theta}^1)^{\mu}=\psi_{-1}^{\4T}$, while if $\theta=1$, then 
$\XX_{\mu,\theta}^1$ has order $\mu$. Since $m=\mu$ is odd in the latter 
case, the centralizer is cyclic of order $2\mu$ in both cases. This is why, 
in the symplectic and orthogonal cases, the first factor in $W_0^*$ is 
always a wreath product of $C_{2\mu}$ with a symmetric group.


We are now ready to check the conditions in case \ref{messy_case} of 
Hypotheses \ref{G-hypoth-X}. 

\noindent\boldd{$N_{G}(T)$ contains a Sylow $p$-subgroup of $G$. } 
Set 
	\[ e = v_p(q^m-1) = v_p(q^\mu-\theta)\,. \]
The second equality holds since if $2|m$, then 
$p\nmid(q^{\mu}-1)$ and hence $e=v_p(q^{\mu}+1)$. 
Recall also that $m|(p-1)$, so $v_p(m)=0$. Consider the information listed 
in Table \ref{tb:III.2b}, where the formulas for $v_p(|T|)=v_p(|T\7|)$ and 
$v_p(|W_0|)$ follow from Table \ref{tb:III.2a}, and those for $|G|$ 
are shown in \cite[Theorems 25 \& 35]{Steinberg-lect} and also in 
\cite[Corollary 10.2.4, Proposition 10.2.5 \& Theorem 14.3.2]{Carter}.
	\begin{table}[ht]
	\[ \renewcommand{\arraystretch}{1.5}
	\begin{array}{|c|c|c|c|c|} \hline
	G & \textup{cond.} & v_p(|G|) & v_p(|T|) & v_p(|W_0|) \\
	\hline\hline
	\SL_n(q) &  & \sum_{i=2}^nv_p(q^i-1) &
	ke & v_p(k!) \\\hline
	\Sp_{2n}(q) & & & &  \\
	\cline{1-1}
	\Spin_{2n+1}(q) &  
	& \halfup{\sum_{i=1}^nv_p(q^{2i}-1)} &  &  \\
	\cline{1-3}
	 & \gee=\theta^\kappa &  & \halfup{\kappa e} & 
	\halfup{v_p(\kappa!)} \\
	\cline{2-2}
	\Spin_{2n}^\gee(q) & \gee\ne\theta^\kappa,~ \mu{\nmid} n & 
	\halfup{v_p(q^n-\gee)}\hfill
	&  &  \\\cline{2-2}\cline{4-5}
	 & \gee\ne\theta^\kappa,~ \mu|n
	& \halfup{\quad+\sum_{i=1}^{n-1}v_p(q^{2i}-1)}
	& (\kappa{-}1)e & v_p((\kappa{-}1)!) \\\hline
	\end{array}
	\]
	\caption{} \label{tb:III.2b}
	\end{table}

For all $i>0$, we have
	\[ v_p(q^i-1)=\begin{cases} 
	e+v_p(i/m) & \textup{if $m|i$}\\ 0 & \textup{if $m{\nmid}i$.}
	\end{cases} \]
The first case follows from Lemma \ref{v(q^i-1)}, and the second case 
since $m=\ordp(q)$. Using this, we check that 
$v_p(q^{2i}-1)=v_p(q^i-1)$ for all $i$ whenever $m$ is odd, and that 
	\[ v_p(q^n-\gee)=\begin{cases} 
	e+v_p(2n/m) & \textup{if $m|2n$ and $\gee=(-1)^{2n/m}$}\\ 
	0 & \textup{otherwise.}
	\end{cases} \]
So in all cases, $v_p(|G|)=v_p(|T|)+v_p(|W_0|)$ by the above relations 
and the formulas in Table \ref{tb:III.2b}. Since $N_G(T)/T\cong W_0$ by 
Lemma \ref{NG(T)}(b), this proves that $v_p(|G|)=v_p(|N_{G}(T)|)$, 
and hence that $N_{G}(T)$ contains a Sylow $p$-subgroup of $G$. 

\smallskip

\noindent\boldd{$\bigl|\gamma|_{\4T}\bigr|=|\tau|=\ordp(q)\ge2$ and 
$[\gamma,\psi_{q_0}]=\Id$ } by construction. Note, when 
$G$ is a spinor group, that these relations hold in $\4G$ if and only 
if they hold in $\4G\7$. 

\smallskip

\noindent\boldd{$C_S(\Omega_1(A))=A$ } by Table \ref{tb:III.2a} and since 
$p\nmid|H|$.

\smallskip

\noindent\boldd{$C_A(O_{p'}(W_0))=1$. } By Table \ref{tb:III.2a}, in all 
cases, there are $r,t\ge1$ and $1\ne s|(p-1)$ such that $A\cong(C_{p^t})^r$, 
and $\Aut_{W_0^*}(A)\cong C_s\wr\Sigma_r$ acts on $A$ by acting on and 
permuting the cyclic factors. In particular, $\Aut_{O_{p'}(W_0)}(A)$ 
contains a subgroup of index at most $2$ in $(C_s)^r$, this subgroup acts 
nontrivially on each of the cyclic factors in $A$, and hence 
$C_A(O_{p'}(W_0))=1$.

\smallskip

\noindent\boldd{A free $\gen{\gamma}$-orbit in $\Sigma$. } This can be 
defined as described in Table \ref{tb:III.2x}. In each case, we use the 
notation of Bourbaki \cite[pp. 250--258]{Bourb4-6} for the roots of $\gg$. 
Thus, for example, the roots of $\SL_n$ are the $\pm(\gee_i-\gee_j)$ for 
$1\le i<j\le n$, and the roots of $\SO_{2n}$ the $\pm\gee_i\pm\gee_j$. Note 
that since $S$ is assumed nonabelian, $p\big||W_0|$, and hence $n\ge pm$ in the 
linear case, and $n\ge p\mu$ in the other cases. 
	\begin{table}[ht]
	\[ \renewcommand{\arraystretch}{1.5}
	\begin{array}{|c||c|c|} \hline
	G & \theta=1 & \theta=-1 \\
	\hline\hline
	\SL_n(q) & \multicolumn{2}{c|}{\{\gee_i-\gee_{m+i}\,|\,1\le i\le m\}} \\\hline
	\Sp_{2n}(q) & \{2\gee_i\,|\,1\le i\le \mu\} & 
	\{\pm2\gee_i\,|\,1\le i\le \mu\} \\ \hline
	\Spin_{2n+1}(q) &  \{\gee_i\,|\,1\le i\le \mu\} & 
	\{\pm \gee_i\,|\,1\le i\le \mu\} \\ \hline
	\Spin_{2n}^\gee(q) & \{\gee_i-\gee_{\mu+i}\,|\,1\le i\le \mu\} & 
	\{\pm(\gee_i-\gee_{\mu+i})\,|\,1\le i\le \mu\} \\ \hline
	\end{array} \]
	\caption{} \label{tb:III.2x}
	\end{table}

\smallskip

\noindent\boldd{$\Aut_{W_0}(A)\cap\Aut\scal(A)\le
\begin{cases} 
\gen{\gamma|_A} & \textup{if $\ordp(q)$ even or $-\Id\notin W_0$}\\
\gen{\gamma|_A,\psi_{-1}^A} & \textup{otherwise.}
\end{cases}$}\\
Set $K^*=\Aut_{W_0^*}(A)\cap\Aut\scal(A)$ and 
$K=\Aut_{W_0}(A)\cap\Aut\scal(A)$ for short. By Table \ref{tb:III.2a}, 
$|K^*|=m$ if $G\cong\SL_n(q)$, and $|K^*|=2\mu$ otherwise. Also, 
$\gen{\gamma|_A}=\gen{\psi_q^{-1}|_A}$ has order $\ordp(q)$. Thus $K\le 
K^*=\gen{\gamma|_A}$ except when $G$ is symplectic or orthogonal and 
$m=\ordp(q)$ is odd. In this last case, $K=K^*$ (so $|K|=2\mu=2m$) if 
$W_0$ contains an element which inverts $A$ (hence which inverts 
$T$ and $\4T$); and $|K^*/K|=2$ ($|K|=m$) otherwise.

\smallskip

\noindent\boldd{$\Aut_{G}(A)=\Aut_{W_0}(A)$. } Since $A=O_p(T)\cong 
O_p(T\7)$ by Table \ref{tb:III.2a}, it suffices to prove this for $G\7$. Fix 
$g\in{}N_{G\7}(A)$.  Since $\9g{\4T\7}$ is a maximal torus in the 
algebraic group $C_{\4G\7}(A)$ (Proposition \ref{p:CG(T)}), there is 
$b\in{}C_{\4G\7}(A)$ such that $\9b{\4T\7}=\9g{\4T\7}$.  Set 
$a=b^{-1}g\in{}N_{\4G\7}(\4T\7)$; thus $c_a=c_g\in\Aut(A)$.  Set 
$w=a\4T\7\in{}W=N_{\4G\7}(\4T\7)/\4T\7$; thus $w\in{}N_W(A)$, and 
$w|_A=c_g|_A$. 

By the descriptions in Table \ref{tb:III.2a}, we can factor 
$\4T\7=\4T_1\times\4T_2$, where $\gamma$ and each element of $N_W(A)$ 
send each factor to itself, $\gamma|_{\4T_2}=\Id$, $A\le\4T_1$, and 
$[C_W(A),\4T_1]=1$.  Since $\sigma(g)=g$, $\sigma(a)\equiv a$ (mod 
$C_{\4G\7}(A)$), and so $\tau(w)\equiv w$ (mod $C_W(A)$).  Thus 
$\tau(w)|_{\4T_1}=w|_{\4T_1}$ since $C_W(A)$ acts trivially on this 
factor, $\tau(w)|_{\4T_2}=w|_{\4T_2}$ since $\gamma|_{\4T_2}=\Id$, 
and so $w\in{}W_0=C_W(\tau)$. 

\smallskip

\noindent\boldd{$N_{\Aut(A)}(\Aut_{W_0}(A))\le\Aut\scal(A) 
\Aut_{\Aut(G)}(A)$. } By Table \ref{tb:III.2a}, for some $r,t\ge1$, 
$A=A_1\times\cdots\times A_r$, where $A_i\cong C_{p^t}$ for each $i$. Also, 
for some $1\ne s|(p-1)$, $\Aut_{W_0^*}(A)\cong C_s\wr\Sigma_r$ acts on $A$ 
via faithful actions of $C_s$ on each $A_i$ and permutations of the $A_i$. 

Let $\Aut^0_{W_0^*}(A)\nsg\Aut_{W_0^*}(A)$ and 
$\Aut^0_{W_0}(A)\nsg\Aut_{W_0}(A)$ be the subgroups of elements which 
normalize each cyclic subgroup $A_i$. Thus $\Aut^0_{W_0^*}(A)\cong(C_s)^r$, 
and contains $\Aut^0_{W_0}(A)$ with index at most $2$.

\noindent\textbf{Case 1: } Assume first that $\Aut^0_{W_0}(A)$ is 
characteristic in $\Aut_{W_0}(A)$. Fix some 
$\alpha\in{}N_{\Aut(A)}(\Aut_{W_0}(A))$. We first show that $\alpha\in 
\Aut_{W_0^*}(A)\Aut\scal(A)$.

Since $\alpha$ normalizes $\Aut_{W_0}(A)$, it also normalizes 
$\Aut^0_{W_0}(A)$. For each $\beta\in\Aut^0_{W_0}(A)$, $[\beta,A]$ is a 
product of $A_i$'s. Hence the factors $A_i$ are characterized as the 
minimal nontrivial intersections of such $[\beta,A]$, and are 
permuted by $\alpha$. So after composing with an appropriate element of 
$\Aut_{W_0^*}(A)$, we can assume that $\alpha(A_i)=A_i$ for each $i$.

After composing $\alpha$ by an element of $\Aut\scal(A)$, 
we can assume that $\alpha|_{A_1}=\Id$.  Fix $i\ne1$ ($2\le i\le r$), 
let $u\in\Z$ be such that $\alpha|_{A_i}=\psi_u^{A_i}=(a\mapsto a^u)$, and 
choose $w\in\Aut_{W_0}(A)$ such that $w(A_1)=A_i$.  Then 
$w^{-1}\alpha w\alpha^{-1}\in\Aut_{W_0}(A)$ since $\alpha$ normalizes 
$\Aut_{W_0}(A)$, and 
$\bigl(w^{-1}\alpha w\alpha^{-1}\bigr)\big|_{A_1} = \psi_u^{A_1}$. Hence 
$u^s\equiv1$ (mod $p^t=|A_1|$), and since this holds for each $i$, 
$\alpha\in\Aut_{W^*_0}(A)$. 

Thus $N_{\Aut(A)}(\Aut_{W_0}(A))\le\Aut_{W_0^*}(A)\Aut\scal(A)$. By Table 
\ref{tb:III.2a}, each element of $\Aut_{W^*_0}(A)$ extends to some 
$\varphi\in\Aut_{W^*}(\4T)$ which commutes with $\sigma|_{\4T}$. So 
$\Aut_{W^*_0}(A)\le\Aut_{\Aut(G)}(A)$ by Lemma \ref{4T->4G}, and this 
finishes the proof of the claim. 

\noindent\textbf{Case 2: } Now assume that $\Aut^0_{W_0}(A)$ is not 
characteristic in $\Aut_{W_0}(A)$. Then $r\le4$, and since $p\le r$, we 
have $p=3$ and $r=3,4$. Also, $s=2$ since $s|(p-1)$ and $s\ne1$. Thus 
$r=4$, since $\Aut^0_{W_0}(A)=O_2(\Aut_{W_0}(A))$ if $r=3$. Thus 
$\Aut_{W_0}(A)\cong 
C_2^3\sd{}\Sigma_4$: the Weyl group of $D_4$. Also, $m=2$ since $p=3$, so 
(in the notation used in the tables) $\mu=1$, $\theta=-1$, and $\kappa=n$. 
By Table \ref{tb:III.2a}, $G\cong\SO_8(q)$ for some $q\equiv2$ (mod $3$) 
(and $W_0=W$). 

Now, $O_2(W)\cong Q_8\times_{C_2}Q_8$, and so $\Out(O_2(W))\cong\Sigma_3\wr 
C_2$. Under the action of $W/O_2(W)\cong\Sigma_3$, the elements of order 
$3$ act on both central factors and those of order $2$ exchange the 
factors. (This is seen by computing their centralizers in $O_2(W)$.) It 
follows that 
$N_{\Out(O_2(W))}(\Out_W(O_2(W)))/\Out_W(O_2(W))\cong\Sigma_3\cong\Gamma_{G}$, 
and all classes in this quotient extend to graph 
automorphisms of $G\cong\Spin_8(q)$. So for each $\alpha\in 
N_{\Aut(A)}(\Aut_{W}(A))$, after composing with a graph automorphism of 
$G$ we can arrange that $\alpha$ commutes with $O_2(W)$, and in 
particular, normalizes $\Aut^0_W(A)$. Hence by the same argument as used in 
Case 1, $\alpha\in\Aut\scal(A)\Aut_{\Aut(G)}(A)$. 


This finishes the proof that this $\sigma$-setup for $G$ 
satisfies case \ref{messy_case} of Hypotheses \ref{G-hypoth-X}.
\end{proof}

\begin{Ex} \label{ex:nonsplit2}
Fix distinct odd primes $p$ and $q_0$, and a prime power $q=q_0^b$ where 
$b$ is even and 
$\ordp{q}$ is even. Set $G=\Spin_{4k}^-(q)$ for some $k\ge2$. Let 
$(\4G,\sigma)$ be the setup for $G$ of Lemma \ref{classical(III.3)}, where 
$\sigma=\psi_q\gamma$ for $\gamma\in\Aut(\4G)$. In the notation 
of Table \ref{tb:III.2a}, $m=\ordp(q)$, $\mu=m/2$, $\theta=-1=\gee$, 
$n=2k$, and $\kappa=[2k/\mu]=[4k/m]$. There are three cases to consider:
\begin{enuma} 

\item If $q^{2k}\equiv-1$ (mod $p$); equivalently, if $m|4k$ and 
$\kappa=4k/m$ is odd, then $\gee=\theta^\kappa$, 
$w_0=\gamma|_{\4T\7}=\XX_{\mu,\theta}^\kappa$, $\rk(A)=\kappa$, and 
$W_0^*\cong C_m\wr\Sigma_\kappa$. Then $W_0^*$ acts faithfully on $A$ while 
$w_0\in W_0^*{\sminus}W_0$, and so $\gamma|_A\notin\Aut_{W_0}(A)$. Hence by 
Proposition \ref{Ker(kappa)}(d), $\4\kappa_G$ is split.

\item If $q^{2k}\equiv1$ (mod $p$); equivalently, if $m|4k$ and 
$\kappa=4k/m$ is even, then $\gee\ne\theta^\kappa$, 
$\gamma|_{\4T\7}=\XX_{\mu,\theta}^{\kappa-1}$, $\rk(A)=\kappa-1$, and 
$W_0^*\cong (C_m\wr\Sigma_{\kappa-1})\times H$ where 
$H\cong(C_2\wr\Sigma_\mu)$. Then $H$ acts trivially on $A$ and contains 
elements in $W_0^*{\sminus}W_0$, so $\gamma|_A\in\Aut_{W_0}(A)$. Hence 
$\4\kappa_G$ is not split.

\item If $q^{2k}\not\equiv\pm1$ (mod $p$); equivalently, if $m\nmid4k$, 
then in either case ($\kappa$ even or odd), the factor $H$ in the last 
column of Table \ref{tb:III.2a} is nontrivial, acts trivially on $A$, and 
contains elements in $W_0^*{\sminus}W_0$. Hence $\gamma|_A\in\Aut_{W_0}(A)$ 
in this case, and $\4\kappa_G$ is not split. 

\end{enuma}
\end{Ex}

We also need the following lemma, which handles the only case of a 
Chevalley group of exceptional type which we must show satisfies case 
\ref{messy_case} of Hypotheses \ref{G-hypoth-X}. 

\begin{Lem} \label{E8(III.3)}
Set $p=5$, let $q$ be an odd prime power such that $q\equiv\pm2$ (mod $5$), 
and set $G=E_8(q)$.  Then $G$ has a $\sigma$-setup which satisfies 
Hypotheses \ref{G-hypoth-X} (case \ref{messy_case}).
\end{Lem}

\begin{proof} We use the notation in \ref{G-setup}, where $q$ is a power of 
the odd prime $q_0$, and $\4G=E_8(\4\F_{q_0})$. 

By \cite[Planche VII]{Bourb4-6}, the of roots of $E_8$ can be given 
the following form, where $\{\gee_1,\dots,\gee_8\}$ denotes the standard orthonormal 
basis of $\R^8$:
	\[ \Sigma = \Bigl\{\pm \gee_i\pm \gee_j \,\Big|\, 1\le i<j\le8\Bigr\} \cup
	\Bigl\{ \frac12\sum_{i=1}^8 (-1)^{m_i}\gee_i \,\Big|\, 
	\sum_{i=1}^8m_i \textup{ even} \Bigr\} \subseteq \R^8~. \]
By the same reference, the Weyl group $W$ is the group of all automorphisms 
of $\R^8$ which permute $\Sigma$ ($A(R)=W(R)$ in the notation of 
\cite{Bourb4-6}). Give $\R^8$ a complex structure by setting 
$i\gee_{2k-1}=\gee_{2k}$ and $i\gee_{2k}=-\gee_{2k-1}$, and set 
$\gee^*_k=\gee_{2k-1}$ for $1\le k\le4$.  Multiplication by $i$ permutes 
$\Sigma$, and hence is the action of an element $w_0\in{}W$. Upon writing 
the elements of $\Sigma$ with complex coordinates,  we get the following 
equivalent subset $\Sigma^*\subseteq\C^4$:
	\begin{multline*}  
	\Sigma^* = \Big\{(\pm1\pm{}i)\gee^*_k\,\Big|\,1\le k\le4\Bigr\} 
	\cup \Big\{i^m\gee^*_k+i^n\gee^*_\ell
	\,\Big|\,1\le k<\ell\le4,~ m,n\in\Z\Bigr\} \\ \cup
	\Big\{\frac{1+i}2\sum_{k=1}^4i^{m_k}\gee^*_k\,\Big|\,\sum m_k 
	\textup{ even}\Bigr\} ~. 
	\end{multline*}

Let $\Z\Sigma\subseteq\R^8$ be the lattice generated by $\Sigma$. By Lemma 
\ref{theta-r}\eqref{hb.hc} (and since $(\alpha,\alpha)=2$ for all 
$\alpha\in\Sigma$), we can identify 
$\4T\cong\Z\Sigma\otimes_{\Z}\fqobar^\times$ by sending $h_\alpha(\lambda)$ 
to $\alpha\otimes\lambda$ for $\alpha\in\Sigma$ and 
$\lambda\in\fqobar^\times$. Set $\Lambda_0=\Z\Sigma\cap\Z^8$, a lattice in 
$\R^8$ of index $2$ in $\Z\Sigma$ and in $\Z^8$. The inclusions of lattices 
induce homomorphisms
	\[ \4T \cong \Z\Sigma\otimes_{\Z}\fqobar^\times \Left5{\chi_1} 
	\Lambda_0\otimes_{\Z}\fqobar^\times \Right5{\chi_2} 
	\Z^8\otimes_{\Z}\fqobar^\times \cong (\fqobar^\times)^8 \]
each of which is surjective with kernel of order $2$ (since 
$\textup{Tor}^1_{\Z}(\Z/2,\fqobar^\times)\cong\Z/2$). 
We can thus identify $\4T=(\fqobar^\times)^8$, modulo $2$-power torsion, 
in a way so that 
	\[ \alpha=\sum_{i=1}^8k_i\gee_i\in\Sigma,~ \lambda\in\fqobar^\times 
	\quad\implies\quad h_\alpha(\lambda) = 
	(\lambda^{k_1},\ldots,\lambda^{k_8})\,. \]
Under this identification, by the formula in Lemma \ref{theta-r}\eqref{^txb(u)}, 
	\beqq \beta=\sum_{i=1}^8\ell_i\gee_i\in\Sigma 
	\quad\implies\quad \theta_\beta(\lambda_1,\ldots,\lambda_8) = 
	\lambda_1^{\ell_1}\cdots\lambda_8^{\ell_8}  \label{e:theta} \eeqq
for $\lambda_1,\ldots,\lambda_8\in\fqobar^\times$.  Also, 
	\[ w_0(\lambda_1,\ldots,\lambda_8)
	=(\lambda_2^{-1},\lambda_1,\lambda_4^{-1},
	\lambda_3,\ldots,\lambda_8^{-1},\lambda_7) \] 
for each $(\lambda_1,\ldots,\lambda_8)$.  


Choose $g_0\in N_{\4G}(\4T)$ such that $g_0\4T=w_0$ and 
$\psi_{q_0}(g_0)=g_0$ (Lemma \ref{l:gT}), and set 
$\gamma=c_{g_0}\in\Inn(\4G)$. Thus 
$\sigma=\psi_q\circ\gamma=\gamma\circ\psi_q$, $G=C_{\4G}(\sigma)$, and 
$T=C_{\4T}(\sigma)$. By the Lang-Steinberg theorem \cite[Theorem 
10.1]{Steinberg-end}, there is $h\in\4G$ such that $g=h\psi_q(h^{-1})$; 
then $\sigma=c_h\psi_qc_h^{-1}$ and $G\cong C_{\4G}(\psi_q)=E_8(q)$. It 
remains to check that the setup $(\4G,\sigma)$ satisfies the list of 
conditions in Hypotheses \ref{G-hypoth-X}.

We identify $W_0=C_W(w_0)$ with the group of $\C$-linear automorphisms of 
$\C^4$ which permute $\Sigma^*$.  The order of $W_0$ is computed in 
\cite[Table 11]{Carter-cj} (the entry $\Gamma=D_4(a_1)^2$), but since we 
need to know more about its structure, we describe it more precisely here.  
Let $W_2\le\GL_4(\C)$ be the group of monomial matrices with nonzero 
entries $\pm1$ or $\pm i$, and with determinant $\pm1$.  Then $W_2\le W_0$, 
$|W_2|=\frac12\cdot4^4\cdot4!=2^{10}\cdot3$, and $W_2$ acts on $\Sigma^*$ 
with three orbits corresponding to the three subsets in the above 
description of $\Sigma^*$.  The (complex) reflection of order 2 in the 
hyperplane orthogonal to $\frac{1+i}2(\gee^*_1+\gee^*_2+\gee^*_3+\gee^*_4)$ 
sends $(1+i)\gee^*_1$ to 
$\frac{1+i}2(\gee^*_1-\gee^*_2-\gee^*_3-\gee^*_4)$, and it sends 
$(\gee^*_1+i\gee^*_2)$ to 
$\frac{1+i}2(i^3\gee^*_1+i\gee^*_2-\gee^*_3-\gee^*_4)$.  Thus $W_0$ acts 
transitively on $\Sigma^*$.  

Let $\4\Sigma\subseteq P(\C^4)$ be the set of projective points representing 
elements of $\Sigma^*$, and let $[\alpha]\in\4\Sigma$ denote the class of 
$\alpha\in{}\Sigma^*$.  To simplify notation, we also write $[x]=[\alpha]$ for 
$x\in\C^4$ representing the same point, also when $x\notin \Sigma^*$.  Let 
$\sim$ denote the relation on $\4\Sigma$:  $[\alpha]\sim[\beta]$ if 
$\alpha=\beta$, or if $\alpha\perp\beta$ and the projective line $\Gen{[\alpha],[\beta]} 
\subseteq P(\C^4)$ contains four other points in $\4\Sigma$.  By inspection, 
$[\gee^*_j]\sim[\gee^*_k]$ for all $j,k\in\{1,2,3,4\}$, and these are the only 
elements $[\alpha]$ such that $[\alpha]\sim[\gee^*_j]$ for some $j$.  Since 
this relation is preserved by $W_0$, and $W_0$ acts transitively on $\4\Sigma$, 
we see that $\sim$ is an equivalence relation on $\4\Sigma$ with 15 classes of 
four elements each.  Set $\Delta=\4\Sigma/{\sim}$, and let $[\alpha]_\Delta$ 
denote the class of $[\alpha]$ in $\Delta$.  Thus $|\4\Sigma|=\frac14|\Sigma|=60$ and 
$|\Delta|=15$.  Since $W_2$ is the stabilizer subgroup of $[\gee^*_1]_\Delta$ 
under the transitive $W_0$-action on $\Delta$, we have
$|W_0|=|W_2|\cdot15=2^{10}\cdot3^2\cdot5$.

Let $W_1\nsg W_0$ be the subgroup of elements which act trivially on 
$\Delta$.  By inspection, $W_1\le W_2$, $|W_1|=2^6$, and $W_1$ is generated 
by $w_0=\diag(i,i,i,i)$, $\diag(1,1,-1,-1)$, $\diag(1,-1,1,-1)$, and the 
permutation matrices for the permutations $(1\,2)(3\,4)$ and 
$(1\,3)(2\,4)$.  Thus $W_1\cong C_4\times_{C_2}D_8\times_{C_2}D_8$.  

By the above computations, $|W_0/W_1|=2^4\cdot3^2\cdot5=|\Sp_4(2)|$.  
There is a bijection from $\Delta$ to the set of maximal isotropic 
subspaces in $W_1/Z(W_1)$ which sends a class $[\alpha]_\Delta$ to the 
subgroup of those elements in $W_1$ which send each of the four projective 
points in $[\alpha]_\Delta$ to itself.  Hence for each 
$w\in{}C_{W_0}(W_1)$, $w$ acts via the identity on $\Delta$, and so 
$w\in{}W_1$ by definition.  Thus $W_0/W_1$ injects into 
$\Out(W_1)\cong\Sigma_6\times C_2$, and injects into the first factor since 
$Z(W_1)=Z(W_0)$ ($\cong C_4$).  So by counting, $W_0/W_1\cong\Sigma_6$.  
Also, $W_1=O_2(W_0)$. 

Set $a=v_5(q^4-1)=v_5(q^2+1)$, and fix $u\in\4\F_{q_0}^\times$ of order 
$5^a$. Let $A$ be as in Notation \ref{G-setup-X}\eqref{not8x}: the subgroup 
of elements in $T$ of $5$-power order. Thus 
	\beqq A = \bigl\{ (u_1,u_1^q,u_2,u_2^q,u_3,u_3^q,u_4,u_4^q) 
	\,\big|\, u_1,u_2,u_3,u_4\in\gen{u}\bigr\} \cong (C_{5^a})^4 ~. 
	\label{e:A} \eeqq
By \eqref{e:A} and \eqref{e:theta}, there is no $\beta\in\Sigma$ such that 
$A\le\Ker(\theta_\beta)$.  Hence $C_{\4G}(A)^0=\4T$ by Proposition 
\ref{p:CG(T)}.  So by Lemma \ref{NG(T)}(b), 
	\beqq N_G(A)=N_G(T)\quad \textup{and}\quad
	N_G(T)/T = W_0\,. 
	\label{e:NG(T)} \eeqq

We are now ready to check the conditions in Case \ref{messy_case} of Hypotheses 
\ref{G-hypoth-X}.

\smallskip

\noindent\boldd{$N_G(T)$ contains a Sylow $p$-subgroup of $G$. } Let $S$ be 
a Sylow $p$-subgroup of $N_G(T)$ which contains $A$. Since $N_G(T)/T=W_0$ 
by \eqref{e:NG(T)}, 
$A\cong(C_{5^a})^4$, and $W_0/O_{2}(W_0)\cong\Sigma_6$, $|S|=5^{4a+1}$. By 
\cite[Theorem 25]{Steinberg-lect} or \cite[Corollary 10.2.4 \& Proposition 
10.2.5]{Carter}, and since $v_5(q^k-1)=0$ when $4\nmid k$ and 
$v_5(q^{4\ell}-1)=a+v_5(\ell)$ (Lemma \ref{v(q^i-1)}),
	\[ v_5(|G|)=v_5\bigl((q^{24}-1)(q^{20}-1)(q^{12}-1)(q^8-1)\bigr)
	=4a+1\,. \]  
Thus $S\in\sylp{G}$.

\smallskip

\noindent\boldd{$\bigl|\gamma|_{\4T}\bigr|=\ordp(q)\ge2$ and
$[\gamma,\psi_{q_0}]=\Id$. } The first is clear, and
the second holds since $\gamma=c_{g_0}$ where $\psi_{q_0}(g_0)=g_0$. 

\noindent\boldd{$C_S(\Omega_1(A))=A$ } by the above description of the 
action of $W_0$ on $A$.

\smallskip

\noindent\boldd{$C_A(O_{p'}(W_0))=1$ } since $w_0\in O_{5'}(W_0)$ and 
$C_A(w_0)=1$.

\smallskip

\noindent\boldd{A free $\gen{\gamma}$-orbit in $\Sigma$. } The subset 
$\{\pm(\gee_1+\gee_3),\pm(\gee_2+\gee_4)\}\subseteq\Sigma$ is a free 
$\gen{\gamma}$-orbit.

\smallskip

\noindent\boldd{$\Aut_{W_0}(A)\cap\Aut\scal(A)\le\gen{\gamma|_A}$. } Recall 
that $\bigl|\gamma|_{\4T}\bigr|=4$ and $|\Aut\scal(A)|=4\cdot5^k$ for some 
$k$, and $W_0$ acts faithfully on $A$. So if this is not true, then there 
is an element of order $5$ in $Z(W_0)$, which is impossible by the above 
description of $W_0$.

\smallskip

\noindent\boldd{$\Aut_G(A)=\Aut_{W_0}(A)$ } by \eqref{e:NG(T)}.


\smallskip


\noindent\boldd{$N_{\Aut(A)}(\Aut_{W_0}(A))\le\Aut\scal(A)\Aut_{W_0}(A)$. } 
For $j=1,2,3,4$, let $A_j<A$ be the cyclic subgroup of all elements as in 
\eqref{e:A} where $u_k=1$ for $k\ne{}j$. The group $W_0$ contains as 
subgroup $C_2\wr\Sigma_4$:  the group which permutes pairs of coordinates 
up to sign.  So each of the four subgroups $A_j$ is the reflection subgroup 
of some reflection in $W_0$.  

For each $\varphi\in{}C_{\Aut(A)}(\Aut_{W_0}(A))$, $\varphi(A_j)=A_j$ 
for each $j$, and $\varphi(a)=a^{n_j}$ for some $n_j\in(\Z/5^a)^\times$.  
Also, $n_1=n_2=n_3=n_4$ since the $A_j$ are permuted transitively by 
elements of $W_0$, and hence $\varphi\in\Aut\scal(A)$.

Now assume $\varphi\in N_{\Aut(A)}(\Aut_{W_0}(A))$.  Since $\varphi$ 
centralizes $Z(W_1)=\gen{w_0}=\gen{\diag(i,i,i,i)}$ (since $\diag(i,i,i,i)\in 
Z(\Aut(A))$), $c_\varphi|_{W_1}\in\Inn(W_1)$, and we can assume (after 
composing by an appropriate element of $W_1$) that $[\varphi,W_1]=1$.  So 
$c_\varphi\in\Aut(W_0)$ has the form 
$c_\varphi(g)=g\chi(\4g)$, where $\4g\in{}W_0/W_1\cong\Sigma_6$ is the 
class of $g\in{}W_0$, and where 
$\chi\in\Hom(W_0/W_1,Z(W_1))\cong\Hom(\Sigma_6,C_4)\cong C_2$ is some 
homomorphism.  Since $(w_0)^2$ inverts the torus $T$, composition with 
$(w_0)^2$ does not send reflections (in $A$) to reflections, and so we 
must have $c_\varphi=\Id_{W_0}$.  Thus $\varphi\in 
C_{\Aut(A)}(\Aut_{W_0}(A))=\Aut\scal(A)$ (modulo $\Aut_{W_0}(A)$).
\end{proof}

The following lemma now reduces the proof of Theorem \ref{ThX} 
to the cases considered in Section \ref{s:X1}, together with certain small 
cases handled at the end of this section.  As before, when $p$ is a prime 
and $p\nmid n$, $\ordp(n)$ denotes the multiplicative order of $n$ in 
$\F_p^\times$. 

\begin{Prop} \label{G-cases-odd}
Fix an odd prime $p$, and assume $G\in\Lie(q_0)$ is of universal type for 
some prime $q_0\ne{}p$.   
Fix $S\in\sylp{G}$, and assume $S$ is nonabelian. Then 
there is a prime $q_0^*\ne{}p$, a group $G^*\in\Lie(q_0^*)$ of 
universal type, and $S^*\in\sylp{G^*}$, such that 
$\calf_S(G)\cong\calf_{S^*}(G^*)$, 
and one of the following holds: either 
\begin{enuma} 
\item $G^*$ has a $\sigma$-setup which satisfies Hypotheses 
\ref{G-hypoth-X} and \ref{G-hypoth-X2}, $G^*\cong\gg(q^*)$ or 
${}^2\gg(q^*)$ where $q^*$ is a power of $q_0^*$, and \medskip
\begin{enumerate}[label=\rm(a.\arabic*) ,leftmargin=10mm]
\item $-\Id\notin W$ and $G^*$ is a Chevalley group, or
\item $-\Id\in W$ and $\ordp(q^*)$ is even, or
\item $p\equiv3$ (mod $p$) and $\ordp(q^*)=1$
\end{enumerate}
\smallskip

\noindent where $W$ is the Weyl group of $\gg$; \quad or

\item $p=3$, $q_0^*=2$, $G\cong\lie3D4(q)$ or $\lie2F4(q)$ for 
$q$ some power of $q_0$, and $G^*\cong\lie3D4(q^*)$ or $\lie2F4(q^*)$ for 
$q^*$ some power of $2$.

\end{enuma}
Moreover, if $p=3$ and $G^*=F_4(q^*)$ where $q^*$ is a power of $q_0^*$, 
then we can assume $q_0^*=2$. In all cases, we can choose $G^*$ to be 
either one of the groups listed in Proposition \ref{list-simp}(a--e), or 
one of $E_7(q^*)$ or $E_8(q^*)$ for some $q^*\equiv-1$ (mod $p$).
\end{Prop}

\begin{proof} We can assume that $G=\gg(q)$ is one of the groups listed in 
one of the five cases (a)--(e) of Proposition \ref{list-simp}. In all cases 
except \ref{list-simp}(c), we can also assume that $G$ satisfies Hypotheses 
\ref{G-hypoth-X2}, with $q_0=2$ if $p=3$ and $\gg=F_4$, and with $q_0$ odd 
in cases (a) and (b) of \ref{list-simp}. If $G=\SL_n(q)$ or 
$\Spin_{2n}^\pm(q)$ where $p|(q-1)$, or $G$ is in case (d), then $G$ 
satisfies Hypotheses \ref{G-hypoth-X} by Lemma \ref{case(III.1)}. If 
$G\cong\SL_n(q)$ or $\Spin_{2n}^\pm(q)$ where $p\nmid(q-1)$, then $G$ 
satisfies Hypotheses \ref{G-hypoth-X} by Lemma \ref{classical(III.3)}. This 
leaves only case (c) in Proposition \ref{list-simp}, which corresponds to 
case (b) here, and case (e) ($p=5$, $G=E_8(q)$, $q\equiv\pm2$ (mod $5$)) 
where $G^*$ satisfies Hypotheses \ref{G-hypoth-X} by Lemma \ref{E8(III.3)}. 

We next show, in cases (a,b,d,e) of Proposition \ref{list-simp}, that we 
can arrange for one of the conditions (a.1), (a.2), or (a.3) to hold. If 
$-\Id\notin W$, then $\gg=A_n$, $D_n$ for $n$ odd, or $E_6$, and $G$ is a 
Chevalley group by the assumptions in cases (a,b,d) of Proposition 
\ref{list-simp}. So (a.1) holds. If $-\Id\in W$ and $\ord_p(q)$ is even, 
then (a.2) holds, while if $p\equiv3$ (mod $4$) and $\ordp(q)=1$, then 
(a.3) holds. 

By inspection, we are left with the following two cases:
\begin{enumerate}[label=\rm(\alph*$'$) ,ref=\alph*,itemsep=6pt,leftmargin=*]
\stepcounter{enumi}
\item $G=\gg(q)$, where $\gg=\Spin_{2n}^+$, $n\ge p$ is even, 
$\ordp(q)$ is odd, and $q^n\equiv1$ (mod $p$); or 

\stepcounter{enumi}

\item $G=\gg(q)$ where $\gg=G_2$, $F_4$, $E_6$, $E_7$, or $E_8$, 
$p\equiv1$ (mod $4$), $p|(q-1)$, and $p\bmid|W(\gg)|$.
\end{enumerate}
In case (d$'$), the above conditions leave only the possibility $p=5$ and 
$\gg=E_7$ or $E_8$ (see the computations of $|W|$ in, e.g., 
\cite[\S\,VI.4]{Bourb4-6}). In either case, by Lemma 
\ref{cond(i-iii)}(a), we can choose a prime power $q^\vee$ which satisfies 
Hypotheses \ref{G-hypoth-X2} and such that $\4{\gen{q^\vee}}=\4{\gen{-q}}$ 
in $\Z_p^\times$, and set $G^\vee=\gg(q^\vee)$. Then $G^\vee\sim_pG$ by 
Theorem \ref{OldThA}(c), $\ord_p(q^\vee)$ is even, so (a.2) holds, and 
$G^\vee$ satisfies Hypotheses \ref{G-hypoth-X} by Lemma \ref{invert(III.2)} 
or \ref{classical(III.3)}. 
\end{proof}

We now consider the two families of groups which appear in Proposition 
\ref{G-cases-odd}(b): those not covered by Hypotheses \ref{G-hypoth-X}.

\begin{Prop} \label{p:lowrank}
Let $G$ be one of the groups $\lie3D4(q)$ where $q$ is a prime power prime 
to $3$, $\lie2F4(2^{2m+1})$ for $m\ge0$, or $\lie2F4(2)'$.  Then the 
$3$-fusion system of $G$ is tame. If $G\cong\lie3D4(2^n)$ ($n\ge1$), 
$\lie2F4(2^{2m+1})$ ($m\ge0$), or $\lie2F4(2)'$, then $\kappa_G$ is split 
surjective, and $\Ker(\kappa_G)$ is the subgroup of field automorphisms of 
order prime to $3$.
\end{Prop}

\begin{proof} Fix $S\in\syl3{G}$, and set $\calf=\calf_S(G)$.  

If $G$ is the Tits group $\lie2F4(2)'$, then $S$ is extraspecial of order 
$3^3$ and exponent $3$, so $\Out(S)\cong\GL_2(3)$.  Also, $\Out_G(S)\cong 
D_8$ and $\Out_{\Aut(G)}(S)\cong SD_{16}$, since the normalizer in 
$\lie2F4(2)$ of an element of order $3$ (the element $t_4$ in 
\cite{Shinoda}) has the form $\SU_3(2):2\cong3^{1+2}_+:\SD_{16}$ by 
\cite[Table IV]{Shinoda} or \cite[Proposition 1.2]{Malle}. Hence 
$\Out(S,\calf)\le N_{\Out(S)}(\Out_G(S))/\Out_G(S)$ has order at most $2$, 
and $\4\kappa_G$ sends $\Out(G)\cong C_2$ (\cite[Theorem 2]{GrL}) 
isomorphically to $\Out(S,\calf)$. If $G=\lie2F4(2)$, then 
$\Out_G(S)\cong\SD_{16}$, so $\Out(S,\calf)=1$ by a similar argument, and 
$\kappa_G$ is an isomorphism between trivial groups. 

Assume now that $G\cong\lie2F4(2^n)$ for odd $n\ge3$ or 
$G\cong\lie3D4(q)$ where $3\nmid q$. 
In order to describe the Sylow $3$-subgroups of these groups, set 
$\zeta=e^{2\pi{}i/3}$, $R=\Z[\zeta]$, and $\pp=(1-\zeta)R$.  Let 
$S_k$ be the semidirect product $R/\pp^k\rtimes{}C_3$, where the 
quotient acts via multiplication by $\zeta$.  Explicitly, set 
	\[ S_k = \{(x,i) \,|\, x\in{}R/\pp^k,\ i\in\Z/3 \} 
	\qquad\textup{and}\qquad
	A_k = R/\pp^k \times\{0\}, \]
where $(x,i)(y,j)=(x+\zeta^iy,i+j)$. Thus $|S_k|=3^{k+1}$. Set $s=(0,1)$, 
so that $s(x,0)s^{-1}=(\zeta{}x,0)$ for each $x\in{}R/\pp^k$. 

Assume $k\ge3$, so that $A_k$ is the unique abelian subgroup of index 
three in $S_k$.  Set $S=S_k$ and $A=A_k$ for short.  We want to describe 
$\Out(S)$.  Define automorphisms $\xi_a$ ($a\in(R/\pp^k)^\times$), $\omega$, 
$\eta$, and $\rho$ by setting
	\beqq \xi_a(x,i)=(xa,i),\quad \eta=\xi_{-1},\quad 
	\omega(x,i)=(-\widebar{x},-i),\quad 
	\rho(x,i)=(x+\lambda(i),i). \label{e:lowrank} \eeqq
Here, $x\mapsto\widebar{x}$ means complex conjugation, and 
$\lambda(i)=1+\zeta+\ldots+\zeta^{i-1}$.  Note, when checking that $\rho$ 
is an automorphism, that $\lambda(i)+\zeta^i\lambda(j)=\lambda(i+j)$. Note 
that $\rho^3\in\Inn(S)$: it is (left) conjugation by $(1-\zeta^2,0)$. 

Let $\Aut^0(S)\nsg\Aut(S)$ be the subgroup of automorphisms 
which induce the identity on $S/[S,S]=S/[s,A]$, and set 
$\Out^0(S)=\Aut^0(S)/\Inn(S)$.  Each element in $s{\cdot}[s,A]$ is 
conjugate to $s$, and thus each class in $\Out^0(S)$ is represented by an 
automorphism which sends $s$ to itself, which is unique modulo $\gen{c_s}$.  
If $\varphi\in\Aut(S)$ and $\varphi(s)=s$, then $\varphi|_A$ commutes with 
$c_s$, thus is $R$-linear under the identification $A\cong{}R/\pp^k$, and 
hence $\varphi=\xi_a$ for some $a\in1+\pp/\pp^k$.  Moreover, since 
	\[ (1+\pp/\pp^k)^\times = (1+\pp^2/\pp^k)^\times \times \gen{\zeta}
	= (1+3R/\pp^k)^\times \times \gen{\zeta} \]
as multiplicative groups (just compare orders, noting that the groups on 
the right have trivial intersection), each class in $\Out^0(S)$ is 
represented by $\xi_a$ for some unique $a\in1+3R/\pp^k$.  

Since the images of $\eta$, $\omega$, and $\rho$ generate 
$\Aut(S)/\Aut^0(S)$ (the group of automorphisms of $S/[s,A]\cong C_3^2$ 
which normalize $A/[s,A]\cong C_3$), this shows that $\Out(S)$ is generated 
by the classes of the automorphisms in \eqref{e:lowrank}.  In fact, a 
straightforward check of the relations among them shows that 
	\[ \Out(S) \cong \Bigl( 
	\underset{}{\Out^0(S)} 
	\rtimes \underset{[\omega]}{C_2}\Bigr) \times 
	\underset{[\rho],[\eta]}{\Sigma_3}
	\qquad\textup{where}\qquad
	\Out^0(S)=\bigl\{[\xi_a]\,\big|\,a\in(1+3R/\pp^k)^\times\bigr\}. \]
Also, $\omega\xi_a\omega^{-1}=\xi_{\bar{a}}$ for $a\in(1+3R/\pp^k)^\times$. 

For each $x\in1+3R$ such that $\widebar{x}\equiv{}x$ (mod $\pp^k$), we can 
write $x=r+s\zeta$ with $r,s\in\Z$, and then 
$s(\zeta-\overline{\zeta})\in{}\pp^k$, so $s\in{}\pp^{k-1}$, and 
$x\in{}r+s+\pp^k\subseteq1+3\Z+\pp^k$.  This proves that 
	\[ C_{\Out(S)}(\omega) = 
	\bigl\{[\xi_a]\,\big|\,a\in\Z\bigr\}\times 
	\Gen{[\omega]}\times\Gen{[\rho],[\eta]}. \]

For any group $G$ with $S\in\syl3{G}$ and $S\cong{}S_k$, $\Out_G(S)$ has 
order prime to $3$, and hence is a 2-group and conjugate to a subgroup of 
$\gen{\omega,\eta}\in\syl2{\Out(S)}$. If $|\Out_G(S)|=4$, then we can 
identify $S$ with $S_k$ in a way so that $\Out_G(S)=\Gen{[\omega],[\eta]}$.  
Then 
	\begin{align*} 
	\Out(S,\calf) &\le 
	N_{\Out(S)}\bigl(\Gen{[\omega],[\eta]}\bigr)/\Gen{[\omega],[\eta]} 
	\\
	&= C_{\Out(S)}\bigl(\Gen{[\omega],[\eta]}\bigr)/\Gen{[\omega],[\eta]}
	= \bigl\{ [\xi_a ]\,\big|\, a\in\Z \bigr\} = \Gen{[\xi_2]}\,, 
	\end{align*}
where the first equality holds since $O_3(\Out(S))$ has index 
four in $\Out(S)$. 

We are now ready to look at the individual groups. Assume $G=\lie2F4(q)$, 
where $q=2^n$ and $n\ge3$ is odd. By \cite[3.2--3.6]{Steinberg-aut}, 
$\Out(G)$ is cyclic of order $n$, generated by the field automorphism 
$\psi_2$.  By the main theorem in \cite{Malle}, there is a subgroup 
$\caln_G(T_8)\cong(C_{q+1})^2\sd{}\GL_2(3)$, the normalizer of a maximal 
torus, which contains a Sylow $3$-subgroup. Hence if we set 
$k=v_3(q+1)=v_3(4^n-1)=1+v_3(n)$ (Lemma \ref{v(q^i-1)}), we have $S\cong 
S_{2k}\cong(C_{3^k})^2\sd{}C_3$, and $\Out_G(S)=\gen{\omega,\eta}$ up to 
conjugacy.  So $\Out(S,\calf)$ is cyclic, generated by 
$\xi_2=\kappa_G(\psi_2)$.  Since $A\cong(C_{3^{k}})^2$, and since 
$\xi_{-1}\in\Out_G(S)$, $|\Out(S,\calf)|=|[\xi_2]|=3^{k-1}$ where 
$k-1=v_3(n)$. Thus $\4\kappa_G$ is surjective, and is split since the Sylow 
3-subgroup of $\Out(G)\cong C_n$ is sent isomorphically to $\Out(S,\calf)$.

Next assume $G=\lie3D4(q)$, where $q=2^n$ for $n\ge1$. By 
\cite[3.2--3.6]{Steinberg-aut}, $\Out(G)$ is cyclic of order $3n$, 
generated by the field automorphism $\psi_2$ (and where the field 
automorphism $\psi_{2^n}$ of order three is also a graph automorphism). Set 
$k=v_3(q^2-1)=v_3(2^{2n}-1)=1+v_3(n)$ (Lemma \ref{v(q^i-1)}). Then $S\cong 
S_{2k+1}$: this follows from the description of the Sylow structure in $G$ 
in \cite[10-1(4)]{GL}, and also from the description (based on 
\cite{Kleidman-3D4}) of its fusion system in \cite[Theorem 2.8]{ind-p12} 
(case (a.ii) of the theorem). Also, $\Out_G(S)=\gen{\omega,\eta}$ up to 
conjugacy.  So $\Out(S,\calf)$ is cyclic, generated by 
$\xi_2=\kappa_G(\psi_2)$.  Since $A\cong{}C_{3^{k}}\times{}C_{3^{k+1}}$, 
and since $\xi_{-1}\in\Out_G(S)$, $|\Out(S,\calf)|=|[\xi_2]|=3^{k}$. Thus 
$\4\kappa_G$ is surjective, and is split since the Sylow 3-subgroup of 
$\Out(G)\cong C_{3n}$ is sent isomorphically to $\Out(S,\calf)$.


By Theorem \ref{OldThA}(b) and Lemma \ref{cond(i-iii)}(a), for each prime 
power $q$ with $3\nmid q$, the $3$-fusion system of $\lie3D4(q)$ is 
isomorphic to that of $\lie3D4(2^n)$ for some $n$.  By \cite[Theorem 
C]{limz-odd}, $\mu_G$ is injective in all cases.  Thus the $3$-fusion 
systems of all of these groups are tame. 
\end{proof}


\newpage

\appendix

\newcommand{\qq}{\mathfrak{q}}
\renewcommand{\ggg}{\mathfrak{g}}
\newcommand{\defect}{\delta}

\renewcommand{\2}[1]{\textup{\textbf{2\uppercase{#1}}}}
\renewcommand{\3}[1]{\textup{\textbf{3\uppercase{#1}}}}
\newcommand{\6}[1]{{\mathversion{bold}\textbf{#1}}}
\def\7[#1,#2,#3]{\def\test{#3}\def\tst{}\ifx\test\tst
 \textup{\6{[$#1$]$^{#2}$}}\else\textup{\6{[$#1$]$^{#2}_{(#3)}$}}\fi}

\newsect{Injectivity of $\mu_G$\\Bob Oliver}
\label{s:mu}

Recall that for any finite group $G$ and any $S\in\sylp{G}$, 
	\[ \mu_G \: \Out\typ(\call_S^c(G)) \Right6{} \Out(S,\calf_S(G)) \]
is the homomorphism which sends the class of 
$\beta\in\Aut\typ^I(\call_S^c(G))$ to the class of $\beta_S|_S$, where 
$\beta_S$ is the induced automorphism of 
$\Aut_{\call_S^c(G)}(S)=N_G(S)/O_{p'}(C_G(S))$. We need to develop tools 
for computing $\Ker(\mu_G)$, taking as starting point \cite[Proposition 
4.2]{AOV1}.

As usual, for a finite group $G$ and a prime $p$, a proper subgroup $H<G$ is 
\emph{strongly $p$-embedded} in $G$ if $p\big||H|$, and $p\nmid|H\cap\9gH|$ 
for $g\in G{\sminus}H$. The following properties of groups with 
strongly embedded subgroups will be needed.

\begin{Lem} \label{l:str.emb.} 
Fix a prime $p$ and a finite group $G$. 
\begin{enuma} 

\item If $G$ contains a strongly $p$-embedded subgroup, then $O_p(G)=1$. 

\item If $H<G$ is strongly $p$-embedded, and $K\nsg{}G$ is a normal 
subgroup of order prime to $p$ such that $KH<G$, then $HK/K$ is strongly 
$p$-embedded in $G/K$.

\end{enuma}
\end{Lem}

\begin{proof} \textbf{(a) } See, e.g., \cite[Proposition A.7(c)]{AKO}. 

\smallskip

\noindent\textbf{(b) } Assume otherwise. Thus there is $g\in{}G{\sminus}HK$ 
such that $p\big||(\9gHK/K)\cap(HK/K)|$, and hence $x\in\9gHK\cap{}HK$ of 
order $p$. Then $H\cap K\gen{x}$ and $\9gH\cap K\gen{x}$ have order a 
multiple of $p$, so there are elements $y\in H$ and $z\in\9gH$ of order $p$ 
such that $y\equiv x\equiv z$ (mod $K$). 

Since $\gen{y},\gen{z}\in\sylp{K\gen{x}}$, there is $k\in{}K$ such that 
$\gen{y}=\9k\gen{z}$.  Then $y\in{}H\cap\9{kg}H$, and $kg\notin H$ since 
$k\in{}K$ and $g\notin HK$. But this is impossible, since $H$ is strongly 
$p$-embedded. 
\end{proof}

For the sake of possible future applications, we 
state the next proposition in terms of abstract 
fusion and linking systems. We refer to \cite{AOV1}, and also to Chapters 
I.2 and III.4 in \cite{AKO}, for the basic definitions. Recall that if 
$\calf$ is a fusion system over a finite $p$-group $S$, and $P\le 
S$, then 
\begin{itemize} 
\item $P$ is \emph{$\calf$-centric} if $C_S(Q)\le Q$ for each $Q$ which is
$\calf$-conjugate to $P$;
\item $P$ is \emph{fully normalized in $\calf$} if $|N_S(P)|\ge|N_S(Q)|$ 
whenever $Q$ is $\calf$-conjugate to $P$; and 
\item $P$ is \emph{$\calf$-essential} if $P<S$, $P$ is $\calf$-centric and 
fully normalized in $\calf$, and if $\outf(P)$ contains a strongly 
$p$-embedded subgroup. 
\end{itemize}

For any saturated fusion system $\calf$ over a finite $p$-group $S$, set 
	\begin{multline*} 
	\5\calz(\calf) = \bigl\{E\le S\,\big|\, \textup{$E$ 
	elementary abelian, fully normalized in $\calf$,} \\
	\textup{$E=\Omega_1(Z(C_S(E)))$, 
	$\autf(E)$ has a strongly $p$-embedded subgroup} \bigr\} \,.
	\end{multline*}
The following proposition is our main tool for proving that $\mu_\call$ is 
injective in certain cases. Point (a) will be used to handle the groups 
$\Spin_n^\pm(q)$, point (c) the linear and symplectic groups, and point (b) 
the exceptional Chevalley groups.

\begin{Prop} \label{Ker(mu)}
Fix a saturated fusion system $\calf$ over a $p$-group $S$ and 
an associated centric linking system $\call$.
Let $E_1,\ldots,E_k\in\5\calz(\calf)$ 
be such that each $E\in\5\calz(\calf)$ is $\calf$-conjugate to $E_i$ for 
some unique $i$.  For each $i$, set $P_i=C_S(E_i)$ and $Z_i=Z(P_i)$. Then the 
following hold.
\begin{enuma} 
\item If $k=0$ ($\5\calz(\calf)=\emptyset$), then $\Ker(\mu_\call)=1$.

\item If $k=1$, $E_1\nsg S$, and 
$\autf(\Omega_1(Z(S)))=1$, then $\Ker(\mu_\call)=1$.

\item Assume, for each $(g_i)_{i=1}^k \in \prod_{i=1}^k 
C_{Z_i}(\Aut_S(P_i))$, that there is $g\in C_{Z(S)}(\autf(S))$ such that 
$g_i\in g\cdot C_{Z_i}(\autf(P_i))$ for each $i$.  Then 
$\Ker(\mu_\call)=1$.

\end{enuma}
\end{Prop}

\begin{proof} We first show that (a) and (b) are special cases of (c), and 
then prove (c). That (a) follows from (c) is immediate.

\textbf{(b) } If $k=1$, $E_1\nsg S$, and $\autf(\Omega_1(Z(S)))=1$, then 
the group $\outf(S)$ of order prime to $p$ acts trivially on 
$\Omega_1(Z(S))$, and hence acts trivially on $Z(S)$ (cf. \cite[Theorem 
5.2.4]{Gorenstein}). Also, $P_1=C_S(E_1)\nsg S$, so 
$C_{Z_1}(\Aut_S(P_1))=Z(S)=C_{Z(S)}(\autf(S))$, and $\Ker(\mu_\call)=1$ by 
(c). 

\smallskip

\noindent\textbf{(c) } Fix a class $[\alpha]\in\Ker(\mu_\call)$. By 
\cite[Proposition 4.2]{AOV1}, there is an automorphism 
$\alpha\in\Aut\typ^I(\call)$ in the class $[\alpha]$ such that 
$\alpha_S=\Id_{\Aut_\call(S)}$.  By the same proposition, there are 
elements $g_P\in C_{Z(P)}(\Aut_S(P))$, defined for each $P\le S$ which is 
fully normalized and $\calf$-centric, such that 
\begin{enumr} 
\item $\alpha_P\in\Aut(\Aut_\call(P))$ is conjugation by $\delta_P(g_P)$; 

\item $\alpha_P=\Id$ if and only if $g_P\in C_{Z(P)}(\autf(P))$; and 

\item if $Q<P$ are both fully normalized and $\calf$-centric, then 
	\[ g_P\equiv{}g_Q \pmod{C_{Z(Q)}(N_{\autf(P)}(Q))}. \]

\end{enumr}
Furthermore, 
\begin{enumr}[resume]
\item $[\alpha]=1\in\Out\typ(\call)$ if and only if there is $g\in 
C_{Z(S)}(\autf(S))$ such that $g_P\in g\cdot C_{Z(P)}(\autf(P))$ for each 
$P<S$ such that $P$ is $\calf$-essential and $P=C_S(\Omega_1(Z(P)))$.
\end{enumr}
Set $g_i=g_{P_i}$ ($1\le i\le k$) for short.  

By hypothesis, we can assume there is an element $g\in{}C_{Z(S)}(\autf(S))$ 
such that $g_i\in{}g{\cdot}C_{Z(P_i)}(\autf(P_i))$ for each $i$.  Upon 
replacing $\alpha$ by its composite with $c_{\delta_S(g)}^{-1}$, we can 
assume $g_i\in{}C_{Z(P_i)}(\autf(P_i))$, and hence 
$\alpha_{P_i}=\Id_{\Aut_\call(P_i)}$ for each $i$.

We claim that $\alpha_P=\Id$ for all $P\le S$, and hence that $[\alpha]=1$ 
by (iv) and (ii).  Assume otherwise, and choose $Q<S$ which is fully 
normalized and of maximal order among all subgroups such that 
$\alpha_Q\ne\Id$.  Thus $\alpha_R=\Id$ for all $R\le S$ with $|R|>|Q|$. By 
Alperin's fusion theorem (cf. \cite[Theorem I.3.6]{AKO}), $Q$ is 
$\calf$-essential, and $\alpha$ is the identity on $\Mor_\call(P,P^*)$ for 
all $P,P^*\in\Ob(\call)$ such that $|P|,|P^*|>|Q|$. Also, for each $Q^*\in 
Q^\calf$, there is (by Alperin's fusion theorem again) an isomorphism 
$\chi\in\Iso_\call(Q,Q^*)$ which is a composite of isomorphisms each of 
which extends to an isomorphism between strictly larger subgroups, and 
hence $\alpha_{Q,Q^*}(\chi)=\chi$. Thus 
	\beqq Q^*\in Q^\calf, \textup{ $Q^*$ fully normalized} 
	\qquad\implies\qquad \alpha_{Q^*}\ne\Id\,. 
	\label{e:A.1c} \eeqq

Set $E=\Omega_1(Z(Q))$. Let $\varphi\in\homf(N_S(E),S)$ be such that 
$\varphi(E)$ is fully normalized (cf. \cite[Lemma I.2.6(c)]{AKO}). Then 
$N_S(Q)\le N_S(E)$, so $|N_S(\varphi(Q))|\ge|N_S(Q)|$, and $\varphi(Q)$ is 
fully normalized since $Q$ is. Since $\alpha_{Q^*}\ne\Id$ by 
\eqref{e:A.1c}, we can replace $Q$ by $Q^*$ and $E$ by $E^*$, and arrange 
that $Q$ and $E$ are both fully normalized in $\calf$ (and $Q$ is still 
$\calf$-essential). 

Set $\Gamma=\autf(Q)$, and set 
	\begin{align*} 
	\Gamma_0 &= C_\Gamma(E) = 
	\bigl\{\varphi\in\autf(Q)\,\big|\,\varphi|_E=\Id_E\bigr\} \\
	\Gamma_1 &= \Gen{ \varphi\in\autf(Q) \,\big|\, \varphi=\4\varphi|_Q~ 
	\textup{for some} ~\4\varphi\in\homf(R,S),~ R>Q } \,.
	\end{align*}
Let $\pi_Q\:\Aut_\call(Q)\Right2{}\autf(Q)$ be the homomorphism induced by 
the functor $\pi$. For each $\varphi\in\Gamma=\autf(Q)$, and each 
$\psi\in\pi_Q^{-1}(\varphi)$, we have
	\beqq \varphi(g_Q)=g_Q \quad\iff\quad 
	[\psi,\delta_Q(g_Q)]=\Id 
	\quad\iff\quad \alpha_Q(\psi)=\psi \,: \label{e:A.1a} \eeqq
the first by axiom (C) in the definition of a linking system (see, e.g., 
\cite[Definition III.4.1]{AKO}) and since $\delta_Q$ is injective, and the 
second by point (i) above.

Now, $\Aut_S(Q)\le\Gamma_1$, since each element of $\Aut_S(Q)$ extends to 
$N_S(Q)$ and $N_S(Q)>Q$ (see \cite[Theorem 2.1.6]{Sz1}). Hence 
	\[ \Gamma_0\Gamma_1 \le O^p(\Gamma_0)\cdot\Aut_S(Q)\cdot\Gamma_1
	= O^p(\Gamma_0)\Gamma_1 \,. \]
For each $\varphi\in\Gamma_0$ of order prime to $p$, 
$\varphi|_{Z(Q)}=\Id_{Z(Q)}$ since $\varphi$ is the identity on 
$E=\Omega_1(Z(Q))$ (cf. \cite[Theorem 5.2.4]{Gorenstein}). Thus 
$g_Q\in C_{Z(Q)}(O^p(\Gamma_0))$. If $\varphi\in\autf(Q)$ extends to 
$\4\varphi\in\homf(R,S)$ for some $R>Q$, then by the maximality of $Q$, 
$\alpha(\4\psi)=\4\psi$ for each $\4\psi\in\Mor_\call(R,S)$ such that 
$\pi(\4\psi)=\4\varphi$, and since $\alpha$ commutes with restriction (it 
sends inclusions to themselves), $\alpha_Q$ is the identity on 
$\4\psi|_{Q,Q}\in\pi_Q^{-1}(\varphi)$. So by \eqref{e:A.1a}, 
$\varphi(g_Q)=g_Q$. Thus $\varphi(g_Q)=g_Q$ for all $\varphi\in\Gamma_1$. 
Since $\alpha_Q\ne\Id$ by assumption, there is some $\varphi\in\autf(Q)$ 
such that $\varphi(g_Q)\ne g_Q$ (by \eqref{e:A.1a} again), and we conclude 
that 
	\beqq g_Q \in C_{Z(Q)}(\Gamma_0\Gamma_1) \qquad\textup{and}\qquad 
	\Gamma_0\Gamma_1 < \Gamma = \autf(Q)\,. \label{e:A.1b} \eeqq

Set $Q^*=N_{C_S(E)}(Q)\ge{}Q$.  Then 
$\Aut_{Q^*}(Q)=\Gamma_0\cap\Aut_S(Q)\in\sylp{\Gamma_0}$ since 
$\Aut_S(Q)\in\sylp{\Gamma}$, and by the Frattini argument, 
$\Gamma=N_\Gamma(\Aut_{Q^*}(Q))\Gamma_0$.  If $Q^*>Q$, then for each 
$\varphi\in N_\Gamma(\Aut_{Q^*}(Q))$, $\varphi$ extends to 
$\4\varphi\in\autf(Q^*)$ by the extension axiom. Thus 
$N_\Gamma(\Aut_{Q^*}(Q))\le\Gamma_1$ in this case, so 
$\Gamma=\Gamma_1\Gamma_0$, contradicting \eqref{e:A.1b}. We conclude that 
$Q^*=Q$. 

The homomorphism $\Gamma=\autf(Q)\Right2{}\autf(E)$ induced by restriction 
is surjective by the extension axiom, so 
$\autf(E)\cong\Gamma/\Gamma_0$. By \cite[Proposition I.3.3(b)]{AKO}, 
$\Gamma_1/\Inn(Q)$ is strongly $p$-embedded in $\Gamma/\Inn(Q)=\outf(Q)$; 
and $\Gamma_0\Gamma_1<\Gamma$ by \eqref{e:A.1b}. Also, 
$p\nmid|\Gamma_0/\Inn(Q)|$, since otherwise we would have $\Gamma_1\ge 
N_\Gamma(T)$ for some $T\in\sylp{\Gamma_0}$, in which case 
$\Gamma_1\Gamma_0\ge N_\Gamma(T)\Gamma_0=\Gamma$ by the Frattini argument. 
Thus $\Gamma_1\Gamma_0/\Gamma_0$ is strongly $p$-embedded in 
$\Gamma/\Gamma_0\cong\autf(E)$ by Lemma \ref{l:str.emb.}(b).

Now, $C_S(E)=Q$ since $N_{C_S(E)}(Q)=Q$ (cf. \cite[Theorem 2.1.6]{Sz1}). 
Thus $\Omega_1(Z(C_S(E)))=\Omega_1(Z(Q))=E$, and we conclude 
that $E\in\5\cale(\calf)$. 
Then $E\in(E_i)^\calf$ for some unique $1\le i\le k$, and $Q\in(P_i)^\calf$ by 
the extension axiom (and since $E$ and $E_i$ are both fully centralized).  
But then $\alpha_{P_i}\ne\Id$ by \eqref{e:A.1c}, contradicting the original 
assumption about $\alpha_{P_i}$. We conclude that $\alpha=\Id$.  
\end{proof}

\bigskip

\newsubb{Classical groups of Lie type in odd characteristic}
{s:mu-classical}

Throughout this subsection, we fix an odd prime power $q$ and an integer 
$n\ge1$.  We want to show $\Ker(\mu_G)=1$ when $G$ is one of the 
quasisimple classical groups of universal type over $\F_q$. By Theorem 
\ref{OldThA}(d), we need not consider the unitary groups.

\begin{Prop}  \label{lim1-classical}
Fix an odd prime power $q$.  Let $G$ be isomorphic to one of the quasisimple 
groups $\SL_n(q)$, $\Sp_n(q)$ ($n=2m$), or $\Spin_n^\pm(q)$ ($n\ge3$).  Then 
$\Ker(\mu_G)=1$.
\end{Prop}

\begin{proof}  Let $V$, $\bb$, and $\5G=\Aut(V,\bb)$ be such that 
$G=[\5G,\5G]$ if $G\cong{}\Sp_n(q)$ or $G\cong{}\SL_n(q)$, and 
$G/\gen{z}=[\5G,\5G]$ for some $z\in{}Z(\5G)$ if $G\cong\Spin_n^\pm(q)$ 
(where $z\in{}Z(G)$).  Thus $V$ is a vector space of dimension $n$ over the 
field $K=\F_q$, $\bb$ is a trivial symplectic, or quadratic form, and $\5G$ 
is one of the groups $\GL_n(q)$, $\Sp_{2n}(q)$, or $\GO_n^{\pm}(q)$.  

Fix $S\in\syl2{G}$, and set $\calf=\calf_S(G)$.  Set 
$\5\calz=\5\calz(\calf)$ for short.

\smallskip

\noindent\textbf{Case 1: }  Assume $G=\Spin(V,\bb)$, where $\bb$ is 
nondegenerate and symmetric.  Set $Z=Z(G)$, and let $z\in{}Z$ be such that 
$G/\gen{z}=\Omega(V,\bb)$.  We claim that $\5\calz=\emptyset$ in this case, 
and hence that $\Ker(\mu_G)=1$ by Proposition \ref{Ker(mu)}(a).

Fix an elementary abelian 2-subgroup $E\le{}G$ where $E\ge Z$.  Let 
$V=\bigoplus_{i=1}^mV_i$ be the decomposition as a sum of eigenspaces for 
the action of $E$ on $V$. Fix indices $j,k\in\{1,\ldots,m\}$ such that 
either $\dim(V_j)\ge2$, or the subspaces have the same 
discriminant (modulo squares). (Since $\dim(V)\ge3$, this can always be 
done.) Then there is an involution $\gamma\in\SO(V,\bb)$ such that 
$\gamma(V_i)=V_i$ for all $i$, $\gamma|_{V_i}=\Id$ for $i\ne{}j,k$, 
$\det(\gamma|_{V_j})=\det(\gamma|_{V_k})=-1$, and such that the 
$(-1)$-eigenspace of $\gamma$ has discriminant a square. This last 
condition ensures that $\gamma\in\Omega(V,\bb)$ (cf. \cite[Lemma A.4(a)]{LO}), 
so we can lift it to $g\in G$. Then for each $x\in{}E$, $c_g(x)=x$ 
if $x$ has the same eigenvalues on $V_j$ and $V_k$, and $c_g(x)=zx$ 
otherwise (see, e.g., \cite[Lemma A.4(c)]{LO}).  Since $z$ is fixed by 
all elements of $\autf(E)$, $c_g\in{}O_2(\autf(E))$, and hence 
$\autf(E)$ has no strongly $2$-embedded subgroups by Lemma 
\ref{l:str.emb.}(a).  Thus $E\notin\5\calz$.

\smallskip

\noindent\textbf{Case 2: }  Now assume $G$ is linear or symplectic, and 
fix $S\in\syl2{G}$. For each $\calv=\{V_1,\ldots,V_k\}$ such that 
$V=\bigoplus_{i=1}^kV_i$, and such that $V_i\perp V_j$ for $i\ne{}j$ if $G$ 
is symplectic, set 
	\[ E(\calv) = \bigl\{\varphi\in G \,\big|\, 
	\varphi|_{V_i}=\pm\Id\ \textup{ for each } i \bigr\} . \]

We claim that each subgroup in $\5\calz$ has this form. To see this, fix 
$E\in\5\calz$, and let $\calv=\{V_1,\ldots,V_k\}$ be the eigenspaces for 
the nonzero characters of $E$.  Then $E\le E(\calv)$, 
$V=\bigoplus_{i=1}^kV_i$, and this is an orthogonal decomposition if $G$ is 
symplectic. Also, $C_{\5G}(E)$ is the product of the groups 
$\Aut(V_i,\bb|_{V_i})$.  Since $E=\Omega_1(Z(P))$ where $P=C_S(E)$, $E$ 
contains the $2$-torsion in the center of $C_G(E)$, and thus $E=E(\calv)$.  
Furthermore, the action of $S$ on each $V_i$ must be irreducible (otherwise 
$\Omega_1(Z(C_S(E)))>E$), so $\dim(V_i)$ is a power of $2$ for each $i$.

Again assume $E=E(\calv)\in\5\calz$ for some $\calv$. Then 
$\Aut_{\widehat{G}}(E)$ is a product of symmetric groups: if $\calv$ 
contains $n_i$ subspaces of dimension $i$ for each $i\ge1$, then 
$\Aut_{\5G}(E(\calv))\cong\prod_{i\ge1}\Sigma_{n_i}$.  Each such 
permutation can be realized by a self map of determinant one (if $G$ is 
linear), so $\Aut_G(E)=\Aut_{\5G}(E)$.  Since $\Aut_G(E)$ contains a 
strongly 2-embedded subgroup by definition of $\5\calz$ (and since a direct 
product of groups of even order contains no strongly $2$-embedded 
subgroup), $\Aut_G(E)=\Aut_{\5G}(E)\cong\Sigma_3$.


Write $n=\dim(V)=2^{k_0}+2^{k_1}+\ldots+2^{k_m}$, where 
$0\le{}k_0<k_1<\cdots<k_m$.  There is an (orthogonal) decomposition 
$V=\bigoplus_{i=0}^mV_m$, where $S$ acts irreducibly on each $V_i$, and 
where $\dim(V_i)=2^{k_i}$ (see \cite[Theorem 1]{CF}).  For each $1\le i\le 
m$, fix an (orthogonal) decomposition $\calw_i$ of $V_i$ whose components 
have dimensions $2^{k_{i-1}},2^{k_{i-1}},2^{k_{i-1}+1},\ldots,2^{k_i-1}$, 
and set 
	\[ \calv_i = \{V_j\,|\,j\ne i\}\cup \calw_i \]
and $E_i=E(\calv_i)$.  Thus $\calv_i$ contains exactly three subspaces of 
dimension $2^{k_{i-1}}$, and the dimensions of the other subspaces are 
distinct. Hence $\Aut_G(E_i)\cong\Sigma_3$, and $E_i\in\5\calz$. 
Conversely, by the above analysis (and since the conjugacy class of 
$E\in\5\calz$ is determined by the dimensions of its eigenspaces), each 
subgroup in $\5\calz$ is $G$-conjugate to one of the $E_i$.

\newcommand{\muu}[2]{\lambda^{(#1)}_{#2}}

For each $1\le i\le m$, set $P_i=C_S(E_i)$ and $Z_i=Z(P_i)$ (so 
$E_i=\Omega_1(Z_i)$). Since each element of $N_G(P_i)=N_G(E_i)$ 
permutes members of $\calv_i$ of equal dimension, and the elements of 
$N_S(P_i)$ do so only within each of the $V_j$, we have 
	\beqq \begin{split} 
	Z_i &= \bigl\{ z\in G \,\big|\, z|_{X}=\muu{z}{X}\Id_{X} 
	\textup{ for all $X\in\calv_i$, some 
	$\muu{z}{X}\in O_2(\F_q^\times)$} \bigr\} \\
	C_{Z_i}(\Aut_S(P_i)) &= \bigl\{ z\in Z_i \,\big|\, 
	\muu{z}{X_i}=\muu{z}{X_i'} \bigr\} \\
	C_{Z_i}(\Aut_G(P_i)) &= \bigl\{ z\in Z_i \,\big|\, 
	\muu{z}{X_i}=\muu{z}{X_i'}=\muu{z}{V_{i-1}} \bigr\} \,,
	\end{split} \label{e:A.3a} \eeqq
where $X_i$, $X_i'$, and $V_{i-1}$ are the three 
members of the decomposition $\calv_i$ of dimension $2^{k_{i-1}}$ (and 
$X_i,X_i'\in\calw_i$). 

Fix $(g_i)_{i=1}^m\in\prod_{i=1}^mC_{Z_i}(\Aut_S(P_i))$. Then $g_i\in 
C_{Z_i}(\Aut_G(P_i))$ if and only if $\muu{g_i}{V_{i-1}}=\muu{g_i}{X_i}$. 
Choose $g\in\5G$ such that $g|_{V_i}=\eta_i\cdot\Id$ for each $i$, where 
the $\eta_i\in O_2(\F_q^\times)$ are chosen so that 
$\eta_i/\eta_{i-1}=\muu{g_i}{X_i}/\muu{g_i}{V_{i-1}}$ for each $1\le i\le 
m$.  If $G$ is linear, then $\det(g)=\theta^{2^{k_0}}$ for some $\theta\in 
O_2(\F_q^\times)$, and upon replacing $g$ by 
$g\circ\theta^{-2^{k_0}/n}\Id_V$ (recall $k_0=v_2(n)$) we can assume 
$g\in{}G$. Then $g\in C_{Z(S)}(\Aut_G(S))$ since it is a multiple of the 
identity on each $V_i$ and has $2$-power order. By construction and 
\eqref{e:A.3a}, $g\equiv g_i$ (mod $C_{Z_i}(\Aut_G(P_i))$) for each $i$; so 
$\Ker(\mu_G)=1$ by Proposition \ref{Ker(mu)}(c). 
\end{proof}

\bigskip

\newsubb{Exceptional groups of Lie type in odd characteristic}
{s:mu-exceptional}

Throughout this subsection, $q_0$ is an odd prime, and $q$ is a power of 
$q_0$. We show that $\Ker(\mu_G)=1$ when $G$ is one of the groups $G_2(q)$, 
$F_4(q)$, $E_6(q)$, $E_7(q)$, or $E_8(q)$ and is of universal type.

The following proposition is a special case of \cite[Theorem 2.1.5]{GLS3}, 
and is stated and proven explicitly in \cite[Proposition 8.5]{limz}. It 
describes, in many cases, the relationship between conjugacy classes and 
normalizers in a connected algebraic group and those in the subgroup fixed 
by a Steinberg endomorphism. 

\begin{Prop} \label{L<->Lhat}
Let $\4G$ be a connected algebraic group over $\fqobar$, 
let $\sigma$ be a Steinberg endomorphism of $\4G$, and set 
$G=C_{\4G}(\sigma)$.  Let $H\le{}G$ be any subgroup, and let $\calh$ be the 
set of $G$-conjugacy classes of subgroups $\4G$-conjugate to $H$.  Let 
$N_{\4G}(H)$ act on $\pi_0(C_{\4G}(H))$ by sending $g$ to 
$xg\sigma(x)^{-1}$ (for $x\in{}N_{\4G}(H)$).  Then there is a bijection
	\[ \omega\: \calh \Right6{\cong} \pi_0(C_{\4G}(H))/N_{\4G}(H), \]
defined by setting $\omega([\9xH])=[x^{-1}\sigma(x)]$ whenever 
$\9xH\le{}C_{\4G}(\sigma)$.  Also, for each $x\in\4G$ such that 
$\9xH\le G$, $\Aut_{G}(\9xH)$ is isomorphic to the 
stabilizer of $[x^{-1}\sigma(x)]\in\pi_0(C_{\4G}(H))/C_{\4G}(H)$ under the 
action of $\Aut_{\4G}(H)$ on this set. 
\end{Prop}

Since we always assume $\4G$ is of universal type in this section, the 
group $G=C_{\4G}(\sigma)$ of Proposition \ref{L<->Lhat} is equal to the 
group $G=O^{q_0'}(C_{\4G}(\sigma))$ of Definition \ref{d:Lie} and Notation 
\ref{G-setup}.

The following definitions will be useful when 
applying Proposition \ref{L<->Lhat}. For any finite group $G$, set 
	\begin{align*} 
	\mathcal{SE}(G) &= \bigl\{H\le G \,\big|\, 
	\textup{$H$ has a strongly $2$-embedded subgroup} \bigr\} \\
	\defect(G) &= \begin{cases} 
	\min \bigl\{[G:H] \,\big|\, H\in\mathcal{SE}(G) \bigr\} & 
	\textup{if $\mathcal{SE}(G)\ne\emptyset$} \\
	\infty & \textup{if $\mathcal{SE}(G)=\emptyset$.}
	\end{cases}
	\end{align*}
Thus by Proposition \ref{L<->Lhat}, if $H<\4G$ is such that 
$|\pi_0(C_{\4G}(H))|>\defect(\Out_{\4G}(H))$, then no subgroup $H^*\le 
C_{\4G}(\sigma)$ which is $\4G$-conjugate to $H$ has the property that 
$\Aut_{C_{\4G}(\sigma)}(H^*)$ has a strongly $2$-embedded subgroup. 
The next lemma provides some tools for finding lower bounds for 
$\defect(G)$.

\begin{Lem} \label{l:d(G)}
\begin{enuma} 
\item For any finite group $G$, 
$\defect(G)\ge|O_2(G)|\cdot\defect(G/O_2(G))$.

\item If $G=G_1\times G_2$ is finite, and $\defect(G_i)<\infty$ for 
$i=1,2$, then 
	\[ \defect(G)=\min\bigl\{ 
	\defect(G_1) \cdot \eta(G_2) \,,\, 
	\defect(G_2) \cdot \eta(G_1) \bigr\} \,, \]
where $\eta(G_i)$ is the smallest index of any odd order subgroup of $G_i$.

\item If $\defect(G)<\infty$, and there is a faithful $\F_2[G]$-module $V$ 
of rank $n$, then $2^{v_2(|G|)-[n/2]}\big|\defect(G)$.

\item More concretely, $\defect(\GL_3(2))=28$, 
$\defect(\GL_4(2))=112$, $\defect(\GL_5(2))=2^8\cdot7\cdot31$, and 
$\defect(\SO_4^+(2))=2=\defect(\SO_4^-(2))$. Also, 
$2^4\le\defect(\SO_6^+(2))<\infty$ and $2^6\le\defect(\SO_7(2))<\infty$.

\end{enuma}
\end{Lem}

\begin{proof} \textbf{(a) } If $H\in\mathcal{SE}(G)$, 
then $H\cap{}O_2(G)=1$ by Lemma \ref{l:str.emb.}(a). Hence 
there is a subgroup $H^*\le G/O_2(G)$ isomorphic to $H$, and 
	\[ [G:H]=|O_2(G)|\cdot[G/O_2(G):H^*]
	\ge|O_2(G)|\cdot\defect(G/O_2(G))\,. \]

\smallskip

\noindent\textbf{(b) } If a finite group $H$ has a strongly 2-embedded 
subgroup, then so does its direct product with any odd order group. Hence 
$\defect(G)\le\defect(G_i)\eta(G_{3-i})$ for $i=1,2$. 

Assume $H\le G$ has a strongly $2$-embedded subgroup 
$K<H$. Set $H_i=H\cap{}G_i$ for $i=1,2$. Since all involutions in $H$ are 
$H$-conjugate (see \cite[6.4.4]{Sz2}), $H_1$ and $H_2$ cannot both have 
even order. Assume $|H_2|$ is odd. Let $\pr_1$ be projection onto 
the first factor. 
If $\pr_1(K)=\pr_1(H)$, then there is $x\in(H{\sminus}K)\cap H_2$, and this 
commutes with all Sylow $2$-subgroups of $H$ since they lie in $G_1$, 
contradicting the assumption that $K$ is strongly $2$-embedded in $H$.
Thus $\pr_1(K)<\pr_1(H)$. Then $\pr_1(H)$ has a strongly 
$2$-embedded subgroup by Lemma \ref{l:str.emb.}(b), and hence 
	\[ [G:H] = [G_1:\pr_1(H)] \cdot [G_2:H_2] \ge
	\defect(G_1) \cdot \eta(G_2) \,. \] 
So $\defect(G)\ge\defect(G_i)\eta(G_{3-i})$ for $i=1$ or $2$.

\smallskip

\noindent\textbf{(c) } This follows from \cite[Lemma 1.7(a)]{OV2}: if $H<G$ 
has a strongly $2$-embedded subgroup, $T\in\syl2{H}$, and $|T|=2^k$, then 
$\dim(V)\ge2k$.

\smallskip

\noindent\textbf{(d) } The formulas for $\defect(\SO_4^\pm(2))$ hold since 
$\SO_4^+(2)\cong\Sigma_3\wr C_2$ contains a subgroup isomorphic to 
$C_3^2\sd{}C_4$ and $\SO_4^-(2)\cong\Sigma_5$ a subgroup isomorphic to 
$A_5$. Since $4|\defect(\GL_3(2))$ by (c), and since $7|\defect(\GL_3(2))$ 
(there are no subgroups of order $14$ or $42$), we have 
$28|\defect(\GL_3(2))$, with equality since $\Sigma_3$ has index $28$. The 
last two (very coarse) estimates follow from (c), and the $6$- and 
$7$-dimensional representations of these groups. 

Fix $n=4,5$, and set $G_n=\GL_n(2)$. Assume $H\le G_n$ has a strongly 
embedded subgroup, where $7\big||H|$ or $31\big||H|$. By (c), 
$2^4|\defect(G_4)$ and $2^8|\defect(G_5)$, and thus $8\nmid|H|$. If $H$ is 
almost simple, then $H\cong A_5$ by Bender's theorem (see \cite[Theorem 
6.4.2]{Sz2}), contradicting the assumption about $|H|$. So by the main 
theorem in \cite{Aschbacher}, $H$ must be contained in a member of one of 
the classes $\calc_i$ ($1\le i\le8$) defined in that paper. One quickly checks 
that since $(7\cdot31,|H|)\ne1$, $H$ is contained in a member of $\calc_1$. 
Thus $H$ is reducible, and since $O_2(H)=1$, either $H$ is isomorphic to a 
subgroup of $\GL_3(2)\times\GL_{n-3}(2)$, or $n=5$ and $H<\GL_4(2)$. By (b) 
and since $7\big||\defect(\GL_3(2))$, we must have 
$H\cong\Sigma_3\times(C_7\sd{}C_3)$, in which case $|H|<180=|\GL_2(4)|$. Thus 
$7|\defect(G_n)$ for $n=4,5$, and $31|\defect(G_5)$. Since $\GL_4(2)$ contains 
a subgroup isomorphic to $\GL_2(4)\cong C_3\times A_5$, we get 
$\defect(G_4)=2^4\cdot7$ and $\defect(G_5)=2^8\cdot7\cdot31$. 
\end{proof}

We illustrate the use of the above proposition and lemma by proving the 
injectivity of $\mu_G$ when $G=G_2(q)$.

\begin{Prop} \label{lim1-G2}
If $G=G_2(q)$ for some odd prime power $q$, then $\Ker(\mu_G)=1$.
\end{Prop}

\begin{proof} Assume $q$ is a power of the prime $q_0$, set 
$\4G=G_2(\fqobar)$, and fix a maximal torus $\4T$. We identify 
$G=C_{\4G}(\psi_q)$, where $\psi_q$ is the field automorphism, and acts via 
$(t\mapsto t^q)$ on $\4T$. Fix $S\in\syl2{G}$, and set 
$\5\calz=\5\calz(\calf_S(G))$.

Let $E\cong C_2^2$ be the $2$-torsion subgroup of $\4T$. By Proposition 
\ref{p:CG(T)}, $C_{\4G}(V)=\4T\gen{\theta}$ where $\theta\in{}N_{\4G}(\4T)$ 
inverts the torus.  Thus by Proposition \ref{L<->Lhat}, there are two 
$G$-conjugacy classes of subgroups $\4G$-conjugate to $E$, represented by 
$E^\pm$ ($E^+=E$), where $\Aut_G(E^\pm)=\Aut(E^\pm)\cong\Sigma_3$ and 
$C_G(E^\pm)=(C_{q\mp1})^2\rtimes{}C_2$.  The subgroups in one of these 
classes have centralizer in $S$ isomorphic to $C_2^3$, hence are not in 
$\5\calz$, while those in the other class do lie in $\5\calz$. The 
latter also have normalizer of order $12(q\pm1)^2$ and hence of odd index 
in $G$, and thus are normal in some choice of Sylow $2$-subgroup.

By \cite[Table I]{Griess}, for each nontoral elementary abelian 2-subgroup 
$E\le\4G$, $\rk(E)=3$, $C_{\4G}(E)=E$, and $\Aut_{\4G}(E)\cong\GL_3(2)$. By 
Proposition \ref{L<->Lhat}, and since 
$\defect(\Aut_{\4G}(E))=28>|C_{\4G}(E)|$ by Lemma \ref{l:d(G)}, $\Aut_G(E)$ 
contains no strongly $2$-embedded subgroup, and thus $E\notin\5\calz$.  

Thus $\5\calz$ is contained in a unique $G$-conjugacy class of subgroups 
of rank $2$, and $\Ker(\mu_G)=1$ by Proposition \ref{Ker(mu)}(b).
\end{proof}

Throughout the rest of this section, fix an odd prime power $q$, and let 
$\gg$ be one of the groups $F_4$, $E_6$, $E_7$, or $E_8$. 

\newcommand{\TT}{T_{(2)}}

\begin{Hyp} \label{G-hypoth-exc}
Assume $\4G=\gg(\fqobar)$ and $G\cong\gg(q)$, where $q$ is a power of the odd 
prime $q_0$, and where $\gg=F_4$, $E_6$, $E_7$, or $E_8$ and is of 
universal type. Fix a maximal torus $\4T<\4G$. 
\begin{enumI}[widest=II]

\item Set $\TT=\{t\in\4T\,|\,t^2=1\}$. 
Let $\2a$ and $\2b$ denote the two $\4G$-conjugacy classes of noncentral 
involutions in $\4G$, as defined in \cite[Table VI]{Griess}, \emph{except} 
that when $\gg=E_7$, they denote the classes labelled $\2b$ and $\2c$, 
respectively, in that table. For each elementary abelian $2$-subgroup 
$E<\4G$, define 
	\[ \qq_E\: E\Right5{}\F_2 \]
by setting $\qq(x)=0$ if $x\in\2b\cup\{1\}$, and $\qq(x)=1$ if 
$x\in\2a\cup(Z(\4G){\sminus}1)$.

\item Assume $G=C_{\4G}(\psi_q)$, where $\psi_q$ is the field endomorphism 
with respect to some root structure with maximal torus $\4T$. Thus 
$\psi_q(t)=t^q$ for all $t\in\4T$. Fix $S\in\syl2{G}$, and set 
$\5\calz=\5\calz(\calf_S(G))$.

\end{enumI}
\end{Hyp}

By \cite[Lemma 2.16]{Griess}, $\qq_{\TT}$ is a quadratic form on $\TT$ in 
all cases, and hence $\qq_E$ is quadratic for each $E\le\TT$.  In general, 
$\qq_E$ need not be quadratic when $E$ is not contained in a maximal torus.  
In fact, Griess showed in \cite[Theorems 7.3, 8.2, \& 9.2]{Griess} that in 
many (but not all) cases, $E$ is contained in a torus if and only if 
$\qq_E$ is quadratic ($\textup{cx}(E)\le2$ in his terminology).

With the above choices of notation for noncentral involutions, all of the 
inclusions $F_4\le{}E_6\le{}E_7\le{}E_8$ restrict to inclusions of the 
classes \2a and of the classes \2b. This follows since the forms are 
quadratic, and also (for $E_7<E_8$) from \cite[Lemma 2.16(iv)]{Griess}.

\begin{Lem} \label{l:q:V->F2}
Assume Hypotheses \ref{G-hypoth-exc}, and let 
$\bb$ be the bilinear form associated to $\qq$. Define
	\begin{align*} 
	V_0 &= \bigl\{v\in \TT\,\big|\, \bb(v,\TT)=0, ~ \qq(v)=0 \bigr\} \\
	\gamma_x &= \bigl( v \mapsto v+\bb(v,x)x \bigr) \in \Aut(\TT,\qq) 
	&& \textup{for $x\in{}\TT$ with $\qq(x)=1$, $q\not\perp \TT$} 
	\end{align*}
Then the following hold.
\begin{enuma} 
\item $\Aut_{\4G}(\TT)=\Aut(\TT,\qq)$. 

\item For each nonisotropic $x\in \TT{\sminus}\TT^\perp$, $\gamma_x$ is the 
restriction to $\TT$ of a Weyl reflection on $\4T$. If $\alpha\in\Sigma$ is 
such that $\gamma_x=w_\alpha|_{\TT}$, then 
$\theta_\alpha(v)=(-1)^{\bb(x,v)}$ for each $v\in{}\TT$. 

\item If $\gg=E_r$ ($r=6,7,8$), then $\qq$ is nondegenerate ($V_0=0$), 
and the restriction to $\TT$ of each Weyl reflection is equal to $\gamma_x$ for 
some nonisotropic $x\in{}\TT{\sminus}\TT^\perp$. 

\item If $\gg=F_4$, then $\dim(V_0)=2$, and $\qq(v)=1$ for all $v\in 
\TT{\sminus}V_0$. 

\end{enuma}
\end{Lem}

\begin{proof} \textbf{(a) } Since $\Aut_{\4G}(\TT)$ has to preserve 
$\4G$-conjugacy classes, it is contained in $\Aut(\TT,\qq)$. Equality will 
be shown while proving (c) and (d).

\smallskip

\noindent\textbf{(c) } If $\gg=E_r$ for $r=6,7,8$, then $\qq$ is 
nondegenerate by \cite[Lemma 2.16]{Griess}. Hence the only orthogonal 
transvections are of the form $\gamma_x$ for nonisotropic $x$, and each 
Weyl reflection restricts to one of them. By a direct count (using the 
tables in \cite{Bourb4-6}), the number of pairs $\{\pm\alpha\}$ of roots in 
$\gg$ (hence the number of Weyl reflections) is equal to 36, 63, or 120, 
respectively. This is equal to the number of nonisotropic elements in 
$\TT\sminus\TT^\perp=\TT\sminus Z(\4G)$ (see 
the formula in \cite[Theorem 11.5]{Taylor} for the number of isotropic 
elements). So all transvections are 
restrictions of Weyl reflections, and $\Aut_{\4G}(\TT)=\Aut(\TT,\qq)$.


\smallskip

\noindent\textbf{(d) } Assume $\gg=F_4$. Then $\dim(V_0)=2$ and 
$\qq^{-1}(1)=\TT{\sminus}V_0$ by \cite[Lemma 2.16]{Griess}. Thus 
$|\Aut(\TT,\qq)|=4^2\cdot|GL_2(2)|^2=2^6\cdot3^2=\frac12|W|$ (see 
\cite[Planche VIII]{Bourb4-6}), so $\Aut_W(\TT)=\Aut(\TT,\qq)$ since $W$ 
also contains $-\Id$. 

There are three conjugacy classes of transvections 
$\gamma\in\Aut(\TT,\qq)$: one of order 36 containing those where 
$\gamma|_{V_0}\ne\Id$ (and hence $[\gamma,\TT]\le V_0$), and two of order 
12 containing those where $\gamma|_{V_0}=\Id$ (one where $[\gamma,\TT]\le 
V_0$ and one where $[\gamma,\TT]\nleq V_0$). Since there are two $W$-orbits 
of roots (long and short), each containing 12 pairs $\pm\alpha$, the 
corresponding Weyl reflections must restrict to the last two classes of 
transvections, of which one is the set of all $\gamma_x$ for 
$x\in\TT\sminus V_0$.

\smallskip

\noindent\textbf{(b) } We showed in the proofs of (c) and (d) that each 
orthogonal transvection $\gamma_x$ is the restriction of a Weyl reflection. 
If $\gamma_x=w_\alpha|_{\TT}$ for some root $\alpha\in\Sigma$, then 
$\theta_\alpha\in\Hom(\4T,\fqobar^\times)$ (Lemma 
\ref{theta-r}\eqref{^txb(u)}), so $[\TT:\Ker(\theta_\alpha|_{\TT})]\le2$. 
Also, $\Ker(\theta_\alpha)\le C_{\4T}(w_\alpha)$ by Lemma 
\ref{theta-r}\eqref{wa(hb)}, so $\Ker(\theta_\alpha|_{\TT})\le 
C_{\TT}(w_\alpha)=C_{\TT}(\gamma_x)=x^\perp$, with equality since 
$[\TT:x^\perp]=2$. Since $\theta_\alpha(\TT)\le\{\pm1\}$, it follows that 
$\theta_\alpha(v)=(-1)^{\bb(x,v)}$ for each $v\in\TT$.
\end{proof}

We are now ready to list the subgroups in $\5\calz(\gg(q))$ in all cases. 
The proof of the following lemma will be given at the end of the 
section.

\begin{Lem} \label{all-piv}
Let $\4G=\gg(\fqobar)$ and $G=\gg(q)$ be as in Hypotheses 
\ref{G-hypoth-exc}.  Assume $E\in\5\calz(G)$.  Then either $\gg\ne{}E_7$, 
$\rk(E)=2$, and $\qq_E=0$; or $\gg=E_7$, $Z=Z(\widebar{G})\cong{}C_2$, and 
$E=Z\times{}E_0$ where $\rk(E_0)=2$ and $\qq_{E_0}=0$.  In all cases, 
$\Aut_{\widebar{G}}(E)\cong\Sigma_3$. 
\end{Lem}

\begin{proof} This will be shown in Lemmas \ref{toral-piv} and 
\ref{nontoral-piv}.
\end{proof}

The next two lemmas will be needed to apply Proposition \ref{Ker(mu)}(b) to 
these groups.  The first is very elementary.

\begin{Lem} \label{isotr.sgr.}
Let $V$ be an $\F_2$-vector space of dimension $k$, and let 
$\qq\:V\Right2{}\F_2$ be a quadratic form on $V$.  For $m\ge1$ such that 
$k>2m$, the number of totally isotropic subspaces of dimension $m$ in $V$ is 
odd.
\end{Lem}

\begin{proof}  This will be shown by induction on $m$, starting with the 
case $m=1$.  Since $k\ge3$, there is an orthogonal splitting 
$V=V_1\perp{}V_2$ where $V_1,V_2\ne0$.  Let $k_i$ be the number of 
isotropic elements in $V_i$ (including $0$), and set $n_i=|V_i|$.  The 
number of isotropic elements in $V$ is then $k_1k_2+(n_1-k_1)(n_2-k_2)$, 
and is even since the $n_i$ are even.  The number of 1-dimensional 
isotropic subspaces is thus odd.

Now fix $m>1$ (such that $k>2m$), and assume the lemma holds for subspaces 
of dimension $m-1$.  For each isotropic element $x\in{}V$, a subspace 
$E\le{}V$ of dimension $m$ containing $x$ is totally isotropic if and only 
if $E\le x^\perp$ \emph{and} $E/\gen{x}$ is isotropic in $x^\perp/\gen{x}$ 
with the induced quadratic form.  By the induction hypothesis, and since 
	\[ 2\cdot\dim(E/\gen{x})=2(m-1)<k-2\le\dim(x^\perp/\gen{x}), \]
the number of isotropic subspaces of dimension $m$ which contain $x$ is odd.  
Upon taking the sum over all $x$, and noting that each subspace has been 
counted $2^m-1$ times, we see that the number of isotropic subspaces of 
dimension $m$ is odd.  
\end{proof}

\begin{Lem} \label{odd-orbit}
Assume Hypotheses \ref{G-hypoth-exc}(I). Let $\sigma$ be a 
Steinberg endomorphism of $\widebar{G}$ such that for some $\gee=\pm1$, 
$\sigma(t)=t^{\gee q}$ for each $t\in\4T$. Set $G=C_{\widebar{G}}(\sigma)$. 
Fix $E\le\TT$ of rank $2$ such that $\qq(E)=0$.  Then the set of subgroups 
of $G$ which are $\widebar{G}$-conjugate to $E$, and the set of subgroups 
which are $G$-conjugate to $E$, both have odd order and contain all totally 
isotropic subgroups of rank $2$ in $\TT$. 
\end{Lem}

\begin{proof}  Let $\4\X\supseteq\X$ be the sets of subgroups of $G$ which 
are $\widebar{G}$-conjugate to $E$ or $G$-conjugate to $E$, respectively. 
Let $\X_0$ be the subset of all totally isotropic subgroups of $\TT$ of 
rank $2$. If $\qq$ is nondegenerate, then by Witt's theorem (see 
\cite[Theorem 7.4]{Taylor}), $\Aut_W(\TT)=\Aut(\TT,\qq)$ permutes $\X_0$ 
transitively, and hence all elements in $\X_0$ are $G$-conjugate to $E$ by 
Lemma \ref{l:gT}. If in addition, $\dim(\TT)\ge5$, then $|\X_0|$ is odd by 
Lemma \ref{isotr.sgr.}. Otherwise, by Lemma \ref{l:q:V->F2}(c,d), $\gg=F_4$ 
and $\X_0=\{E\}$. Thus in all cases, $\X_0\subseteq\X$ and $|\X_0|$ is odd. 

Assume $\gg=E_6$. Then $C_{\widebar{G}}(\TT)=\widebar{T}$ by Proposition 
\ref{p:CG(T)}. Consider the conjugation action of $\TT$ on $\4\X$, and let 
$\X_1$ be its fixed point set. Since $\TT\le G$ by the assumptions on 
$\sigma$, this action also normalizes $\X$. For $F\in\X_1$, either the 
action of $\TT$ fixes $F$ pointwise, in which case $F\in\X_0$, or there are 
$x,y\in{}F$ such that $[x,\TT]=1$ and $[y,\TT]=\gen{x}$.  In particular, 
$c_y\in\Aut_{\widebar{G}}(\TT)=SO(\TT,\qq)$.  For each $v\in{}\TT$ such 
that $[y,v]=x$, $\qq(v)=\qq(vx)$ and $\qq(x)=0$ imply $x\perp{}v$, so 
$x\perp{}\TT$ since $\TT$ is generated by those elements.  This is 
impossible since $\qq$ is nondegenerate by Lemma \ref{l:q:V->F2}(c), and 
thus $\X_1=\X_0$. 

Now assume $\gg=F_4$, $E_7$, or $E_8$. Then $-\Id\in W$, so there is 
$\theta\in{}N_{\widebar{G}}(\widebar{T})$ which inverts $\widebar{T}$.  
Then $C_{\4G}(\TT)=\4T\gen{\theta}$.  By the Lang-Steinberg theorem, there 
is $g\in\widebar{G}$ such that $g^{-1}\sigma(g)\in\theta\widebar{T}$; then 
$\sigma(gtg^{-1})=gt^{\mp{}q}g^{-1}$ for $t\in\widebar{T}$, and thus 
$\sigma$ acts on $g\widebar{T}g^{-1}$ via $t\mapsto{}t^{\mp{}q}$.  We can 
thus assume $\widebar{T}$ was chosen so that 
$G\cap\widebar{T}=C_{\widebar{T}}(\sigma)$ contains the 4-torsion subgroup 
$\4T_{(4)}\le\widebar{T}$.  Let $\X_1\subseteq\4\X$ be the fixed point set 
of the conjugation action of $\4T_{(4)}$ on $\4\X$.  For $F\in\X_1$, either 
the action of $\4T_{(4)}$ fixes $F$ pointwise, in which case $F\in\X_0$, or 
there are $x,y\in{}F$ such that $[x,\4T_{(4)}]=1$ and 
$[y,\4T_{(4)}]=\gen{x}$.  But then $[F,\4T_{(4)}^*]=1$ for some 
$\4T_{(4)}^*<\4T_{(4)}$ of index two, $[F,\TT]=1$ implies 
$F\le\TT\gen{\theta}$; and $F\le\TT$ since no element in 
$\4T_{(4)}{\sminus}\TT$ commutes with any element of $\TT\theta$.  So 
$\X_1=\X_0$ in this case.

Thus in both cases, $\X_0$ is the fixed point set of an action of a 
2-group on $\4\X$ which normalizes $\X$.  Since $|\X_0|$ is odd, so are 
$|\4\X|$ and $|\X|$.
\end{proof}

We are now ready to prove:

\begin{Prop}  \label{lim1-exceptional}
Fix an odd prime power $q$.  Assume $G$ is a quasisimple group of 
universal type isomorphic to $G_2(q)$, $F_4(q)$, $E_6(q)$, $E_7(q)$, or 
$E_8(q)$. Then $\Ker(\mu_G)=1$.
\end{Prop}

\begin{proof} This holds when $G\cong G_2(q)$ by Proposition \ref{lim1-G2}, 
so we can assume Hypotheses \ref{G-hypoth-exc}. 
Let $\X$ be the set of all elementary abelian $2$-subgroups 
$E\le{}G$ such that either $\gg\ne{}E_7$, $\rk(E)=2$, and $\qq_E=0$; or 
$\gg=E_7$, $\rk(E)=3$, and $E=Z(G)\times{}E_0$ where $\qq_{E_0}=0$.  By 
Lemma \ref{odd-orbit}, $|\X|$ is odd.  In all cases, by Lemma 
\ref{all-piv}, $\5\calz(G)\subseteq\X$.  By Proposition \ref{Ker(mu)}(a,b), 
to prove $\mu_G$ is injective, it remains to show that if 
$\5\calz(G)\ne\emptyset$, then $\5\calz(G)$ has odd order and is contained 
in a single $G$-conjugacy class, and $\Aut_G(Z(S))=1$.

Fix $E\in\X$ such that $E\le \TT$.  We first claim that if $\gg=F_4$, 
$E_6$, or $E_7$, then $C_{\4G}(E)$ is connected, and hence all elements in 
$\X$ are $G$-conjugate to $E$ by Proposition \ref{L<->Lhat}. If $\gg=E_7$, 
then $C_{\widebar{G}}(E)$ is connected by \cite[Proposition 
9.5(iii)(a)]{Griess}. If $\gg=F_4$ or $E_6$, then for $x\in{}E$, 
$C_{\widebar{G}}(x)\cong\Spin_9(\fqobar)$ or 
$\fqobar\times_{C_4}\Spin_{10}(\fqobar)$, respectively (see \cite[Table 
VI]{Griess}).  Since the centralizer of each element in the simply connected 
groups $\Spin_9(\fqobar)$ and $\Spin_{10}(\fqobar)$ is connected 
\cite[Theorem 8.1]{Steinberg-end}, $C_{\widebar{G}}(E)$ is connected in 
these cases.

Now assume $\gg=E_8$. We can assume $G=C_{\widebar{G}}(\psi_q)$, where 
$\psi_q$ is the field automorphism; in particular, $\psi_q(t)=t^q$ for 
$t\in\4T$. Fix $x,y\in{}E$ such that $E=\gen{x,y}$. By \cite[Lemma 
2.16(ii)]{Griess}, $(\TT,\qq)$ is of positive type (has a 4-dimensional 
totally isotropic subspace). Hence $E^\perp=E\times{}V_1\times V_2$, where 
$\dim(V_i)=2$ and $\qq(V_i{\sminus}1)=1$ for $i=1,2$, and $V_1\perp V_2$. 
Thus $(\qq_{E^\perp})^{-1}(1)=\bigcup_{i=1}^2\bigl((V_i{\sminus}1)\times 
E\bigr)$, and by Lemma \ref{l:q:V->F2}(b,c), these are the restrictions to 
$\TT$ of Weyl reflections $w_\alpha$ for $\alpha\in\Sigma$ such that 
$E\le\Ker(\theta_\alpha)$. Also, $C_W(E)\cong W(D_4)\wr C_2$. By 
Proposition \ref{p:CG(T)}, $C_{\4G}(E)^0$ has type $D_4\times D_4$ and 
$|\pi_0(C_{\4G}(E))|=2$. More precisely, 
$C_{\4G}(E)=(\4H_1\times_E\4H_2)\gen{\delta}$, where 
$\4H_i\cong\Spin_8(\fqobar)$ and $Z(\4H_i)=E$ for $i=1,2$, and conjugation 
by $\delta\in N_{\4G}(\4T)$ exchanges $V_1$ and $V_2$ and hence exchanges 
$\4H_1$ and $\4H_2$.


By Proposition \ref{L<->Lhat}, the two connected components in the 
centralizer give rise to two $G$-conjugacy classes of subgroups which are 
$\widebar{G}$-conjugate to $E$, represented by $E$ and $gEg^{-1}$ where 
$g^{-1}\sigma(g)$ lies in the nonidentity component of 
$C_{\widebar{G}}(E)$.  Then $C_G(E)$ contains a subgroup
$\Spin_8^+(q)\times_{C_2^2}\Spin_8^+(q)$ with index $8$ (the extension by 
certain pairs of diagonal automorphisms of the $\Spin_8^+(q)$-factors, as 
well as an automorphism which switches the factors). 
So $E=Z(T)$ for $T\in\syl2{C_G(E)}$, and $E\in\5\calz(G)$.  
Also, $gyg^{-1}\in{}C_G(gEg^{-1})$ if and only if 
$y\in{}C_{\widebar{G}}(E)$ and $\til\tau(y)=y$ where 
$\til\tau=c_{g^{-1}\sigma(g)}\circ\sigma$.  Then $\til\tau$ switches the 
central factors in $C_{\widebar{G}}(E)$, and the group 
$C_{C_{\4G}(E)}(\til\tau)$ splits as a product of $E$ times the group of 
elements which are invariant after lifting $\til\tau$ to the 4-fold cover 
$\Spin_8(\fqobar)\wr{}C_2$.  Since $gEg^{-1}$ intersects trivially with the 
commutator subgroup of $C_G(gEg^{-1})$, $\Omega_1(Z(T))>gEg^{-1}$ for any 
$T\in\syl2{C_G(gEg^{-1})}$ (since $Z(T)\cap[T,T]\ne1$); and thus 
$gEg^{-1}\notin\5\calz(G)$. Thus $\5\calz(G)$ is the $G$-conjugacy class of 
$E$, and has odd order by Lemma \ref{odd-orbit}.

Thus, in all cases, if $\5\calz(G)$ is nonempty, it has odd order and is 
contained in one $G$-conjugacy class. Also, $Z(S)\le C_E(\Aut_S(E))<E$ for 
$E\in\5\calz(G)$, so either $|Z(S)|=2$, or $\gg=E_7$, $Z(S)\cong C_2^2$, 
and the three involutions in $Z(S)$ belong to three different 
$\4G$-conjugacy classes. Hence $\Aut_G(Z(S))=1$. 
\end{proof}

It remains to prove Lemma \ref{all-piv}, which is split into the two Lemmas 
\ref{toral-piv} and \ref{nontoral-piv}. The next proposition will be used 
to show that certain elementary abelian subgroups are not in $\5\calz$.

\begin{Prop} \label{not_pivotal}
Assume Hypotheses \ref{G-hypoth-exc}. Let $E\le\TT$ and 
$x\in\TT{\sminus}E$ be such that the orbit of $x$ under the 
$C_W(E)$-action on $\TT$ has odd order. Then no subgroup of $S$ 
which is $\widebar{G}$-conjugate to $E$ is in $\5\calz$.  More generally, 
if $\widebar{E}\ge{}E$ is also elementary abelian, and is such that $x$ is 
not $C_{\widebar{G}}(E)$-conjugate to any element of $\widebar{E}$, then 
for any $L\nsg{}G$ which contains 
$\{gxg^{-1}\,|\,g\in{}\widebar{G}\}\cap{}G$, no subgroup of $S$ which is 
$\widebar{G}$-conjugate to $\widebar{E}$ is in $\5\calz$.
\end{Prop}

\begin{proof}  In \cite{limz}, an elementary abelian $p$-subgroup $E<G$ 
is called \emph{pivotal} if $O_p(\Aut_G(E))=1$, and $E=\Omega_1(Z(P))$ 
for some $P\in\sylp{C_G(E)}$. In particular, by Lemma 
\ref{l:str.emb.}(a), the subgroups in $\5\calz$ are all pivotal. 
Note that $\TT\le G$ by Hypotheses \ref{G-hypoth-exc}. By 
\cite[Proposition 8.9]{limz}, no subgroup satisfying the above 
conditions can be pivotal, and hence they cannot be in $\5\calz$.  
\end{proof}

In the next two lemmas, we show that in all cases, $E\in\5\calz$ implies 
$\rk(E)=2$ and $\qq_E=0$ if $\gg\ne{}E_7$, with a similar result when 
$\gg=E_7$.  We first handle 
those subgroups which are \emph{toral} (contained in a maximal torus in 
$\widebar{G}$), and then those which are not toral.  By a $\2a^k$-subgroup 
or subgroup of type $\2a^k$ ($\2b^k$-subgroup or subgroup of type $\2b^k$) 
is meant an elementary abelian 2-subgroup of rank $k$ all of whose 
nonidentity elements are in class \2a (class \2b).  

\begin{Lem} \label{toral-piv}
Assume Hypotheses \ref{G-hypoth-exc}. Fix some $E\in\5\calz$ which is 
contained in a maximal torus of $\widebar{G}$.  Then either $\gg\ne{}E_7$, 
$\rk(E)=2$, and $\qq_E=0$; or $\gg=E_7$, $Z=Z(\widebar{G})\cong{}C_2$, and 
$E=Z\times{}E_0$ where $\rk(E_0)=2$ and $\qq_{E_0}=0$.  In all cases, 
$\Aut_{\widebar{G}}(E)\cong\Sigma_3$. 
\end{Lem}

\begin{proof}  Set $Z=O_2(Z(\4G))\le\TT$. Thus $|Z|=2$ if $\gg=E_7$, and 
$|Z|=1$ otherwise. Recall that $\Aut_G(\TT)=\Aut_{\widebar{G}}(\TT)= 
\Aut(\TT,\qq)$ by Lemmas \ref{l:gT} and \ref{l:q:V->F2}(a). 

The following notation will be used to denote isomorphism types of 
quadratic forms over $\F_2$.  Let $\7[n,\pm,]$ denote the isomorphism class 
of a nondegenerate form of rank $n$.  When $n$ is even, $\7[n,+,]$ denotes 
the hyperbolic form (with maximal Witt index), and $\7[n,-,]$ the 
form with nonmaximal Witt index.  Finally, a subscript ``\6{$(k)$}'' 
denotes sum with a $k$-dimensional trivial form.  By 
\cite[Lemma 2.16]{Griess}, $\qq_{\TT}$ has type $\7[2,-,2]$, $\7[6,-,]$, 
$\7[7,,]$, or $\7[8,+,]$ when $\gg=F_4$, $E_6$, $E_7$, or $E_8$, 
respectively.  

Fix $E\le\TT$; we want to determine whether $E$ can be 
$\widebar{G}$-conjugate to an element of $\5\calz$.  Set 
$E_1=E\cap{}E^\perp$ (the orthogonal complement taken with respect to 
$\qq$), and set $E_0=\Ker(\qq_{E_1})$.  Note that $E_1>E_0$ if 
$\gg=E_7$ ($E\ge{}Z$).  

Assume first that $E_0=1$.  If $\gg=F_4$, then $\TT\cap\2b$ is a 
$C_W(E)$-orbit of odd order.  If $\gg=E_r$ and $E_1=1$, then 
$E\times{}E^\perp$, $E^\perp$ is $C_W(E)$-invariant, and hence there is 
$1\ne{}x\in{}E^{\perp}$ whose $C_W(E)$-orbit has odd order.  If $\gg=E_r$ 
and $\rk(E_1)=1$, then $E\cap{}E^\perp=E_1$, there is an odd number of 
involutions in $E^\perp{\sminus}E_1$ of each type (isotropic or not), and 
again there is $1\ne{}x\in{}E^{\perp}$ whose $C_W(E)$-orbit has odd order. 
In all cases, $x$ has the property that $C_W(\gen{E,x})$ has odd index 
in $C_W(E)$.  So by Proposition \ref{not_pivotal}, no subgroup of $G$ 
which is $\widebar{G}$-conjugate to $E$ can be in $\5\calz$.

Thus $E_0\ne1$.  Set $k=\rk(E_0)$. Then 
	\begin{align} 
	\bigl| \pi_0(C_{\4G}(E)) \bigr| 
	&= \bigl| C_W(E) \big/ \Gen{w_\alpha \,\big|\, \alpha\in\Sigma, ~ 
	E\le\Ker(\theta(\alpha)) } \bigr| \tag{Proposition \ref{p:CG(T)}} \\
	&\le \bigl| C_W(\TT) \bigr| \cdot \bigl| C_{\SO(\TT,\qq)}(E) \big/ 
	\Gen{\gamma_v \,\big|\, v\in\2a\cap E^\perp } \bigr| 
	\tag{Lemma \ref{l:q:V->F2}(a,b)} \\
	&\le \bigl| C_W(\TT) \bigr| \cdot 
	\bigl| C_{\SO(\TT,\qq)}(E_0^\perp) \bigr| \cdot 
	\bigl| C_{\SO(E_0^\perp,\qq)}(E) \big/ 
	\Gen{\gamma_v \,\big|\, v\in\2a\cap E^\perp } \bigr| \,.
	\label{e:|pi0|} \end{align}
The first factor is easily described: 
	\beqq \bigl| C_W(\TT) \bigr| = 2^\gee \qquad
	\textup{where} \quad 
	\gee=\begin{cases} 
	1 & \textup{if $-\Id\in W$ (if $\gg=F_4$, $E_7$, $E_8$)} \\
	0 & \textup{if $-\Id\notin W$ (if $\gg=E_6$).} 
	\end{cases} \label{e:|pi0|-1} \eeqq

We next claim that
	\beqq \bigl| C_{\SO(\TT,\qq)}(E_0^\perp) \bigr| \le 
	2^{\binom{k}2}\,,
	\label{e:|pi0|-2} \eeqq
with equality except possibly when $\gg=F_4$. To see this, let $F_1<\TT$ be a 
subspace complementary to $E_0^\perp$. Each $\alpha\in 
C_{\Aut(\TT)}(E_0^\perp)$ has the form $\alpha(x)=x\psi(x)$ for some 
$\psi\in\Hom(F_1,E_0)$, and $\alpha$ is orthogonal if and only if 
$x\perp\psi(x)$ for each $x$. The space of such homomorphisms has dimension 
at most $\binom{k}2$ (corresponding to symmetric $k\times k$ matrices with 
zeros on the diagonal); with dimension equal to $\binom{k}2$ if 
$\dim(F_1)=\dim(E_0)$ (which occurs if $\qq$ is nondegenerate).

Write $(E_0)^\perp=E\times F_2$, where $E^\perp=E_0\times F_2$ and the form 
$\qq_{F_2}$ is nondegenerate. By \cite[Theorem 11.41]{Taylor}, 
$\SO(F_2,\qq_{F_2})$ is generated by transvections unless $\qq_{F_2}$ is of 
type $\7[4,+,]$, in which case the reflections 
generate a subgroup of $\SO(F_2,\qq_{F_2})\cong\Sigma_3\wr C_2$ isomorphic 
to $\Sigma_3\times\Sigma_3$. Also, $F_2$ is generated by nonisotropic 
elements except when $\qq_{F_2}$ is of type $\7[2,+,]$, and when this is 
the case, all automorphisms of $(E_0)^\perp$ which induce the identity on 
$E$ and on $(E_0)^\perp/E_0$ are composites of transvections. (Look at the 
composites $\gamma_{vx}\circ\gamma_v$ for $v\in F_2$ and $x\in{}E_0$.) 
Hence 
	\beq \bigl| C_{\SO(E_0^\perp,\qq)}(E) \big/ 
	\Gen{\gamma_v \,\big|\, v\in\2a\cap E^\perp } \bigr| \le 2^\eta
	\eeq
where $\eta=1$ if $\qq_{E^\perp}$ has type $\7[4,+,k]$, $\eta=k$ if 
$\qq_{E^\perp}$ has type $\7[2,+,k]$, and $\eta=0$ otherwise. 
Together with \eqref{e:|pi0|}, \eqref{e:|pi0|-1}, and \eqref{e:|pi0|-2}, 
this proves that 
	\beqq |\pi_0(C_{\widebar{G}}(E))| \le 2^{\binom{k}2+\gee+\eta} 
	\qquad\qquad \textup{where $\gee\le1$.}
	\label{e:|pi0|-X} \eeqq

Now, $N_{\4G}(E)\le C_{\4G}(E)^0N_{\4G}(\4T)$ by the Frattini argument: 
each maximal torus which contains $E$ lies in $C_{\4G}(E)^0$ and hence is 
$C_{\4G}(E)^0$-conjugate to $\4T$. So each element of $\Aut_{\4G}(E)$ is 
represented by a coset of $\4T$ in $N_{\4G}(\4T)$, and can be chosen to lie 
in $G$ by Lemma \ref{l:gT}. Thus the action described in Proposition 
\ref{L<->Lhat} which determines the automizers $\Aut_G(E^*)$ for $E^*$ 
$\4G$-conjugate to $E$ is the conjugation action of $\Aut_{\4G}(E)$ on the 
set of conjugacy classes in $\pi_0(C_{\4G}(E))$. In particular, this action 
is not transitive, since the identity is fixed.

Set $\ell=\rk(E/E_0)-1$ if $\gg=E_7$ and $\ell=\rk(E/E_0)$ otherwise.  
Every automorphism of $E$ which induces the identity on $E_0Z$ and on 
$E/E_0$ is orthogonal, and hence the restriction of an element of 
$O_2(C_W(E))$.  Thus $|O_2(\Out_{\widebar{G}}(E))|\ge2^{k\ell}$.  If 
$E^*\in\calz$ is $\widebar{G}$-conjugate to $E$, then since $\Aut_G(E^*)$ 
has a strongly $2$-embedded subgroup, 
$2^{k\ell}\le\defect(\Aut_{\widebar{G}}(E))< 
\bigl|\pi_0(C_{\widebar{G}}(E))\bigr|$ by Proposition \ref{L<->Lhat} and 
Lemma \ref{l:d(G)}(a), with strict inequality since the action of 
$N_{\widebar{G}}(E)$ on $\pi_0(C_{\widebar{G}}(E))$ is not transitive.  
Together with \eqref{e:|pi0|-X}, and since $\gee\le1$, this implies that 
$k\ell\le\binom{k}2+\eta\le\binom{k}2+k$. Thus $\ell\le\frac{k+1}2$, and 
$\ell\le\frac{k-1}2$ if $\eta=0$. By definition, $\eta=0$ whenever 
$\rk(E_1/E_0)=1$, which is the case if $\gg=E_7$ or $\ell$ is odd.  Since 
$2k+\ell\le8$, we are thus left with the following possibilities.  
\begin{itemize} 

\item If $(k,\ell)=(3,2)$, then $\gg=E_8$, $E$ has form of type 
$\7[2,+,3]$, so $E^\perp=E_0$ has trivial form, and $\eta=0$. Thus 
$k\ell\nleq\binom{k}2+\eta$, so this cannot occur. 

\item If $(k,\ell)=(3,1)$, then $\gg=E_8$, $E$ has form of type $*+3$, and 
$\rk(E)=\rk(E_1)=4$.  Then 
$\Aut_{\widebar{G}}(E)\cong{}C_2^3\rtimes{}\GL_3(2)$, so 
$\defect(\Aut_{\4G}(E))\ge2^3\cdot28$ by Lemma \ref{l:d(G)}(a,d). Since 
$|\pi_0(C_{\widebar{G}}(E))|\le16$, this case is also impossible. 

\item If $(k,\ell)=(4,0)$, then $\gg=E_8$ and $E=E_0$ is isotropic of rank 
$4$. By Proposition \ref{p:CG(T)} and Lemma \ref{l:q:V->F2}(c), 
$C_{\4G}(E)^0=\4T$. By \cite[Proposition 3.8(ii)]{CG}, $\pi_0(C_{\4G}(E))$ 
is extraspecial of order $2^7$ and $\Aut_{\4G}(E)\cong\GL_4(2)$. (This is 
stated for subgroups of $E_8(\C)$, but the same argument applies in our 
situation.) In particular, $\pi_0(C_{\4G}(E))$ has just $65$ conjugacy 
classes. Since $\defect(\GL_4(2))=112$ by Lemma \ref{l:d(G)}(d), 
Proposition \ref{L<->Lhat} implies that $\Aut_G(E^*)$ cannot have a 
strongly 2-embedded subgroup.

\item If $(k,\ell)=(3,0)$, then $E=Z\times E_0$ where $\dim(E_0)=3$, and 
$\Aut_{\widebar{G}}(E)\cong\GL_3(2)$. If $\gg=E_6$ or $E_7$, then 
$E^\perp=E$, and $|\pi_0(C_{\4G}(E))|\le16$ by \eqref{e:|pi0|-X}. 

If $\gg=E_8$, then $(E^\perp,\qq_{E^\perp})$ has type $\7[2,+,3]$. 
By the arguments used to prove \eqref{e:|pi0|-X}, 
	\[ \qquad |C_W(E)| = | C_W(\TT) | \cdot 
	| C_{\SO(\TT,\qq)}(E_0^\perp) | \cdot 
	| C_{\SO(E_0^\perp,\qq)}(E) | =
	2\cdot2^3\cdot2^7=2^{11}. \]
Also, $E^\perp$ contains exactly $8$ nonisotropic elements, they are 
pairwise orthogonal, and hence determine $8$ pairwise commuting 
transvections on $\TT$. These extend to $8$ Weyl reflections which are 
pairwise commuting since no two can generate a dihedral subgroup of order 
$8$ (this would imply two roots of different lengths). Hence by Proposition 
\ref{p:CG(T)}, $C_{\4G}(E)^0$ has type $(A_1)^8$ and 
$|\pi_0(C_{\4G}(E))|=2^{11}/2^8=2^3$. Since $\defect(\GL_3(2))=28$ by Lemma 
\ref{l:d(G)}(c), this case cannot occur.

\item If $(k,\ell)=(2,0)$, then $E=Z\times E_0$ where $\dim(E_0)=2$. Then 
$E$ is as described in the statement of the lemma. 
\qedhere

\end{itemize}
\end{proof}

It remains to handle the nontoral elementary abelian subgroups.

\begin{Lem} \label{nontoral-piv}
Assume Hypotheses \ref{G-hypoth-exc}.  Let $E\le{}G$ 
be an elementary abelian 2-group which is not contained in a maximal torus 
of $\4G$.  Then $E\notin\5\calz$. 
\end{Lem}

\begin{proof} To simplify notation, we write $\K=\fqobar$. 
Set $Z=O_2(Z(\4G))\le\TT$. Thus $|Z|=2$ if $\gg=E_7$, and 
$|Z|=1$ otherwise. The maximal nontoral subgroups of $\4G$ are 
described in all cases by Griess \cite{Griess}.
\begin{enumA}[parsep=3pt] 

\item \label{ntor:F4E6}
If $\gg=F_4$ or $E_6$, then by \cite[Theorems 7.3 \& 
8.2]{Griess}, $\4G$ contains a unique conjugacy class of 
maximal nontoral elementary abelian 2-subgroups, represented by $W_5$ 
of rank five. There is a subgroup $W_2\le{}W_5$ of 
rank two such that $W_5\cap\2a=W_5{\sminus}W_2$. Also, 
$\Aut_{\4G}(E_5)=\Aut(E_5,\qq_{E_5})$: the group of all 
automorphisms of $W_5$ which normalize $W_2$. A subgroup 
$E\le{}W_5$ is nontoral if and only if it contains a $\2a^3$-subgroup.

When $\gg=F_4$, we can assume $W_5=\TT\gen{\theta}$, where $\theta\in 
N_{\4G}(\4T)$ inverts the torus.

\item \label{ntor:E7}
If $\gg=E_7$, then by \cite[Theorem 9.8(i)]{Griess}, $\4G$ 
contains a unique maximal nontoral elementary abelian 2-subgroup $W_6$, of 
rank six. For any choice of $E_6(\K)<G$, $W_5<E_6(\K)$ (as just 
described) has rank $5$, is nontoral since it contains a $\2a^3$-subgroup, 
and so we can take $W_6=Z\times W_5$. 

Each coset of $Z$ of involutions in $\4G{\sminus}Z$ contains one element of 
each class $\2a$ and $\2b$. Together with the above description of $E_5$, 
this shows that all $\2a^2$-subgroups of $W_6$ are contained in 
$W_5$. Hence for each nontoral subgroup $E\le W_6$ which contains $Z$, 
$E\cap W_5$ is the subgroup generated by $\2a^2$-subgroups of $E$, thus 
is normalized by $\Aut_{\4G}(E)$, and so 
	\begin{align*} 
	\qquad \Aut_{\4G}(E) &\cong \Aut_{\4G}(E\cap W_5) = 
	\Aut(E\cap{}W_5,\qq_{E\cap{}W_5}) \cong \Aut(E,\qq_E) \\
	\Aut_{\4G}(W_6) &\cong \Aut(W_6,\qq_{W_6})\cong 
	C_2^6\rtimes(\Sigma_3\times\GL_3(2)) 
	\end{align*}

For $Z\le E\le{}W_6$, the subgroup $E$ is nontoral exactly when it contains 
a $\2a^3$-subgroup. This is immediate from the analogous statement in 
\eqref{ntor:F4E6} for $E_6(\K)$.

\item \label{ntor:E8}
If $\gg=E_8$, then by \cite[Theorem 2.17]{Griess}, $\4G$ 
contains two maximal elementary abelian subgroups $W_8$ and $W_9$, neither 
of which is toral \cite[Theorem 9.2]{Griess}.  
An elementary abelian $2$-subgroup $E\le\4G$ is nontoral if and only if 
$\qq_E$ is not quadratic or $E$ has type $\2b^5$ \cite[Theorem 9.2]{Griess}.

We refer to \cite[Theorem 2.17]{Griess} for descriptions of $W_8$ and 
$W_9$. There are subgroups 
$F_0\le{}F_1,F_2\le{}W_8$ such that $\rk(F_0)=2$, $\rk(F_1)=\rk(F_2)=5$, 
$F_1\cap{}F_2=F_0$, and 
$W_8\cap\2a=(F_1{\sminus}F_0)\cup(F_2{\sminus}F_0)$.  Also, 
$\Aut_{\4G}(W_8)$ is the group of those automorphisms of $W_8$ 
which leave $F_0$ invariant, and either leave $F_1$ and $F_2$ invariant or 
exchange them.  

We can assume that $W_9=\TT\gen{\theta}$, where $\theta\in{}N_{\4G}(\4T)$ 
inverts $\4T$.  Also, $W_9{\sminus}\TT\subseteq\2b$. Hence 
$\TT=\gen{W_9\cap\2a}$ is $\Aut_{\4G}(W_9)$-invariant.  Each automorphism 
of $W_9$ which is the identity on $\TT$ is induced by conjugation by some 
element of order 4 in $\4T$, and thus $\Aut_{\4G}(W_9)$ is the group of all 
automorphisms whose restriction to $\TT$ lies in $\Aut_{\4G}(\TT)$.  

\end{enumA}

We next list other properties of elementary abelian subgroups of $\4G$, and 
of their centralizers and normalizers, which will be needed in the proof.
\begin{enumA}[resume,parsep=3pt]

\item \label{ntor:dim(C)} \label{e:dimC(E)} 
\emph{If $\gg=E_8$, $E\le\4G$, $E\cong C_2^r$, and $|E\cap\2a|=m$, then 
$\dim(C_{\4G}(E))=2^{8-r}+2^{5-r}m-8$.}

This follows from character computations: if $\ggg$ denotes the Lie algebra 
of $\4G=E_8(\K)$, then 
$\dim(C_{\4G}(E))=\dim(C_\ggg(E))=|E|^{-1}\sum_{x\in{}E}\chi_\ggg(x)$. By 
\cite[Table VI]{Griess}, $\chi_\ggg(1)=\dim(\4G)=248$, and 
$\chi_\ggg(x)=24$ or $-8$ when $x\in\2a$ or $\2b$, respectively.

\item \label{ntor:Et<E<W9}
\emph{If $\gg=E_8$, $E\le\4G$ is an elementary abelian $2$-group, and 
$E_t<E$ has index $2$ and is such that $E{\sminus}E_t\subseteq\2b$, then 
there is $g\in\4G$ such that $\9gE\le W_9=\TT\gen\theta$ and 
$\9gE_t\le\TT$.} 

It suffices to prove this when $E$ is maximal among such such pairs 
$E_t<E$. We can assume that $E$ is contained in $W_8$ or $W_9$. 

If $E\le W_8$, then in the notation of \eqref{ntor:E8}, $F_0\le E$ (since 
$E$ is maximal), and either $\rk(E\cap{}F_i)=3$ for $i=1,2$ and $\rk(E)=6$, 
or $\rk(E\cap{}F_i)=4$ for $i=1,2$ and $\rk(E)=7$. These imply that 
$|E\cap\2a|=8$ or $24$, respectively, and hence by \eqref{ntor:dim(C)} that 
$\dim(C_{\4G}(E_t))=8$ ($C_{\4G}(E_t)^0=\4T$) and $\dim(C_{\4G}(E))=0$. 
Hence in either case, if $g\in\4G$ is such that $\9gE_t\le\TT$, then 
$\9gE{\sminus}\9gE_t\subseteq\theta\4T$, and there is $t\in\4T$ such that 
$\9{tg}E\le\TT\gen\theta=W_9$. 

If $E\le W_9$, set $E_2=\gen{E\cap\2a}$. Then $E_2\le E\cap\TT$ and $E_2\le 
E_t$, so there is nothing to prove unless $\rk(E/E_2)\ge2$. In this case, 
from the maximality of $E$, we see that $E_t=E_a\times E_b$, where 
$E_a\cong C_2^2$ has type \textbf{2ABB}, $E_b$ is a $\2b^3$-group, and 
$E_a\perp E_b$ with respect to the form $\qq$. Thus $\rk(E)=6$, 
$|E\cap\2a|=8$, and the result follows by the same argument as in the last 
paragraph.

\item \label{ntor:max-ntor}
\emph{If $\gg=E_8$, and $E\le\4G$ is a nontoral elementary abelian 
$2$-group, then either $E$ contains a $\2a^3$-subgroup, or $E$ is 
$\4G$-conjugate to a subgroup of $W_9$.} 

Assume $E\le W_8$ is nontoral and contains no $\2a^3$-subgroup. We use the 
notation $F_0<F_1,F_2<W_8$ of \eqref{ntor:E8}. Set $E_i=E\cap F_i$ for 
$i=0,1,2$. Then $\qq_{E_1E_2}$ is quadratic: it is the orthogonal direct 
sum of $\qq_{E_0}$, $\qq_{E_1/E_0}$, and $\qq_{E_2/E_0}$, each of which is 
quadratic since $\rk(E_i/E_0)\le2$ for $i=1,2$ ($E$ has no 
$\2a^3$-subgroup). Hence $E>E_1E_2\ge\gen{E\cap\2a}$ since $E$ is nontoral, 
so $E$ is conjugate to a subgroup of $W_9$ by \eqref{ntor:Et<E<W9}.

\item \label{e:cc} 
\noindent \emph{Let $E\le\4G$ be an elementary abelian $2$-subgroup, and let 
$E_t\le E$ be maximal among toral subgroups of $E$. Assume that 
$E_t\cap E_t^\perp\cap\2b=\emptyset$, and that either 
$\rk(\4T)-\rk(E_t)\ge2$ or $E_t\cap E_t^\perp=1$. Then $E\notin\5\calz$.}

To see this, choose $F\ge F_t$ which is $\4G$-conjugate to $E\ge E_t$ and 
such that $F_t=F\cap\TT$. By maximality, no element of $F{\sminus}F_t$ is 
$C_{\4G}(F_t)$-conjugate to an element of $\4T$. If $F_t\cap F_t^\perp=1$, 
then some $C_W(F_t)$-orbit in $F_t^\perp{\sminus}1$ has odd order. 
Otherwise, since $\qq$ is linear on $F_t\cap F_t^\perp$, we have $F_t\cap 
F_t^\perp=\gen{y}$ for some $y\in\2a$, in which case 
$|\qq_{F_t^\perp}^{-1}(0)|=|F_t^\perp|/2$ is 
even since $\rk(F_t^\perp)\ge\rk(\4T)-\rk(F_t)\ge2$. So again, some 
$C_W(F_t)$-orbit in $F_t^\perp{\sminus}1$ has odd order in this case. Point 
\eqref{e:cc} now follows from Proposition \ref{not_pivotal}.

\item \label{e:dd} 
\emph{Assume $\gg=E_8$. Let $1\ne{}E_0\le E\le\4G$ be elementary abelian 
$2$-subgroups, where $\rk(E)=3$, and $E\cap\2a=E_0{\sminus}1$. Then}
	\[ C_{\4G}(E) \cong \begin{cases}  
	E\times F_4(\K) & \textup{if $\rk(E_0)=3$} \\
	E\times \PSp_8(\K) & \textup{if $\rk(E_0)=2$} \\
	E\times \PSO_8(\K) & \textup{if $\rk(E_0)=1$.} \\
	\end{cases} \]

To see this, fix $1\ne{}y\in{}E_0$, and identify $C_{\4G}(y)\cong 
\SL_2(\K)\times_{C_2}E_7(\K)$.  For each $x\in{}E{\sminus}\gen{y}$, since 
$x$ and $xy$ are $\4G$-conjugate, $x\ne(1,b)$ for $b\in E_7(\K)$. Thus 
$x=(a,b)$ for some $a\in\SL_2(\K)$ and $b\in E_7(\K)$ both of order 4, and 
(in the notation of \cite[Table VI]{Griess}) $b$ is in class \textbf{4A} or 
\textbf{4H} since $b^2\in Z(E_7(\K))$. By \eqref{e:dimC(E)} and \cite[Table 
VI]{Griess}, 
	\[ \dim(C_{\4G}(E))= \begin{cases} 
	80 = \dim(C_{E_7(\K)}(\textbf{4H}))+1 & \textup{if $E$ has type 
	\textbf{2AAA}} \\
	64 = \dim(C_{E_7(\K)}(\textbf{4A}))+1 & \textup{if $E$ has type 
	\textbf{2ABB},} 
	\end{cases} \]
and thus $x\in\2a$ if $b\in\textbf{4H}$ and $x\in\2b$ if $b\in\textbf{4A}$. 
Thus if $E=\gen{y,x_1,x_2}$, and $x_i=(a_i,b_i)$, then 
$\gen{a_1,a_2}\le{}\SL_2(\K)$ and $\gen{b_1,b_2}\le{}E_7(\K)$ are both 
quaternion of order 8.  Point \eqref{e:dd} now follows using the 
description in \cite[Proposition 9.5(i)]{Griess} of centralizers of certain 
quaternion subgroups of $E_7(\K)$. When combined with the 
description in \cite[Table VI]{Griess} of $C_{E_7(\K)}(\textbf{4A})$, this 
also shows that 
	\beqq \qquad F\cong C_2^2 \textup{ of type \textbf{2ABB} $\implies$ 
	$C_{\4G}(F)^0$ is of type $A_7T^1$} \label{e:ee} \eeqq
(i.e., $C_{\4G}(F)^0\cong(\SL_8(\K)\times\K^\times)/Z$,
for some finite subgroup $Z\le Z(\SL_8(\K))\times\K^\times$). 

\item \label{e:ff} 
\emph{If $U<\4G$ is a $\2a^3$-subgroup, then $C_{\4G}(U)=U\times H$, where 
$H$ is as follows:}
	\[ \renewcommand{\arraystretch}{1.5}
	\begin{array}{|c||c|c|c|c|} \hline
	\gg & F_4 & E_6 & E_7 & E_8 \\ \hline
	H & SO_3(\K) & \SL_3(\K) & \Sp_6(\K) 
	& F_4(\K) \\ \hline
	\end{array} \]

When $\gg=E_8$, this is a special case of \eqref{e:dd}.  For $x\in\2a\cap 
F_4(\K)$, 
$C_{E_8(\K)}(x)\cong{}\SL_2(\K)\times_{C_2}E_7(\K)$ by 
\cite[2.14]{Griess}.  Since 
$C_{F_4(\K)}(x)\cong{}\SL_2(\K)\times_{C_2}\Sp_6(\K)$, this 
shows that $C_{E_7(\K)}(U)\cong{}U\times{}\Sp_6(\K)$.  

Similarly, $C_{E_8(\K)}(y)\cong{}\SL_3(\K)\times_{C_3}E_6(\K)$ by 
\cite[2.14]{Griess} again (where $y$ is in class \textbf{3B} in his 
notation).  There is only one class of element of order three in $F_4(\K)$ 
whose centralizer contains a central factor $\SL_3(\K)$ --- 
$C_{F_4(\K)}(y)\cong\SL_3(\K)\times_{C_3}\SL_3(\K)$ for $y$ of type \3c 
in $F_4(\K)$ --- and thus $C_{E_6(\K)}(U)\cong{}U\times\SL_3(\K)$.  

If $\gg=F_4$, then by \cite[2.14]{Griess}, for $y\in\3c$, 
$C_{\4G}(y)\cong\SL_3(\K)\times_{C_3}\SL_3(\K)$. Also, the involutions in 
one factor must all lie in the class $\2a$ and those in the other in $\2b$. 
This, together with Proposition \ref{p:CG(T)}, shows that for $U_2<U$ of 
rank $2$, $C_{\4G}(U_2)\cong(T^2\times_{C_3}\SL_3(\K))\gen{\theta}$, where 
$\theta$ inverts a maximal torus. Thus $C_{\4G}(U)=U\times 
C_{\SL_3(\K)}(\theta)$, where by \cite[Proposition 2.18]{Griess}, 
$C_{\SL_3(\K)}(\theta)\cong\SO_3(\K)$. This finishes the proof of 
\eqref{e:ff}.

\end{enumA}

For the rest of the proof, we fix a nontoral elementary abelian 
$2$-subgroup $E<\4G$. We must show that $E\notin\5\calz$. In almost all 
cases, we do this either by showing that the hypotheses of \eqref{e:cc} 
hold, or by showing that $\defect(\Aut_{\4G}(E))>|\pi_0(C_{\4G}(E))|$ 
(where $\defect(-)$ is as in Lemma \ref{l:d(G)}), in which case $\Aut_G(E)$ 
has no strongly $2$-embedded subgroup by Proposition \ref{L<->Lhat}, and 
hence $E\notin\5\calz$.

By \eqref{ntor:F4E6}, \eqref{ntor:E7}, and \eqref{ntor:max-ntor}, either 
$E$ contains a $\2a^3$-subgroup of rank three, or $\gg=E_8$ and $E$ is 
$\4G$-conjugate to a subgroup of $W_9$.  These two cases will be handled 
separately.

\smallskip

\noindent\textbf{Case 1: }  Assume first that $E$ contains a 
$\2a^3$-subgroup $U\le{}E$. From the lists in 
(\ref{ntor:F4E6},\ref{ntor:E7},\ref{ntor:E8}) of maximal nontoral 
subgroups, there are the following possibilities.

\begin{description}[leftmargin=5mm,itemsep=6pt,parsep=3pt] 

\item[\underbar{$\gg=F_4$, $E_6$, or $E_7$}] By 
(\ref{ntor:F4E6},\ref{ntor:E7}), we can write $E=U\times E_0\times Z$, 
where $E_0$ is a $\2b^k$ subgroup (some $k\le2$) and 
$UE_0{\sminus}E_0\subseteq\2a$ (and where $Z=1$ unless $\gg=E_7$). If 
$k=0$, then $E\notin\5\calz$ by \eqref{e:cc}, so assume $k\ge1$. By 
\eqref{e:ff}, and since each elementary abelian 2-subgroup of $\SL_3(\K)$ 
and of $\Sp_6(\K)$ has connected centralizer, $\pi_0(C_{\4G}(E))\cong U$ if 
$\gg=E_6$ or $E_7$. If $\gg=F_4$, then by \eqref{e:ff} again, and since the 
centralizer in $\SO_3(\K)\cong\PSL_2(\K)$ of any $C_2^k$ has $2^k$ 
components, $|\pi_0(C_{\4G}(E))|=2^{3+k}$. 


By (\ref{ntor:F4E6},\ref{ntor:E7}) again, 
$\Aut_{\4G}(E)$ is the group of all automorphisms which 
normalize $E_0$ and $UE_0$ and fix $Z$. Hence 
	\[ |O_2(\Aut_{\4G}(E))|=2^{3k} \qquad\textup{and}\qquad
	\Aut_{\4G}(E)/O_2(\Aut_{\4G}(E))\cong\GL_3(2)\times\GL_k(2)\,. \]
So $\defect(\Aut_{\4G}(E))\ge2^{3k+3}>|\pi_0(C_{\4G}(E))|$ by Lemma 
\ref{l:d(G)}, and $E\notin\5\calz$. 


\item[\underbar{$\gg=E_8$}] By \eqref{e:ff}, $C_{\4G}(U)=U\times H$ where 
$H\cong F_4(\K)$. Set $E_2=E\cap H$, and let 
$E_0=\gen{E_2\cap\2b}$. Set $k=\rk(E_0)$ and $\ell=\rk(E_2/E_0)$.

If $k=0$, then $E_2$ has type $\2a^\ell$, and $E{\sminus}(U\cup 
E_2)\subseteq\2b$. So each maximal toral subgroup $E_t<E$ has the form 
$E_t=U_1\times U_2$, where $\rk(U_1)=2$, $\rk(U_2)\le2$, and 
$E_t\cap\2a=(U_1\cup U_2){\sminus}1$. The hypotheses of \eqref{e:cc} thus 
hold, and so $E^*\notin\5\calz$.

Thus $k=1,2$. If $\ell\le2$, then $E_2$ is toral, and 
$|\pi_0(C_{\4G}(E))|=8\cdot|\pi_0(C_{H}(E_2))|\le2^{3+k}$ by formula 
\eqref{e:|pi0|-X} in the proof of Lemma \ref{toral-piv}. (Note that 
$\gee=1$ and $\eta=0$ in the notation of that formula.) If $\ell=3$, then 
$|\pi_0(C_{\4G}(E))|=2^{6+k}$ by the argument just given for $F_4(\K)$. 
Also, $\Aut_{\4G}(E)$ contains all automorphisms of $E$ which normalize 
$E_0$, and either normalize $UE_0$ and $E_2$ or (if $\ell=3$) exchange 
them: since in the notation of \eqref{ntor:E8}, each such automorphism 
extends to an automorphism of $W_8$ which normalizes $F_1$ and $F_2$. So 
$|O_2(\Aut_{\4G}(E))|\ge2^{k(3+\ell)}$, and 
$\Aut_{\4G}(E)/O_2(\Aut_{\4G}(E))\cong 
\GL_3(2)\times\GL_k(2)\times\GL_\ell(2)$ or (if $\ell=3$) $(\GL_3(2)\wr 
C_2)\times\GL_k(2)$. In all cases, 
$\defect(\Aut_{\4G}(E))\ge2^{3k+\ell{}k+3}>|\pi_0(C_{\4G}(E))|$, so 
$E\notin\calz$.
 

\end{description}

\smallskip

\noindent\textbf{Case 2: }  Now assume that $\gg=E_8$, and that $E$ is 
$\4G$-conjugate to a subgroup of $W_9$. To simplify the argument, we assume 
that $E\le W_9$, and then prove that no subgroup $E^*\in\5\calz$ can be 
$\4G$-conjugate to $E$. Recall that $W_9=\TT\gen{\theta}$, where $\theta\in 
N_{\4G}(\4T)$ inverts the torus and $\theta\TT\subseteq\2b$. 

If $E\cap\2a=\emptyset$, then $\rk(E)=5$. In this case, 
$\Aut_{\4G}(E)\cong\GL_5(2)$ and $|C_{\4G}(E)|=2^{15}$ \cite[Proposition 
3.8]{CG}. (Cohen and Griess work in $E_8(\C)$, but their argument also 
holds in our situation.) Since $\delta(\GL_5(2))>2^{15}$ by Lemma 
\ref{l:d(G)}(d), no $E^*\in\5\calz$ can be $\4G$-conjugate to $E$. 

Now assume $E$ has $\2a$-elements, and set $E_2=\gen{E\cap\2a}$. Then  
$E_2\le\TT$ (hence $\qq_{E_2}$ is quadratic) by the above remarks.  Set 
$E_1=E_2^\perp\cap{}E_2$ and $E_0=\Ker(\qq_{E_1})$. If $E_0=1$ and 
$\rk(E_2)\ne7$, then by \eqref{e:cc}, no subgroup of $S$ which 
is $\4G$-conjugate to $E$ lies in $\5\calz$. 

It remains to consider the subgroups $E$ for which $E_0\ne1$ or 
$\rk(E_2)=7$. Information about $|O_2(\Aut_{\4G}(E))|$ and 
$|\pi_0(C_{\4G}(E))|$ for such $E$ is summarized in Table \ref{tb:E8case2}. 
By the ``type of $\qq_E$'' is meant the type of quadratic 
form, in the notation used in the proof of Lemma \ref{toral-piv}.

\begin{table}[ht]
\[ \renewcommand{\arraystretch}{1.5}
\renewcommand{\arraycolsep}{1mm}
\newcommand{\Sm}[1]{\textup{\begin{Small}$#1$\end{Small}}}
\newcommand{\sm}[1]{\textup{\begin{small}$#1$\end{small}}}
\begin{array}{|@{~}l||c|c|c|c|c|c|c|c|c|c|c|c|c|} \hline
\textup{Case nr.} &1&2&3&4&5&6&7&8&9&10&11&12&13\\\hline\hline
\rk(E/E_2) &1&1&1&1&1&1&1&1&1&1&2&2&2\\ \hline
\rk(E_2/E_0) &7&6^+&5&4^+&4^+&4^-&3&3&2^-&2^-&1&1&1\\ \hline
\rk(E_0) &0&1&1&2&1&1&2&1&2&1&3&2&1\\ \hline
\textup{type of } \qq_{E_2} & \,\8\7[7,,]\, & \8{\7[6,+,1]} & \8\7[5,,1] & 
\8\7[4,+,2] & \8\7[4,+,1] & \8\7[4,-,1] & \8\7[3,,2] & \8\7[3,,1] & 
\8\7[2,-,2] & \8\7[2,-,1] & \8\7[1,,3] & \8\7[1,,2] & \8\7[1,,1] \\\hline
\sm{|\pi_0(C_{\4G}(E^*))|\le} & 2^9 & 2^9 & 2^{10} & 2^{10} & 2^8 & 2^7 & 
2^6 & 2^5 & 2^5 & 2^4 & 2^{12} & 2^8 & 2^5 \\ \hline
\sm{|O_2(\Aut_{\4G}(E^*))|} & 2^7 & 2^{13} & 2^{11} & 2^{14} & 2^9 & 
2^9 & 2^{11} & 2^7 & 2^8 & 2^5 & 2^{11} & 2^8 & 2^5 \\ \hline
{\defect(\Aut_{\4G}(E^*)){\ge}} & 2^{13} & 2^{17} & 3{\cdot}2^{13} 
& 2^{16} & 2^{10} & 2^{10} & 2^{12} & 2^7 & 2^9 & 2^5 & 2^{14} 
& 2^9 & 2^5 \\ \hline
\end{array}
\]
\caption{} \label{tb:E8case2}
\end{table}

We first check that the table includes all cases. If $\rk(E/E_2)=1$, 
then $E_2=E\cap\TT$, and the table lists all types which the form 
$\qq_{E_2}$ can have. Note that since $E_2$ is generated by 
nonisotropic vectors, $\qq_{E_2}$ cannot have type $\7[2,+,k]$. If 
$\rk(E/E_2)=2$, then $\qq_{E_2}$ is linear, and must be one of the 
three types listed. Since $\qq_{E\cap\TT}$ is quadratic and $\qq_E$ is 
not, $E_2$ has index at most $2$ in $E\cap\TT$.

We claim that
	\beqq \parbox{\short}{$E,F<W_9$, $\alpha\in\Iso(E,F)$ such that 
	$\alpha(E\cap\TT)=F\cap\TT$ and $\alpha(E\cap\2a)=F\cap\2a$ 
	~$\implies$~ $\alpha=c_{tg}$ for some $t\in\4T$ and some 
	$g\in N_G(\4T)=G\cap N_{\4G}(\4T)$.}
	\label{e:c_tg} \eeqq
By \eqref{ntor:E8} and Witt's theorem (see \cite[Theorem 7.4]{Taylor}), 
there is $g\in N_{\4G}(\4T)$ such that $\alpha|_{E\cap\TT}=c_g$, and we can 
assume $g\in{}G$ by Lemma \ref{l:gT}. Then $\9gE{\sminus}\9g(E\cap\TT)\le 
\theta\4T$ since $\theta\4T\in Z(N_{\4G}(\4T))/\4T$, so $\alpha=c_{tg}$ for 
some $t\in\4T$. This proves \eqref{e:c_tg}. In particular, any two 
subgroups of $W_9$ which have the same data as listed in the first three 
rows of Table \ref{tb:E8case2} are $\4G$-conjugate.

By \eqref{e:c_tg}, together with \eqref{ntor:Et<E<W9} when $\rk(E/E_2)=2$, 
we have $\Aut_{\4G}(E)=\Aut(E,\qq_E)$ in all cases. Thus $\Aut_{\4G}(E)$ is 
the group of all automorphisms of $E$ which normalize $E_0$ and $E_2$ and 
preserve the induced quadratic form on $E_2/E_0$. This gives the value for 
$|O_2(\Aut_{\4G}(E))|$ in the table, and the lower bounds for 
$\defect(\Aut_{\4G}(E))$ then follow from Lemma \ref{l:d(G)}. 

In cases 1--6, the upper bounds for $|\pi_0(C_{\4G}(E))|$ given in the 
table are proven in \cite[p. 78--79]{limz}. In all cases, 
$|\pi_0(C_{\4G}(E_2))|$ is first computed, using Proposition \ref{p:CG(T)} 
or the upper bound given in formula \eqref{e:|pi0|-X} in the proof of the 
last lemma, and then \cite[Proposition 8.8]{limz} is used to compute an 
upper bound for $|\pi_0(C_{\4G}(E))|\big/|\pi_0(C_{\4G}(E_2))|$. There is in 
fact an error in the table on \cite[p. 79]{limz} (the group $C_G(E_0)^0_s$ 
in the third-to-last column should be $\SL_2\times\SL_2$ up to finite 
cover), but correcting this gives in fact a better estimate 
$|\pi_0(C_{\4G}(E))|\le2^9$. 

Case nr. 11 can be handled in a similar way. Set $E_t=E\cap\TT<E$, so that 
$|E/E_t|=2=|E_t/E_2|$. The form $\qq_{E_t}$ has type $\7[2,+,3]$, while 
$E_t^\perp$ has type $\2b^3$. Hence $|\pi_0(C_{\4G}(E_t))|\le2^4$ by 
\eqref{e:|pi0|-X}. By \cite[Proposition 8.8]{limz}, 
$|\pi_0(C_{\4G}(E))|\le2^{4+r}$, where $r=\dim(\4T)=8$.

To handle the remaining cases, fix rank $2$ subgroups $F_1,F_2\le\TT<\4G$ 
with involutions of type \textbf{AAA} and \textbf{ABB}, respectively, and 
consider the information in Table \ref{tb:E8case2b}.
	\begin{table}[ht]
	\[ \renewcommand{\arraystretch}{1.5}
	\begin{array}{|l||c|c|c|c|c|c|} \hline
	 & & 
	\multicolumn{5}{c|}{\dim(C_{\4G}(F_i)\gen{\theta,g}) \textup{ for 
	$g$ as follows:}} \\\cline{3-7}
	\halfup{i} & \halfup{C_{\4G}(F_i\gen{\theta})} & -I_4\oplus I_4
	& -I_2\oplus I_6 & \textup{order $4$} & \2a & \2b \\\hline\hline
	1 & F_1\gen{\theta}\times\PSp_8(\K) & 20 & 24 & 16 & 16 & 20 \\\hline
	2 & F_2\gen{\theta}\times\PSO_8(\K) & 12 & 16 & 16 & 16 & 12 \\\hline
	\end{array} \]
	\caption{} \label{tb:E8case2b}
	\end{table}
The description of $C_{\4G}(F_i\gen\theta)$ follows from \eqref{e:dd}. The 
third through fifth columns give dimensions of centralizers of 
$F_i\gen\theta\gen{g}$, for $g$ as described after lifting to $\Sp_8(\K)$ 
or $\SO_8(\K)$. (Here, $I_m$ denotes the $m\times m$ identity matrix.) 
The last two columns do this for $g\in\2a$ or $\2b$, 
respectively, when $g\in\TT$ is orthogonal to $F_i$ with respect to the 
form $\qq$, and the dimensions follow from 
\eqref{e:dimC(E)}. Thus elements of class \2b\ lift to involutions in 
$\Sp_8(\K)$ or $\SO_8(\K)$ with $4$-dimensional $(-1)$-eigenspace, while 
for $i=1$ at least, elements of class \2a\ lift to elements of order 
$4$ in $\Sp_8(\K)$. 

Thus in all of the cases nr. 7--13 in Table \ref{tb:E8case2}, we can 
identify $E=F_i\gen{\theta}\times F^*$, where $i=1$ in nr. 7--10 or $i=2$ 
in nr. 11--13, and where $F^*$ lifts to an abelian subgroup of $\Sp_8(\K)$ 
or $\SO_8(\K)$ (elementary abelian except for nr. 7--8). 
This information, together with the following:
	\begin{multline*} 
	\textup{$H$ a group, $Z\le Z(H)$, $|Z|=p$, $Z\le P\le H$ a 
	$p$-subgroup} \\ \textup{$\implies$ 
	$\bigl|C_{H/Z}(P)\big/C_H(P)/Z\bigr| \le |P/\Fr(P)|$} 
	\end{multline*}
(applied with $H=\Sp_8(\K)$ or $\SO_8(\K)$), imply the remaining
bounds in the last line of Table \ref{tb:E8case2}.

In all but the last case in Table \ref{tb:E8case2}, 
$\defect(\Aut_{\4G}(E))>|\pi_0(C_{\4G}(E))|$, so no $E^*\in\5\calz$ is 
$\4G$-conjugate to $E$ by Proposition \ref{L<->Lhat}. In the last case, by 
the same proposition, $E$ can be $\4G$-conjugate to some $E^*\in\5\calz$ 
only if $\Aut_{\4G}(E)$ acts transitively on $\pi_0(C_{\4G}(E))\cong C_2^5$ 
with point stabilizers isomorphic to $\Sigma_3$. By \eqref{e:c_tg}, each 
class in $O_2(\Aut_{\4G}(E))$ is represented by some element $tg\in 
N_{\4G}(E)$, where $g\in{}N_{G}(\4T)$ and $t\in\4T$. In particular, 
$(tg)\sigma(tg)^{-1}=t\sigma(t)^{-1}\in\4T$. So each class in the 
$O_2(\Aut_{\4G}(E))$-orbit of $1\in\pi_0(C_{\4G}(E))$ 
has nonempty intersection with $\4T$. But by 
\eqref{e:ee}, $C_{\4G}(F_2)^0\cap\theta\4T=\emptyset$, so 
$\theta{}C_{\4G}(E)^0\cap\4T=\emptyset$. Thus the action is not transitive 
on $\pi_0(C_{\4G}(E))$, and hence $E^*\notin\5\calz$. 
\end{proof}


\end{document}